\documentclass{amsbook}
\usepackage{amssymb}
\usepackage{mathrsfs}
\usepackage{stmaryrd}
\usepackage{imakeidx}
\usepackage{hyperref}
\usepackage{enumitem}

\usepackage{tikz}
\usetikzlibrary{cd}

\makeindex

\newcommand{\Z}{\mathbb{Z}}
\newcommand{\C}{\mathbb{C}}
\newcommand{\F}{\mathbb{F}}
\newcommand{\Q}{\mathbb{Q}}
\newcommand{\Ql}{\mathbb{Q}_\ell}
\newcommand{\Fp}{\mathbb{F}_p}
\newcommand{\bk}{\Bbbk}

\newcommand{\sA}{\mathsf{A}}
\newcommand{\lfrown}{\mathbin{\acute{\frown}}}
\newcommand{\rfrown}{\mathbin{\grave{\frown}}}
\newcommand{\Mod}{\mathrm{Mod}}
\DeclareMathOperator{\Sym}{Sym}

\newcommand{\Kb}{K^{\mathrm{b}}}
\newcommand{\mix}{{\mathrm{mix}}}
\newcommand{\mon}{{\mathrm{mon}}}
\newcommand{\Db}{D^{\mathrm{b}}}
\newcommand{\Dmix}{D^\mix}
\newcommand{\cone}{\mathrm{cone}}
\newcommand{\coH}{\mathsf{H}}
\newcommand{\Kar}{{\mathrm{Kar}}}

\newcommand{\fh}{\mathfrak{h}}
\newcommand{\uu}[1]{\underline{#1}}
\newcommand{\uv}{{\underline{v}}}
\newcommand{\uw}{{\underline{w}}}
\newcommand{\ux}{{\underline{x}}}

\newcommand{\s}[1]{{}_s\widehat{#1}}
\newcommand{\sov}[1]{{}_s\overline{#1}}
\newcommand{\uov}[1]{{}_u\overline{#1}}
\newcommand{\vov}[1]{{}_v\overline{#1}}
\newcommand{\uvov}[1]{{}_{\check u}\overline{#1}}
\renewcommand{\ss}[1]{\widehat{#1}_s}
\newcommand{\ssov}[1]{\overline{#1}_s}
\newcommand{\ttov}[1]{\overline{#1}_t}
\renewcommand{\t}[1]{{}_t\widehat{#1}}
\newcommand{\tov}[1]{{}_t\overline{#1}}
\renewcommand{\tt}[1]{\widehat{#1}_t}
\newcommand{\tdelta}{\widetilde{\delta}}
\renewcommand{\bot}{\mathrm{bot}}
\newcommand{\JW}{\mathrm{JW}}
\newcommand{\Br}{\mathrm{Br}}
\newcommand{\TD}{\mathrm{TD}}

\newcommand{\LL}{\mathrm{LL}}

\newcommand{\hW}{\widehat{W}}

\newcommand{\ovu}[1]{\overline{#1}_u}
\newcommand{\ovuv}[1]{\overline{#1}_{\check u}}
\newcommand{\huv}[1]{\widehat{#1}_{\check u}}

\newcommand{\cB}{\mathcal{B}}
\newcommand{\GKM}{\mathscr{G}}
\newcommand{\BKM}{\mathscr{B}}
\newcommand{\TKM}{\mathscr{T}}
\newcommand{\UKM}{\mathscr{U}}
\newcommand{\bX}{\mathbf{X}}
\newcommand{\Zdem}{{\Z'}}

\makeatletter
\newcommand*\leftdash{\!\rotatebox[origin=c]{-45}{$\dabar@\dabar@\dabar@$}\!}
\newcommand*\rightdash{\!\rotatebox[origin=c]{45}{$\dabar@\dabar@\dabar@$}\!}
\makeatother
\renewcommand{\fatslash}{\!\mathord{\mathchar"2728}\;}
\renewcommand{\fatbslash}{\mathord{\mathchar"2729}}
\newcommand{\BGB}{\BKM \backslash \GKM / \BKM}
\newcommand{\UGB}{\UKM \fatbslash \GKM/\BKM}
\newcommand{\UGBold}{\UKM \backslash \GKM/\BKM}
\newcommand{\UGU}{\UKM \fatbslash \GKM\fatslash \UKM}
\newcommand{\UGUby}{\UKM\leftdash\GKM\rightdash\UKM}
\newcommand{\UGUvee}{\UKM^\vee \fatbslash \GKM^\vee\fatslash \UKM^\vee}
\newcommand{\UGUveeby}{\UKM^\vee \leftdash \GKM^\vee\rightdash \UKM^\vee}
\newcommand{\BGBvee}{\BKM^\vee\backslash \GKM^\vee/\BKM^\vee}
\newcommand{\BGUvee}{\BKM^\vee\backslash \GKM^\vee/\UKM^\vee}

\newcommand{\loc}{\mathrm{loc}}
\newcommand{\Loc}{\mathsf{Loc}}

\newcommand{\Parity}{\mathrm{Parity}}
\newcommand{\ParityBS}{\mathrm{Parity}_{\mathrm{BS}}}
\newcommand{\cE}{\mathcal{E}}
\newcommand{\Diag}{\mathscr{D}}
\newcommand{\DiagBS}{\mathscr{D}_{\mathrm{BS}}}
\newcommand{\DiagBSp}{\mathscr{D}_{\mathrm{BS}}^\oplus}
\newcommand{\oDiagBS}{\overline{\mathscr{D}}_{\mathrm{BS}}}
\newcommand{\oDiagBSp}{\overline{\mathscr{D}}{}^\oplus_{\mathrm{BS}}}
\newcommand{\oDiag}{\overline{\mathscr{D}}}
\newcommand{\even}{\mathrm{ev}}
\newcommand{\odd}{\mathrm{od}}
\newcommand{\ustar}{\mathbin{\underline{\star}}}
\newcommand{\Perv}{\mathrm{Perv}}

\newcommand{\Tilt}{\mathrm{Tilt}}

\newcommand{\TiltBSp}{\mathrm{\Tilt}_{\mathrm{BS}}^\oplus}
\newcommand{\cT}{\mathcal{T}}
\newcommand{\Tmon}{\widetilde{\mathcal{T}}}
\newcommand{\hatstar}{\mathbin{\widehat{\star}}}
\newcommand{\tD}{\widetilde{\Delta}}
\newcommand{\tN}{\widetilde{\nabla}}
\newcommand{\Conv}{\mathrm{Conv}}

\newcommand{\BE}{\mathsf{BE}}
\newcommand{\RE}{\mathsf{RE}}
\newcommand{\LM}{\mathsf{LM}}
\newcommand{\FM}{\mathsf{FM}}
\newcommand{\ForBERE}{\mathsf{For}^{\BE}_{\RE}}
\newcommand{\ForBELM}{\mathsf{For}^{\BE}_{\LM}}
\newcommand{\ForLMRE}{\mathsf{For}^{\LM}_{\RE}}
\newcommand{\ForFMLM}{\mathsf{For}^{\FM}_{\LM}}

\newcommand{\bs}{\mathbf{s}}
\newcommand{\bt}{\mathbf{t}}
\newcommand{\bu}{\mathbf{u}}

\newcommand{\cF}{\mathcal{F}}
\newcommand{\cG}{\mathcal{G}}
\newcommand{\cH}{\mathcal{H}}

\newsavebox\lowerdot
\savebox\lowerdot{%
\begin{tikzpicture}[scale=0.3,thick,baseline]
 \draw (0,-0.5) to (0,0.5);
 \node at (0,-0.5) {$\bullet$};
\end{tikzpicture}%
}
\newsavebox\upperdot
\savebox\upperdot{%
\begin{tikzpicture}[scale=0.3,thick,baseline]
 \draw (0,-0.5) to (0,0.5);
 \node at (0,0.5) {$\bullet$};
\end{tikzpicture}%
}
\newsavebox\upperlowerdot
\savebox\upperlowerdot{%
\begin{tikzpicture}[scale=0.3,thick,baseline]
 \draw (0,-1) to (0,-0.4);
 \draw (0,0.4) to (0,1);
 \node at (0,-0.4) {$\bullet$};
 \node at (0,0.4) {$\bullet$};
\end{tikzpicture}%
}
\newsavebox\lowerupperdot
\savebox\lowerupperdot{%
\begin{tikzpicture}[scale=0.3,thick,baseline]
 \draw (0,-0.5) to (0,0.5);
 \node at (0,-0.5) {$\bullet$};
 \node at (0,0.5) {$\bullet$};
\end{tikzpicture}%
}
\newsavebox\capmor
\savebox\capmor{%
\begin{tikzpicture}[yscale=0.1,xscale=0.1,baseline,thick] \draw[black] (-1,0) to[out=90, in=180] (0,2) to[out=0, in=90] (1,0); \end{tikzpicture}%
}
\newsavebox\cupmor
\savebox\cupmor{%
\begin{tikzpicture}[yscale=0.1,xscale=0.1,thick] \draw[black] (-1,2) to[out=90, in=180] (0,0) to[out=0, in=90] (1,2); \end{tikzpicture}%
}
\newsavebox\invymor
\savebox\invymor{%
\begin{tikzpicture}[yscale=0.2,xscale=0.1,baseline,thick] \draw (-1,-1) -- (0,0) -- (1,-1); \draw (0,0) -- (0,1); \end{tikzpicture}%
}
\newsavebox\ymor
\savebox\ymor{%
\begin{tikzpicture}[yscale=-0.2,xscale=0.1,baseline,thick] \draw (-1,-1) -- (0,0) -- (1,-1); \draw (0,0) -- (0,1); \end{tikzpicture}%
}

\newcommand{\id}{\mathrm{id}}
\newcommand{\pt}{\mathrm{pt}}
\newcommand{\simto}{\overset{\sim}{\to}}
\newcommand{\la}{\langle}
\newcommand{\ra}{\rangle}

\DeclareMathOperator{\Hom}{Hom}
\DeclareMathOperator{\End}{End}
\DeclareMathOperator{\uHom}{\underline{Hom}}
\DeclareMathOperator{\uEnd}{\underline{End}}
\DeclareMathOperator{\gHom}{\mathbb{H}\mathsf{om}}
\DeclareMathOperator{\gEnd}{\mathbb{E}\mathsf{nd}}

\newcommand{\sdots}{\ldots}

\makeatletter
\def\lotimes{\@ifnextchar_{\@lotimessub}{\@lotimesnosub}}
\def\@lotimessub_#1{\mathchoice{\mathbin{\mathop{\otimes}^L}_{#1}}%
  {\otimes^L_{#1}}{\otimes^L_{#1}}{\otimes^L_{#1}}}
\def\@lotimesnosub{\mathbin{\mathop{\otimes}^L}}
\makeatother

\numberwithin{equation}{chapter}
\numberwithin{section}{chapter}
\numberwithin{figure}{chapter}
\newtheorem{thm}{Theorem}[section]
\newtheorem{lem}[thm]{Lemma}
\newtheorem{prop}[thm]{Proposition}
\newtheorem{cor}[thm]{Corollary}

\theoremstyle{definition}
\newtheorem{defn}[thm]{Definition}

\theoremstyle{remark}
\newtheorem{rmk}[thm]{Remark}
\newtheorem{ex}[thm]{Example}

\title{Free-monodromic mixed tilting sheaves \\ on flag varieties}

\author[P. N. Achar]{Pramod N. Achar}
\address{Department of Mathematics\\
  Louisiana State University\\
  Baton Rouge, LA 70803\\
  U.S.A.}
\email{pramod@math.lsu.edu}

\author[S. Makisumi]{Shotaro Makisumi}
\address{Department of Mathematics\\
Columbia University\\
New York, NY 10027\\
U.S.A.}
\email{makisumi@math.columbia.edu}

\author[S. Riche]{Simon Riche}
\address{Universit\'e Clermont Auvergne, CNRS, LMBP, F-63000 Clermont-Ferrand, France.
}
\email{simon.riche@uca.fr}

\author[G. Williamson]{Geordie Williamson}
\address{School of Mathematics and Statistics F07, University of
  Sydney NSW 2006, Australia. }
\email{g.williamson@sydney.edu.au}

\begin{document}
\frontmatter

\begin{abstract}
In this paper we propose a construction of a monoidal category of
``free-monodromic'' tilting perverse sheaves on (Kac--Moody) flag
varieties in the setting of the ``mixed modular derived category''
introduced by the first and third authors. This category shares most of the
properties of their counterpart in characteristic~$0$, defined by
Bezrukavnikov--Yun using certain pro-objects in triangulated
categories. This construction is the main new ingredient in the
construction of a ``modular Koszul duality'' equivalence for
constructible sheaves on flag varieties, see~\cite{mkdkm}.
\end{abstract}

\maketitle
\tableofcontents

\mainmatter
\chapter{Introduction}

\section{Koszul duality}

Algebraically, Koszul duality is a phenomenon that can link two seemingly unrelated rings $A$ and $A^!$. It often takes the form of an equivalence between the bounded derived categories of finitely generated \emph{graded} modules over $A$ and $A^!$, for suitable nonnegative gradings on $A$ and $A^!$. This equivalence sends simple $A$-modules concentrated in degree $0$ to injective graded $A^!$-modules, and their projective covers to simple graded $A^!$-modules concentrated in degree $0$.  We summarize this informally by writing
\begin{align*}
\Db \Mod^{\mathrm{fg}, \Z}(A) &\simto \Db \Mod^{\mathrm{fg}, \Z}(A^!) \\
\text{simple} &\mapsto \text{injective} \\
\text{projective} &\mapsto \text{simple.}
\end{align*}
See e.g.~\cite{bgs} for a detailed account of this construction.  A key point is that in this setting, Koszul duality is revealed by the construction of the gradings on $A$ and $A^!$.

\section{Koszul duality for constructible sheaves on flag varieties}
\label{sec:intro-cat-BGS}

The importance of this construction for the geometry of flag varieties was advocated in~\cite{bgs}. Consider a connected reductive algebraic group $G$ over an algebraically closed field $\F$ of characteristic $p>0$, with a Borel subgroup $B$, and set $\cB:=G/B$. Let $G^\vee$ be the Langlands dual connected reductive group over $\F$; let $B^\vee$ be a Borel subgroup of $G^\vee$; and set $\cB^\vee := G^\vee/B^\vee$. Finally, let
\[
 \Perv_{(B)}(\cB, \overline{\Ql}), \quad \text{resp.} \quad \Perv_{(B^\vee)}(\cB^\vee, \overline{\Ql}),
\]
be the category of perverse (\'etale) $\overline{\Ql}$-sheaves on $\cB$, resp.~$\cB^\vee$, which are constant along the $B$-orbits, resp.~$B^\vee$-orbits.

In~\cite{bgs}, in an attempt to interpret in a categorical way the inversion formula for Kazhdan--Lusztig polynomials, the authors proved that these two categories of perverse sheaves are Koszul dual to one another in the following sense: they constructed ``graded versions'' of these categories, denoted by $\Perv^\mix_{(B)}(\cB, \overline{\Ql})$ and $\Perv^\mix_{(B^\vee)}(\cB^\vee, \overline{\Ql})$ respectively, along with an equivalence of triangulated categories
\begin{equation}
\label{eqn:intro-koszul1}
\begin{aligned}
 \Db \Perv^\mix_{(B)}(\cB, \overline{\Ql}) &\simto \Db \Perv^\mix_{(B^\vee)}(\cB^\vee, \overline{\Ql}) \\
\text{simple} &\mapsto \text{injective} \\
\text{projective} &\mapsto \text{simple.}
\end{aligned}
\end{equation}
(As above, this should be understood as saying that simple objects \emph{of weight $0$} are sent to injective objects.)

Once again, Koszul duality is revealed by the construction of the graded version $\Perv^\mix_{(B)}(\cB, \overline{\Ql})$. This construction is related to (but different from) the category of mixed perverse sheaves $P_{\mathrm{m}}(\cB_0, \overline{\Ql})$ as developed in~\cite[\S5]{bbd}.  (Here, $\cB_0$ is the flag variety for a split $\Fp$-form of $G$, where $\Fp \subset \F$ is the prime field.)  Specifically, recall that every object in $P_{\mathrm{m}}(\cB_0, \overline{\Ql})$ is equipped with a (functorial) \emph{weight filtration}.  Then $\Perv^\mix_{(B)}(\cB, \overline{\Ql})$ is defined to be the full subcategory of the Serre subcategory of $P_{\mathrm{m}}(\cB_0,\overline{\Ql})$ generated by (half-)Tate twists of intersection cohomology complexes associated with constant local systems on the orbits of an $\Fp$-form of $B$, consisting of the objects $\cF$ such that the associated graded $\mathrm{gr}^\bullet(\cF)$ of the weight filtration on $\cF$ is a semisimple perverse sheaf.\footnote{In fact, the category considered in~\cite{bgs} is defined in slightly different terms, involving certain parity conditions. See~\cite[\S 7.2]{ar:kdsf} for the comparison with the definition given above.}

The category $\Perv_{(B)}(\cB, \overline{\Ql})$ also participates in another interesting derived equivalence called \emph{Ringel duality}, whose features are summarized as follows:
\begin{equation}
\label{eqn:intro-ringel}
\begin{aligned}
\Db\Perv_{(B)}(\cB, \overline{\Ql}) &\simto \Db \Perv_{(B)}(\cB, \overline{\Ql}) \\
\text{injective} &\mapsto \text{tilting} \\
\text{costandard} &\mapsto \text{standard} \\
\text{tilting} &\mapsto \text{projective}
\end{aligned}
\end{equation}
There is an analogous derived self-equivalence for $\Perv^\mix_{(B)}(\cB, \overline{\Ql})$ commuting with Tate twist, but unlike Koszul duality, Ringel duality is visible in the ungraded setting.

Following a suggestion of Be{\u\i}linson--Ginzburg~\cite{bg}, we can now form a variant of~\eqref{eqn:intro-koszul1} by composing it with the graded version of~\eqref{eqn:intro-ringel}.  Of course, the resulting functor sends (suitably normalized) tilting perverse sheaves to simple perverse sheaves (of weight~$0$).  Remarkably, it also sends simple objects to tilting objects:
\begin{equation}
\label{eqn:intro-koszul2}
\begin{aligned}
 \Db \Perv^\mix_{(B)}(\cB, \overline{\Ql}) &\simto \Db \Perv^\mix_{(B^\vee)}(\cB^\vee, \overline{\Ql}) \\
\text{simple} &\mapsto \text{tilting} \\
\text{tilting} &\mapsto \text{simple.}
\end{aligned}
\end{equation}
The additional symmetry exhibited by~\eqref{eqn:intro-koszul2} is rather useful as we pursue generalizations in the following sections.

\section{The mixed derived category}

In the setting of sheaves with coefficients of positive characteristic, Deligne's definition of mixed perverse sheaves is not available. In~\cite{modrap2}, the first and third authors of the present article propose a way to overcome this difficulty.

This solution has its roots in the earlier paper~\cite{ar:kdsf}, 
dedicated to a detailed study of the category $\Db \Perv^\mix_{(B)}(\cB,
\overline{\Ql})$ defined by Be{\u\i}linson--Ginzburg--Soergel. The results of~\cite{ar:kdsf} imply that there is an equivalence of triangulated categories
\begin{equation}\label{eqn:intro-semis}
\Db\Perv^\mix_{(B)}(\cB,\overline{\Ql}) \cong \Kb\mathrm{Semis}_{(B)}(\cB,\overline{\Ql}),
\end{equation}
where $\mathrm{Semis}_{(B)}(\cB,\overline{\Ql})$ is the full additive subcategory of $\Db\Perv_{(B)}(\cB, \overline{\Ql})$ consisting of semisimple complexes.

Notably, the right-hand side of~\eqref{eqn:intro-semis} does not refer to Deligne's theory of weights.  In fact, on the right-hand side, we can replace $G$ by the corresponding \emph{complex} connected reductive group, $\cB$ by its flag variety, and the \'etale
derived category by the usual constructible derived category (in the analytic topology) to obtain a description which does not even refer to the \'etale topology. 

Let us therefore change our set-up: assume now that $G$, $B$, $T$, $\cB$ and the Langlands dual data are defined over $\C$, and all sheaves are considered with respect to
the analytic topology. We seek to imitate the right-hand side of~\eqref{eqn:intro-semis} for sheaves with positive-characteristic coefficients. As suggested implicitly
in~\cite{soergel}, and more explicitly in~\cite{mkd}, the correct
replacement for semisimple complexes in this setting are the
\emph{parity complexes} of Juteau--Mautner--Williamson~\cite{jmw}. If
$\bk$ is an arbitrary field, this suggests defining the ``mixed
derived category'' by 
\[
\Dmix_{(B)}(\cB,\bk) := \Kb \Parity_{(B)}(\cB, \bk),
\]
where $\Parity_{(B)}(\cB, \bk)$ is the full subcategory of the $B$-constructible derived category
$\Db_{(B)}(\cB, \bk)$ whose objects are the parity complexes. This is
the definition 
chosen in~\cite{modrap2}.

One difficulty with this definition is that the perverse t-structure on $\Dmix_{(B)}(\cB,\bk)$ is not immediately visible, but this problem is solved in~\cite{modrap2}.  The heart of the resulting t-structure, denoted by $\Perv^\mix_{(B)}(\cB, \bk)$, is a graded highest weight category, so it makes sense to consider \emph{tilting mixed perverse sheaves} in $\Dmix_{(B)}(\cB,\bk)$. Then, in the case when the characteristic of $\bk$ is good for $G$, building on the results of~\cite{modrap1}, an equivalence of triangulated categories
\begin{equation}\label{eqn:intro-koszul-mod}
\begin{aligned}
\Dmix_{(B)}(\cB,\bk) &\simto \Dmix_{(B^\vee)}(\cB^\vee, \bk) \\
\text{parity} &\mapsto \text{tilting} \\
\text{tilting} &\mapsto \text{parity}
\end{aligned}
\end{equation}
is constructed.
We call this equivalence \emph{Koszul duality}, even though it does not involve any Koszul rings in the algebraic sense in general.  (In particular, the behavior of \emph{simple} objects under~\eqref{eqn:intro-koszul-mod} is rather opaque.)

\section{The case of Kac--Moody flag varieties}
\label{sec:intro-KM}

Let $\GKM$ be a Kac--Moody group (over either $\F$ or $\C$, as appropriate), with Borel subgroup $\BKM$ and maximal torus $\TKM$. Let $\UKM \subset \BKM$ be the maximal unipotent subgroup.  Recall that the $\UKM$-orbits on the flag variety $\GKM/\BKM$ are the same as the $\BKM$-orbits (and that these orbits are affine spaces).  Any complex of sheaves that is constant along $\BKM$-orbits is automatically $\UKM$-equivariant, so we can replace $\Db_{(\BKM)}(\GKM/\BKM,\overline{\Ql})$ by the $\UKM$-equivariant derived category $\Db_{\UKM}(\GKM/\BKM,\overline{\Ql})$, or, in ``stacky'' notation, by
\[
\Db(\UGBold,\overline{\Ql}).
\]
The benefit of this change in notation will be seen below.

In~\cite{by}, the authors prove an analogue of~\eqref{eqn:intro-koszul2} for $\Perv(\UGBold,\overline{\Ql})$.  In general, this category no longer has enough projective (or injective) objects, so the proof cannot go through analogues of~\eqref{eqn:intro-koszul1} or~\eqref{eqn:intro-ringel}.

Instead, one of the main ideas of~\cite{by} is to exhibit $\Dmix(\UGBold,\overline{\Ql})$ as a module category on both the left and right for certain \emph{monoidal} categories.  On the right, the construction is easy: the $\BKM$-equivariant derived category of $\GKM/\BKM$, denoted by
\[
\Dmix(\BGB,\overline{\Ql}),
\]
has a monoidal structure given by the \emph{convolution product} $\star$.  It also acts on $\Dmix(\UGBold,\overline{\Ql})$ on the right by convolution.

The situation on the left is much more complicated.  One starts with the monodromy construction of Verdier~\cite{verdier}, which equips each $\cF$ in $\Perv(\UGBold,\overline{\Ql})$ with a functorial group homomorphism
\[
X_*(\TKM) \to \End(\cF)^\times,
\]
where each element of $X_*(\TKM)$ acts by a unipotent automorphism. Taking the logarithm of this action provides an algebra homomorphism
\begin{equation}\label{eqn:intro-leftmon}
\mathrm{S}(V) \to \End(\cF),
\end{equation}
where $V:=\overline{\Ql} \otimes_\Z X_*(\TKM)$. Via this map, called the \emph{left monodromy map}, each element of $V$ acts by a nilpotent endomorphism.  Building on Verdier's construction, Bezrukavnikov and Yun explain how to define a ``free-monodromic completion'' of $\Dmix(\UGBold,\overline{\Ql})$.  This category,\footnote{To be precise, in~\cite{by} the authors do not work with the category of perverse sheaves considered in~\S\ref{sec:intro-cat-BGS}, but drop the semisimplicity condition on the associated graded of the weight filtration. This change is not relevant for our constructions, hence will be neglected in this introduction.} denoted by
\[
\Dmix(\UGUby, \overline{\Ql}),
\]
adds in certain pro-objects that live on the $\TKM$-torsor $\UKM\backslash\GKM \to \BKM\backslash\GKM$, and that are equipped with free actions on both the left and the right by $\mathrm{S}(V)$.  This category has a monoidal structure, given by a \emph{monodromic convolution product} $\star^\mon$.  It contains lifts of tilting perverse sheaves, and it acts on $\Dmix(\UGBold, \overline{\Ql})$ on the left by monodromic convolution.

The main result of~\cite{by} is a \emph{monoidal} equivalence of categories
\begin{equation}
\label{eqn:intro-koszul-by1}
\begin{aligned}
\Dmix(\UGUby,\overline{\Ql}),\star^\mon &\simto \Dmix(\BGBvee,\overline{\Ql}),\star \\
\text{tilting} &\mapsto \text{simple.}
\end{aligned}
\end{equation}
Note that even when $\GKM$ is finite-dimensional (i.e., a reductive group), this is a new result, not contained in~\eqref{eqn:intro-koszul1} or~\eqref{eqn:intro-koszul2}.  The Kac--Moody version of~\eqref{eqn:intro-koszul2} is then obtained as the middle column of the following commutative diagram:
\begin{equation}
\label{eqn:intro-koszul-by2}
\hspace{-5pt}
\begin{tikzpicture}[baseline=(a.base)]
\node[scale=0.875] (a) at (0,0){
\begin{tikzcd}[column sep=small, row sep=large]
\Dmix(\UGUby,\overline{\Ql}) \ar[r] \ar[d, "\wr"', "\eqref{eqn:intro-koszul-by1}"]&
\Dmix(\UGBold,\overline{\Ql}) 
\ar[d, "\substack{\text{tilting} \\ {\rotatebox{-90}{\text{$\mapsto$}}} \\ \text{simple}}\,\wr"', "\hphantom{\wr}\,\substack{\text{simple} \\ {\rotatebox{-90}{\text{$\mapsto$}}} \\ \text{tilting}}"] &
\Dmix(\BGB,\overline{\Ql}) \ar[l] \ar[d, "\wr"', "\eqref{eqn:intro-koszul-by1}"] \\
\Dmix(\BGBvee,\overline{\Ql}) \ar[r] &
\Dmix(\BGUvee,\overline{\Ql}) &
\Dmix(\UGUveeby,\overline{\Ql}), \ar[l]
\end{tikzcd}};
\end{tikzpicture}
\end{equation}
where the horizontal arrows are induced by the action on the skyscraper sheaf.

\section{Monodromy in the mixed modular derived category}

One of the primary goals of this paper and its sequel~\cite{mkdkm} is to prove analogues of~\eqref{eqn:intro-koszul-by1} and~\eqref{eqn:intro-koszul-by2} with coefficients in positive characteristic. The first step is to understand the proper replacement for the left monodromy map~\eqref{eqn:intro-leftmon}.  We cannot simply copy the construction of~\cite{by}, because ``logarithm'' does not make sense in general in positive characteristic.

Henceforth, let $V:=\bk \otimes_\Z X_*(\TKM)$.  The solution to our problem is based on the following observation: the category $\Parity(\UGBold, \bk)$ is obtained from the category of equivariant parity sheaves $\Parity(\BGB, \bk)$ by ``killing'' the action of $\coH_{\BKM}^\bullet(\pt, \bk) = \mathrm{S}(V^*)$. Instead of tensoring with the trivial $\mathrm{S}(V^*)$-module, one might lose less information by tensoring with its Koszul resolution $\Lambda \otimes_\bk \mathrm{S}(V^*)$, where $\Lambda$ is the exterior algebra of $V^*$. This idea is made precise in~\S\ref{sec:lmon-mixed-complexes}, where we define the \emph{left monodromic category} $\Dmix(\UGB, \bk)$. The objects of this category are certain pairs $(\cF,\delta)$ which are ``almost'' (but not quite) complexes: $\cF$ is a sequence of objects in $\Parity(\BGB,\bk)$, and $\delta$ is a ``differential'' living in
\[
\Lambda \otimes \bigoplus_{i,j,k \in \Z} \Hom_{\Parity(\BGB, \bk)}(\cF^i, \cF^j \{k\})
\]
and which satisfies $\kappa(\delta) + \delta \circ \delta = 0$, where $\kappa$ is induced by the Koszul differential on $\Lambda \otimes_\bk \mathrm{S}(V^*)$. (Here, $\{k\}$ is the cohomological shift by $k$ in $\Parity(\BGB,\bk)$.) The $\Hom$-space from $(\cF,\delta_\cF)$ to $(\cG,\delta_\cG)$ is defined as the degree-$(0,0)$ cohomology of a complex of $\Z$-graded $\bk$-vector spaces whose underlying vector space is
\[
\Lambda \otimes \bigoplus_{i,j,k \in \Z} \Hom_{\Parity(\BGB, \bk)}(\cF^i, \cG^j \{k\}).
\]
The idea that complexes in $\Dmix(\UGBold,\bk)$ should be ``deformed'' by allowing the square of the differential not to vanish is one of the keys to our constructions.

If $(\cF, \delta)$ is an object of $\Dmix(\UGB, \bk)$, the ``contraction'' morphism $V \otimes \Lambda \to \Lambda$ acting on $\delta$ defines a morphism $\mathrm{S}(V) \to \bigoplus_{j \in \Z} \Hom(\cF, \cF \langle j \rangle)$, which has all the required properties to be called the left monodromy morphism. (Here, $\langle j \rangle$ is the $j$-th power of the ``Tate twist'' $\langle 1 \rangle:=\{-1\}[1]$.)

\section{The free-monodromic derived category}

Our construction of the right monodromy action is more brutal: 
such an action is forced to exist by the definition of morphisms in the \emph{free-monodromic category} $\Dmix(\UGU, \bk)$. The objects of this category are certain pairs $(\cF, \delta)$, where $\cF$ is again a sequence of objects in $\Parity(\BGB,\bk)$, and the ``differential'' $\delta$ now lives in
\[
\Lambda \otimes \bigoplus_{i,j,k \in \Z} \Hom_{\Parity(\BGB, \bk)}(\cF^i, \cF^j \{k\}) \otimes \mathrm{S}(V)
\]
and is required to satisfy
\[
\kappa(\delta) + \delta \circ \delta = \sum_{i \in I} 1 \otimes (\id \ustar e_i) \otimes \check e_i,
\]
where $(e_i : i \in I)$ and $(\check e_i : i \in I)$ are dual bases of $V^*$ and $V$ respectively, and $\ustar$ is a graded version of the convolution product $\star$ on $\Parity(\BGB,\bk)$. The $\Hom$-space from $(\cF,\delta_\cF)$ to $(\cG,\delta_\cG)$ is defined as the degree-$(0,0)$ cohomology of a complex of $\Z$-graded $\bk$-vector spaces whose underlying vector space is
\begin{equation}
\label{eqn:intro-Hom-fm}
\Lambda \otimes \bigoplus_{i,j,k \in \Z} \Hom_{\Parity(\BGB, \bk)}(\cF^i, \cG^j \{k\}) \otimes \mathrm{S}(V).
\end{equation}

For $(\cF,\delta)$ in this category, we now have a natural algebra morphism
\[
\mathrm{S}(V) \otimes \mathrm{S}(V) \to \bigoplus_{j \in \Z} \Hom(\cF, \cF \langle j \rangle),
\]
defined on the first copy of $\mathrm{S}(V)$ in terms of the contraction morphism $V \otimes \Lambda \to \Lambda$, and on the second copy using the right-most factor in~\eqref{eqn:intro-Hom-fm}.

This definition is rather elementary, but it has a serious drawback: the asymmetry between the left and right monodromy morphisms makes it difficult to define a monodromic convolution product in $\Dmix(\UGU, \bk)$. The main result of this paper partially solves this problem.

\section{Convolution}

Let us briefly discuss the ideas that go into the construction of the monodromic convolution product.  The first step is to lift the monodromy action on $\Hom$-spaces to an action on the underlying complexes of graded vector spaces.  This is possible on a certain class of objects that we call \emph{convolutive} (see~\S\ref{sec:convolutive-UGU}), but not on arbitrary objects of $\Dmix(\UGU,\bk)$.  This class contains the \emph{free-monodromic tilting sheaves}, which are certain ``lifts'' of the tilting perverse sheaves in $\Perv^\mix(\UGB,\bk)$.

On the subcategory of convolutive objects, we are able to produce a candidate for monodromic convolution, denoted by $\hatstar$.  This operation is well defined on objects and on morphisms, but one runs into serious problems when trying to check that the product is bifunctorial.  Specifically, we do not know how to show that it obeys an ``interchange law'' on morphisms, except in the special case of certain ``nice'' morphisms.

To proceed further, we restrict to the full subcategory $\Tilt(\UGU,\bk)$ of free-monodromic tilting sheaves.  The intuitive idea is to show that there are ``enough'' morphisms in $\Tilt(\UGU,\bk)$ that are ``nice'' in the sense considered above.  We will not make the notion of ``enough'' precise here; our actual proof 
involves a localization argument and a reduction to a computation in a rank-$2$ subgroup (of finite type).  We solve the problem in the rank-$2$ case by developing a theory of free-monodromic lifts of ``minimal Rouquier complexes.''  As a consequence, we obtain the desired monoidal structure on $\Tilt(\UGU,\bk)$.

A summary of the principal categories considered in this paper is shown in Figure~\ref{fig:categories}.  Note that we still do not know how to equip all of $\Dmix(\UGU,\bk)$ with a monoidal structure.

\begin{figure}
\[
\begin{tikzcd}[column sep=0pt]
\hspace{-1em}\Tilt(\UGU,\bk), \hatstar\hspace{-1em} \ar[d, hook] &&&&
\hspace{-1em}\Parity(\BGB,\bk), \star\hspace{-1em} \ar[d, hook] \\
\hspace{-1em}\Dmix(\UGU,\bk)\hspace{-1em} \ar[dr, "\text{forget}"'] &&&&
\hspace{-1em}\Dmix(\BGB,\bk),\ustar\hspace{-1em} \ar[dl, "\text{forget}"] \\
& \hspace{-1em}\Dmix(\UGB,\bk) \ar[rr, "\sim"] & \hspace{0.5em} &
\Dmix(\UGBold,\bk)\hspace{-1em}
\end{tikzcd}
\]
\caption{Various categories of sheaves on flag varieties}\label{fig:categories}
\end{figure}

\section{Application to Koszul duality}
\label{sec:intro-KD}

In the subsequent paper~\cite{mkdkm}, we will prove that there is an equivalence of monoidal additive categories
\begin{equation}\label{eqn:intro-koszul-eqmon}
\Tilt(\UGU,\bk) \simto \Parity(\BGBvee,\bk);
\end{equation}
this result is a positive-characteristic analogue of~\eqref{eqn:intro-koszul-by1}. (At the moment, it cannot be upgraded to an equivalence of triangulated categories, since we do not know how to make $\Dmix(\UGU,\bk)$ into a monoidal category.)  We will then use this to establish the following analogue of~\eqref{eqn:intro-koszul-by2}:
\[
\hspace{-5pt}
\begin{tikzcd}[column sep=small, row sep=large]
\Tilt(\UGU,\bk) \ar[r] \ar[d, "\wr"', "\eqref{eqn:intro-koszul-eqmon}"]&
\Dmix(\UGBold,\bk) 
\ar[d, "\substack{\text{tilting} \\ \rotatebox{-90}{\text{$\mapsto$}} \\ \text{parity}}\,\wr"', "\hphantom{\wr}\,\substack{\text{parity} \\ \rotatebox{-90}{\text{$\mapsto$}} \\ \text{tilting}}"] &
\Parity(\BGB,\bk) \ar[l] \ar[d, "\wr"', "\eqref{eqn:intro-koszul-eqmon}"] \\
\Parity(\BGBvee,\bk) \ar[r] &
\Dmix(\BGUvee,\bk) &
\Tilt(\UGUvee,\bk). \ar[l]
\end{tikzcd}
\]
The middle column is called ``modular Koszul duality for Kac--Moody groups''; it generalizes~\eqref{eqn:intro-koszul2}, \eqref{eqn:intro-koszul-mod}, and~\eqref{eqn:intro-koszul-by2}.

\section{The diagrammatic language}

In this introduction, we have discussed our main results in the language of equivariant parity complexes on Kac--Moody flag varieties.  As shown in~\cite{rw}, these parity sheaves are equivalent to a special case of the \emph{Elias--Williamson diagrammatic category} defined in~\cite{ew}, and one could ask whether it is possible to work in that more general setting instead.  In fact, the majority of the present paper---the first nine chapters---are written in the diagrammatic language. This has two main benefits:
\begin{enumerate}
\item
it prepares for an expected generalization of the Koszul duality formalism developed in~\cite{mkdkm} to more general Coxeter groups;
\item
it makes certain explicit computations much easier.
\end{enumerate}
However, our main result relies on some results of~\cite{modrap2} whose current proofs require geometry.  In Chapter~\ref{chap:kac-moody}, we translate the main constructions into the language of parity complexes, in preparation for the main result in Chapter~\ref{chap:functoriality}.

\section{Contents}

We begin in Chapter~\ref{chap:Soergel-calculus} with a review of the construction of the Elias--Williamson diagrammatic category associated to a Coxeter group and a realization of this group. In Chapter~\ref{chap:dgg}, we fix our conventions on graded modules and differential graded graded algebras. 

The next four chapters contain the core material on left- and free-monodromic complexes.  Chapter~\ref{chap:soergel-diagrams} explains the construction of the left monodromy action in the mixed derived category, and Chapter~\ref{chap:fm-complexes} introduces our main object of study, the free-monodromic category. In Chapter~\ref{chap:fm-convolution}, we construct the convolution operation for convolutive free-monodromic complexes. In Chapter~\ref{chap:murmurs}, we lay the groundwork for showing that free-monodromic convolution is bifunctorial, by constructing the unitor and associator isomorphisms. 

Next, in Chapters~\ref{chap:dihedral} and~\ref{chap:rouquier}, we focus on finite dihedral groups.  Chapter~\ref{chap:dihedral} contains a detailed study of the diagrammatic category in this case. In Chapter~\ref{chap:rouquier}, we use this study to construct ``minimal Rouquier complexes'' (in both the biequivariant and the free-monodromic settings). 

In the last two chapters, we restrict our attention to the case of Cartan realizations of crystallographic Coxeter groups.  Chapter~\ref{chap:kac-moody} explains the relation with parity complexes and the constructions of~\cite{modrap2}. Finally, in Chapter~\ref{chap:functoriality}, we prove the main result: for Cartan realizations of crystallographic Coxeter groups, the convolution operation is a bifunctor on free-monodromic tilting sheaves.

\section{Acknowledgements}

This paper was initially posted on arXiv in 2017, and used in a crucial way in the companion paper~\cite{mkdkm}. It was submitted soon after that, but we were never able to obtain a report on it. It was revised in 2022 to take into account a better understanding of the foundations of the Elias--Williamson diagrammatic category obtained with~\cite{ew2}.

We are grateful to I.~Loseu for a number of enlightening conversations at an early stage of this work, B.~Elias for his help regarding the content of Chapter~\ref{chap:Soergel-calculus} (at a much later stage of this work), and C.~Vay for motivating discussions.

The first, second, and fourth authors discussed some of the ideas in
this paper during a workshop held at the American Institute for
Mathematics in March 2016. We are grateful to AIM for its hospitality
and fruitful working environment.

Part of this work was done during the second author's visit to Yale University for Fall 2016. The second author would like to thank the Yale mathematics department for their hospitality and great working conditions, as well as his dissertation advisor, Z.~Yun, for support through his Packard Foundation fellowship and for countless helpful conversations.

The fourth author would like to thank RIMS, Kyoto for an excellent
working environment and the opportunity to give a series of lectures
on this work.

P.A. was supported by NSF Grant No.~DMS-1500890. S.R. was partially supported by ANR Grant No.~ANR-13-BS01-0001-01. This project has received funding from the European Research Council (ERC) under the European Union's Horizon 2020 research and innovation programme (grant agreements No.~677147 and~101002592).

\chapter{Coxeter groups and Elias--Williamson calculus}
\label{chap:Soergel-calculus}

In this chapter, we discuss definitions, background, and notation related to realizations of Coxeter groups and the Elias--Williamson diagrammatic category.

\section{Realizations of Coxeter groups}
\label{sec:realizations}

The \emph{2-colored quantum numbers} $[n]_x, [n]_y \in \Z[x, y]$\index{quantum numbers!nx&@{$[n]_x, [n]_y$}} are defined by
\[
 [0]_x = [0]_y = 0, \quad [1]_x = [1]_y = 1, \quad [2]_x = x, \quad [2]_y = y,
\]
the recursive formula
\begin{equation} \label{eq:2q-recursion}
 [n]_x = \begin{cases}
          [2]_x[n-1]_x - [n-2]_x &\mbox{ for } n \ge 2 \mbox{ even;} \\
          [2]_y[n-1]_x - [n-2]_x &\mbox{ for } n \ge 2 \mbox{ odd,}
         \end{cases}
\end{equation}
and the same formula with $x$ and $y$ swapped. The main examples of such quantum numbers that will be relevant for us are the following:
\begin{itemize}
 \item $[3]_x=[3]_y=xy-1$;
 \item $[4]_x = x^2y-2x$, $[4]_y=xy^2-2y$;
 \item $[5]_x=[5]_y = x^2y^2-3xy+1$;
 \item $[6]_x = x^3 y^2 - 4 x^2 y + 3x$; $[6]_y = x^2 y^3 - 4 x y^2 + 3y$.
\end{itemize}
By induction one can easily prove that
\begin{gather}
 \label{eq:2q-s-t-odd}
 [n]_x = [n]_y \qquad \text{for $n$ odd;}\\
 \label{eq:2q-s-t-even}
 [2]_y[n]_x = [2]_x[n]_y \quad \text{for  $n$ even.}
\end{gather}
If $n$ is odd, we will sometimes write $[n]$ for $[n]_x=[n]_y$.

Let $(W,S)$ be a Coxeter system, where $S$ is finite (this assumption will be in force throughout the whole paper, though we will not repeat it), and let $\bk$ be a commutative ring. We assume we are given a free $\bk$-module $V$ of finite rank, together with subsets $\{\alpha_s^\vee : s \in S\} \subset V$ and $\{\alpha_s : s \in S\} \subset V^*:=\Hom_\bk(V,\bk)$. Given a pair $(s,t)$ of distinct simple reflections and $n \in \Z_{\geq 0}$, we denote by $[n]_s$, resp.~$[n]_t$,\index{quantum numbers!nxs&@{$[n]_s, [n]_t$}} the value of $[n]_x$, resp.~$[n]_y$, at $x=-\la\alpha_s^\vee, \alpha_t\ra$ and $y=-\la\alpha_t^\vee, \alpha_s\ra$. (Note that $[n]_s$ really depends on the pair $(s,t)$, and not only on $s$.) These are elements of $\bk$.

We will call any word $\uw$ in $S$ an \emph{expression}\index{expression}. We will denote its length by $\ell(\uw)$\index{lw@$\ell(\uw)$}, and the element of $W$ defined as the product of these simple reflections by $\pi(\uw)$\index{piw@$\pi(\uw)$}.
We will denote by $\hW$\index{What@$\hW$} the set of reduced expressions for elements of $W$.

Following~\cite{elias, ew}, we will say that $\fh:=(V, \{\alpha_s^\vee : s \in S\}, \{\alpha_s : s \in S\})$ is a \emph{realization}\index{realization} of $(W,S)$ over $\bk$ if
\begin{enumerate}
\item
\label{it:realization-cond-1}
$\langle \alpha_s^\vee, \alpha_s \rangle = 2$ for all $s \in S$;
\item
the assignment
\[
s \mapsto \bigl( v \mapsto v-\langle \alpha_s, v \rangle \alpha_s^\vee \bigr)
\]
defines a representation of $W$ on $V$;
\item
\label{it:realization-cond-3}
for any pair $(s,t)$ of distinct simple reflections such that $st$ has finite order $m_{st}$ in $W$, we have
\[
[m_{st}]_s = [m_{st}]_t=0.
\]
\end{enumerate}
(See~\cite[\S 3.1]{ew} for a discussion of condition~\eqref{it:realization-cond-3}.)

We will say that \emph{Demazure surjectivity}\index{Demazure surjectivity} holds if the morphisms of $\bk$-modules
\[
\alpha_s : V \to \bk \quad \text{and} \quad \alpha_s^\vee : V^* \to \bk
\]
are surjective for any $s \in S$. Note that this property follows from condition~\eqref{it:realization-cond-1} above if $2$ is invertible in $\bk$. We will say that the realization is \emph{balanced}\index{balanced} if for any pair $(s,t)$ as in~\eqref{it:realization-cond-3} we have
\[
[m_{st}-1]_s=[m_{st}-1]_t=1.
\]
(See~\cite[\S 5.2]{ew} for a discussion of this condition.) 

Throughout this paper, all realizations are tacitly assumed to be balanced and to satisfy Demazure surjectivity. The main examples of realizations that we want to consider are Cartan realizations of crystallographic Coxeter groups; see~\S\ref{sec:Cartan-realizations} below for details.

\section{Jones--Wenzl projectors}
\label{sec:JW}

Let $A$ be a $\Z[x,y]$-algebra, where $x$ and $y$ are indeterminates. Associated to this algebra, we have the two-colored Temperley--Lieb category $2\mathscr{T} \hspace{-2pt} \mathscr{L}_A$, see~\cite[\S 6.4]{ew2}. (See also~\cite[\S 2.6]{el} for an extension of this definition to more colors.) The colors in this category will be denoted $s$ and $t$ (which should be thought of as the two generators of a dihedral group); $x$ is associated with $s$ and $y$ with $t$. The morphism spaces (when nonzero) admit a basis consisting of two-colored crossingless matchings. Composition in this category is induced by concatenation, where a circle with color $s$, resp.~$t$, inside is evaluated as $-y$, resp.~$-x$.

We will need special notation for the ``interesting'' words in $\{s,t\}$:
for $n \ge 0$, we consider the words\index{ns@{$\s{n}$}}\index{nt@{$\t{n}$}}
\begin{equation*}
 \s{n} = (s,t, \sdots) \mbox{ ($n$ terms)}, \quad \t{n} = (t,s, \sdots) \mbox{ ($n$ terms)}. \\
\end{equation*}
We define $\ss{n}$, $\tt{n}$
in a similar way, with the subscript indicating the last term in the expression.
In the category $2\mathscr{T} \hspace{-2pt} \mathscr{L}_A$ one has a notion of \emph{Jones--Wenzl projector} $\mathcal{JW}_{\ss{n}}$ and $\mathcal{JW}_{\tt{n}}$ associated with words $\ss{n}$ and $\tt{n}$, see~\cite[Claim~2.14]{el} or~\cite[Definition~6.9]{ew2}. Such elements do not always exist (the answer depends on the choice of $A$), and it is a difficult problem to determine precisely when they do, except in the symmetric case (when $[2]_s = [2]_t$), see~\cite[Appendix]{el}.

One solution to this question, applicable for small values of $n$, is as follows. First, consider the case when $A=\Q(x,y)$. In this case, it is known that $\mathcal{JW}_{\ss{n}}$ and $\mathcal{JW}_{\tt{n}}$ exist for all $n$, and induction formulas to compute these morphisms are given in~\cite[Theorem~6.10 and Remark~6.11]{ew2}. Assume that, for a given $n$, one has an explicit expression, of $\mathcal{JW}_{\ss{n}}$ and $\mathcal{JW}_{\tt{n}}$, and that the coefficients in the expansion of these morphisms in the basis of crossingless matchings all belong to $\Z[x,y][1/a]$ for some $a \in \Z[x,y]$. Then if the given morphism $\Z[x,y] \to A$ extends to a morphism $\Z[x,y][1/a] \to A$ (in other words, if the image of $a$ in $A$ is invertible), one obtains morphisms in $2\mathscr{T} \hspace{-2pt} \mathscr{L}_A$ by evaluating all coefficients in $A$ using such an extension. It is clear from definitions that these morphisms are Jones--Wenzl projectors for $\ss{n}$ and $\tt{n}$, showing existence in this case.

For our purposes, the most important cases will be when $n \in \{2,3,4,6\}$ (see Section~\ref{sec:Cartan-realizations}). In these cases, in $2\mathscr{T} \hspace{-2pt} \mathscr{L}_{\Q(x,y)}$ the Jones--Wenzl projectors are as follows.
%
%
(We will only write projectors for words starting with $s$, and will not indicate the colors of the regions since they can be easily determined. The projectors for the words starting with $t$ can be obtained by switching $s$ and $t$ and $x$ and $y$.) One finds that
\[
\mathcal{JW}_{(s,t)} = \begin{array}{c}
\begin{tikzpicture}[scale=0.3,thick]
\draw (0,-1) -- (0,2);
\end{tikzpicture}
\end{array}, \qquad
\mathcal{JW}_{(s,t,s)} =
\begin{array}{c}
\begin{tikzpicture}[scale=0.3,thick]
\draw (0,-1) -- (0,2);
\draw (1,-1) -- (1,2);
\end{tikzpicture}
\end{array}
+ \frac{1}{x} 
\begin{array}{c}
\begin{tikzpicture}[scale=0.3,thick]
\draw (0,-1) to[out=90,in=180] (0.5,0) to[out=0,in=90] (1,-1);
\draw (0,2) to[out=-90,in=180] (0.5,1) to[out=0,in=-90] (1,2);
\end{tikzpicture}
\end{array},
\]
\[
\mathcal{JW}_{(s,t,s,t)} =
\begin{array}{c}
\begin{tikzpicture}[scale=0.3,thick]
\draw (0,-1) -- (0,2);
\draw (1,-1) -- (1,2);
\draw (2,-1) -- (2,2);
\end{tikzpicture}
\end{array}
+ \frac{[2]_y}{[3]} 
\begin{array}{c}
\begin{tikzpicture}[scale=0.3,thick]
\draw (0,-1) to[out=90,in=180] (0.5,0) to[out=0,in=90] (1,-1);
\draw (0,2) to[out=-90,in=180] (0.5,1) to[out=0,in=-90] (1,2);
\draw (2,-1) -- (2,2);
\end{tikzpicture}
\end{array}
+ \frac{[2]_x}{[3]} 
\begin{array}{c}
\begin{tikzpicture}[scale=0.3,thick]
\draw (0,-1) to[out=90,in=180] (0.5,0) to[out=0,in=90] (1,-1);
\draw (0,2) to[out=-90,in=180] (0.5,1) to[out=0,in=-90] (1,2);
\draw (-1,-1) -- (-1,2);
\end{tikzpicture}
\end{array}
+ \frac{1}{[3]} 
\begin{array}{c}
\begin{tikzpicture}[scale=0.3,thick]
\draw (0,-1) to[out=90,in=180] (0.5,0) to[out=0,in=90] (1,-1);
\draw (3,2) to[out=-90,in=180] (3.5,1) to[out=0,in=-90] (4,2);
\draw (2,-1) -- (2,2);
\end{tikzpicture}
\end{array}
+ \frac{1}{[3]} 
\begin{array}{c}
\begin{tikzpicture}[scale=0.3,thick]
\draw (3,-1) to[out=90,in=180] (3.5,0) to[out=0,in=90] (4,-1);
\draw (0,2) to[out=-90,in=180] (0.5,1) to[out=0,in=-90] (1,2);
\draw (2,-1) -- (2,2);
\end{tikzpicture}
\end{array}.
\]
The next relevant case is $\mathcal{JW}_{(s,t,s,t,s,t)}$, whose expression is shown on Figure~\ref{fig:JW5}.

\begin{figure}
\label{fig:JW5}
\begin{multline*}
\mathcal{JW}_{(s,t,s,t,s,t)} =
\begin{array}{c}
\begin{tikzpicture}[scale=0.3,thick]
\draw (0,-1) -- (0,2);
\draw (1,-1) -- (1,2);
\draw (2,-1) -- (2,2);
\draw (3,-1) -- (3,2);
\draw (4,-1) -- (4,2);
\end{tikzpicture}
\end{array}
+ \frac{1}{[5]}
\begin{array}{c}
\begin{tikzpicture}[scale=0.3,thick]
\draw (0,-1) to[out=90,in=180] (0.5,0) to[out=0,in=90] (1,-1);
\draw (5,2) to[out=-90,in=180] (5.5,1) to[out=0,in=-90] (6,2);
\draw (2,-1) -- (2,2);
\draw (3,-1) -- (3,2);
\draw (4,-1) -- (4,2);
\end{tikzpicture}
\end{array}
+ \frac{1}{[5]}
\begin{array}{c}
\begin{tikzpicture}[scale=0.3,thick]
\draw (5,-1) to[out=90,in=180] (5.5,0) to[out=0,in=90] (6,-1);
\draw (0,2) to[out=-90,in=180] (0.5,1) to[out=0,in=-90] (1,2);
\draw (2,-1) -- (2,2);
\draw (3,-1) -- (3,2);
\draw (4,-1) -- (4,2);
\end{tikzpicture}
\end{array} \\
+ \frac{[2]_x}{[4]_x [5]}
\begin{array}{c}
\begin{tikzpicture}[scale=0.3,thick]
\draw (0,-1) to[out=90,in=180] (0.5,0) to[out=0,in=90] (1,-1);
\draw (-1,-1) to[out=90,in=180] (0.5,1) to[out=0,in=90] (2,-1);
\draw (5,2) to[out=-90,in=180] (5.5,1) to[out=0,in=-90] (6,2);
\draw (4,2) to[out=-90,in=180] (5.5,0) to[out=0,in=-90] (7,2);
\draw (3,-1) -- (3,2);
\end{tikzpicture}
\end{array}
+ \frac{[2]_x}{[4]_x [5]}
\begin{array}{c}
\begin{tikzpicture}[scale=0.3,thick]
\draw (5,-1) to[out=90,in=180] (5.5,0) to[out=0,in=90] (6,-1);
\draw (4,-1) to[out=90,in=180] (5.5,1) to[out=0,in=90] (7,-1);
\draw (0,2) to[out=-90,in=180] (0.5,1) to[out=0,in=-90] (1,2);
\draw (-1,2) to[out=-90,in=180] (0.5,0) to[out=0,in=-90] (2,2);
\draw (3,-1) -- (3,2);
\end{tikzpicture}
\end{array}
+ \frac{[2]_y [4]_x}{[2]_x [5]}
\begin{array}{c}
\begin{tikzpicture}[scale=0.3,thick]
\draw (0,-1) to[out=90,in=180] (0.5,0) to[out=0,in=90] (1,-1);
\draw (0,2) to[out=-90,in=180] (0.5,1) to[out=0,in=-90] (1,2);
\draw (2,-1) -- (2,2);
\draw (3,-1) -- (3,2);
\draw (4,-1) -- (4,2);
\end{tikzpicture}
\end{array} \\
+ \frac{[2]_x}{[5]}
\begin{array}{c}
\begin{tikzpicture}[scale=0.3,thick]
\draw (0,-1) to[out=90,in=180] (0.5,0) to[out=0,in=90] (1,-1);
\draw (4,2) to[out=-90,in=180] (4.5,1) to[out=0,in=-90] (5,2);
\draw (2,-1) -- (2,2);
\draw (3,-1) -- (3,2);
\draw (-1,-1) -- (-1,2);
\end{tikzpicture}
\end{array}
+ \frac{[2]_x}{[5]}
\begin{array}{c}
\begin{tikzpicture}[scale=0.3,thick]
\draw (4,-1) to[out=90,in=180] (4.5,0) to[out=0,in=90] (5,-1);
\draw (0,2) to[out=-90,in=180] (0.5,1) to[out=0,in=-90] (1,2);
\draw (2,-1) -- (2,2);
\draw (3,-1) -- (3,2);
\draw (-1,-1) -- (-1,2);
\end{tikzpicture}
\end{array}
+ \frac{[2]_x [2]_y}{[4]_x [5]}
\begin{array}{c}
\begin{tikzpicture}[scale=0.3,thick]
\draw (5,-1) to[out=90,in=180] (5.5,0) to[out=0,in=90] (6,-1);
\draw (4,-1) to[out=90,in=180] (5.5,1) to[out=0,in=90] (7,-1);
\draw (1,2) to[out=-90,in=180] (1.5,1) to[out=0,in=-90] (2,2);
\draw (-1,2) to[out=-90,in=180] (-0.5,1) to[out=0,in=-90] (0,2);
\draw (3,-1) -- (3,2);
\end{tikzpicture}
\end{array} \\
+ \frac{[2]_x [2]_y}{[4]_x [5]}
\begin{array}{c}
\begin{tikzpicture}[scale=0.3,thick]
\draw (5,2) to[out=-90,in=180] (5.5,1) to[out=0,in=-90] (6,2);
\draw (4,2) to[out=-90,in=180] (5.5,0) to[out=0,in=-90] (7,2);
\draw (1,-1) to[out=90,in=180] (1.5,0) to[out=0,in=90] (2,-1);
\draw (-1,-1) to[out=90,in=180] (-0.5,0) to[out=0,in=90] (0,-1);
\draw (3,-1) -- (3,2);
\end{tikzpicture}
\end{array}
+ \frac{[4]_x}{[5]}
\begin{array}{c}
\begin{tikzpicture}[scale=0.3,thick]
\draw (0,-1) to[out=90,in=180] (0.5,0) to[out=0,in=90] (1,-1);
\draw (0,2) to[out=-90,in=180] (0.5,1) to[out=0,in=-90] (1,2);
\draw (-1,-1) -- (-1,2);
\draw (-2,-1) -- (-2,2);
\draw (-3,-1) -- (-3,2);
\end{tikzpicture}
\end{array}
+ \frac{[2]_y}{[5]}
\begin{array}{c}
\begin{tikzpicture}[scale=0.3,thick]
\draw (4,-1) to[out=90,in=180] (4.5,0) to[out=0,in=90] (5,-1);
\draw (0,2) to[out=-90,in=180] (0.5,1) to[out=0,in=-90] (1,2);
\draw (2,-1) -- (2,2);
\draw (3,-1) -- (3,2);
\draw (6,-1) -- (6,2);
\end{tikzpicture}
\end{array} \\
+ \frac{[2]_y}{[5]}
\begin{array}{c}
\begin{tikzpicture}[scale=0.3,thick]
\draw (0,-1) to[out=90,in=180] (0.5,0) to[out=0,in=90] (1,-1);
\draw (4,2) to[out=-90,in=180] (4.5,1) to[out=0,in=-90] (5,2);
\draw (2,-1) -- (2,2);
\draw (3,-1) -- (3,2);
\draw (6,-1) -- (6,2);
\end{tikzpicture}
\end{array}
+ \frac{[2]_x^2}{[4]_x[5]}
\begin{array}{c}
\begin{tikzpicture}[scale=0.3,thick]
\draw (4,-1) to[out=90,in=180] (4.5,0) to[out=0,in=90] (5,-1);
\draw (6,-1) to[out=90,in=180] (6.5,0) to[out=0,in=90] (7,-1);
\draw (0,2) to[out=-90,in=180] (0.5,1) to[out=0,in=-90] (1,2);
\draw (-1,2) to[out=-90,in=180] (0.5,0) to[out=0,in=-90] (2,2);
\draw (3,-1) -- (3,2);
\end{tikzpicture}
\end{array}
+ \frac{[2]_x^2}{[4]_x[5]}
\begin{array}{c}
\begin{tikzpicture}[scale=0.3,thick]
\draw (4,2) to[out=-90,in=180] (4.5,1) to[out=0,in=-90] (5,2);
\draw (6,2) to[out=-90,in=180] (6.5,1) to[out=0,in=-90] (7,2);
\draw (0,-1) to[out=90,in=180] (0.5,0) to[out=0,in=90] (1,-1);
\draw (-1,-1) to[out=90,in=180] (0.5,1) to[out=0,in=90] (2,-1);
\draw (3,-1) -- (3,2);
\end{tikzpicture}
\end{array} \\
+ \frac{[2]_x([5]+2)}{[4]_x [5]}
\begin{array}{c}
\begin{tikzpicture}[scale=0.3,thick]
\draw (0,-1) to[out=90,in=180] (0.5,0) to[out=0,in=90] (1,-1);
\draw (0,2) to[out=-90,in=180] (0.5,1) to[out=0,in=-90] (1,2);
\draw (2,-1) -- (2,2);
\draw (3,-1) to[out=90,in=180] (3.5,0) to[out=0,in=90] (4,-1);
\draw (3,2) to[out=-90,in=180] (3.5,1) to[out=0,in=-90] (4,2);
\end{tikzpicture}
\end{array}
+ \frac{[2]_x [2]_y}{[5]}
\begin{array}{c}
\begin{tikzpicture}[scale=0.3,thick]
\draw (0,-1) to[out=90,in=180] (0.5,0) to[out=0,in=90] (1,-1);
\draw (3,2) to[out=-90,in=180] (3.5,1) to[out=0,in=-90] (4,2);
\draw (2,-1) -- (2,2);
\draw (-1,-1) -- (-1,2);
\end{tikzpicture}
\end{array}
+ \frac{[2]_x [2]_y}{[5]}
\begin{array}{c}
\begin{tikzpicture}[scale=0.3,thick]
\draw (3,-1) to[out=90,in=180] (3.5,0) to[out=0,in=90] (4,-1);
\draw (0,2) to[out=-90,in=180] (0.5,1) to[out=0,in=-90] (1,2);
\draw (2,-1) -- (2,2);
\draw (-1,-1) -- (-1,2);
\end{tikzpicture}
\end{array} \\
+ \frac{[2]_x^2 [2]_y}{[4]_x [5]}
\begin{array}{c}
\begin{tikzpicture}[scale=0.3,thick]
\draw (0,-1) to[out=90,in=180] (0.5,0) to[out=0,in=90] (1,-1);
\draw (-2,-1) to[out=90,in=180] (-1.5,0) to[out=0,in=90] (-1,-1);
\draw (3,2) to[out=-90,in=180] (3.5,1) to[out=0,in=-90] (4,2);
\draw (5,2) to[out=-90,in=180] (5.5,1) to[out=0,in=-90] (6,2);
\draw (2,-1) -- (2,2);
\end{tikzpicture}
\end{array}
+ \frac{[2]_x^2 [2]_y}{[4]_x [5]}
\begin{array}{c}
\begin{tikzpicture}[scale=0.3,thick]
\draw (3,-1) to[out=90,in=180] (3.5,0) to[out=0,in=90] (4,-1);
\draw (5,-1) to[out=90,in=180] (5.5,0) to[out=0,in=90] (6,-1);
\draw (0,2) to[out=-90,in=180] (0.5,1) to[out=0,in=-90] (1,2);
\draw (-2,2) to[out=-90,in=180] (-1.5,1) to[out=0,in=-90] (-1,2);
\draw (2,-1) -- (2,2);
\end{tikzpicture}
\end{array}
+ \frac{[2]_x [2]_y^2 [3]}{[4]_x [5]}
\begin{array}{c}
\begin{tikzpicture}[scale=0.3,thick]
\draw (0,-1) to[out=90,in=180] (0.5,0) to[out=0,in=90] (1,-1);
\draw (0,2) to[out=-90,in=180] (0.5,1) to[out=0,in=-90] (1,2);
\draw (2,-1) to[out=90,in=180] (2.5,0) to[out=0,in=90] (3,-1);
\draw (2,2) to[out=-90,in=180] (2.5,1) to[out=0,in=-90] (3,2);
\draw (4,-1) -- (4,2);
\end{tikzpicture}
\end{array} \\
+ \frac{[2]_x^3 [3]}{[4]_x [5]}
\begin{array}{c}
\begin{tikzpicture}[scale=0.3,thick]
\draw (0,-1) to[out=90,in=180] (0.5,0) to[out=0,in=90] (1,-1);
\draw (0,2) to[out=-90,in=180] (0.5,1) to[out=0,in=-90] (1,2);
\draw (2,-1) to[out=90,in=180] (2.5,0) to[out=0,in=90] (3,-1);
\draw (2,2) to[out=-90,in=180] (2.5,1) to[out=0,in=-90] (3,2);
\draw (-1,-1) -- (-1,2);
\end{tikzpicture}
\end{array}
+ \frac{[3]}{[5]}
\begin{array}{c}
\begin{tikzpicture}[scale=0.3,thick]
\draw (3,-1) to[out=90,in=180] (3.5,0) to[out=0,in=90] (4,-1);
\draw (0,2) to[out=-90,in=180] (0.5,1) to[out=0,in=-90] (1,2);
\draw (2,-1) -- (2,2);
\draw (5,-1) -- (5,2);
\draw (6,-1) -- (6,2);
\end{tikzpicture}
\end{array}
+ \frac{[3]}{[5]}
\begin{array}{c}
\begin{tikzpicture}[scale=0.3,thick]
\draw (0,-1) to[out=90,in=180] (0.5,0) to[out=0,in=90] (1,-1);
\draw (3,2) to[out=-90,in=180] (3.5,1) to[out=0,in=-90] (4,2);
\draw (2,-1) -- (2,2);
\draw (5,-1) -- (5,2);
\draw (6,-1) -- (6,2);
\end{tikzpicture}
\end{array}
+ \frac{[3]}{[5]} \begin{array}{c}
\begin{tikzpicture}[scale=0.3,thick]
\draw (0,-1) to[out=90,in=180] (0.5,0) to[out=0,in=90] (1,-1);
\draw (3,2) to[out=-90,in=180] (3.5,1) to[out=0,in=-90] (4,2);
\draw (2,-1) -- (2,2);
\draw (-1,-1) -- (-1,2);
\draw (-2,-1) -- (-2,2);
\end{tikzpicture}
\end{array} \\
+ \frac{[3]}{[5]}
\begin{array}{c}
\begin{tikzpicture}[scale=0.3,thick]
\draw (3,-1) to[out=90,in=180] (3.5,0) to[out=0,in=90] (4,-1);
\draw (0,2) to[out=-90,in=180] (0.5,1) to[out=0,in=-90] (1,2);
\draw (2,-1) -- (2,2);
\draw (-1,-1) -- (-1,2);
\draw (-2,-1) -- (-2,2);
\end{tikzpicture}
\end{array}
+ \frac{[2]_x [3]}{[4]_x [5]}
\begin{array}{c}
\begin{tikzpicture}[scale=0.3,thick]
\draw (4,-1) to[out=90,in=180] (4.5,0) to[out=0,in=90] (5,-1);
\draw (0,-1) to[out=90,in=180] (0.5,0) to[out=0,in=90] (1,-1);
\draw (0,2) to[out=-90,in=180] (0.5,1) to[out=0,in=-90] (1,2);
\draw (-1,2) to[out=-90,in=180] (0.5,0.5) to[out=0,in=-90] (2,2);
\draw (3,-1) -- (3,2);
\end{tikzpicture}
\end{array}
+ \frac{[2]_x [3]}{[4]_x [5]}
\begin{array}{c}
\begin{tikzpicture}[scale=0.3,thick]
\draw (-4,-1) to[out=90,in=180] (-3.5,0) to[out=0,in=90] (-3,-1);
\draw (0,-1) to[out=90,in=180] (0.5,0) to[out=0,in=90] (1,-1);
\draw (0,2) to[out=-90,in=180] (0.5,1) to[out=0,in=-90] (1,2);
\draw (-1,2) to[out=-90,in=180] (0.5,0.5) to[out=0,in=-90] (2,2);
\draw (-2,-1) -- (-2,2);
\end{tikzpicture}
\end{array}\\
+ \frac{[2]_x [3]}{[4]_x [5]}
\begin{array}{c}
\begin{tikzpicture}[scale=0.3,thick]
\draw (0,-1) to[out=90,in=180] (0.5,0) to[out=0,in=90] (1,-1);
\draw (-1,-1) to[out=90,in=180] (0.5,0.5) to[out=0,in=90] (2,-1);
\draw (4,2) to[out=-90,in=180] (4.5,1) to[out=0,in=-90] (5,2);
\draw (0,2) to[out=-90,in=180] (0.5,1) to[out=0,in=-90] (1,2);
\draw (3,-1) -- (3,2);
\end{tikzpicture}
\end{array}
+ \frac{[2]_x [3]}{[4]_x [5]}
\begin{array}{c}
\begin{tikzpicture}[scale=0.3,thick]
\draw (5,-1) to[out=90,in=180] (5.5,0) to[out=0,in=90] (6,-1);
\draw (4,-1) to[out=90,in=180] (5.5,0.5) to[out=0,in=90] (7,-1);
\draw (1,2) to[out=-90,in=180] (1.5,1) to[out=0,in=-90] (2,2);
\draw (5,2) to[out=-90,in=180] (5.5,1) to[out=0,in=-90] (6,2);
\draw (3,-1) -- (3,2);
\end{tikzpicture}
\end{array}
+ \frac{[2]_x [3]}{[4]_x [5]}
\begin{array}{c}
\begin{tikzpicture}[scale=0.3,thick]
\draw (4,-1) to[out=90,in=180] (5.5,0.2) to[out=0,in=90] (7,-1);
\draw (5,-1) to[out=90,in=180] (5.5,-0.2) to[out=0,in=90] (6,-1);
\draw (5,2) to[out=-90,in=180] (5.5,1.2) to[out=0,in=-90] (6,2);
\draw (4,2) to[out=-90,in=180] (5.5,0.8) to[out=0,in=-90] (7,2);
\draw (3,-1) -- (3,2);
\end{tikzpicture}
\end{array} \\
+ \frac{[2]_x [3]}{[4]_x [5]}
\begin{array}{c}
\begin{tikzpicture}[scale=0.3,thick]
\draw (4,-1) to[out=90,in=180] (5.5,0.2) to[out=0,in=90] (7,-1);
\draw (5,-1) to[out=90,in=180] (5.5,-0.2) to[out=0,in=90] (6,-1);
\draw (5,2) to[out=-90,in=180] (5.5,1.2) to[out=0,in=-90] (6,2);
\draw (4,2) to[out=-90,in=180] (5.5,0.8) to[out=0,in=-90] (7,2);
\draw (8,-1) -- (8,2);
\end{tikzpicture}
\end{array}
+ \frac{[2]_y [3]}{[5]}
\begin{array}{c}
\begin{tikzpicture}[scale=0.3,thick]
\draw (0,-1) to[out=90,in=180] (0.5,0) to[out=0,in=90] (1,-1);
\draw (0,2) to[out=-90,in=180] (0.5,1) to[out=0,in=-90] (1,2);
\draw (-1,-1) -- (-1,2);
\draw (-2,-1) -- (-2,2);
\draw (2,-1) -- (2,2);
\end{tikzpicture}
\end{array}
+ \frac{[2]_x [3]}{[5]}
\begin{array}{c}
\begin{tikzpicture}[scale=0.3,thick]
\draw (0,-1) to[out=90,in=180] (0.5,0) to[out=0,in=90] (1,-1);
\draw (0,2) to[out=-90,in=180] (0.5,1) to[out=0,in=-90] (1,2);
\draw (-1,-1) -- (-1,2);
\draw (2,-1) -- (2,2);
\draw (3,-1) -- (3,2);
\end{tikzpicture}
\end{array} \\
+ \frac{[2]_x^2 [3]}{[4]_x [5]}
\begin{array}{c}
\begin{tikzpicture}[scale=0.3,thick]
\draw (0,-1) to[out=90,in=180] (0.5,0) to[out=0,in=90] (1,-1);
\draw (2,-1) -- (2,2);
\draw (4,-1) to[out=90,in=180] (4.5,0) to[out=0,in=90] (5,-1);
\draw (3,2) to[out=-90,in=180] (3.5,1) to[out=0,in=-90] (4,2);
\draw (5,2) to[out=-90,in=180] (5.5,1) to[out=0,in=-90] (6,2);
\end{tikzpicture}
\end{array}
+ \frac{[2]_x^2 [3]}{[4]_x [5]}
\begin{array}{c}
\begin{tikzpicture}[scale=0.3,thick]
\draw (5,-1) to[out=90,in=180] (5.5,0) to[out=0,in=90] (6,-1);
\draw (0,2) to[out=-90,in=180] (0.5,1) to[out=0,in=-90] (1,2);
\draw (2,-1) -- (2,2);
\draw (3,-1) to[out=90,in=180] (3.5,0) to[out=0,in=90] (4,-1);
\draw (4,2) to[out=-90,in=180] (4.5,1) to[out=0,in=-90] (5,2);
\end{tikzpicture}
\end{array}
+ \frac{[2]_x [2]_y [3]}{[4]_x [5]}
\begin{array}{c}
\begin{tikzpicture}[scale=0.3,thick]
\draw (8,-1) to[out=90,in=180] (8.5,0) to[out=0,in=90] (9,-1);
\draw (7,-1) -- (7,2);
\draw (4,-1) to[out=90,in=180] (4.5,0) to[out=0,in=90] (5,-1);
\draw (3,2) to[out=-90,in=180] (3.5,1) to[out=0,in=-90] (4,2);
\draw (5,2) to[out=-90,in=180] (5.5,1) to[out=0,in=-90] (6,2);
\end{tikzpicture}
\end{array} \\
+ \frac{[2]_x [2]_y [3]}{[4]_x [5]}
\begin{array}{c}
\begin{tikzpicture}[scale=0.3,thick]
\draw (5,-1) to[out=90,in=180] (5.5,0) to[out=0,in=90] (6,-1);
\draw (8,2) to[out=-90,in=180] (8.5,1) to[out=0,in=-90] (9,2);
\draw (7,-1) -- (7,2);
\draw (3,-1) to[out=90,in=180] (3.5,0) to[out=0,in=90] (4,-1);
\draw (4,2) to[out=-90,in=180] (4.5,1) to[out=0,in=-90] (5,2);
\end{tikzpicture}
\end{array}
+ \frac{[2]_x^2 [3]}{[4]_x [5]}
\begin{array}{c}
\begin{tikzpicture}[scale=0.3,thick]
\draw (4,-1) to[out=90,in=180] (4.5,0) to[out=0,in=90] (5,-1);
\draw (6,-1) to[out=90,in=180] (6.5,0) to[out=0,in=90] (7,-1);
\draw (5,2) to[out=-90,in=180] (5.5,1) to[out=0,in=-90] (6,2);
\draw (4,2) to[out=-90,in=180] (5.5,0.5) to[out=0,in=-90] (7,2);
\draw (3,-1) -- (3,2);
\end{tikzpicture}
\end{array}
+ \frac{[2]_x [2]_y [3]}{[4]_x [5]}
\begin{array}{c}
\begin{tikzpicture}[scale=0.3,thick]
\draw (4,-1) to[out=90,in=180] (4.5,0) to[out=0,in=90] (5,-1);
\draw (6,-1) to[out=90,in=180] (6.5,0) to[out=0,in=90] (7,-1);
\draw (5,2) to[out=-90,in=180] (5.5,1) to[out=0,in=-90] (6,2);
\draw (4,2) to[out=-90,in=180] (5.5,0.5) to[out=0,in=-90] (7,2);
\draw (8,-1) -- (8,2);
\end{tikzpicture}
\end{array}\\
+ \frac{[2]_x^2 [3]}{[4]_x [5]}
\begin{array}{c}
\begin{tikzpicture}[scale=0.3,thick]
\draw (4,-1) to[out=90,in=180] (5.5,0.5) to[out=0,in=90] (7,-1);
\draw (5,-1) to[out=90,in=180] (5.5,0) to[out=0,in=90] (6,-1);
\draw (6,2) to[out=-90,in=180] (6.5,1) to[out=0,in=-90] (7,2);
\draw (4,2) to[out=-90,in=180] (4.5,1) to[out=0,in=-90] (5,2);
\draw (3,-1) -- (3,2);
\end{tikzpicture}
\end{array}
+ \frac{[2]_x [2]_y [3]}{[4]_x [5]}
\begin{array}{c}
\begin{tikzpicture}[scale=0.3,thick]
\draw (4,-1) to[out=90,in=180] (5.5,0.5) to[out=0,in=90] (7,-1);
\draw (5,-1) to[out=90,in=180] (5.5,0) to[out=0,in=90] (6,-1);
\draw (6,2) to[out=-90,in=180] (6.5,1) to[out=0,in=-90] (7,2);
\draw (4,2) to[out=-90,in=180] (4.5,1) to[out=0,in=-90] (5,2);
\draw (8,-1) -- (8,2);
\end{tikzpicture}
\end{array}.
\end{multline*}
\caption{Jones--Wenzel projector for $(s,t,s,t,s,t)$}
\end{figure}

From these formulas one sees that the Jones--Wenzl projectors for length-$m$ alternating words in $s$ and $t$ exist provided the images of the following elements in $A$ are invertible:
\begin{itemize}
\item
if $m=2$: no condition;
\item
if $m=3$: $[2]_x$ and $[2]_y$;
\item
if $m=4$: $[3]$;
\item
if $m=6$: $[5]$ and $\frac{[4]_x}{[2]_x}=\frac{[4]_y}{[2]_y}$.
\end{itemize}


Assuming that the Jones--Wenzl projectors for $\ss{n}$ and $\tt{n}$ exist, one can ask whether they are rotatable, see~\cite[Definition~6.16]{ew2}. Once again it is a delicate question in general to determine when this condition is satisfied. But if one has an explicit formula for the projectors, checking the rotatability condition is just a matter of computation. Using the formulas given above, one sees that, provided the conditions above are satisfied, the Jones--Wenzl projectors for length-$m$ alternating words in $s$ and $t$ are rotatable provided the images of the following elements in $A$ vanish:
\begin{itemize}
\item
if $m=2$: $[2]_x$ and $[2]_y$;
\item
if $m=3$: $[2]_x [2]_y - 1$;
\item
if $m=4$: $[2]_x [2]_y - 2$;
\item
if $m=6$: $[2]_x [2]_y - 3$.
\end{itemize}
One can remark that these conditions in fact automatically imply that those considered above are satisfied.

\section{The Elias--Williamson diagrammatic category}
\label{sec:ew-diagram}

Let $(W,S)$ be a Coxeter system, let $\bk$ be an integral domain,\footnote{The assumption that $\bk$ is an integral domain is not explicit in~\cite{ew}, but it is necessary for the ``localization'' techniques used in~\cite[Section~5]{ew} to make sense.} and let $\fh=(V, \{\alpha_s^\vee : s \in S\}, \{\alpha_s : s \in S\})$ be a balanced realization of $(W,S)$ over $\bk$ which satisfies Demazure surjectivity. We will consider the symmetric algebra $\mathrm{S}(V^*)$ as a graded ring with $V^*$ in degree $2$.

In~\cite{ew}, the authors associate to such data a $\bk$-linear graded (strict) monoidal category $\DiagBS(\fh, W)$\index{diagrammatic categories!DBS@$\DiagBS(\fh,W)$}. The assumptions that one needs for this category to be well behaved are made more explicit in~\cite[\S 5]{ew2}. In addition to the conditions considered in Section~\ref{sec:realizations}, namely Demazure surjectivity and balancedness (which is not strictly necessary, but simplifies the combinatorics), one needs to assume two extra conditions. First, for any pair $s,t \in S$ of distinct elements generating a finite subgroup of $W$, denoting by $m$ the order of $st$ in $W$, and considering $\bk$ as a $\Z[x,y]$-algebra via
\[
x \mapsto - \langle \alpha_s^\vee, \alpha_t \rangle, \quad y \mapsto - \langle \alpha_t^\vee, \alpha_s \rangle,
\]
one wants to assume that in the category $2\mathscr{T} \hspace{-2pt} \mathscr{L}_\bk$
the Jones--Wenzl projectors $\mathcal{JW}_{\ss{m}}$ and $\mathcal{JW}_{\tt{m}}$ exist and are rotatable. In this case, taking a deformation retract (see~\cite[\S 5.2]{ew}) one obtains a diagram of the form used below in the definition of $\DiagBS(\fh, W)$. (More precisely, what we will consider is the diagram obtained by adding vertical lines to the left and to the right, as in~\cite[(5.16)]{ew}.) The diagram obtained from $\mathcal{JW}_{(s,t, \cdots, t)}$ (if $m$ is even) or $\mathcal{JW}_{(s,t, \cdots, s)}$ (if $m$ is odd) will be denoted
\[
\begin{array}{c}
    \begin{tikzpicture}[yscale=0.5,xscale=0.3,baseline,thick]
    \draw (-3,-0.5) rectangle (3,0.5);
    \node at (0,0) {$\JW_{(s,t,\sdots, t)}$};
    \draw[red] (-2.5,-1.5) to (-2.5,-0.5);
    \draw (-1.5,-1.5) to (-1.5,-0.5);
    \draw (2.5,-1.5) to (2.5,-0.5);
        \draw[red] (-2.5,0.5) to (-2.5,1.5);
    \draw (-1.5,1.5) to (-1.5,0.5);
    \draw (2.5,1.5) to (2.5,0.5);
    \node at (0.2,-1) {$\cdots$};
    \node at (0.2,1) {$\cdots$};
        \node at (-2.5,-1.8) {\tiny $s$};
                \node at (-2.5,1.8) {\tiny $s$};
    \node at (-1.5,-1.8) {\tiny $t$};
    \node at (2.5,-1.8) {\tiny $t$};
    \node at (-1.5,1.8) {\tiny $t$};
    \node at (2.5,1.8) {\tiny $t$};
    \end{tikzpicture}
\end{array}
\quad \text{or} \quad
\begin{array}{c}
    \begin{tikzpicture}[yscale=0.5,xscale=0.3,baseline,thick]
    \draw (-3,-0.5) rectangle (3,0.5);
    \node at (0,0) {$\JW_{(s,t,\sdots, s)}$};
    \draw[red] (-2.5,-1.5) to (-2.5,-0.5);
    \draw (-1.5,-1.5) to (-1.5,-0.5);
    \draw[red] (2.5,-1.5) to (2.5,-0.5);
        \draw[red] (-2.5,0.5) to (-2.5,1.5);
    \draw (-1.5,1.5) to (-1.5,0.5);
    \draw[red] (2.5,1.5) to (2.5,0.5);
    \node at (0.2,-1) {$\cdots$};
    \node at (0.2,1) {$\cdots$};
        \node at (-2.5,-1.8) {\tiny $s$};
                \node at (-2.5,1.8) {\tiny $s$};
    \node at (-1.5,-1.8) {\tiny $t$};
    \node at (2.5,-1.8) {\tiny $s$};
    \node at (-1.5,1.8) {\tiny $t$};
    \node at (2.5,1.8) {\tiny $s$};
    \end{tikzpicture}
\end{array},
\]
and again called a Jones--Wenzl projector.
The last condition has to be considered only in case $W$ admits a parabolic subgroup of type $H_3$. In this case, the ``Zamolodchikov" relation one needs to impose (see~\cite[(5.12)]{ew}) is not known explicitly. One therefore needs to assume that there exists a linear combination of this form that is sent to $0$ by the operation described in~\cite[\S 2]{ew2}. Such a linear combination is then fixed, and its vanishing is imposed in the category $\DiagBS(\fh, W)$ (see~\eqref{it:Zamolodchikov} below).

The objects of the category $\DiagBS(\fh, W)$ are parametrized by expressions $\uw$; the object attached to $\uw$ will be denoted by $B_\uw$\index{Buw@$B_\uw$}. The morphisms are generated (under horizontal and vertical concatenation, and $\bk$-linear combinations) by four kinds of morphisms depicted by diagrams (to be read from bottom to top):
\begin{enumerate}
\item
for any homogeneous $f \in \mathrm{S}(V^*)$, a morphism
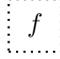
\[
    \begin{tikzpicture}[thick,scale=0.07,baseline]
      \node at (0,0) {$f$};
      \draw[dotted] (-5,-5) rectangle (5,5);
    \end{tikzpicture}
\]
from $B_\varnothing$ to itself, of degree $\deg(f)$;
\item
for any $s \in S$, ``dot'' morphisms
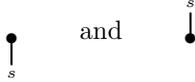
\[
       \begin{tikzpicture}[thick,scale=0.07,baseline]
      \draw (0,-5) to (0,0);
      \node at (0,0) {$\bullet$};
      \node at (0,-6.7) {\tiny $s$};
    \end{tikzpicture}
    \qquad \text{and} \qquad
      \begin{tikzpicture}[thick,baseline,xscale=0.07,yscale=-0.07]
      \draw (0,-5) to (0,0);
      \node at (0,0) {$\bullet$};
      \node at (0,-6.7) {\tiny $s$};
    \end{tikzpicture}
\]
from $B_s$ to $B_\varnothing$ and from $B_\varnothing$ to $B_s$, respectively, of degree $1$;
\item
for any $s \in S$, trivalent morphisms
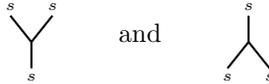
\[
        \begin{tikzpicture}[thick,baseline,scale=0.07]
      \draw (-4,5) to (0,0) to (4,5);
      \draw (0,-5) to (0,0);
      \node at (0,-6.7) {\tiny $s$};
      \node at (-4,6.7) {\tiny $s$};
      \node at (4,6.7) {\tiny $s$};      
    \end{tikzpicture}
    \qquad \text{and} \qquad
        \begin{tikzpicture}[thick,baseline,scale=-0.07]
      \draw (-4,5) to (0,0) to (4,5);
      \draw (0,-5) to (0,0);
      \node at (0,-6.7) {\tiny $s$};
      \node at (-4,6.7) {\tiny $s$};
      \node at (4,6.7) {\tiny $s$};    
    \end{tikzpicture}
\]
from $B_s$ to $B_{(s,s)}$ and from $B_{(s,s)}$ to $B_s$, respectively, of degree $-1$;
\item
for any pair $(s,t)$ of distinct simple reflections such that $st$ has finite order $m_{st}$ in $W$, a morphism
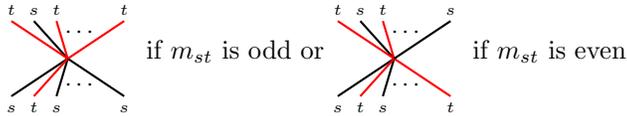
\[
    \begin{tikzpicture}[yscale=0.5,xscale=0.3,baseline,thick]
\draw (-2.5,-1) to (0,0) to (-1.5,1);
\draw (-0.5,-1) to (0,0);
\draw[red] (-1.5,-1) to (0,0) to (-2.5,1);
\draw[red] (0,0) to (-0.5,1);
\node at (-2.5,-1.3) {\tiny $s$\vphantom{$t$}};
\node at (-1.5,1.3) {\tiny $s$\vphantom{$t$}};
\node at (-0.5,-1.3) {\tiny $s$\vphantom{$t$}};
\node at (-1.5,-1.3) {\tiny $t$};
\node at (-2.5,1.3) {\tiny $t$};
\node at (-0.5,1.3) {\tiny $t$};
\node at (0.6,-0.7) {\small $\cdots$};
\node at (0.6,0.7) {\small $\cdots$};
\draw (2.5,-1) -- (0,0);
\draw[red] (2.5,1) -- (0,0);
\node at (2.5,-1.3) {\tiny $s$\vphantom{$t$}};
\node at (2.5,1.3) {\tiny $t$};
\end{tikzpicture} \text{ if $m_{st}$ is odd or}
    \begin{tikzpicture}[yscale=0.5,xscale=0.3,baseline,thick]
\draw (-2.5,-1) to (0,0) to (-1.5,1);
\draw (-0.5,-1) to (0,0);
\draw[red] (-1.5,-1) to (0,0) to (-2.5,1);
\draw[red] (0,0) to (-0.5,1);
\node at (-2.5,-1.3) {\tiny $s$\vphantom{$t$}};
\node at (-1.5,1.3) {\tiny $s$\vphantom{$t$}};
\node at (-0.5,-1.3) {\tiny $s$\vphantom{$t$}};
\node at (-1.5,-1.3) {\tiny $t$};
\node at (-2.5,1.3) {\tiny $t$};
\node at (-0.5,1.3) {\tiny $t$};
\node at (0.6,-0.7) {\small $\cdots$};
\node at (0.6,0.7) {\small $\cdots$};
\draw[red] (2.5,-1) -- (0,0);
\draw (2.5,1) -- (0,0);
\node at (2.5,-1.3) {\tiny $t$};
\node at (2.5,1.3) {\tiny $s$\vphantom{$t$}};
\end{tikzpicture} \text{ if $m_{st}$ is even}
\]
from $B_{(s,t,\sdots)}$ to $B_{(t,s,\sdots)}$ (where each expression has length $m_{st}$, and colors alternate), of degree $0$.
\end{enumerate}
(Below we will sometimes omit the labels ``$s$'' or ``$t$'' when they
do not play any role.) We set:
\[
\begin{array}{c}
\begin{tikzpicture}[scale=0.3,thick]
\draw (-1,-1) to[out=90,in=180] (0,1) to[out=0,in=90] (1,-1);
\end{tikzpicture}
\end{array}
:=
\begin{array}{c}
\begin{tikzpicture}[scale=0.3,thick]
\draw (-1,-1) to (0,0) to (1,-1);
\draw (0,0) to (0,1);
\node at (0,1) {$\bullet$};
\end{tikzpicture}
\end{array},
\qquad
\begin{array}{c}
\begin{tikzpicture}[scale=-0.3,thick]
\draw (-1,-1) to[out=90,in=180] (0,1) to[out=0,in=90] (1,-1);
\end{tikzpicture}
\end{array}
:=
\begin{array}{c}
\begin{tikzpicture}[scale=-0.3,thick]
\draw (-1,-1) to (0,0) to (1,-1);
\draw (0,0) to (0,1);
\node at (0,1) {$\bullet$};
\end{tikzpicture}
\end{array}.
\]
These morphisms are subject to the following relations:
\begin{enumerate}
\item
the boxes add and multiply in the obvious way;
\item
Frobenius unit:
\[
\begin{array}{c}
\begin{tikzpicture}[scale=0.3,thick]
\draw (0,-1) -- (0,2);
\draw (0,0) -- (1,1);
\node at (1,1) {$\bullet$};
\end{tikzpicture}
\end{array}
=
\begin{array}{c}
\begin{tikzpicture}[scale=0.3,thick]
\draw (0,-1) -- (0,2);
\draw (0,1) -- (1,0);
\node at (1,0) {$\bullet$};
\end{tikzpicture}
\end{array}
=
\begin{array}{c}
\begin{tikzpicture}[scale=0.3,thick]
\draw (0,-1) -- (0,2);
\end{tikzpicture}
\end{array};
\]
\item
Frobenius associativity:
\[
\begin{array}{c}
\begin{tikzpicture}[scale=0.3,thick]
\draw (0,-1) -- (0,2);
\draw (0,0) -- (2,1);
\draw (2,-1) -- (2,2);
\end{tikzpicture}
\end{array}
=
\begin{array}{c}
\begin{tikzpicture}[scale=0.3,thick]
\draw (0,-1) -- (0,2);
\draw (0,1) -- (2,0);
\draw (2,-1) -- (2,2);
\end{tikzpicture}
\end{array}
=
\begin{array}{c}
\begin{tikzpicture}[scale=0.3,thick]
\draw (-1.5,2) -- (0,1) -- (0,0) -- (-1.5,-1);
\draw (0,1) -- (1.5,2);
\draw (0,0) -- (1.5,-1);
\end{tikzpicture}
\end{array};
\]
\item
needle relation:
\[
\begin{array}{c}
\begin{tikzpicture}[scale=0.3,thick]
\draw (0,2) -- (0,1) -- (-1,0) -- (0,-1) -- (0,-2);
\draw (0,1) -- (1,0) -- (0,-1);
\end{tikzpicture}
\end{array}
= 0;
\]
\item
barbell relation:
\[
\begin{array}{c}
\begin{tikzpicture}[scale=0.3,thick]
\draw (0,1) -- (0,-1);
\node at (0,1) {$\bullet$};
\node at (0,-1) {$\bullet$};
\end{tikzpicture}
\end{array}
=
\begin{array}{c}
    \begin{tikzpicture}[thick,scale=0.3]
      \node at (0,0) {$\alpha_s$};
      \draw[dotted] (-1,-1) rectangle (1,1);
    \end{tikzpicture}
    \end{array}
\]
(where $s$ is the color of the diagram on the left-hand side);
\item
nil-Hecke relation:
\[
\begin{array}{c}
\begin{tikzpicture}[scale=0.3,thick]
\draw (0,-2) -- (0,2);
\node at (0,-2.5) {\tiny $s$};
\node at (0,2.5) {\tiny $s$};
\node at (-1.5,0) {$f$};
    \end{tikzpicture}
    \end{array}
    =
    \begin{array}{c}
\begin{tikzpicture}[scale=0.3,thick]
\draw (0,-2) -- (0,2);
\node at (0,-2.5) {\tiny $s$};
\node at (0,2.5) {\tiny $s$};
\node at (1.5,0) {$s(f)$};
    \end{tikzpicture}
    \end{array}
    +
    \begin{array}{c}
\begin{tikzpicture}[scale=0.3,thick]
\draw (0,-2) -- (0,-1);
\draw (0,2) -- (0,1);
\node at (0,-2.5) {\tiny $s$};
\node at (0,2.5) {\tiny $s$};
\node at (0,-1) {$\bullet$};
\node at (0,1) {$\bullet$};
\node at (0,0) {$\partial_s(f)$};
    \end{tikzpicture}
    \end{array}
\]
(where $\partial_s : f \mapsto \frac{f-s(f)}{\alpha_s}$ is the Demazure operator\index{Demazure operator@{$\partial_s$}} associated with $s$);
\item
cyclicity of the $2m_{st}$-valent vertex:
\begin{gather*}
\begin{array}{c}
    \begin{tikzpicture}[yscale=0.5,xscale=0.3,thick]
\draw (-2.5,-1.5) to[out=90, in=-90] (-2.5,-1) to[out=60, in=-160] (0,0) to[out=140,in=-60] (-1.5,1) to[in=-90,out=90] (-1.5,1.5);
\draw (-0.5,-1.5) to[out=90,in=-90] (-0.5,-1) to[out=80,in=-110] (0,0);
\draw[red] (-1.5,-1.5) to[out=90, in=-90] (-1.5,-1) to[out=60, in=-140] (0,0) to[out=160, in=0] (-2,0.5) to[out=180, in=90] (-3.5,-1.5);
\draw[red] (0,0) to[out=110, in=-80] (-0.5,1) to[in=-90,out=90] (-0.5,1.5);
\node at (0.4,-1) {\small $\cdots$};
\node at (0.6,0.7) {\small $\cdots$};
\draw (3.5,1.5) to[out=-90,in=0] (2,-0.5) to[out=180,in=-20] (0,0);
\draw[red] (2.5,1.5) to[out=90,in=-90] (2.5,1) to[out=-120,in=20] (0,0);
\draw[red] (0,0) to[out=-40,in=120] (1.5,-1) to[out=-90,in=90] (1.5,-1.5); 
\end{tikzpicture}
\end{array}
=
    \begin{tikzpicture}[yscale=0.5,xscale=0.3,baseline,thick]
\draw[red] (-2.5,-1.5) to (0,0) to (-1.5,1.5);
\draw[red] (-0.5,-1.5) to (0,0);
\draw (-1.5,-1.5) to (0,0) to (-2.5,1.5);
\draw (0,0) to (-0.5,1.5);
\draw[red] (2.5,-1.5) -- (0,0);
\draw (2.5,1.5) -- (0,0);
\node at (0.6,-1) {\small $\cdots$};
\node at (0.6,1) {\small $\cdots$};
\end{tikzpicture}
=
\begin{array}{c}
    \begin{tikzpicture}[yscale=0.5,xscale=-0.3,thick]
\draw (-2.5,-1.5) to[out=90, in=-90] (-2.5,-1) to[out=60, in=-160] (0,0) to[out=140,in=-60] (-1.5,1) to[in=-90,out=90] (-1.5,1.5);
\draw (-0.5,-1.5) to[out=90,in=-90] (-0.5,-1) to[out=80,in=-110] (0,0);
\draw[red] (-1.5,-1.5) to[out=90, in=-90] (-1.5,-1) to[out=60, in=-140] (0,0) to[out=160, in=0] (-2,0.5) to[out=180, in=90] (-3.5,-1.5);
\draw[red] (0,0) to[out=110, in=-80] (-0.5,1) to[in=-90,out=90] (-0.5,1.5);
\node at (0.4,-1) {\small $\cdots$};
\node at (0.6,0.7) {\small $\cdots$};
\draw (3.5,1.5) to[out=-90,in=0] (2,-0.5) to[out=180,in=-20] (0,0);
\draw[red] (2.5,1.5) to[out=90,in=-90] (2.5,1) to[out=-120,in=20] (0,0);
\draw[red] (0,0) to[out=-40,in=120] (1.5,-1) to[out=-90,in=90] (1.5,-1.5); 
\end{tikzpicture}
\end{array}
\text{if $m_{st}$ is odd;}
\\
\begin{array}{c}
    \begin{tikzpicture}[yscale=0.5,xscale=0.3,thick]
\draw (-2.5,-1.5) to[out=90, in=-90] (-2.5,-1) to[out=60, in=-160] (0,0) to[out=140,in=-60] (-1.5,1) to[in=-90,out=90] (-1.5,1.5);
\draw (-0.5,-1.5) to[out=90,in=-90] (-0.5,-1) to[out=80,in=-110] (0,0);
\draw[red] (-1.5,-1.5) to[out=90, in=-90] (-1.5,-1) to[out=60, in=-140] (0,0) to[out=160, in=0] (-2,0.5) to[out=180, in=90] (-3.5,-1.5);
\draw[red] (0,0) to[out=110, in=-80] (-0.5,1) to[in=-90,out=90] (-0.5,1.5);
\node at (0.4,-1) {\small $\cdots$};
\node at (0.6,0.7) {\small $\cdots$};
\draw[red] (3.5,1.5) to[out=-90,in=0] (2,-0.5) to[out=180,in=-20] (0,0);
\draw (2.5,1.5) to[out=90,in=-90] (2.5,1) to[out=-120,in=20] (0,0);
\draw (0,0) to[out=-40,in=120] (1.5,-1) to[out=-90,in=90] (1.5,-1.5); 
\end{tikzpicture}
\end{array}
=
    \begin{tikzpicture}[yscale=0.5,xscale=0.3,baseline,thick]
\draw[red] (-2.5,-1.5) to (0,0) to (-1.5,1.5);
\draw[red] (-0.5,-1.5) to (0,0);
\draw (-1.5,-1.5) to (0,0) to (-2.5,1.5);
\draw (0,0) to (-0.5,1.5);
\draw (2.5,-1.5) -- (0,0);
\draw[red] (2.5,1.5) -- (0,0);
\node at (0.6,-1) {\small $\cdots$};
\node at (0.6,1) {\small $\cdots$};
\end{tikzpicture}
=
\begin{array}{c}
    \begin{tikzpicture}[yscale=0.5,xscale=-0.3,thick]
\draw[red] (-2.5,-1.5) to[out=90, in=-90] (-2.5,-1) to[out=60, in=-160] (0,0) to[out=140,in=-60] (-1.5,1) to[in=-90,out=90] (-1.5,1.5);
\draw[red] (-0.5,-1.5) to[out=90,in=-90] (-0.5,-1) to[out=80,in=-110] (0,0);
\draw (-1.5,-1.5) to[out=90, in=-90] (-1.5,-1) to[out=60, in=-140] (0,0) to[out=160, in=0] (-2,0.5) to[out=180, in=90] (-3.5,-1.5);
\draw (0,0) to[out=110, in=-80] (-0.5,1) to[in=-90,out=90] (-0.5,1.5);
\node at (0.4,-1) {\small $\cdots$};
\node at (0.6,0.7) {\small $\cdots$};
\draw (3.5,1.5) to[out=-90,in=0] (2,-0.5) to[out=180,in=-20] (0,0);
\draw[red] (2.5,1.5) to[out=90,in=-90] (2.5,1) to[out=-120,in=20] (0,0);
\draw[red] (0,0) to[out=-40,in=120] (1.5,-1) to[out=-90,in=90] (1.5,-1.5); 
\end{tikzpicture}
\end{array}
\text{if $m_{st}$ is even;}
\end{gather*}
\item
2-color associativity:
\begin{gather*}
\begin{array}{c}
    \begin{tikzpicture}[yscale=0.5,xscale=0.3,baseline,thick]
\draw[red] (-2.5,-1.5) to (0,0) to (-1.5,1.5);
\draw[red] (-0.5,-1.5) to (0,0);
\draw (-1.5,-1.5) to (0,0) to (-2.5,1.5);
\draw[red] (0,0) to (1,1.5);
\draw[red] (2.5,-1.5) -- (0,0);
\draw (3,1.5) -- (2,0.7) -- (0,0);
\draw (2,1.5) -- (2,0.7);
\node at (0.6,-1) {\small $\cdots$};
\node at (0,1) {\small $\cdots$};
\end{tikzpicture}
\end{array}
=
\begin{array}{c}
    \begin{tikzpicture}[yscale=0.5,xscale=0.3,baseline,thick]
    \draw[red] (-2.5,-1.5) to (-2,-0.9) to (0.5,-0.2);
    \draw[red] (-0.5,-1.5) to (0.5,-0.2);
    \draw (-1.5,-1.5) to (0.5,-0.2);
    \draw (0.5,-0.2) to (3,1.5);
    \draw (0.5,-0.2) to (-0.4,0) to (-0.5,0.6);
    \draw[red] (0.5,-0.2) to (0.7,0.3) to (-0.5,0.6);
    \draw[red] (2.5,-1.5) to (0.5,-0.2);
    \draw[red] (-2,-0.9) to (-0.5,0.6) to (-1.5,1.5);
    \draw (-0.5,0.6) to (-2.5,1.5);
     \draw (-0.5,0.6) to (2,1.5);
     \node at (0.8,-1) {$\cdots$};
     \node at (0.1,1.2) {$\cdots$};
     \node at (0.2,0.2) {\tiny $\cdots$};
    \end{tikzpicture}
\end{array}
\text{if $m_{st}$ is odd;}
\\
\begin{array}{c}
    \begin{tikzpicture}[yscale=0.5,xscale=0.3,baseline,thick]
\draw[red] (-2.5,-1.5) to (0,0) to (-1.5,1.5);
\draw[red] (-0.5,-1.5) to (0,0);
\draw (-1.5,-1.5) to (0,0) to (-2.5,1.5);
\draw (0,0) to (1,1.5);
\draw (2.5,-1.5) -- (0,0);
\draw[red] (3,1.5) -- (2,0.7) -- (0,0);
\draw[red] (2,1.5) -- (2,0.7);
\node at (0.6,-1) {\small $\cdots$};
\node at (0,1) {\small $\cdots$};
\end{tikzpicture}
\end{array}
=
\begin{array}{c}
    \begin{tikzpicture}[yscale=0.5,xscale=0.3,baseline,thick]
    \draw[red] (-2.5,-1.5) to (-2,-0.9) to (0.5,-0.2);
    \draw[red] (-0.5,-1.5) to (0.5,-0.2);
    \draw (-1.5,-1.5) to (0.5,-0.2);
    \draw[red] (0.5,-0.2) to (3,1.5);
    \draw (0.5,-0.2) to (-0.4,0) to (-0.5,0.6);
    \draw (0.5,-0.2) to (0.7,0.3) to (-0.5,0.6);
    \draw (2.5,-1.5) to (0.5,-0.2);
    \draw[red] (-2,-0.9) to (-0.5,0.6) to (-1.5,1.5);
    \draw (-0.5,0.6) to (-2.5,1.5);
     \draw[red] (-0.5,0.6) to (2,1.5);
     \node at (0.8,-1) {$\cdots$};
     \node at (0.1,1.2) {$\cdots$};
     \node at (0.2,0.2) {\tiny $\cdots$};
    \end{tikzpicture}
\end{array}
\text{if $m_{st}$ is even;}
\end{gather*}
\item
Jones--Wenzl relations:
\begin{gather*}
\begin{array}{c}
    \begin{tikzpicture}[yscale=0.5,xscale=0.3,baseline,thick]
\draw[red] (-2.5,-1.5) to (0,0) to (-1.5,1.5);
\draw[red] (-0.5,-1.5) to (0,0);
\draw (-1.5,-1.5) to (0,0) to (-2.5,1.5);
\draw[red] (0,0) to (1,1.5);
\draw[red] (2.5,-1.5) -- (0,0);
\draw (1.7,0.7) -- (0,0);
\node at (1.7,0.7) {$\bullet$};
\node at (0.6,-1) {\small $\cdots$};
\node at (0,1) {\small $\cdots$};
\node at (-2.5,-1.8) {\tiny $s$};
\node at (-2.5,1.8) {\tiny $t$};
\node at (-.5,-1.8) {\tiny $s$};
\node at (-1.5,-1.8) {\tiny $t$};
\node at (-1.5,1.8) {\tiny $s$};
\node at (2.5,-1.8) {\tiny $s$};
\node at (1,1.8) {\tiny $s$};
\end{tikzpicture}
\end{array}
=
\begin{array}{c}
    \begin{tikzpicture}[yscale=0.5,xscale=0.3,baseline,thick]
    \draw (-3,-0.5) rectangle (3,0.5);
    \node at (0,0) {$\JW_{(s,t,\sdots, s)}$};
    \draw[red] (-2.5,-1.5) to (-2.5,-0.5);
    \draw (-1.5,-1.5) to (-1.5,-0.5);
    \draw[red] (2.5,-1.5) to (2.5,-0.5);
        \draw[red] (-2.5,1) to (-2.5,0.5);
        \node[red] at (-2.5,1) {$\bullet$};
    \draw (-1.5,1.5) to (-1.5,0.5);
    \draw[red] (2.5,1.5) to (2.5,0.5);
    \node at (0.2,-1) {$\cdots$};
    \node at (0.2,1) {$\cdots$};
    \node at (-2.5,-1.8) {\tiny $s$};
    \node at (-1.5,-1.8) {\tiny $t$};
    \node at (2.5,-1.8) {\tiny $s$};
    \node at (-1.5,1.8) {\tiny $t$};
    \node at (2.5,1.8) {\tiny $s$};
    \end{tikzpicture}
\end{array}
\text{if $m_{st}$ is odd;}
\\
\begin{array}{c}
    \begin{tikzpicture}[yscale=0.5,xscale=0.3,baseline,thick]
\draw[red] (-2.5,-1.5) to (0,0) to (-1.5,1.5);
\draw[red] (-0.5,-1.5) to (0,0);
\draw (-1.5,-1.5) to (0,0) to (-2.5,1.5);
\draw (0,0) to (1,1.5);
\draw (2.5,-1.5) -- (0,0);
\draw[red] (1.7,0.7) -- (0,0);
\node[red] at (1.7,0.7) {$\bullet$};
\node at (0.6,-1) {\small $\cdots$};
\node at (0,1) {\small $\cdots$};
\node at (-2.5,-1.8) {\tiny $s$};
\node at (-2.5,1.8) {\tiny $t$};
\node at (-.5,-1.8) {\tiny $s$};
\node at (-1.5,-1.8) {\tiny $t$};
\node at (-1.5,1.8) {\tiny $s$};
\node at (2.5,-1.8) {\tiny $t$};
\node at (1,1.8) {\tiny $t$};
\end{tikzpicture}
\end{array}
=
\begin{array}{c}
    \begin{tikzpicture}[yscale=0.5,xscale=0.3,baseline,thick]
    \draw (-3,-0.5) rectangle (3,0.5);
    \node at (0,0) {$\JW_{(s,t,\sdots, t)}$};
    \draw[red] (-2.5,-1.5) to (-2.5,-0.5);
    \draw (-1.5,-1.5) to (-1.5,-0.5);
    \draw (2.5,-1.5) to (2.5,-0.5);
        \draw[red] (-2.5,1) to (-2.5,0.5);
        \node[red] at (-2.5,1) {$\bullet$};
    \draw (-1.5,1.5) to (-1.5,0.5);
    \draw (2.5,1.5) to (2.5,0.5);
    \node at (0.2,-1) {$\cdots$};
    \node at (0.2,1) {$\cdots$};
        \node at (-2.5,-1.8) {\tiny $s$};
    \node at (-1.5,-1.8) {\tiny $t$};
    \node at (2.5,-1.8) {\tiny $t$};
    \node at (-1.5,1.8) {\tiny $t$};
    \node at (2.5,1.8) {\tiny $t$};
    \end{tikzpicture}
\end{array}
\text{if $m_{st}$ is even}
\end{gather*}
(see~\S\ref{sec:JW-projectors} below for a review of the main properties of the ``Jones--Wenzl projectors\index{Jones--Wenzl projector}'' appearing on the right-hand sides of these relations);
\item
\label{it:Zamolodchikov}
Zamolodchikov relations: see~\cite[\S 5.1]{ew}.
\end{enumerate}

The composition of morphisms is induced by vertical concatenation.
The monoidal product in $\DiagBS(\fh,W)$ is induced by the assignment $B_\uv \star B_\uw:=B_{\uv \uw}$\index{convolution!star@$\star$}, and horizontal concatenation of diagrams. 

When we consider $\DiagBS(\fh,W)$ as a graded category as above, the graded $\bk$-module of morphisms from $B_\uw$ to $B_\uv$ will be denoted
\[
\Hom^\bullet_{\DiagBS(\fh,W)}(B_\uw, B_\uv).
\]
But sometimes it will be more convenient to consider it as a usual category endowed with a ``shift of grading'' autoequivalence $(1)$, whose $n$-th power will be denoted by $(n)$. From this perspective the objects of $\DiagBS(\fh,W)$ are the $B_{\uw}(n)$ where $\uw$ is an expression and $n \in \Z$. The morphism space from $B_{\uw}(n)$ to $B_{\uv}(m)$, denoted
\[
\Hom_{\DiagBS(\fh,W)}(B_{\uw}(n),B_{\uv}(m)),
\]
is the $\bk$-submodule of $\Hom^\bullet_{\DiagBS(\fh,W)}(B_\uw, B_\uv)$ consisting of elements of degree $m-n$.

\section{Additive hull of \texorpdfstring{$\DiagBS(\fh,W)$}{DBS(h,W)}}
\label{sec:additive-hull-Diag}

We will denote by $\DiagBSp(\fh,W)$\index{diagrammatic categories!DBSp@$\DiagBSp(\fh,W)$} the additive envelope of $\DiagBS(\fh,W)$.  In other words, $\DiagBSp(\fh,W)$ is the category in which:
\begin{itemize}
 \item
 an object $\cF$ is a finite (formal) direct sum of objects of $\DiagBS(\fh,W)$;
 \item
 a morphism from $\cF$ to $\cG$ is a matrix of morphisms in $\DiagBS(\fh,W)$ from the summands of $\cF$ to those of $\cG$.
\end{itemize}
The bifunctor $\star: \DiagBS(\fh,W) \times \DiagBS(\fh,W) \to \DiagBS(\fh,W)$ and the autoequivalence $(1): \DiagBS(\fh,W) \to \DiagBS(\fh,W)$ extend to formal direct sums in an obvious way.  In particular, the bifunctor\index{convolution!star@$\star$}
\[
\star: \DiagBSp(\fh,W) \times \DiagBSp(\fh,W) \to \DiagBSp(\fh,W),
\]
together with the unit object $B_\varnothing$, and the unitor and associator isomorphisms induced by those of $\DiagBS(\fh,W)$, make $\DiagBSp(\fh,W)$ into a monoidal category. As in any monoidal category, this operation obeys an ``interchange law'' with respect to composition. That is, given morphisms $f: \cF \to \cG$, $g: \cG \to \cH$, $f': \cF' \to \cG'$, and $g': \cG' \to \cH'$, we have an equality
\[
(g \circ f) \star (g' \circ f') = (g \star g') \circ (f \star f')
\]
of morphisms $\cF \star \cF' \to \cH \star \cH'$. (This property is part of the statement that $\star$ is a bifunctor.)

It will sometimes be convenient to pick out the ``even'' and ``odd'' parts of an object $\cF \in \DiagBSp(\fh,W)$.  Suppose $\cF = \bigoplus_{i=1}^k B_{\uw_i}(n_i)$.  We put
\[
\cF_\even = \bigoplus_{\substack{i \in \{1,\sdots, k\}\\ \text{$\ell(\uw_i) + n_i$ even}}} B_{\uw_i}(n_i)
\qquad\text{and}\qquad
\cF_\odd = \bigoplus_{\substack{i \in \{1,\sdots, k\}\\ \text{$\ell(\uw_i) + n_i$ odd}}} B_{\uw_i}(n_i),
\]
so that $\cF = \cF_\even \oplus \cF_\odd$.  Note that $\Hom_{\DiagBS(\fh,W)}(B_\uw(n), B_\uv(m))=0$ unless $\ell(\uw)+n$ and $\ell(\uv)+m$ have the same parity.  As a consequence, morphisms in $\DiagBSp(\fh,W)$ respect the decomposition above.  Explicitly, any morphism $f: \cF \to \cG$ in $\DiagBSp(\fh,W)$ can be written uniquely as
\begin{equation}\label{eqn:morphism-even-odd}
f = f_\even + f_\odd
\end{equation}
where $f_\even$ is a morphism $\cF_\even \to \cG_\even$, and $f_\odd$ is a morphism $\cF_\odd \to \cG_\odd$.  As an example, this decomposition interacts with $\star$ as follows:
\begin{align*}
(\cF \star \cG)_\even &= \cF_\even \star \cG_\even \oplus \cF_\odd \star \cG_\odd, \\
(\cF \star \cG)_\odd &= \cF_\even \star \cG_\odd \oplus \cF_\odd \star \cG_\even.
\end{align*}

If $\bk'$ is another integral domain and $\varphi : \bk \to \bk'$ is a
ring morphism, then $\bk' \otimes_\bk V$ affords a natural
realization of $W$, which will be denoted $\bk' \otimes_\bk \fh$, and
we have an ``extension of scalars'' functor
\[
\bk' : \DiagBS(\fh,W) \to \DiagBS(\bk' \otimes_\bk \fh,W),
\]
which induces a functor denoted similarly
\[
\bk' : \DiagBSp(\fh,W) \to \DiagBSp(\bk' \otimes_\bk \fh,W).
\]
These functors are compatible with the monoidal products. The functor $\bk'$ induces an isomorphism
\[
\bk' \otimes_\bk \Hom_{\DiagBSp(\fh,W)}(\cF, \cG) \simto \Hom_{\DiagBSp(\bk' \otimes_\bk \fh,W)}(\bk'(\cF), \bk'(\cG))
\]
for any $\cF, \cG$ in $\DiagBSp(\fh,W)$. (This fact follows from the existence of the ``double leaves basis'' of~\cite[Theorem~6.11]{ew}, which does not depend on the choice of coefficients.)

We will also denote by $\oDiagBS(\fh,W)$\index{diagrammatic categories!DBSbar@$\oDiagBS(\fh,W)$} the category obtained from $\DiagBS(\fh,W)$ by tensoring all morphism spaces over $\mathrm{S}(V^*)$ (for the action induced by adding a box on the left of any diagram) with the trivial module $\bk$ (a right module category over $\DiagBS(\fh,W)$), and by $\oDiagBSp(\fh,W)$\index{diagrammatic categories!DBSpbar@$\oDiagBSp(\fh,W)$} the additive closure of $\oDiagBS(\fh,W)$.

Finally, if $\bk$ is a field or a complete local ring, we will denote by $\Diag(\fh,W)$\index{diagrammatic categories!D@$\Diag(\fh,W)$} the Karoubian envelope of $\DiagBSp(\fh,W)$, and by $\oDiag(\fh,W)$ the Karoubian envelope of $\oDiagBSp(\fh,W)$.\index{diagrammatic categories!Dbar@$\oDiag(\fh,W)$} According to~\cite[Theorem~6.25]{ew}, there is a bijection
\[
\left\{
\begin{array}{c}
\text{indecomposable objects in $\Diag(\fh,W)$} \\
\text{up to isomorphism and grading shift}
\end{array}
\right\}
\overset{\sim}{\leftrightarrow}
W.
\]
The indecomposable object corresponding to $w \in W$ will be denoted by\index{Bw@$B_w$}
\[
B_w.
\]
It can be characterized as follows: for any reduced expression $\uw \in \hW$ with $\pi(\uw) = w$, $B_w$ is the unique indecomposable summand of $B_\uw$ that does not occur as a summand of any $B_\uv(n)$ with $\uv \in \hW$ and $\ell(\uv) < \ell(\uw)$.  We will study these objects further in the case of dihedral groups in Chapter~\ref{chap:dihedral}.

\section{More on 2-colored quantum numbers}

We conclude this chapter with brief statements of a few additional properties of the $2$-colored quantum numbers.  These properties will be needed in Chapter~\ref{chap:dihedral}.

Using the formulas~\eqref{eq:2q-s-t-odd}--\eqref{eq:2q-s-t-even},
\eqref{eq:2q-recursion} may be rewritten
\begin{equation} \label{eq:2q-recursion-alt}
 [n]_x = [2]_x[n-1]_y - [n-2]_x \qquad \mbox{for } n \ge 2
\end{equation}
(and similarly for $x$ and $y$ swapped).
It also follows that
\begin{equation} \label{eq:2q-product-0}
 [n]_x[m]_y = [m]_x[n]_y \qquad \mbox{for } n \equiv m \bmod 2;
\end{equation}
\begin{equation} \label{eq:2q-product-1}
 [2]_x[n]_y[m]_y = [2]_y[n]_x[m]_x \qquad \mbox{for } n \not\equiv m \bmod 2.
\end{equation}

The following are 2-color analogues of the classical Clebsch--Gordan formulas. They can be proved together by induction on $m$ (with $n$ fixed).

\begin{lem} Suppose that $n \ge m \ge 0$.
  \begin{enumerate}
  \item If $n, m$ are of the same parity, then
    \begin{equation}
      \label{eq:CB1}
      [n]_x[m]_y = \sum_{i \in [n-m+1,n+m-1] \cap (2\Z + 1)} [i]_x = \sum_{i \in [n-m+1,n+m-1] \cap (2\Z + 1)} [i]_y.
    \end{equation}
  \item  If $n, m$ are of different parity, then
    \begin{equation}
      \label{eq:CB2}
      [n]_x[m]_x = \sum_{i \in [n-m+1,n+m-1] \cap (2\Z)} [i]_x,
    \end{equation}
    and similarly for $y$.
  \end{enumerate}
\end{lem}

The following formulas are obtained by applying \eqref{eq:CB1} to each product: for all $n \ge 1$, we have
\begin{equation} \label{eq:2q-square-identity}
 [n]_x[n]_y = [n-1]_x[n+1]_y + 1 = [n-1]_y[n+1]_x + 1.
\end{equation}

\chapter{Bigraded modules and dgg algebras}
\label{chap:dgg}

This chapter contains preliminaries, notation, and conventions (especially sign conventions) for bigraded modules and dgg (differential graded graded) modules.

\section{Gradings and differentials}

Let $\bk$ be a commutative ring. Consider a bigraded $\bk$-module
\[
M = \bigoplus_{i,j \in \Z} M^i_j.
\]
Elements of $M^i_j$ are said to be \emph{homogeneous of bidegree $(i,j)$}.  The \emph{cohomological degree} of an element $m \in M^i_j$, denoted by $|m|$, is defined to be $i$.  (The integer $j$ is called its \emph{internal degree}.  We will not need separate notation for the internal degree.)  We define three kinds of shift-of-grading functors, acting on objects as follows:
\begin{align*}
M[n]^i_j &= M^{i+n}_j, &
M\la n\ra^i_j &= M^i_{j-n}, & 
M(n)^i_j &= M^{i+n}_{j+n}.
\end{align*}
(The actions on morphisms are the obvious ones.)
The three shift functors commute with each other and are related by
\[
\la n\ra = (-n)[n].
\]
If $M$ and $N$ are two bigraded $\bk$-modules, their tensor product $M \otimes N$ is equipped with the bigrading given by
\[
(M \otimes N)^i_j = \bigoplus_{\substack{p+q=i\\ r+s = j}} M^p_r \otimes N^q_s.
\]
Moreover, we equip the category of bigraded $\bk$-modules with the usual ``super'' symmetric monoidal structure: that is, the isomorphism 
\[
\tau: M \otimes N \to N \otimes M
\]
is given by $\tau(m \otimes n) = (-1)^{|m||n|} n \otimes m$ for
homogeneous elements $m \in M$, $n \in N$.  This sign rule makes
numerous appearances in formulas below.

A \emph{dgg $\bk$-module} is a bigraded $\bk$-module equipped with a map $d: M \to M[1]$, called the \emph{differential}, such that $d \circ d = 0$. When $M$ and $N$ are dgg $\bk$-modules, we equip $M \otimes N$ with the differential
\[
d_{M \otimes N}(m \otimes n) = d_M(m) \otimes n + (-1)^{|m|} m \otimes d_N(n)
\]
for $m \in M$ and $n \in N$.
A \emph{dgg algebra} is a dgg $\bk$-module $A$ that is also a $\bk$-algebra, with the unit $1$ homogeneous of bidegree $(0,0)$, and such that the multiplication map $A \otimes A \to A$ is a map of dgg $\bk$-modules. 

Given a dgg $\bk$-module $M$, the tensor algebra $\mathrm{T}(M) = \bigoplus_{n \ge 0} M^{\otimes n}$ has a natural structure of a dgg algebra.
The \emph{symmetric algebra} on $M$, denoted by $\Sym(M)$\index{Sym@$\Sym$}, is the quotient of $\mathrm{T}(M)$ by the ideal generated by elements of the form
\[
x \otimes y - (-1)^{|x||y|} y \otimes x
\]
for $x,y \in M$, and elements of the form
\[
x \otimes x
\]
for $x \in M$ such that $|x|$ is odd.
The differential on $\mathrm{T}(M)$ induces a differential on $\Sym(M)$, making it into a dgg $\bk$-algebra.

\section{Symmetric algebras associated to realizations}

We fix a Coxeter system $(W,S)$, a commutative ring $\bk$, and a realization $\fh=(V, \{\alpha_s^\vee : s \in S\}, \{\alpha_s : s \in S\})$ of $(W,S)$ over $\bk$.

We will consider the dgg algebras\index{R@$R$}\index{Rv@$R^\vee$}\index{Lambda@$\Lambda$}\index{Lambdav@$\Lambda^\vee$}
\begin{equation}\label{eqn:poly}
\begin{aligned}
R &= \Sym \bigl( V^*(-2) \bigr), &
 R^\vee &= \Sym \bigl( V\la-2\ra \bigr), \\
\Lambda &= \Sym \bigl( V^*(-2)[1] \bigr), &
  \Lambda^\vee &= \Sym \bigl( V\la-2\ra[1] \bigr).
\end{aligned}
\end{equation}
In other words, $R$ is the symmetric algebra on the dgg $\bk$-module that consists of $V^*$ in bidegree $(2,2)$, and likewise for the other rings. In particular, $\Lambda$ is generated by elements of bidegree $(1,2)$, $R^\vee$ by elements of bidegree $(0,-2)$, and $\Lambda^\vee$ by elements of bidegree $(-1,-2)$. Note that $R$ and $R^\vee$ are isomorphic to polynomial rings, whereas $\Lambda$ and $\Lambda^\vee$ are exterior algebras.  We will also occasionally use\index{Rnat@$R^\natural$}
\[
R^\natural = \Sym(V^*\la 2\ra).
\]
All of these are dgg algebras with trivial differential.

Let $\mathsf{a}$ be the cone of the identity map $V^*(-2) \to V^*(-2)$.  Thus, we have $\mathsf{a}^2_2 = \mathsf{a}^1_2 = V^*$; all other $\mathsf{a}^i_j$ vanish; and the differential $\mathsf{a}^1_2 \to \mathsf{a}^2_2$ is the identity map.  Let $\sA$\index{A@{$\sA$}} be the symmetric algebra on $\mathsf{a}$:
\[
\sA = \Sym\left(
\begin{tikzcd}[column sep=tiny, row sep=small]
\text{deg.~$(2,2)$:} & V^* \\
\text{deg.~$(1,2)$:} & V^* \ar[u, "\id"']
\end{tikzcd}
\right)
\]
Note that $R$ is naturally a dgg subalgebra of $\sA$, whereas $\Lambda$ is a quotient of $\sA$.

Let $B$ denote one of the rings defined so far. Each of these rings admits obvious unit and counit maps, denoted by\index{eta@$\eta$}\index{epsilon@$\epsilon$}
\[
\eta_B: \bk \to B
\qquad\text{and}\qquad
\epsilon_B: B \to \bk.
\]
For $\sA$, these maps are both quasi-isomorphisms.  Moreover, the isomorphism
\[
\coH^\bullet_\bullet(\sA) \xrightarrow[\sim]{\epsilon_{\sA}} \bk
\]
is an isomorphism of bigraded $R$-modules.  Thus, $\sA$ can be regarded as a $K$-projective resolution of the trivial $R$-module in the sense of~\cite{spaltenstein}.

\section{Derivations of symmetric and exterior algebras}
\label{sec:derivations}

Define a ``contraction'' map\index{frownl@$\lfrown$}
\[
\lfrown: V\la-2\ra[1] \otimes \Lambda \to \Lambda
\]
as follows: for $x \in V$ and $r_1, \sdots, r_k \in V^*$, we put
\[
x \lfrown (r_1 \wedge \cdots \wedge r_k) = \sum_{i=1}^k (-1)^{i+1} r_i(x) r_1 \wedge \cdots \wedge \widehat{r_{i}} \wedge \cdots \wedge r_k.
\]
Given $x, y \in V$, the maps $x \lfrown({-}), y \lfrown({-}): \Lambda \to \Lambda\la 2\ra[-1]$ anticommute, so there is an induced map
\[
\lfrown: \Lambda^\vee \otimes \Lambda \to \Lambda.
\]
It is easy to see that for $x \in V$, $x \lfrown ({-})$ is a graded derivation of $\Lambda$:
\begin{equation}
\label{eqn:frown-deriv}
x \lfrown (rs) = (x \lfrown r)s + (-1)^{|r|}r(x \lfrown s).
\end{equation}

Similarly, we define a map\index{frownr@$\rfrown$}
\[
\rfrown: R^\vee \otimes V^*\la 2\ra \to R^\vee
\]
as follows: for $x_1, \sdots, x_k \in V$ and $r \in V^*$, we put
\[
(x_1 \cdots x_k) \rfrown r = \sum_{i=1}^k r(x_i) x_1 \cdots \widehat{x_i} \cdots x_k.
\]
Given $r,s \in V^*$, the maps $({-}) \rfrown r, ({-})\rfrown s: R^\vee \to R^\vee\la-2\ra$ commute, so there is an induced map
\[
\rfrown: R^\vee \otimes R^\natural \to R^\vee.
\]
It is easy to see that for $r \in V^*$, $({-}) \rfrown r$ is a derivation of $R^\vee$:
\begin{equation}
\label{eqn:rfrown-deriv}
(xy) \rfrown r = (x \rfrown r) y + x (y \rfrown r).
\end{equation}

\section{Graded modules over polynomial rings}

Let $\bk$ be a commutative ring as above, and consider the $\bk$-algebra $A:=\bk[x_1, \sdots, x_r]$ (for some $r \in \Z_{\geq 0}$), with the grading such that each $x_i$ has degree $-2$. Below we consider $\bk$ as a graded $A$-module, with each $x_i$ acting by $0$.

The following lemma is an application of the graded Nakayama lemma. It is similar to~\cite[Lemma~2.8(2)]{mr:etsps}, but the present assumptions are slightly different since we do not assume that our complex $M$ is bounded.

\begin{lem}
\label{lem:cohomology}
Let $M$ be a complex of graded $A$-modules. Assume that 
\begin{enumerate}
\item
\label{it:lem-cohomology-assumption-1}
for each $i \in \Z$, the graded $A$-module $M^i$ is free of finite rank;
\item
\label{it:lem-cohomology-assumption-2}
for any $i \in \Z \smallsetminus \{0\}$, we have $\coH^i(\bk \otimes_{A} M) = 0$;
\item
\label{it:lem-cohomology-assumption-3}
the graded $\bk$-module $\coH^0(\bk \otimes_A M)$ is free of finite rank.
\end{enumerate}
Then $\coH^i(M)=0$ for any $i \in \Z \smallsetminus \{0\}$, $\coH^0(M)$ is graded free over $A$, and moreover the natural morphism
\[
\bk \otimes_A \coH^0(M) \to \coH^0(\bk \otimes_A M)
\]
is an isomorphism of $\bk$-modules.
\end{lem}

\begin{proof}
For $0 \le j \le r$, let $A_{(j)} = \bk[x_1, \sdots, x_j]$, regarded as a quotient of $A$ by the ideal $(x_{j+1}, \sdots, x_r)$.  We also let $M_{(j)} = M/(x_{j+1},\sdots,x_r)M$.  This is a chain complex whose differentials are denoted by $d_{(j)}^i: M_{(j)}^i \to M_{(j)}^{i+1}$. Note that each term $M_{(j)}^i$ of this complex is a free graded $A_{(j)}$-module of finite rank.  We will begin by proving that
\begin{equation}\label{eqn:cohomology-free}
\coH^i(M_{(j)}) = 0
\qquad\text{for all $i \ne 0$.}
\end{equation}
We proceed by induction on $j$.  If $j = 0$, then this is just assumption~\eqref{it:lem-cohomology-assumption-2} above.  Now suppose that $j > 0$, and that $\coH^i(M_{(j-1)}) = 0$ for all $i \ne 0$.  Let $p_j: M_{(j)} \to M_{(j-1)}$ be the obvious quotient map, and consider the commutative diagram
\begin{equation}\label{eqn:cohomology-free-lift}
\begin{tikzcd}
M_{(j)}^{i-1} \ar[r, "d_{(j)}^{i-1}"] \ar[d, "p_j^{i-1}"] &
M_{(j)}^{i} \ar[r, "d_{(j)}^{i}"] \ar[d, "p_j^{i}"] &
M_{(j)}^{i+1} \ar[d, "p_j^{i+1}"] \\
M_{(j-1)}^{i-1} \ar[r, "d_{(j-1)}^{i-1}"] &
M_{(j-1)}^{i} \ar[r, "d_{(j-1)}^{i}"] &
M_{(j-1)}^{i+1}
\end{tikzcd}
\end{equation}
The bottom row is exact by assumption.  We will prove that the top row is as well by a diagram chase.  If the top row is not exact, there is a homogeneous element $y \in M_{(j)}^i$ such that $d_{(j)}^i(y) = 0$ but that is not in the image of $d_{(j)}^{i-1}$.  Choose such a $y$ of maximal internal degree.  Of course, $p_j^i(y)$ is in the image of $d_{(j-1)}^{i-1}$.  In other words, there is an element $z \in M_{(j)}^{i-1}$ such that $y - d_{(j)}^{i-1}(z) \in \ker p_j^i$.  That is, there is an element $u \in M_{(j)}^i$ such that
\[
y - d_{(j)}^{i-1}(z) = x_ju.
\]
Now apply $d_{(j)}^i$ to this equation: we obtain
\[
0 = d_{(j)}^i(x_ju) = x_j d_{(j)}^i(u).
\]
Since $M_{(j)}^{i+1}$ is a free $A_{(j)}$-module, this implies that $d_{(j)}^i(u) = 0$.  The equation above shows that $\deg u = \deg y + 2$, so by the maximality of the degree of $y$, $u$ must be in the image of $d_{(j)}^{i-1}$, say $u = d_{(j)}^{i-1}(v)$.  Then
\[
y = d_{(j)}^{i-1}(z + x_j v),
\]
contradicting the assumption that $y$ is not in the image of $d_{(j)}^{i-1}$.  Thus, the top row of~\eqref{eqn:cohomology-free-lift} is exact, as desired.

The claim in~\eqref{eqn:cohomology-free} is now proved.  In the special case where $j = r$, it says that
\[
\coH^i(M) = 0
\qquad\text{for all $i \ne 0$.}
\]

Now, let us prove that $\ker(d^0)$ and $\ker(d^1)$ are acyclic for the functor $\bk \otimes_A (-)$. The two cases are similar, so we only consider $\ker(d^0)$. Consider the exact sequence
\[
0 \to \ker(d^0) \to M^0 \xrightarrow{d^0} M^1 \xrightarrow{d^1} M^2 \xrightarrow{d^2} \cdots \xrightarrow{d^{r-2}} M^{r-1} \xrightarrow{d^{r-1}} \ker(d^r) \to 0.
\]
Then for any $i>0$ we deduce an isomorphism
\[
\mathrm{Tor}^A_i(\bk, \ker(d^0)) \cong \mathrm{Tor}^A_{i+r}(\bk, \ker(d^r))=0,
\]
proving the desired vanishing. (Here the fact that
$\mathrm{Tor}^A_{i+r}(\bk, \ker(d^r))=0$ follows from the fact that
one can compute this $\bk$-module using the Koszul resolution of $\bk$
over $A$, which has length $r$.)

Next, we consider the truncated complex $\tau^{\leq 0} M$ with
\[
(\tau^{\leq 0} M)^i = \begin{cases}
M^{i} & \text{if $i<0$;} \\
\ker(d^0) & \text{if $i=0$;} \\
0 & \text{if $i>0$.}
\end{cases}
\]
Then the natural morphism $\tau^{\leq 0} M \to M$ is a quasi-isomorphism since $\coH^i(M)=0$ for $i>0$. We claim that the induced morphism
\[
\bk \otimes_A (\tau^{\leq 0} M) \to \bk \otimes_A M
\]
is also a quasi-isomorphism. Indeed, it is clear by construction and our assumption~\eqref{it:lem-cohomology-assumption-2} respectively that this morphism induces an isomorphism
\[
\coH^i\bigl( (\tau^{\leq 0} M) \otimes_A \bk \bigr) \simto \coH^i \bigl( M \otimes_A \bk \bigr)
\]
when $i<-1$ or when $i>0$. The fact that this map is also an isomorphism when $i=0$ and $i=-1$ follows from the fact that the natural morphism $\bk \otimes_A \ker(d^0) \to \ker(\bk \otimes_A d^0)$ is an isomorphism, which itself follows from the observation that $\ker(d^1) \cong M^0 / \ker(d^0)$ is acyclic for the functor $\bk \otimes_A (-)$.

Combining what we have proved so far we obtain the following isomorphisms in the derived category of graded $\bk$-modules:
\[
\bk \lotimes_A \coH^0(M) \cong
\bk \lotimes_A (\tau^{\leq 0} M) \cong \bk \otimes_A (\tau^{\leq 0} M) \cong \bk \otimes_A M.
\]
(Here the second isomorphism uses the fact that $\ker(d^0)$ is acyclic for $\bk \otimes_A (-)$.)
Then using~\cite[Lemma~2.8(1)]{mr:etsps} and our assumptions~\eqref{it:lem-cohomology-assumption-2} and~\eqref{it:lem-cohomology-assumption-3} we deduce that $\coH^0(M)$ is graded free over $A$. Applying $\coH^0(-)$ to the previous isomorphisms we also deduce the desired isomorphism $\bk \otimes_A \coH^0(M) \cong \coH^0(\bk \otimes_A M)$.
\end{proof}

\chapter{Complexes of Elias--Williamson diagrams}
\label{chap:soergel-diagrams}

Fix a Coxeter system $(W,S)$, an integral domain $\bk$, and a realization $\fh=(V, \{\alpha_s^\vee : s \in S\}, \{\alpha_s : s \in S\})$ of $(W,S)$ over $\bk$.  In this chapter, we will define three triangulated categories in terms of this data, called $\BE(\fh,W)$, $\RE(\fh,W)$, and $\LM(\fh,W)$. They should be regarded as generalizations of the ``mixed modular derived categories'' defined in~\cite{modrap2}.  The main result of this chapter, Theorem~\ref{thm:leftmon-con}, asserts that $\RE(\fh,W)$ and $\LM(\fh,W)$ are actually equivalent.

The definition of $\RE(\fh,W)$ is quite a bit easier than that of $\LM(\fh,W)$, but there is an advantage to working with the latter: it reveals hidden structure in the category, called the \emph{left monodromy action}.  The construction of this action extends an earlier construction due to the second author~\cite{makisumi}.

\section{\texorpdfstring{$\DiagBSp$}{DBS+}-sequences}
\label{sec:ps-convolution}

If $\mathscr{A}$ is an additive category, we define an \emph{$\mathscr{A}$-sequence}\index{sequence@{$\mathscr{A}$-sequence}} to be a sequence $(X^i)_{i \in \Z}$ of objects of $\mathscr{A}$ such that $X^i = 0$ for all but finitely many $i$'s.
In particular we will consider $\DiagBSp(\fh,W)$-sequences, which we will simply call $\DiagBSp$-sequences. Given a $\DiagBSp$-sequence
$\cF$, we define three shift-of-grading functors as follows:\index{shifts!a@{$(1)$}}\index{shifts!b@{$\langle 1\rangle$}}\index{shifts!c@{$[1]$}}
\[
\cF[n]^i = \cF^{i+n},
\qquad
\cF\la n\ra^i = \cF^{i+n}(-n).
\qquad
\cF(n)^i = \cF^i(n).
\]
These shift functors commute with each other and are related by the formula $\la 1 \ra = [1](-1)$.

Given two $\DiagBSp$-sequences $\cF = (\cF^i)_{i \in \Z}$ and $\cG = (\cG^i)_{i \in \Z}$, we define a bigraded $\bk$-module $\uHom_{\BE}(\cF,\cG)$\index{HomuBE@$\uHom_\BE$} by
\[
\uHom_{\BE}(\cF,\cG)^i_j = \prod_{q-p=i-j} \Hom_{\DiagBSp(\fh,W)}(\cF^p,\cG^q(j)).
\]
In this product, only a finite number of terms are nonzero, so the product can also be written as a direct sum. 

\begin{rmk}
The notation ``$\BE$'' stands for ``biequivariant.''  It is motivated by the case of Cartan realizations of crystallographic Coxeter groups; see Chapter~\ref{chap:kac-moody}.  
\end{rmk}

Note that we have canonical identifications
\begin{equation}
\label{eqn:Hom-shift}
\uHom_\BE(\cF,\cG)[n] = \uHom_\BE(\cF,\cG[n]) = \uHom_\BE(\cF[-n],\cG),
\end{equation}
and similarly for $\la n\ra$ and $(n)$.  In particular, for $f \in \uHom_\BE(\cF,\cG)$, we denote by\index{t@$\bt$}
\[
\bt^n(f)
\] 
the corresponding element in any of the spaces in~\eqref{eqn:Hom-shift}. Of course, if $f$ is homogeneous, we have
\[
|\bt^n(f)| = |f| - n.
\]
Suppose $\cH$ is another $\DiagBSp$-sequence.  There is an associative and bigraded ``composition'' map
\begin{equation}\label{eqn:composition-mor}
\circ: \uHom_\BE(\cG,\cH) \otimes \uHom_\BE(\cF,\cG) \to \uHom_\BE(\cF,\cH)
\end{equation}
defined (in the obvious way) as follows: if $f$ is in $\Hom_{\DiagBSp(\fh,W)}(\cF^p,\cG^q(j))$ and $g$ is in $\Hom_{\DiagBSp(\fh,W)}(\cG^q,\cH^r(k))$, then 
\[
g \circ f = (g (j)) \bullet f \in \Hom_{\DiagBSp(\fh,W)}(\cF^p,\cH^r(j+k)),
\]
where on the right-hand side $\bullet$ means composition in $\DiagBSp(\fh,W)$.  Obviously, composition commutes with $\bt$:
\begin{equation}
\label{eqn:composition-shift}
\bt^m(g) \circ \bt^n(f) = \bt^{m+n}(g \circ f).
\end{equation}

As a special case, the composition map makes $\uEnd_\BE(\cF) := \uHom_\BE(\cF,\cF)$ into a bigraded ring. Consider in particular the ring $\uEnd_\BE(B_\varnothing)$, where we regard $B_\varnothing$ as the sequence $(\sdots, 0, B_\varnothing, 0, \sdots)$ concentrated in degree~$0$. It is easy to check that there is a canonical identification
\begin{equation}
\label{eqn:R-End-E1}
R \cong \uEnd_\BE(B_\varnothing).
\end{equation}

Given $\DiagBSp$-sequences $\cF$ and $\cG$, we define their convolution product $\cF \ustar \cG$\index{convolution!staru@$\ustar$} by
\begin{equation}\label{eqn:convolution-obj}
(\cF \ustar \cG)^i = \bigoplus_{p+q=i} \cF^p \star \cG^q.
\end{equation}
For $\DiagBSp$-sequences $\cF$, $\cF'$, $\cG$, $\cG'$, we define a map
\begin{equation}\label{eqn:convolution-mor}
\ustar: \uHom_\BE(\cF,\cG) \otimes \uHom_\BE(\cF',\cG') \to \uHom_\BE(\cF \ustar \cF', \cG \ustar \cG')
\end{equation}
as follows. Let $f \in \Hom(\cF^p,\cG^q(j))$ and $h \in \Hom((\cF')^{p'}, (\cG')^{q'}(j'))$, 
and write $f = f_\even + f_\odd$ as in~\eqref{eqn:morphism-even-odd}. Then we set
\[
f \ustar h = (-1)^{p(q'-p'+j')} f_\even \star h \oplus (-1)^{(p+1)(q'-p'+j')} f_\odd \star h.
\]
One can check from the definitions that~\eqref{eqn:composition-mor} and~\eqref{eqn:convolution-mor} obey a ``signed interchange law'': if $f \in \uHom_\BE(\cF,\cG)$, $g \in \uHom_\BE(\cG,\cH)$, $f' \in \uHom_\BE(\cF',\cG')$, and $g' \in \uHom_\BE(\cG',\cH')$, we have
\begin{equation}\label{eqn:interchange}
(g \circ f) \ustar (g' \circ f') = (-1)^{|f||g'|}(g \ustar g') \circ (f \ustar f').
\end{equation}
Note that convolution with $B_\varnothing$ on the left or on the right makes any $\uHom_\BE(\cF,\cG)$ into both a left and right $R$-module.  In fact, $\uHom_\BE(\cF,\cG)$ is free as both a left and right $R$-module.  (Again, this follows from the existence of the ``double leaves basis''; see~\cite[Theorem~6.11]{ew}.)

\section{Biequivariant complexes}
\label{sec:equ-mixed-complexes}

In this and the following sections, we will define various triangulated categories in terms of $\DiagBSp$-sequences.  In each case, the objects will be $\DiagBSp$-sequences with additional data.

\begin{defn}
The \emph{biequivariant category} of $(\fh,W)$, denoted by\index{categories@{categories of $\DiagBSp$-sequences}!BEhW@$\BE(\fh,W)$}\index{biequivariant}
\[
\BE(\fh,W),
\]
is the category whose objects are pairs $(\cF,\delta)$, where $\cF$ is a $\DiagBSp$-sequence, and $\delta \in \uEnd_\BE(\cF)^{1}_{0}$ is an element such that $\delta \circ \delta = 0$.

Before defining morphisms, we observe that for any two objects $(\cF,\delta_\cF)$, $(\cG,\delta_\cG)$ in $\BE(\fh,W)$, the bigraded $\bk$-module $\uHom_\BE(\cF,\cG)$ can be made into a dgg $\bk$-module with the differential
\[
d_{\uHom_\BE(\cF,\cG)}(f) = \delta_\cG \circ f + (-1)^{|f|+1} f \circ \delta_\cF.
\]
Then morphisms in $\BE(\fh,W)$ are given by
\[
\Hom_{\BE(\fh,W)}(\cF,\cG) = \coH^{0}_0(\uHom_\BE(\cF,\cG)).
\]
Composition is defined in the obvious way (see below).
\end{defn}

If $\bk'$ is another integral domain and $\varphi : \bk \to \bk'$ is a ring morphism, then the functor $\bk'$ of~\S\ref{sec:additive-hull-Diag} induces a functor
\[
\BE(\fh,W) \to \BE( \bk' \otimes_\bk \fh, W),
\]
which will also be denoted $\bk'$. In fact, for $\cF$ and $\cG$ in $\BE(\fh,W)$, the morphism
\[
\Hom_{\BE(\fh,W)}(\cF, \cG) \to \Hom_{\BE(\bk' \otimes_\bk \fh,W)}(\bk'(\cF), \bk'(\cG))
\]
induced by $\bk'$ is obtained from the isomorphism
\[
\bk' \otimes_\bk \uHom_\BE(\cF, \cG) \simto \uHom_\BE(\bk'(\cF), \bk'(\cG))
\]
induced by the functor of~\S\ref{sec:additive-hull-Diag}.

Note that $\BE(\fh,W)$ is simply the bounded homotopy category of the additive category $\DiagBSp(\fh,W)$; in particular it has a natural triangulated structure, with shift functor defined by
\[
(\cF, \delta) [1] = (\cF[1], -\delta).
\]
The operation $\langle n \rangle$ on $\DiagBSp$-sequences induces an autoequivalence of $\BE(\fh,W)$, denoted similarly and which satisfies
\[
(\cF, \delta) \langle n \rangle = (\cF \langle n \rangle, \delta).
\]
Then we set $(n):=\langle -n \rangle[n]$ (so that this equivalence multiplies differentials by $(-1)^n$).

We will denote the total cohomology of $\uHom_\BE(\cF,\cG)$ by
\[
\gHom_\BE(\cF,\cG) = \coH^{\bullet}_{\bullet}(\uHom_\BE(\cF,\cG)).
\]
This is a bigraded $\bk$-module.  It is easy to check that with the differentials defined above, the map~\eqref{eqn:composition-mor} is a map of chain complexes, so it induces a map
\[
\circ: \gHom_\BE(\cG,\cH) \otimes \gHom_\BE(\cF,\cG) \to \gHom_\BE(\cF,\cH).
\]
In particular, composition of morphisms in $\BE(\fh,W)$ is defined by taking the bidegree-$(0,0)$ component of each term above. 
Under the natural identification $\BE(\fh,W) \cong \Kb(\DiagBSp(\fh,W))$, we clearly have
\[
 \gHom_\BE(\cF,\cG)^i_j = \Hom_{\Kb(\DiagBSp(\fh,W))} \bigl( \cF, \cG[i-j](j) \bigr).
\]
We also set $\gEnd_\BE(\cF) := \gHom_\BE(\cF, \cF)$.

We define a convolution operation on objects of $\BE(\fh,W)$ by equipping the $\DiagBSp$-sequence $\cF \ustar \cG$ with the differential
\[
\delta_{\cF \ustar \cG} := \delta_\cF \ustar \id + \id \ustar \delta_\cG.
\]
Using the signed interchange law~\eqref{eqn:interchange}, one can check that $\delta_{\cF \ustar \cG} \circ \delta_{\cF \ustar \cG} = 0$, as required.  Moreover, one can check using again the signed interchange law that the convolution map~\eqref{eqn:convolution-mor} is a map of chain complexes, so we get an induced map
\[
\ustar: \gHom_\BE(\cF,\cG) \otimes \gHom_\BE(\cF',\cG') \to \gHom_\BE(\cF \ustar \cF', \cG \ustar \cG')
\]
that again obeys a signed interchange law.  These considerations give us a well-defined bifunctor\index{convolution!staru@$\ustar$}
\[
\ustar: \BE(\fh,W) \times \BE(\fh,W) \to \BE(\fh,W),
\]
making $\BE(\fh,W)$ into a monoidal category, with unit object $B_\varnothing$, considered as a $\DiagBSp$-sequence concentrated in position $0$, and endowed with the trivial differential.
  
\begin{ex}
\label{ex:std-costd}
Let $s \in S$.  Consider the $\DiagBSp$-sequence
\[
(\sdots, 0, B_s, B_\varnothing(1), 0, \sdots)
\]
with nonzero terms in positions $0$ and $1$.  We define $\Delta_s \in \BE(\fh,W)$ to be the pair consisting of this sequence together with the differential given by
\[
\delta = 
\begin{tikzpicture}[scale=-0.3,thick,baseline]
 \draw (0,0) to (0,1);
 \node at (0,0) {$\bullet$};
 \node at (0,1.4) {\tiny $s$};
\end{tikzpicture}
\in \Hom_{\DiagBS(\fh,W)}(B_s,B_\varnothing(1)) = 
\uEnd_{\BE}(\Delta_s)^1_0.
\]
A more concise way to describe this is with the following picture:
\[
\Delta_s = 
\begin{tikzcd}
B_\varnothing(1) \\
B_s. \arrow{u}{\usebox\upperdot}
\end{tikzcd}
\]
Similarly, we define $\nabla_s$ to be the object
\[
\nabla_s =
\begin{tikzcd}
B_s \\
B_\varnothing(-1), \arrow{u}{\usebox\lowerdot}
\end{tikzcd}
\]
where the underlying $\DiagBSp$-sequence is concentrated in positions $-1$ and $0$.

More generally, for any expression $\uw = (s_1, \sdots, s_k)$, we define\index{standard object!Deltauw@$\Delta_\uw$}\index{costandard object!nablauw@$\nabla_\uw$}
\[
\Delta_\uw = \Delta_{s_1} \ustar \cdots \ustar \Delta_{s_k}
\qquad\text{and}\qquad
\nabla_\uw = \nabla_{s_1} \ustar \cdots \ustar \nabla_{s_k}.
\]
In the case where $\uw$ is a reduced expression, we call $\Delta_\uw$ a \emph{standard object}, and $\nabla_\uw$ a \emph{costandard object}.
\end{ex}

\begin{rmk}
It is likely that, under suitable assumptions, for any reduced expression $\uw \in \hW$ the objects $\Delta_\uw$ and $\nabla_\uw$ depend only on $\pi(\uw) \in W$ (up to isomorphism).  We will see that this holds in the case of Cartan realizations of crystallographic Coxeter groups in Chapter~\ref{chap:kac-moody}. It also follows from Proposition~\ref{prop:Rouquier-convolution-new} that this property holds if the conditions of Chapter~\ref{chap:dihedral} are satisfied for any pair of simple reflections generating a finite subgroup of $W$.
\end{rmk}

\begin{lem}
\label{lem:convolution-DN-new}
Let $s \in S$. There exists isomorphisms
\[
\Delta_s \ustar \nabla_s \cong \nabla_s \ustar \Delta_s \cong B_\varnothing
\]
in $\BE(\fh,W)$.
\end{lem}

\begin{proof}
We will prove that $\Delta_s \ustar \nabla_s \cong B_\varnothing$; the proof that $\nabla_s \ustar \Delta_s \cong B_\varnothing$ is similar.

By definition, the biequivariant complex $\Delta_s \ustar \nabla_s$ can be depicted as follows:
\[
B_s(-1) \xrightarrow{
\begin{bmatrix}
\begin{array}{c}
\begin{tikzpicture}[yscale=0.3,xscale=0.2,thick]
\draw (0,-1) -- (0,0);
\node at (0,0) {\small $\bullet$};
\end{tikzpicture}
\end{array} \\
- \begin{array}{c}
\begin{tikzpicture}[yscale=0.3,xscale=0.2,thick]
\draw (-1,-1) -- (-1,1);
\draw (0,1) -- (0,0);
\node at (0,0) {\small $\bullet$};
\end{tikzpicture}
\end{array}
\end{bmatrix}
} B_\varnothing \oplus B_{(s,s)} \xrightarrow{
\begin{bmatrix}
\begin{array}{c}
\begin{tikzpicture}[yscale=0.3,xscale=0.2,thick]
\draw (0,1) -- (0,0);
\node at (0,0) {\small $\bullet$};
\end{tikzpicture}
\end{array} &
\begin{array}{c}
\begin{tikzpicture}[yscale=0.3,xscale=0.2,thick]
\draw (1,-1) -- (1,1);
\draw (0,-1) -- (0,0);
\node at (0,0) {\small $\bullet$};
\end{tikzpicture}
\end{array}
\end{bmatrix}
}
B_s(1).
\]
We consider the morphisms $f : \Delta_s \ustar \nabla_s \to B_\varnothing$ and $g : B_\varnothing \to \Delta_s \ustar \nabla_s$ defined by the following diagram:
\[
\begin{tikzcd}[row sep=1.5cm, ampersand replacement=\&]
B_s(-1) \ar[r] \ar[d, shift right=0.5ex] \&
  B_\varnothing \oplus B_{(s,s)} \ar[r] \ar[d, shift right=0.5ex, "{\begin{bmatrix}\displaystyle 1 & \usebox\capmor \end{bmatrix}}"'] \&
  B_s(1) \ar[d, shift right=0.5ex] \\
0 \ar[r] \ar[u, shift right=0.5ex] \&
  B_\varnothing \ar[r] \ar[u, shift right=0.5ex, "{\begin{bmatrix} \displaystyle 1 \\ -\usebox\cupmor\end{bmatrix}}"'] \&
  0. \ar[u, shift right=0.5ex]
\end{tikzcd}
\]
Then $f \circ g=\id$, and the following morphism $h$ satisfies $\id-g \circ f=d(h)$:
\[
\begin{tikzcd}[row sep=1.5cm, ampersand replacement=\&]
B_s(-1) \ar[r] \&
  B_\varnothing \oplus B_{(s,s)} \ar[r] \ar[ld, "{\begin{bmatrix} \displaystyle 0 & -\usebox\invymor \end{bmatrix}}" description] \&
  B_s(1) \ar[ld, "{\begin{bmatrix} \displaystyle 0 \\ \usebox\ymor \end{bmatrix}}" description] \\
B_s(-1) \ar[r] \& B_\varnothing \oplus B_{(s,s)} \ar[r] \& B_s(1).
\end{tikzcd}
\]
(This computation uses the equality
\begin{equation}
\label{eqn:formula-BsBs}
\begin{array}{c} \begin{tikzpicture}[xscale=0.3,yscale=0.5,thick]
\draw (-1,-1) -- (-1,1);
\draw (0,-1) -- (0,1);
\end{tikzpicture}
\end{array}
+
\begin{array}{c} \begin{tikzpicture}[xscale=0.3,yscale=0.5,thick]
\draw (-1,-1) to[out=90, in=180] (0,-0.3) to[out=0, in=90] (1,-1);
\draw (-1,1) to[out=-90, in=180] (0,0.3) to[out=0,in=-90] (1,1);
\end{tikzpicture}
\end{array}
=
\begin{array}{c} \begin{tikzpicture}[xscale=0.3,yscale=0.5,thick]
\draw (0,-1) -- (0,0);
\node at (0,0) {\small $\bullet$};
\draw (1,-1) -- (1,0);
\draw (0,1) -- (1,0) -- (2,1);
\end{tikzpicture}
\end{array}
+
\begin{array}{c} \begin{tikzpicture}[xscale=-0.3,yscale=-0.5,thick]
\draw (0,-1) -- (0,0);
\node at (0,0) {\small $\bullet$};
\draw (1,-1) -- (1,0);
\draw (0,1) -- (1,0) -- (2,1);
\end{tikzpicture}
\end{array}
\end{equation}
in $\DiagBS(\fh,W)$.)
\end{proof}

\section{Right-equivariant complexes}
\label{sec:con-mixed-complexes}

Given $\DiagBSp$-sequences $\cF, \cG$, we define a new bigraded $\bk$-module $\uHom_{\RE}(\cF,\cG)$ by\index{HomuRE@$\uHom_\RE$}
\[
\uHom_{\RE}(\cF,\cG) := \bk \otimes_R \uHom_{\BE}(\cF,\cG).
\]
Note that the composition map~\eqref{eqn:composition-mor} induces a map
\[
\circ: \uHom_{\RE}(\cG,\cH) \otimes \uHom_{\RE}(\cF,\cG) \to \uHom_{\RE}(\cF,\cH).
\]
We will also write $\uEnd_{\RE}(\cF)$ for $\uHom_{\RE}(\cF,\cF)$.

\begin{rmk}
 Here, ``$\RE$'' stands for ``right equivariant.'' Again, this terminology is motivated by the case of Cartan realizations of crystallographic Coxeter groups; see Chapter~\ref{chap:kac-moody}. 
\end{rmk}

\begin{defn}\label{defn:constructible}
The \emph{right-equivariant category} of $(\fh,W)$, which we denote by\index{categories@{categories of $\DiagBSp$-sequences}!BREhW@$\RE(\fh,W)$}\index{right-equivariant}
\[
\RE(\fh,W),
\]
is the category whose objects are pairs $(\cF,\delta)$, where $\cF$ is a $\DiagBSp$-sequence, and $\delta \in \uEnd_{\RE}(\cF)^{1}_{0}$ is an element such that $\delta \circ \delta = 0$.

For two such objects $(\cF,\delta_\cF)$ and $(\cG,\delta_\cG)$, we make $\uHom_{\RE}(\cF,\cG)$ into a chain complex with the differential
\[
d_{\uHom_{\RE}(\cF,\cG)}(f) = \delta_\cG \circ f + (-1)^{|f|+1} f \circ \delta_\cF.
\]
Then morphisms in $\RE(\fh,W)$ are given by
\[
\Hom_{\RE(\fh,W)}(\cF,\cG) = \coH^{0}_{0}(\uHom_{\RE}(\cF,\cG)),
\]
with the natural composition maps.
\end{defn}

As in~\S\ref{sec:equ-mixed-complexes},
if $\bk'$ is another integral domain and $\varphi : \bk \to \bk'$ is a ring morphism, then we have an ``extension of scalars'' functor
\[
\bk' : \RE(\fh,W) \to \RE(\bk' \otimes_\bk \fh, W),
\]
induced on morphisms by the natural isomorphism of complexes
\[
\bk' \otimes_\bk \uHom_{\RE}(\cF, \cG) \simto \uHom_{\RE}(\bk'(\cF), \bk'(\cG))
\]
for $\cF, \cG$ in $\RE(\fh,W)$.

It is clear from the construction that $\RE(\fh,W)$ is equivalent to the bounded homotopy category of the additive category $\oDiagBSp(\fh,W)$; in particular this category has a natural triangulated structure.
The operation $\langle n \rangle$ on $\DiagBSp$-sequences induces an autoequivalence of $\RE(\fh,W)$, which we will also denote $\langle n \rangle$. Then we set $(n):=\langle -n \rangle[n]$.

As in~\S\ref{sec:equ-mixed-complexes}, we also put
\[
\gHom_{\RE}(\cF,\cG) := \coH^{\bullet}_{\bullet}(\uHom_{\RE}(\cF,\cG)), \quad \gEnd_{\RE}(\cF) := \gHom_{\RE}(\cF, \cF).
\]

For any two $\DiagBSp$-sequences $\cF,\cG$, there is an obvious map
\begin{equation}\label{eqn:uhom-eqvt-con}
\uHom_{\BE}(\cF,\cG) \to \uHom_{\RE}(\cF,\cG)
\end{equation}
that commutes with composition. This lets us define a functor\index{forgetful functor!ForBERE@$\ForBERE$}
\[
\ForBERE: \BE(\fh,W) \to \RE(\fh,W)
\]
as follows: it sends an object $(\cF,\delta)$ to the pair $(\cF,\bar\delta)$ where $\bar\delta$ is the image of $\delta$ under $\uEnd_\BE(\cF) \to \uEnd_{\RE}(\cF)$, and it acts on morphisms by the map
\[
\coH^{0}_{0}(\uHom_\BE(\cF,\cG)) \to \coH^{0}_{0}(\uHom_{\RE}(\cF,\cG))
\]
induced by~\eqref{eqn:uhom-eqvt-con}. This functor is triangulated.

The category~$\RE(\fh,W)$ is not monoidal, but it retains a convolution action of $\BE(\fh,W)$ on the right:\index{convolution!staru@$\ustar$}
\[
\ustar: \RE(\fh,W) \times \BE(\fh,W) \to \RE(\fh,W).
\]
The following result is immediate from the definitions.

\begin{lem}
\label{lem:convolution-be-re}
For $\cF, \cG \in \BE(\fh,W)$, there is a natural isomorphism
\[
\ForBERE(\cF \ustar \cG) \cong \ForBERE(\cF) \ustar \cG.
\]
\end{lem}

\begin{ex}
\label{ex:ts-con}
Let $s \in S$. Consider the $\DiagBSp$-sequence
\[
(\sdots, 0, B_\varnothing(-1), B_s, B_\varnothing (1), 0, \sdots),
\]
with nonzero terms in positions $-1$, $0$, and $1$. We define the object $\cT_s'$ as the pair consisting of this $\DiagBSp$-sequence together with the differential\index{tilting object!Tsp@$\cT_s'$}
given by
\begin{multline*}
\delta = 
\begin{tikzpicture}[scale=0.3,thick,baseline]
 \draw (0,0) to (0,1);
 \node at (0,0) {$\bullet$};
 \node at (0,1.4) {\tiny $s$};
\end{tikzpicture}
\oplus
\begin{tikzpicture}[scale=-0.3,thick,baseline]
 \draw (0,0) to (0,1);
 \node at (0,0) {$\bullet$};
 \node at (0,1.4) {\tiny $s$};
\end{tikzpicture}
\in \Hom_{\oDiagBS(\fh,W)}(B_\varnothing(-1),B_s)
\oplus \Hom_{\oDiagBS(\fh,W)}(B_s,B_\varnothing(1)) = 
\\
\uEnd_{\RE}(\cT'_s)^1_0.
\end{multline*}
A more concise way to describe this is with the following picture:
\[
\begin{tikzcd}
B_\varnothing(1) \\
B_s \arrow{u}{\usebox\upperdot}
\\
B_\varnothing(-1). \arrow{u}{\usebox\lowerdot}
\end{tikzcd}
\]

In the setting of Cartan realizations of crystallographic Coxeter groups, this object corresponds to a tilting mixed perverse sheaf in the sense of~\cite{modrap2}; see~\S\ref{sec:parity-sequences} for details. In general, writing explicitly the complex $\uEnd_\RE(\cT_s')$, it is not difficult to check that $\gEnd_\RE(\cT_s')^i_j=\{0\}$ unless $(i,j) \in \{(0,0), (0,-2)\}$. The $\bk$-module $\gEnd_\RE(\cT_s')^0_0$ is free of rank $1$, generated by $\id_{\cT_s'}$, while $\gEnd_\RE(\cT_s')^0_{-2}$ is also free of rank $1$, and generated by the morphism
\[
\begin{tikzcd}[column sep=large]
B_\varnothing(1) \\
B_s \ar[u, "\usebox\upperdot" swap] \\
B_\varnothing(-1) \ar[u, "\usebox\lowerdot" swap]
\ar[r, "\id"]
&B_\varnothing(-1) \\
&B_s(-2) \ar[u, "\usebox\upperdot" swap] \\
&B_\varnothing(-3). \ar[u, "\usebox\lowerdot" swap]
\end{tikzcd}
\]
\end{ex}

\begin{lem}
\label{lem:End-Delta}
For any expression $\uw$, we have
\[
\gEnd_\RE(\ForBERE(\Delta_\uw)) \cong \bk.
\]
\end{lem}
\begin{proof}
Suppose $\uw = (s_1, \sdots, s_r)$.  We have
\[
\ForBERE(\Delta_\uw) \cong \ForBERE(B_\varnothing \ustar \Delta_{s_1} \ustar \cdots \ustar \Delta_{s_r})
\cong \ForBERE(B_\varnothing) \ustar \Delta_{s_1} \ustar \cdots \ustar \Delta_{s_r}
\]
by Lemma~\ref{lem:convolution-be-re}.
Now, Lemma~\ref{lem:convolution-DN-new} implies that the functor $\RE(\fh,W) \to \RE(\fh,W)$ given by $\cF \mapsto \cF \ustar \Delta_s$ is an equivalence of categories.  The calculation above then implies that $\gEnd_\RE(\ForBERE(\Delta_\uw)) \cong \gEnd_\RE(B_\varnothing) \cong \bk$.
\end{proof}

\section{Left-monodromic complexes}
\label{sec:lmon-mixed-complexes}

Given $\DiagBSp$-sequences $\cF, \cG$, we define a new bigraded $\bk$-module $\uHom_{\LM}(\cF,\cG)$ by\index{HomuLM@$\uHom_\LM$}
\[
\uHom_\LM(\cF,\cG) := \Lambda \otimes_\bk \uHom_\BE(\cF,\cG).
\]
We equip it with a composition map
\begin{equation}
\label{eqn:composition-uHomU}
\circ: \uHom_\LM(\cG,\cH) \otimes \uHom_\LM(\cF,\cG) \to \uHom_\LM(\cF,\cH)
\end{equation}
given by
\[
(r \otimes f) \circ (s \otimes g) = (-1)^{|f||s|} (r \wedge s) \otimes (f \circ g),
\]
where $r,s \in \Lambda$, $f \in \uHom_\BE(\cG,\cH)$, and $g \in \uHom_\BE(\cF,\cG)$. We will also write $\uEnd_\LM(\cF)$ for $\uHom_\LM(\cF,\cF)$. (Here ``$\LM$'' stands for ``left monodromic.'')

Let $\bs: \uHom_\LM(\cF,\cG) \to \uHom_\LM(\cF,\cG)$\index{s@$\bs$} be the map given by
\[
\bs(r \otimes f) = (-1)^{|r|}r \otimes f
\]
for $r \in \Lambda$ and $f \in \uHom_\BE(\cF,\cG)$.  It is clear that $\bs$ is an involution, and it is easy to check that for $f \in \uHom_\LM(\cF,\cG)$ and $g \in \uHom_\LM(\cG,\cH)$, we have
\begin{equation}\label{eqn:composition-bs}
\bs(g \circ f) = \bs(g) \circ \bs(f).
\end{equation}
The morphism $\bt : \uHom_\BE(\cF,\cG) \to \uHom_\BE(\cF,\cG[1])$ induces a morphism from $\uHom_\LM(\cF,\cG)$ to $\uHom_\LM(\cF,\cG[1])$, which will also be denoted $\bt$.
Note that, unlike in~\eqref{eqn:composition-shift}, composition in $\uHom_\LM$ does \emph{not} commute with $\bt$.  Instead, it is easy to check that we have
\begin{equation}\label{eqn:composition-shift-leftmon}
\bt^m(g) \circ \bt^n(f) = \bt^{m+n}(g \circ \bs^m(f)).
\end{equation}
(This formula is easier to remember if one thinks of $\bt$ as an operation \emph{on the right}. Then pulling $\bt$ across a morphism $f$ results in applying $\bs$ to $f$.)
Note also that $\bt$ and $\bs$ commute.

Let $\kappa: \uHom_\LM(\cF, \cG) \to \uHom_\LM(\cF,\cG)[1]$\index{kappa@{$\kappa$}} be the map given by
\[
\kappa((r_1 \wedge \cdots \wedge r_k) \otimes f) = \sum_{i = 1}^k (-1)^{i+1} (r_1 \wedge \cdots \wedge \widehat{r_i} \wedge \cdots \wedge r_k) \otimes (r_i \ustar f)
\]
for $r_1, \sdots, r_k \in V^*$ and $f \in \uHom_\BE(\cF, \cG)$.  (Note that on the left-hand side of the tensor sign, the $r_j$'s belong to $\Lambda^{1}_{2}$, while on the right-hand side, they belong to $R^{2}_{2}$.) It is easy to check that
\begin{equation}\label{eqn:kappa-square}
\kappa \circ \kappa = 0
\end{equation}
and that for all $f \in \uHom_\LM(\cF,\cG)$ and $g \in \uHom_\LM(\cG, \cH)$ we have
\begin{align}
\kappa(\bt(f)) &=\bt(\kappa(f)), \label{eqn:kappa-bt} \\
\kappa(\bs(f)) &= -\bs(\kappa(f)), \label{eqn:kappa-bs} \\
\kappa(g \circ f) &= \kappa(g) \circ f + (-1)^{|g|} g \circ \kappa(f). \label{eqn:kappa-derivation}
\end{align}

\begin{lem}\label{lem:uhom-mon-con}
Let $\cF$ and $\cG$ be two $\DiagBSp$-sequences.  Regard $\uHom_\LM(\cF,\cG)$ as a chain complex with differential $\kappa$.  Then $\uHom_\LM(\cF,\cG)$ is naturally a $K$-projective left dgg $\sA$-module.  Moreover, there is a natural isomorphism of bigraded $\bk$-modules
\[
\bk \otimes_\sA \uHom_\LM(\cF,\cG) \simto \uHom_{\RE}(\cF,\cG).
\]
\end{lem}

\begin{proof}
The first assertion follows from the last sentence in~\S\ref{sec:ps-convolution}. The second assertion is clear from the definitions.
\end{proof}

\begin{defn}\label{defn:left-mon}
The \emph{left-monodromic category} of $(\fh,W)$, denoted by\index{categories@{categories of $\DiagBSp$-sequences}!BLMhW@$\LM(\fh,W)$}\index{left-monodromic}
\[
\LM(\fh,W),
\]
is the category defined as follows.
\begin{itemize}
\item
The objects are pairs $(\cF,\delta)$, where $\cF$ is a $\DiagBSp$-sequence and $\delta$ is an element of $\uEnd_\LM(\cF)^{1}_{0}$ such that $\delta \circ \delta + \kappa(\delta) = 0$.
\item
For two such objects $(\cF,\delta_\cF)$ and $(\cG,\delta_\cG)$, we make $\uHom_{\LM}(\cF,\cG)$ into a dgg $\bk$-module with the differential
\[
d_{\uHom_\LM(\cF,\cG)}(f) = \kappa(f) + \delta_\cG \circ f + (-1)^{|f|+1} f \circ \delta_\cF.
\]
Then morphisms in $\LM(\fh,W)$ are given by
\[
\Hom_{\LM(\fh,W)}(\cF,\cG) = \coH^{0}_{0}(\uHom_\LM(\cF,\cG)).
\]
\item
For three objects $(\cF,\delta_\cF)$, $(\cG,\delta_\cG)$ and $(\cH,
\delta_\cH)$, using~\eqref{eqn:kappa-derivation} one can check that
the composition map~\eqref{eqn:composition-uHomU} is a morphism of dgg $\bk$-modules; hence it induces a morphism
\[
\Hom_{\LM(\fh,W)}(\cG,\cH) \otimes \Hom_{\LM(\fh,W)}(\cF,\cG) \to \Hom_{\LM(\fh,W)}(\cF,\cH),
\]
which defines the composition in $\LM(\fh,W)$.
\end{itemize}
\end{defn}

Let us verify that the differential given above actually makes $\uHom_\LM(\cF,\cG)$ into a chain complex.  We have:
\begin{multline}\label{eqn:leftmon-diff}
d(d(f)) = \kappa(\kappa(f)) + \kappa(\delta_\cG f) + (-1)^{|f|+1} \kappa(f \delta_\cF) \\
+ \delta_\cG \kappa(f) + \delta_\cG^2 f + (-1)^{|f|+1} \delta_\cG f\delta_\cF \\
+ (-1)^{|f|} \kappa(f) \delta_\cF + (-1)^{|f|} \delta_\cG f \delta_\cF - f \delta_\cF^2.
\end{multline}
Using~\eqref{eqn:kappa-square} and~\eqref{eqn:kappa-derivation}, this simplifies to
\[
\kappa(\delta_\cG)f + \delta_\cG^2 f  - f \kappa(\delta_\cF) - f \delta_\cF^2.
\]
The definition of left-monodromic complexes tells us that this vanishes.

We will use the term \emph{chain map}\index{chain map} to refer to an element $f \in \uHom_\LM(\cF,\cG)^0_0$ satisfying $d_{\uHom_\LM(\cF,\cG)}(f) = 0$. Such an element defines a morphism from $\cF$ to $\cG$ in $\LM(\fh,W)$.

As in~\S\S\ref{sec:equ-mixed-complexes}--\ref{sec:con-mixed-complexes},
if $\bk'$ is another integral domain and $\varphi : \bk \to \bk'$ is a ring morphism, then we have an ``extension of scalars'' functor
\[
\bk' : \LM(\fh,W) \to \LM(\bk' \otimes_\bk \fh, W),
\]
induced on morphisms by the natural isomorphism of complexes
\begin{equation}
\label{eqn:ext-scalars-uHomU}
\bk' \otimes_\bk \uHom_{\LM}(\cF, \cG) \simto \uHom_{\LM}(\bk'(\cF), \bk'(\cG))
\end{equation}
for $\cF$, $\cG$ in $\LM(\fh,W)$.

The operation $\langle n \rangle$ on $\DiagBSp$-sequences defines an autoequivalence of the category $\LM(\fh,W)$ (which does not change the differential); this autoequivalence will also be denoted by $\langle n \rangle$.

As in the previous sections, we also put
\[
\gHom_\LM(\cF,\cG) := \coH^{\bullet}_{\bullet}(\uHom_\LM(\cF,\cG)), \quad \gEnd_\LM(\cF) := \gHom_\LM(\cF, \cF).
\]

\section{Triangulated structure}
\label{sec:LM-triangulated}

We now explain how to equip $\LM(\fh,W)$ with the structure of a triangulated category.  For an object $(\cF,\delta_\cF) \in \LM(\fh,W)$, we define the shift operation by
\[
(\cF,\delta_\cF)[1] = (\cF[1], -\bs(\delta_\cF)).
\]
(Using~\eqref{eqn:composition-bs} and~\eqref{eqn:kappa-bs}, one can see that $-\bs(\delta_\cF)$ satisfies the defining condition for a differential, so that the right-hand side is indeed an object of $\LM(\fh,W)$.)  For a morphism $f: \cF \to \cG$, we define
\[
f[1]: \cF[1] \to \cG[1]
\qquad\text{by}\qquad
f[1] = \bs(f).
\]
To be more precise, choose a chain map $\tilde f$ corresponding to $f$, and define $f[1]$ to be the morphism corresponding to $\bs(\tilde f) \in \uHom_\LM(\cF,\cG) = \uHom_\LM(\cF [1],\cG[1])$.  It is left to the reader to check that $\bs(\tilde f)$ is indeed a chain map $\cF[1] \to \cG[1]$, and that the corresponding morphism in $\LM(\fh,W)$ is independent of the choice of $\tilde f$.

Next, given a chain map $f: \cF \to \cG$, we will define a new left-monodromic complex denoted by $\cone(f)$.  Its underlying $\DiagBSp$-sequence will be $\cF[1] \oplus \cG$.  We will write elements of $\uEnd_\LM(\cF[1] \oplus \cG)$ by writing them as matrices
\[
\begin{bmatrix}
a & b \\ c & d
\end{bmatrix}
\in 
\begin{bmatrix}
\uEnd_\LM(\cF[1]) & \uHom_\LM(\cG,\cF[1]) \\
\uHom_\LM(\cF[1],\cG) & \uEnd_\LM(\cG)
\end{bmatrix}.
\]
Let $\cone(f) = (\cF[1]\oplus \cG, \delta_{\cone(f)})$ be the complex with differential given by
\[
\delta_{\cone(f)} =
\begin{bmatrix}
-\bs(\delta_\cF) & 0 \\
\bt^{-1}(f) & \delta_\cG
\end{bmatrix}.
\]
Let us check that this is indeed a differential: we have
\begin{multline*}
\delta_{\cone(f)} \circ \delta_{\cone(f)} + \kappa(\delta_{\cone(f)}) \\
=
\begin{bmatrix}
-\bs(\delta_\cF) & 0 \\
\bt^{-1}(f) & \delta_\cG
\end{bmatrix}
\begin{bmatrix}
-\bs(\delta_\cF) & 0 \\
\bt^{-1}(f) & \delta_\cG
\end{bmatrix}
+
\begin{bmatrix}
\kappa(-\bs(\delta_\cF)) & 0 \\
\kappa(\bt^{-1}(f)) & \kappa(\delta_\cG)
\end{bmatrix} \\
= \begin{bmatrix}
\bs(\delta_\cF \circ \delta_\cF + \kappa(\delta_\cF)) & 0 \\
\kappa(\bt^{-1}(f)) + \delta_\cG \circ \bt^{-1}(f) - \bt^{-1}(f) \circ \bs(\delta_\cF) &
\delta_\cG \circ \delta_\cG + \kappa(\delta_\cG)
\end{bmatrix}
\end{multline*}
The diagonal entries clearly vanish.  For the lower left entry, using~\eqref{eqn:composition-shift-leftmon} and~\eqref{eqn:kappa-bt}, we have
\begin{multline*}
\kappa(\bt^{-1}(f)) + \delta_\cG \circ \bt^{-1}(f) - \bt^{-1}(f) \circ \bs(\delta_\cF)  \\
= \bt^{-1} \bigl( \kappa(f) + \delta_\cG \circ f - f \circ \delta_\cF \bigr) = \bt^{-1} \bigl( d_{\uHom_\LM(\cF,\cG)}(f) \bigr) = 0,
\end{multline*}
as desired.

Any diagram in $\LM(\fh,W)$ isomorphic to one of the form
\[
\cF \xrightarrow{f} \cG \xrightarrow{\alpha(f)} \cone(f) \xrightarrow{\beta(f)} \cF[1]
\]
for $f$ a chain map is called a \emph{distinguished triangle}.
Here, $\alpha(f)$ and $\beta(f)$ are the obvious chain maps, given by
\[
\alpha(f) =
\begin{bmatrix}
0 \\ \id_\cG
\end{bmatrix},
\qquad
\beta(f) = \begin{bmatrix} \id_{\cF[1]} & 0 \end{bmatrix}.
\]

\begin{prop} \label{prop:UGB-left-triangulated-structure}
With the definitions above, $\LM(\fh,W)$ has the structure of a triangulated category.
\end{prop}

\begin{proof}[Proof sketch]
The arguments of~\cite[Lemma~I.4.2 and Proposition~I.4.4]{ks} can readily be adapted to our setting.  For instance, the proof of axiom (TR3) from~\cite[Proposition~I.4.4]{ks} involves the study of the following diagram of chain maps:
\[
\begin{tikzcd}[column sep=large]
\cG \ar[r, "\alpha(f)"] \ar[d, equal] & \cone(f) \ar[r, "\beta(f)"] \ar[d, equal]
  & \cF[1] \ar[r, "-f{[1]}"] \ar[d, "\phi", bend left=15] & \cG[1] \ar[d, equal] \\
\cG \ar[r, "\alpha(f)"] & \cone(f) \ar[r, "\alpha(\alpha(f))"]
  & \cone(\alpha(f)) \ar[r, "\beta(\alpha(f))"] \ar[u, "\psi", bend left=15] & \cG[1].
\end{tikzcd}
\]
Here, $\cone(\alpha(f))$ has underlying $\DiagBSp$-sequence $\cG[1] \oplus \cF[1] \oplus \cG$, and $\phi: \cF[1] \to \cone(\alpha(f))$ and $\psi: \cone(\alpha(f)) \to \cF[1]$ are the chain maps given by
\[
\phi =
\begin{bmatrix}
-f[1] \\ \id_{\cF[1]} \\ 0 
\end{bmatrix},
\qquad
\psi = 
\begin{bmatrix} 0 & \id_{\cF[1]} & 0 \end{bmatrix}.
\]
One can check directly that $\psi \circ \phi = \id_{\cF[1]}$, that $\psi \circ \alpha(\alpha(f)) = \beta(f)$, and that $\beta(\alpha(f)) \circ \phi = -f[1]$.  Finally, we claim that the chain map $\phi \circ \psi$ gives rise to the identity morphism of $\cone(\alpha(f))$ in $\LM(\fh,W)$.  This claim follows from the observation that
\begin{equation}\label{eqn:tr3-proof}
\id_{\cone(\alpha(f))} - \phi \circ \psi = d_{\uEnd_\LM(\cone(\alpha(f)))}(r),
\end{equation}
where
{\small
\[
r =
\begin{bmatrix}
0 & 0 & \bt(\id_\cG) \\ 0 & 0 & 0 \\ 0 & 0 & 0
\end{bmatrix}
\in
\begin{bmatrix}
\uEnd_\LM(\cG[1]) & \uHom_\LM(\cF[1],\cG[1]) & \uHom_\LM(\cG,\cG[1]) \\
\uHom_\LM(\cG[1],\cF[1]) & \uEnd_\LM(\cF[1]) & \uHom_\LM(\cG,\cF[1]) \\
\uHom_\LM(\cG[1],\cG) & \uHom_\LM(\cF[1],\cG) & \uEnd_\LM(\cG)
\end{bmatrix}.
\]}%
To check~\eqref{eqn:tr3-proof}, observe that
\[
\delta_{\cone(\alpha(f))} =
\begin{bmatrix}
-\bs(\delta_\cG) & 0 & 0 \\
0 & -\bs(\delta_\cF) & 0 \\
\bt^{-1}(\id_\cG) & \bt^{-1}(f) & \delta_\cG
\end{bmatrix}.
\]
Using~\eqref{eqn:composition-shift-leftmon}, one can check that
\[
\id_{\cone(\alpha(f))} - \phi \circ \psi = \kappa(r) + \delta_{\cone(\alpha(f))} \circ r + r \circ \delta_{\cone(\alpha(f))},
\]
as claimed.  We omit the proof of the other axioms.
\end{proof}

\begin{rmk}
\label{rmk:Hom-shift}
For $\cF,\cG \in \LM(\fh,W)$, we have a canonical identification of bigraded $\bk$-modules
\[
\uHom_\LM(\cF,\cG[1]) = \uHom_\LM(\cF,\cG)[1]
\]
induced by~\eqref{eqn:Hom-shift},
but the two sides do \emph{not} have the same differential: for $f \in \uHom_\LM(\cF,\cG[1])$ we have
\begin{align*}
d_{\uHom_\LM(\cF,\cG[1])}(f) &= \kappa(f) - \bs(\delta_\cG) \circ f + (-1)^{|f|+1} f \circ \delta_\cF, \\
d_{\uHom_\LM(\cF,\cG)[1]}(f) &= -\bt(d_{\uHom_\LM(\cF,\cG)}(\bt^{-1}(f))) \\
&=
-\kappa(f) - \bt(\delta_\cG \circ \bt^{-1}(f)) - (-1)^{|\bt^{-1}(f)|+1} \bt(\bt^{-1}(f) \circ \delta_\cF) \\
&= -\kappa(f) - \delta_\cG \circ f - (-1)^{|f|} f \circ \bs(\delta_\cF).
\end{align*}
From these equations, one can see that
\[
d_{\uHom_\LM(\cF,\cG[1])}(\bs(f)) = \bs(d_{\uHom_\LM(\cF,\cG)[1]}(f)).
\]
In other words, $\bs: \uHom_\LM(\cF,\cG[1]) \to \uHom_\LM(\cF,\cG)[1]$ is an isomorphism of dgg $\bk$-modules, which therefore induces an isomorphism of $\bk$-modules
\[
\Hom_{\LM(\fh,W)}(\cF, \cG[1]) \simto \gHom_\LM(\cF,\cG)^1_0.
\]
From this we obtain an isomorphism of bigraded $\bk$-modules
\[
\gHom_\LM(\cF,\cG) \simto \bigoplus_{n,m \in \Z} \Hom_{\LM(\fh,W)}(\cF, \cG[n]\langle -m \rangle).
\]
Below we will use this isomorphism to identify the two sides.
\end{rmk}

\section{Right-equivariant versus left-monodromic complexes}

In this section, we will show that the categories $\RE(\fh,W)$ and $\LM(\fh,W)$ are actually equivalent.

\begin{lem}
\label{lem:dmixUGB-generate}
The category $\LM(\fh,W)$ is generated as a triangulated category by the objects in $\DiagBS(\fh,W)$.
\end{lem}

Here, we regard an object of $\DiagBS(\fh,W)$ as an object of $\LM(\fh,W)$ by viewing it as a $\DiagBSp$-sequence concentrated in degree $0$, and then equipping it with the zero differential.

\begin{proof}
Let $\cF \in \LM(\fh,W)$.  In the underlying $\DiagBSp$-sequence
\[
(\sdots, \cF^{-1}, \cF^0, \cF^1, \sdots),
\]
let $n$ be the largest integer such that $\cF^n \ne 0$.  Define two new $\DiagBSp$-sequences
\[
\cF' = (\sdots, 0, \cF^n, 0, \sdots)
\qquad\text{and}\qquad
\cF'' = (\sdots, \cF^{n-2}, \cF^{n-1}, 0, \sdots)
\]
where in each case $\cF^i$ is the $i$-th entry,
so that $\cF = \cF'' \oplus \cF'$.  Write its differential as a matrix
\[
\delta_\cF =
\begin{bmatrix}
a & b \\ c & d
\end{bmatrix}
\in 
\begin{bmatrix}
\uEnd_\LM(\cF'') & \uHom_\LM(\cF',\cF'') \\
\uHom_\LM(\cF'',\cF') & \uEnd_\LM(\cF')
\end{bmatrix}.
\]
We claim that $b = 0$.  Indeed, degree considerations show that $\uHom_\BE(\cF',\cF'')^i_j$ can be nonzero only when $i - j \le -1$, but
\[
b \in \uHom_\LM(\cF',\cF'')^1_0 = \bigoplus_{k \ge 0} \Lambda^k_{2k} \otimes \uHom_\BE(\cF',\cF'')^{1-k}_{-2k} = \{0\}.
\]
Similar reasoning shows that $d = 0$ as well.

Next, the condition $\delta_\cF \circ \delta_\cF + \kappa(\delta_\cF) = 0$ implies that $a$ also satisfies the defining condition for differentials.  In other words, $(\cF'',a)$ is an object of $\LM(\fh,W)$, as is $(\cF',0)$.

Finally, consider the element $\bt(c) \in \uHom_\LM(\cF''[-1],\cF')$.  The condition on $\delta_\cF$ implies that $\kappa(c) + ca = 0$, or
\[
\kappa(\bt(c)) + 0 \circ \bt(c) - \bt(c) \circ (-\bs(a)) = 0.
\]
This says that $\bt(c)$ is a chain map $\cF''[-1] \to \cF'$. Moreover, the diagram
\[
\cF''[-1] \xrightarrow{\bt(c)} \cF' \to \cF \to \cF''
\]
is a distinguished triangle, since it can be identified with $\cF''[-1] \xrightarrow{\bt(c)} \cF' \to \cone(\bt(c)) \to \cF''$.

Now, $\cF'$ is a cohomological shift of an object of $\DiagBSp(\fh,W)$, and $\cF''$ has an underlying $\DiagBSp$-sequence with fewer nonzero terms than that of $\cF$.  By an induction argument on the number of nonzero terms, we conclude that $\LM(\fh,W)$ is generated by $\DiagBS(\fh,W)$.
\end{proof}

For any two $\DiagBSp$-sequences $\cF,\cG$, there is a natural map
\begin{equation}\label{eqn:uhom-eqvt-mon}
\uHom_\BE(\cF,\cG) \to \uHom_\LM(\cF,\cG)
\end{equation}
induced by the unit $\eta_\Lambda: \bk \to \Lambda$. If $\delta \in \uEnd_\BE(\cF)$, then its image $\bar\delta \in \uEnd_\LM(\cF)$ under~\eqref{eqn:uhom-eqvt-mon} satisfies $\kappa(\bar\delta) = 0$.  In particular, if $\delta \circ \delta = 0$, then we deduce that $\bar\delta \circ \bar\delta + \kappa(\bar\delta) = 0$.  We can therefore define a functor\index{forgetful functor!ForBELM@$\ForBELM$}
\[
\ForBELM: \BE(\fh,W) \to \LM(\fh,W)
\]
that sends $(\cF,\delta)$ to $(\cF,\bar\delta)$. (Once this notation is introduced, one can restate Lemma~\ref{lem:dmixUGB-generate} as saying that $\LM(\fh,W)$ is generated by the essential image of the functor $\ForBELM$.)

Next, consider the natural map
\[
\uHom_\LM(\cF,\cG) \to \uHom_{\RE}(\cF,\cG)
\]
that is induced by the counit $\epsilon_\sA: \sA \to \bk$ together with the isomorphism from Lemma~\ref{lem:uhom-mon-con}. Considerations similar to those above give us a functor\index{forgetful functor!ForLMRE@$\ForLMRE$}
\[
\ForLMRE: \LM(\fh,W) \to \RE(\fh,W).
\]

It is clear that the functors $\ForLMRE$ and $\ForBELM$ are triangulated, that they commute with the functors $\langle n \rangle$, and that we have
\[
\ForBERE = \ForLMRE \circ \ForBELM.
\]

\begin{thm}
\label{thm:leftmon-con}
The functor $\ForLMRE: \LM(\fh,W) \to \RE(\fh,W)$ is an equivalence of triangulated categories.
\end{thm}

\begin{proof}
In view of Lemma~\ref{lem:dmixUGB-generate} and the fact that objects of $\DiagBS(\fh,W)$ also generate $\RE(\fh,W)$, it is enough to show that $\ForLMRE$ is fully faithful on cohomological shifts of these objects.  In fact, we will show that if $\cF,\cG \in \DiagBS(\fh,W)$, then the natural map
\[
\uHom_\LM(\cF,\cG) \to \uHom_{\RE}(\cF,\cG)
\]
is a quasi-isomorphism, which will be enough by Remark~\ref{rmk:Hom-shift}.  Since $\cF$ and $\cG$ have zero differentials, the differential on $\uHom_\LM(\cF,\cG)$ is just $\kappa$, while $\uHom_{\RE}(\cF,\cG)$ and $\uHom_\BE(\cF,\cG)$ have zero differential.

Recall that $\sA$ is isomorphic as a bigraded $\bk$-module to $\Lambda \otimes R$.  We define an isomorphism of bigraded $\bk$-modules
\begin{equation}\label{eqn:parity-ugb-h}
h: \sA \otimes_R \uHom_\BE(\cF,\cG) \simto \uHom_\LM(\cF,\cG)
\end{equation}
by setting
\[
h(r \otimes x \otimes f) = r \otimes (x \star f)
\qquad\text{where $r \in \Lambda$, $x \in R$, and $f \in \uHom_\BE(\cF,\cG)$.}
\]
It is easy to see that $h$ is a morphism of dgg $\bk$-modules. Consider the following commutative diagram:
\[
\begin{tikzcd}[column sep=small]
\sA \otimes_R \uHom_\BE(\cF,\cG) \ar[rr, "h", "\sim"'] \ar[dr, "\epsilon_\sA \otimes \id"'] 
  && \uHom_\LM(\cF,\cG) \ar[dl, "\ForLMRE"] \\
& \bk \otimes_R \uHom_{\BE}(\cF,\cG) = \uHom_{\RE}(\cF,\cG).
\end{tikzcd}
\]
Since $\epsilon_\sA: \sA \to \bk$ is a quasi-isomorphism of $R$-modules, and since $\uHom_\BE(\cF,\cG)$ is free over $R$ (see~\S\ref{sec:ps-convolution}), the left-hand diagonal arrow is a quasi-isomorphism.  We conclude that the right-hand diagonal arrow is a quasi-isomorphism as well, as desired.
\end{proof}

\begin{ex}
\label{ex:t1-leftmon}
Consider $B_\varnothing$ as an object of $\LM(\fh,W)$. (This object will often be denoted by $\cT_\varnothing$\index{tilting object!Tnoth@$\cT_\varnothing$}.) As a special case of~\eqref{eqn:parity-ugb-h}, using~\eqref{eqn:R-End-E1}, we obtain an isomorphism of dgg algebras
\[
\uEnd_\LM(\cT_\varnothing) \cong \sA.
\]
Taking cohomology, we find that
\begin{equation}\label{eqn:t1-leftmon}
\gEnd_\LM(\cT_\varnothing) \cong \bk.
\end{equation}
\end{ex}

\begin{ex}\label{ex:ts-leftmon}
Let us write down concretely the object in $\LM(\fh,W)$ that corresponds to the object $\cT_s' \in \RE(\fh,W)$ from Example~\ref{ex:ts-con} under the equivalence of Theorem~\ref{thm:leftmon-con}. (This object will be denoted $\cT_s$\index{tilting object!Ts@$\cT_s$}.) The $\DiagBSp$-sequence is the same as in Example~\ref{ex:ts-con}, but the differential is now given by
\begin{multline*}
\delta = \begin{tikzpicture}[scale=0.3,thick,baseline]
 \draw (0,0) to (0,1);
 \node at (0,0) {$\bullet$};
 \node at (0,1.4) {\tiny $s$};
\end{tikzpicture} \oplus \begin{tikzpicture}[scale=-0.3,thick,baseline]
 \draw (0,0) to (0,1);
 \node at (0,0) {$\bullet$};
 \node at (0,1.4) {\tiny $s$};
\end{tikzpicture} \oplus (-\alpha_s \otimes \id) \\
\in \Hom(B_\varnothing(-1),B_s) \oplus \Hom(B_s,B_\varnothing(1)) \oplus (\Lambda^{1}_{2} \otimes \Hom(B_\varnothing(-1), (B_\varnothing(1))(-2))) \\
= \uEnd_\LM(\cT_s)^1_0.
\end{multline*}
We depict this differential as follows:
\[
\begin{tikzcd}
B_\varnothing(1) \\
B_s \ar[u, "\usebox\upperdot" swap] \\
B_\varnothing(-1). \ar[u, "\usebox\lowerdot" swap] \ar[uu, bend left=60, "(-2)" description, "-\alpha_s \otimes \id" near end]
\end{tikzcd}
\]
Using the fact that $\epsilon_s \circ \eta_s = \alpha_s \star \id$ (see the barbell relation in~\S\ref{sec:ew-diagram}), one can check that $\delta^2$ and $-\kappa(\delta)$ are both given by
\[
\begin{tikzcd}
B_\varnothing(1) \\
B_s \\
B_\varnothing(-1). \ar[uu, bend left=60, "1 \otimes (\alpha_s \star \id)"]
\end{tikzcd}
\]
\end{ex}

As for the category $\RE(\fh,W)$, the category $\LM(\fh,W)$
admits a natural convolution action of $\BE(\fh,W)$ on the right
\begin{equation}
\label{eqn:ustar-UGB}
\ustar: \LM(\fh,W) \times \BE(\fh,W) \to \LM(\fh,W),
\end{equation}
defined as follows. First, for $\DiagBSp$-sequences $\cF, \cF', \cG, \cG'$ and for $f=r \otimes f' \in \uHom_\LM(\cF, \cG)$ (where $r \in \Lambda$ and $f' \in \uHom_\BE(\cF, \cG)$), $g \in \uHom_\BE(\cF', \cG')$ we set
\[
f \ustar g := r \otimes (f' \ustar g) \in \uHom_\LM(\cF \ustar \cF', \cG \ustar \cG').
\]
Then the convolution product of the objects $(\cF, \delta_\cF) \in \LM(\fh,W)$ and $(\cG, \delta_\cG) \in \BE(\fh,W)$ is defined as the pair
\[
(\cF \ustar \cG, \delta_\cF \ustar \id_\cG + \id_\cF \ustar \delta_\cG).
\]

The functors $\ForBELM$ and $\ForLMRE$ are compatible with this action in the sense that if $\cF$, $\cG$ belong to $\BE(\fh,W)$ and $\cH$ belongs to $\LM(\fh,W)$, then there exist functorial isomorphisms
\[
\ForBELM(\cF \ustar \cG) \cong \ForBELM(\cF) \ustar \cG, \qquad \ForLMRE(\cH \ustar \cG) \cong \ForLMRE(\cH) \ustar \cG.
\]

\section{The left monodromy action}
\label{sec:left-monodromy}

The main motivation for introducing the left-monodromic category $\LM(\fh,W)$ is the construction given in Theorem~\ref{thm:frown} below. It provides an analogue in our setting of the construction of the monodromy action in~\cite{verdier}.

Recall from~\S\ref{sec:derivations} the map $\lfrown: \Lambda^\vee \otimes \Lambda \to \Lambda$.  This map induces a morphism of bigraded $\bk$-modules
\[
\lfrown: \Lambda^\vee \otimes \uHom_\LM(\cF,\cG) \to \uHom_\LM(\cF,\cG)
\]
for any two $\DiagBSp$-sequences $\cF, \cG$.

\begin{lem}
\label{lem:kappa-frown}
Let $f \in \uHom_\LM(\cF,\cG)$ and $g \in \uHom_\LM(\cG,\cH)$.  For any $x \in V = (\Lambda^\vee)^{-1}_{-2}$, we have
\[
x \lfrown (g \circ f) = (x \lfrown g) \circ f + (-1)^{|g|} g \circ (x \lfrown f)
\]
and
\[
x \lfrown \kappa(f) = -\kappa(x \lfrown f).
\]
\end{lem}

\begin{proof}
The first formula follows from~\eqref{eqn:frown-deriv}; the second one can be checked easily.
\end{proof}

\begin{thm}
\label{thm:frown}
Let $\cF \in \LM(\fh,W)$.  There is a bigraded $\bk$-algebra homomorphism\index{muF@$\mu_\cF$}
\[
\mu_\cF: R^\vee \to \gEnd_\LM(\cF)
\]
determined by setting $\mu_\cF(x) = x \lfrown \delta$ for $x \in V = (R^\vee)^{0}_{-2}$.  More generally, for any $\cF, \cG \in \LM(\fh,W)$, the bigraded $\bk$-module $\gHom_\LM(\cF,\cG)$ is equipped with a canonical structure of a left $R^\vee$-module\index{convolution!starhat@$\hatstar$}
\begin{equation}\label{eqn:leftmon-rv-action}
\hatstar: R^\vee \otimes \gHom_\LM(\cF,\cG) \to \gHom_\LM(\cF,\cG)
\end{equation}
given by $x \hatstar f = \mu_\cG(x) \circ f = f \circ \mu_\cF(x)$.  This action is compatible with composition in the following way: if $f \in \gHom_\LM(\cF,\cG)$, $g \in \gHom_\LM(\cG,\cH)$, and $x \in R^\vee$, then
\begin{equation}\label{eqn:leftmon-interchange}
x \hatstar (g \circ f) = (x \hatstar g) \circ f = g \circ (x \hatstar f).
\end{equation}
\end{thm}

The action of $R^\vee$ on $\gHom_\LM(\cF,\cG)$ defined in this theorem is called the \emph{left monodromy action}\index{monodromy action!left}.

\begin{proof}
Consider the linear map
\[
\tilde\mu_\cF: V\la-2\ra \to \uEnd_\LM(\cF)
\]
given by $\tilde\mu_\cF(x) = x \lfrown \delta$.  We first observe that $d(\tilde\mu_\cF(x)) = 0$:
\[
d(x \lfrown \delta) = \kappa(x \lfrown \delta) + \delta \circ (x \lfrown \delta) - (x \lfrown \delta) \circ \delta
= -x \lfrown \kappa(\delta) - x \lfrown \delta^2 = 0
\]
by Lemma~\ref{lem:kappa-frown}.
Therefore, $\tilde\mu_\cF$ does indeed induce a map $\mu_\cF: V\la -2\ra \to \gEnd_\LM(\cF)$.  To prove that this map extends to a $\bk$-algebra morphism $R^\vee \to \gEnd_\LM(\cF)$, we must show that $\mu_\cF(x) \circ \mu_\cF(y) = \mu_\cF(y) \circ \mu_\cF(x)$.  We will instead prove more generally that for any $f \in \gHom_\LM(\cF,\cG)$, we have
\begin{equation}\label{eqn:mon-commute}
\mu_\cG(x) \circ f = f \circ \mu_\cF(x).
\end{equation}
Choose a representative $\tilde f \in \uHom_\LM(\cF,\cG)$ of $f$.  Using Lemma~\ref{lem:kappa-frown}, it is easy to check that
\[
x \lfrown d(\tilde f) = d(-x \lfrown \tilde f) + (x \lfrown \delta_\cG) \circ \tilde f - \tilde f \circ (x \lfrown \delta_\cF).
\]
But the left-hand side is zero (since $d(\tilde f) = 0$), so now this equation says that $(x \lfrown \delta_\cG) \circ \tilde f$ and $\tilde f \circ (x \lfrown \delta_\cF)$ become equal in $\gHom_\LM(\cF,\cG)$.  In other words,~\eqref{eqn:mon-commute} holds.  This shows that we have the desired map $\mu_\cF: R^\vee \to \gEnd_\LM(\cF)$. It also shows that $\hatstar$ is well-defined.  Lastly,~\eqref{eqn:leftmon-interchange} is an easy consequence of~\eqref{eqn:mon-commute}.
\end{proof}

\begin{lem}\label{lem:eqvt-mon-action}
If $\cF \in \BE(\fh,W)$, then the map $\mu_\cF: R^\vee \to \gEnd_\LM(\ForBELM(\cF))$ factors through $\epsilon_{R^\vee}: R^\vee \to \bk$.
\end{lem}

\begin{proof}
The differential $\delta$ for $\ForBELM(\cF)$ lies in $\Lambda^{0}_{0} \otimes \uEnd_\BE(\cF)^{1}_{0} \subset \uEnd_\LM(\cF)^{1}_{0}$, so it is clear from the definition of $\lfrown$ that $x \lfrown \delta = 0$ for all $x \in V$.
\end{proof}

\begin{ex}
\label{ex:ts-mon-action}
Consider the object $\cT_s$ from Example~\ref{ex:ts-leftmon}.
For $x \in V$, the map $\mu_{\cT_s}(x) \in
\gEnd_\LM(\cT_s)^0_{-2}$ or $\mu_{\cT_s}(x): \cT_s \to \cT_s\la 2\ra$ is given by
\[
\begin{tikzcd}[column sep=large]
B_\varnothing(1) \\
B_s \ar[u, "\usebox\upperdot" swap] \\
B_\varnothing(-1) \ar[u, "\usebox\lowerdot" swap] \ar[uu, bend left=60, "(-2)" description, "-\alpha_s \otimes \id" near end]
\ar[r, "-\alpha_s(x)"]
&B_\varnothing(-1) \\
&B_s(-2) \ar[u, "\usebox\upperdot" swap] \\
&B_\varnothing(-3). \ar[u, "\usebox\lowerdot" swap] \ar[uu, bend left=60, "(-2)" description, "-\alpha_s \otimes \id" near end]
\end{tikzcd}
\]
In particular, the monodromy action of $R^\vee$ on $\gEnd_\LM(\cT_s)$ is nontrivial.  In light of Lemma~\ref{lem:eqvt-mon-action}, this shows that the object $\cT_s \in \LM(\fh,W)$ is not in the essential image of $\ForBELM$, and then that the object $\cT'_s \in \RE(\fh,W)$ is not in the essential image of $\ForBERE$; see~\S\ref{sec:parity-sequences} for more details.

Using the description of $\gEnd_\RE(\cT_s')$ in Example~\ref{ex:ts-con} and the ``Demazure surjectivity'' assumption, we see that
$\mu_{\cT_s} : R^\vee \to \gEnd_\LM(\cT_s)$ is surjective, and induces an
isomorphism
\[
R^\vee / ((R^\vee)^s_+ \cdot R^\vee) \simto \gEnd_\LM(\cT_s),
\]
where $(R^\vee)^s_+ = \{x \in \ker(\epsilon_{R^\vee}) \mid s \cdot x = x\}$.
\end{ex}

\section{Another triangulated structure on \texorpdfstring{$\LM(\fh,W)$}{LM(h,W)}}

For later use, in this section we introduce a second triangulated structure on the category $\LM(\fh,W)$. For clarity, here we
write $\Sigma_\ell$ for the functor on $\LM(\fh,W)$ that was denoted by $[1]$ in~\S\ref{sec:LM-triangulated}, reserving $[1]$ for the shift of a $\DiagBSp$-sequence. Using this notation, in~\S\ref{sec:LM-triangulated} we defined a triangulated structure on the category with shift $(\LM(\fh,W), \Sigma_\ell)$. Call this the \emph{left triangulated structure} and use a subscript $\ell$ to denote the associated constructions; for example, for any chain map $f$, we have the left cone $\cone_\ell(f)$ and left standard triangle involving morphisms $\alpha_\ell(f)$, $\beta_\ell(f)$.

Let $\cF, \cG$ be $\DiagBSp$-sequences. Define
\[
 \bu_{\cF,\cG} : \uHom_\LM(\cF, \cG) \to \uHom_\LM(\cF, \cG)
\]
by
\[
\bu_{\cF,\cG}(f) = (-1)^pf_\even + (-1)^{p+1}f_\odd
 \qquad\text{for}\qquad f \in \Hom_{\DiagBSp(\fh,W)}(\cF^p, \cG^q(j)),
\]
and extending $\Lambda$-linearly. (See~\eqref{eqn:morphism-even-odd} for the decomposition $f = f_\even + f_\odd$. The subscript ``$\cF,\cG$'' is necessary in this notation since for instance under the canonical identification $\uHom_\LM(\cF,\cG) = \uHom_\LM(\cF[1],\cG[1])$, the maps $\bu_{\cF,\cG}$ and $\bu_{\cF[1],\cG[1]}$ do \emph{not} coincide.) The map $\bu_{\cF,\cG}$ is clearly an involution. Let $f \in \uHom_\LM(\cF, \cG)$ and $g \in \uHom_\LM(\cG, \cH)$. The following formulas are clear:
\begin{equation} \label{eq:kappa-bu}
 \kappa(\bu_{\cF,\cG}(f)) = \bu_{\cF,\cG}(\kappa(f))
\end{equation}
\begin{equation} \label{eq:composition-bu}
 g \circ \bu_{\cF,\cG}(f) = \bu_{\cF,\cH}(g \circ f).
\end{equation}
One can also check directly that
\begin{equation} \label{eq:bu-composition}
 \bu_{\cG,\cH}(g) \circ f = (-1)^{|f|} \bu_{\cF,\cH}(g \circ \bs(f))
\end{equation}
if $f$ is homogeneous.

Let $\cF, \cG, \cF', \cG'$ be $\DiagBSp$-sequences, let $f \in \Hom_{\DiagBSp(\fh,W)}(\cF^p, \cG^q(j))$, and let $h \in \uHom_\BE(\cF', \cG')$. The following formulas may be checked directly:
\begin{gather*}
 \bt^n(f \ustar h) = (-1)^{np} f_\even \ustar \bt^n(h) + (-1)^{n(p+1)} f_\odd \ustar \bt^n(h);
\\
 \bu_{\cF \ustar \cF', \cG \ustar \cG'}(f \ustar h) = (-1)^p f_\even \ustar \bu_{\cF',\cG'}(h) + (-1)^{p+1} f_\odd \ustar \bu_{\cF',\cG'}(h).
\end{gather*}
It follows that for any $f \in \uHom_\LM(\cF, \cG)$ we have
\begin{equation} \label{eq:btbu-conv}
 \bt^n(\bu_{\cF \ustar \cF', \cG \ustar \cG'}(f \ustar h)) = f \ustar \bt^n(\bu_{\cF',\cG'}(h)) \qquad\text{for } n \text{ odd}.
\end{equation}

We now define a new shift autoequivalence $\Sigma_r$ on $\LM(\fh,W)$. For an object $(\cF, \delta_\cF) \in \LM(\fh,W)$, define
\[
 \Sigma_r(\cF, \delta_\cF) = (\cF[1], \delta_\cF).
\]
For a morphism $f : \cF \to \cG$, define
\[
 \Sigma_r(f) : \Sigma_r\cF \to \Sigma_r\cG \qquad\text{by}\qquad \Sigma_r(f) = f.
\]

Let us now define the \emph{right triangulated structure}, a triangulated structure on the category with shift $(\LM(\fh,W), \Sigma_r)$. Given a chain map $f : \cF \to \cG$, the \emph{right cone} $\cone_r(f)$ is the left-monodromic complex with underlying $\DiagBSp$-sequence $\cF[1] \oplus \cG$ and differential
\[
 \delta_{\cone_r(f)} =
 \begin{bmatrix}
  \delta_\cF & 0 \\
  \bt^{-1}(\bu_{\cF,\cG}(f)) & \delta_\cG
 \end{bmatrix}
 \in
 \begin{bmatrix}
  \uEnd_\LM(\cF[1]) & \uHom_\LM(\cG, \cF[1]) \\
  \uHom_\LM(\cF[1], \cG) & \uEnd_\LM(\cG)
 \end{bmatrix}.
\]
Let us check that this is indeed a differential. We have
\begin{multline*}
 \delta_{\cone_r(f)} \circ \delta_{\cone_r(f)} + \kappa(\delta_{\cone_r(f)}) =\\
 \begin{bmatrix}
  \delta_\cF \circ \delta_\cF + \kappa(\delta_\cF) & \\
  \bt^{-1}(\bu_{\cF,\cG}(f)) \circ \delta_\cF + \delta_\cG \circ \bt^{-1}(\bu_{\cF,\cG}(f)) + \kappa(\bt^{-1}(\bu_{\cF,\cG}(f))) & \delta_\cG \circ \delta_\cG + \kappa(\delta_\cG)
 \end{bmatrix}.
\end{multline*}
The diagonal entries clearly vanish. By~\eqref{eqn:composition-shift-leftmon} and~\eqref{eqn:kappa-bt}, the lower left entry equals
\[
 \bt^{-1} \bigl( \kappa(\bu_{\cF,\cG}(f)) + \delta_\cG \circ \bu_{\cF,\cG}(f) + \bu_{\cF,\cG}(f) \circ \bs(\delta_\cF) \bigr),
\]
and by~\eqref{eq:kappa-bu}, \eqref{eq:composition-bu}, and \eqref{eq:bu-composition}, this equals
\[
 \bt^{-1} \circ \bu_{\cF,\cG} \bigl( \kappa(f) + \delta_\cG \circ f - f \circ \delta_\cF \bigr) = \bt^{-1} \circ \bu \bigl(d(f) \bigr) = 0,
\]
as desired.

A \emph{right distinguished triangle} is a diagram in $\LM(\fh,W)$ isomorphic to a \emph{right standard triangle}, one of the form
\[
 \cF \xrightarrow{f} \cG \xrightarrow{\alpha_r(f)} \cone_r(f) \xrightarrow{\beta_r(f)} \Sigma_r\cF
\]
for some chain map $f$, where $\alpha_r(f)$ and $\beta_r(f)$ are the obvious chain maps (involving no sign).

\begin{prop}
\label{prop:left-vs-right-triangulated-structure}
 The definitions above give a triangulated structure on $(\LM(\fh,W), \Sigma_r)$. Moreover, there is a natural isomorphism $\eta : \Sigma_\ell \simto \Sigma_r$ such that $(\id_{\LM(\fh,W)}, \eta)$ is a triangulated equivalence
 \[
 (\LM(\fh,W), \Sigma_\ell) \simto (\LM(\fh,W), \Sigma_r).
 \]
\end{prop}

\begin{proof}
 For $\cF \in \LM(\fh,W)$, define $\eta_\cF \in \uHom_\LM(\Sigma_\ell\cF, \Sigma_r\cF) \cong \uHom_\LM(\cF, \cF)$ by
 $\eta_\cF = \bu_{\cF,\cF}(\id_\cF)$. By definition,
 \begin{multline*}
  d(\eta_\cF) = \kappa(\eta_\cF) + \delta_{\Sigma_r\cF} \circ \eta_\cF - \eta_\cF \circ \delta_{\Sigma_\ell\cF}
  \\
  = \kappa(\bu_{\cF,\cF}(\id_\cF)) + \delta_\cF \circ \bu_{\cF,\cF}(\id_\cF) + \bu_{\cF,\cF}(\id_\cF) \circ \bs(\delta_\cF).
 \end{multline*}
 By \eqref{eq:kappa-bu}, \eqref{eq:composition-bu}, and \eqref{eq:bu-composition}, this equals
 \[
  \bu_{\cF,\cF} \bigl( \kappa(\id_\cF) + \delta_\cF \circ \id_\cF - \id_\cF \circ \delta_\cF \bigr) = 0.
 \]
 Thus $\eta_\cF$ is a chain map and defines a morphism $\eta_\cF : \Sigma_\ell\cF \to \Sigma_r\cF$ in $\LM(\fh,W)$, which is clearly an isomorphism. 
 
 Now we check that $\eta$ is an isomorphism of functors
 $\eta : \Sigma_\ell \simto \Sigma_r$. In fact, let $\cF, \cG \in \LM(\fh,W)$ and let $f : \cF \to \cG$ be a morphism. Then by definition,
 \[
  \eta_\cG \circ \Sigma_\ell(f) - \Sigma_r(f) \circ \eta_\cF = \bu_{\cG,\cG}(\id_\cG) \circ \bs(f) - f \circ \bu_{\cF,\cF}(\id_\cF),
 \]
 and by \eqref{eq:composition-bu} and \eqref{eq:bu-composition}, this equals
 \[
  \bu_{\cF,\cG}(\id_\cG \circ f - f \circ \id_\cF) = 0,
 \]
 as desired.
 
 Let $\cF, \cG \in \LM(\fh,W)$. For any chain map $f : \cF \to \cG$, define
 \[
 \gamma_f \in \uHom_\LM(\cone_\ell(f), \cone_r(f))
 \]
 by
 \[
  \gamma_f =
  \begin{bmatrix}
   \bu_{\cF,\cF}(\id_\cF) & 0 \\
   0 & \id_\cG
  \end{bmatrix}
  \in
  \begin{bmatrix}
   \uEnd_\LM(\cF[1]) & \uHom_\LM(\cG, \cF[1]) \\
   \uHom_\LM(\cF[1], \cG) & \uEnd_\LM(\cG)
  \end{bmatrix}.
 \]
By definition,
 \begin{multline*}
  d(\gamma_f) = \kappa(\gamma_f) + \delta_{\cone_r(f)} \circ \gamma_f - \gamma_f \circ \delta_{\cone_\ell(f)} \\
  = \begin{bmatrix}
     \kappa(\bu_{\cF,\cF}(\id_\cF)) & \\
     & \kappa(\id_\cG)
    \end{bmatrix}
    +
    \begin{bmatrix}
     \delta_\cF & \\
     \bt^{-1}(\bu_{\cF,\cG}(f)) & \delta_\cG
    \end{bmatrix}
    \begin{bmatrix}
     \bu_{\cF,\cF}(\id_\cF) & \\
     & \id_\cG
    \end{bmatrix}
    \\
    -
    \begin{bmatrix}
     \bu_{\cF,\cF}(\id_\cF) & \\
     & \id_\cG
    \end{bmatrix}
    \begin{bmatrix}
     -\bs(\delta_\cF) & \\
     \bt^{-1}(f) & \delta_\cG
    \end{bmatrix} \\
  = \begin{bmatrix}
     \delta_\cF \circ \bu_{\cF,\cF}(\id_\cF) + \bu_{\cF,\cF}(\id_\cF) \circ \bs(\delta_\cF) & \\
     \bt^{-1}(\bu_{\cF,\cG}(f)) \circ \bu_{\cF,\cF}(\id_\cF) - \id_\cG \circ \bt^{-1}(f) & \delta_\cG \circ \id_\cG - \id_\cG \circ \delta_\cG
    \end{bmatrix}.
 \end{multline*}
 Using \eqref{eqn:composition-shift-leftmon}, \eqref{eq:composition-bu}, and \eqref{eq:bu-composition}, one easily checks that each entry vanishes. Thus $\gamma_f$ is a chain map and defines a morphism $\gamma_f : \cone_\ell(f) \to \cone_r(f)$ in $\LM(\fh,W)$, which is clearly an isomorphism.
 
 This shows that we have the following isomorphism of diagrams in $\LM(\fh,W)$:
 \[
  \begin{tikzcd}[column sep=large]
   \cF \ar[r, "f"] \ar[d, equal] & \cG \ar[r, "\alpha_\ell(f)"] \ar[d, equal]
    & \cone_\ell(f) \ar[r, "\beta_\ell(f)"] \ar[d, "\gamma_f"] & \Sigma_\ell\cF \ar[d, "\eta_\cF"] \\
   \cF \ar[r, "f"] & \cG \ar[r, "\alpha_r(f)"]
    & \cone_r(f) \ar[r, "\beta_r(f)"] & \Sigma_r\cF
  \end{tikzcd}
 \]
 It follows that the isomorphism
 \[
 (\id_{\LM(\fh,W)}, \eta) : (\LM(\fh,W), \Sigma_\ell) \simto (\LM(\fh,W), \Sigma_r)
 \]
of categories with shift identifies the collection of left distinguished triangles with the collection of right distinguished triangles. The result now follows since, by Proposition~\ref{prop:UGB-left-triangulated-structure}, the former defines a triangulated structure on $(\LM(\fh,W), \Sigma_\ell)$.
\end{proof}

\section{Karoubian envelopes}
\label{sec:karoubian1}

In this section, we assume that $\bk$ is a field or a complete local
ring, so that the indecomposable objects $B_w$ in the Karoubian
envelope $\Diag(\fh,W)$ of $\DiagBSp(\fh,W)$ are defined (see~\S\ref{sec:additive-hull-Diag}). All the constructions carried
out earlier in this chapter can be repeated with $\DiagBSp(\fh,W)$
replaced by $\Diag(\fh,W)$.  However, we will see below that the
resulting triangulated categories are equivalent to those obtained
from $\DiagBSp(\fh,W)$. 

In particular, we will consider $\Diag(\fh,W)$-sequences, which we will simply call $\Diag$-sequences.
Given two $\Diag$-sequences $\cF$ and $\cG$, we define the shift-of-grading functors $[n]$, $\la n\ra$, and $(n)$, as well as the bigraded $\bk$-modules
\[
\uHom_\BE(\cF,\cG),
\qquad
\uHom_\RE(\cF,\cG),
\qquad
\uHom_\LM(\cF,\cG)
\]
in the same way as before.  We then define categories\index{Kar@$({-})_\Kar$}
\[
\BE(\fh,W)_\Kar,
\qquad
\RE(\fh,W)_\Kar,
\qquad
\LM(\fh,W)_\Kar.
\]
In each case, objects of the category are pairs $(\cF,\delta)$ where $\cF$ is a $\Diag$-sequence and $\delta$ is a ``differential'' in an appropriate sense.

\begin{lem}
\label{lem:idem-equiv}
The obvious functors
\begin{align*}
\BE(\fh,W) &\to \BE(\fh,W)_\Kar, \\
\RE(\fh,W) &\to \RE(\fh,W)_\Kar, \\
\LM(\fh,W) &\to \LM(\fh,W)_\Kar
\end{align*}
are all equivalences of categories.
\end{lem}

\begin{proof}
We first consider the functor $\BE(\fh,W) \to \BE(\fh,W)_\Kar$.
Recall that $\DiagBSp(\fh,W)$ is a full subcategory of $\Diag(\fh,W)$.  As a consequence, it is easy to see that for any two objects 
$\cF$, $\cG$ in $\BE(\fh,W)$, our functor induces an isomorphism
\[
\Hom_{\BE(\fh,W)}(\cF,\cG) \simto \Hom_{\BE(\fh,W)_\Kar}(\cF,\cG).
\]
Hence $\BE(\fh,W) \to \BE(\fh,W)_\Kar$ is fully faithful.

To show that it is essentially surjective, it suffices to show that each $B_w$ lies in its essential image.  We proceed by induction on $w$ with respect to the Bruhat order.  If $w = 1$, then $B_1 = B_\varnothing$ certainly lies in the image.  Otherwise, choose a reduced expression $\uw$ for $w$.  We have $B_\uw \cong B_w \oplus \cF$, where $\cF$ is a direct sum of various objects of the form $B_v(n)$ with $v < w$. In particular, $B_w$ is isomorphic to the cone of a certain morphism $\cF \to B_\uw$. Now, $B_\uw$ is in the image of our functor, and so is $\cF$, by induction. Using the fact that our functor $\BE(\fh,W) \to \BE(\fh,W)_\Kar$ is fully faithful and triangulated, we deduce that $B_w$ indeed belongs to its essential image.

The same argument shows that $\RE(\fh,W) \to \RE(\fh,W)_\Kar$ is an equivalence.  For the left-monodromic case, one must first imitate the proof of Lemma~\ref{lem:dmixUGB-generate} to show that $\LM(\fh,W)_\Kar$ is generated by the objects of $\Diag(\fh,W)$.  Then one can repeat the argument above.
\end{proof}

\chapter{Free-monodromic complexes}
\label{chap:fm-complexes}

In this chapter we introduce our main object of study: the category $\FM(\fh,W)$ of ``free-monodromic complexes'' associated to a realization of a Coxeter group.  In this category, as in $\LM(\fh,W)$, there is a left monodromy action of $R^\vee$; but in addition, we find a new \emph{right monodromy action}.  The last section contains a number of nontrivial examples of objects and morphisms in this category, including examples of the monodromy actions.

\section{Definitions}
\label{sec:fm-mixed-complexes}

Let $\cF$ and $\cG$ be $\DiagBSp$-sequences. We set\index{HomuFM@$\uHom_\FM$}
\[
\uHom_\FM(\cF,\cG) = \Lambda \otimes \uHom_\BE(\cF,\cG) \otimes R^\vee.
\]
(Here, ``$\FM$'' stands for ``free-monodromic.'')
If $\cH$ is another $\DiagBSp$-sequence,
we equip these bigraded $\bk$-modules with a composition map
\begin{equation}
\label{eqn:composition-mon}
\circ: \uHom_\FM(\cG,\cH) \otimes \uHom_\FM(\cF,\cG) \to \uHom_\FM(\cF,\cH)
\end{equation}
given by
\[
(r \otimes f \otimes x) \circ (s \otimes g \otimes y) = (-1)^{|x||s| + |f||s| + |x||g|} (r \wedge s) \otimes (f \circ g) \otimes (xy).
\]
where $r,s \in \Lambda$, $f \in \uHom_\BE(\cG,\cH)$, $g \in \uHom_\BE(\cF,\cG)$, and $x, y \in R^\vee$.  (In fact, in this formula, $|x|$ is even, so the sign is equal to $(-1)^{|f||s|}$.) We also set $\uEnd_\FM(\cF):=\uHom_\FM(\cF, \cF)$.

In a minor abuse of notation, we denote by
\[
\kappa: \uHom_\FM(\cF,\cG) \to \uHom_\FM(\cF,\cG)[1]
\]
the map $\kappa \otimes \id_{R^\vee}: \uHom_\LM(\cF,\cG) \otimes R^\vee \to \uHom_\LM(\cF,\cG)[1] \otimes R^\vee$. It is clear that~\eqref{eqn:kappa-square} and~\eqref{eqn:kappa-derivation} are again satisfied in this setting.

Choose a basis $e_1, \sdots, e_r$ for $V^*$, and let $\check e_1, \sdots, \check e_r$ be the dual basis for $V$.  Consider the following elements of $\uEnd_\FM(\cF)$:\index{theta1@{$\theta$}}\index{theta2@{$\Theta$}}
\begin{align*}
\Theta = \Theta_\cF &= \sum_{i = 1}^r 1 \otimes (\id \ustar e_i) \otimes \check e_i \in \uEnd_\FM(\cF)^{2}_{0}, \\
\theta = \theta_\cF &= \sum_{i = 1}^r e_i \otimes \id \otimes \check e_i \in \uEnd_\FM(\cF)^{1}_{0}.
\end{align*}
We will also consider
\[
\kappa(\theta) = \sum_{i = 1}^r 1 \otimes (e_i \ustar \id) \otimes \check e_i \in \uEnd_\FM(\cF)^{2}_{0}.
\]
Of course, these elements are independent of the choice of basis. All three are ``graded-central'' in the following sense: for any $f \in \uHom_\FM(\cF,\cG)$, we have
\begin{equation}\label{eqn:Theta-center}
f \Theta_\cF = \Theta_\cG f, \qquad
f \theta_\cF = (-1)^{|f|}\theta_\cG f, \qquad
f \kappa(\theta_\cF) = \kappa(\theta_\cG) f.
\end{equation}

\begin{defn}
The \emph{free-monodromic category} of $(\fh,W)$, denoted by\index{categories@{categories of $\DiagBSp$-sequences}!BFMhW@$\FM(\fh,W)$}\index{free-monodromic}
\[
\FM(\fh,W),
\]
is the category defined as follows.
\begin{itemize}
\item
The objects are pairs $(\cF,\delta)$, where $\cF$ is a $\DiagBSp$-sequence and $\delta$ is an element in $\uEnd_\FM(\cF)^{1}_{0}$ such that $\delta \circ \delta + \kappa(\delta) = \Theta$.
\item
For two such objects $(\cF,\delta_\cF)$ and $(\cG,\delta_\cG)$, we make $\uHom_\FM(\cF,\cG)$ into a dgg $\bk$-module with the differential
\[
d_{\uHom_\FM(\cF,\cG)}(f) = \kappa(f) + \delta_\cG \circ f + (-1)^{|f|+1} f \circ \delta_\cF.
\]
Then morphisms in $\FM(\fh,W)$ are given by
\[
\Hom_{\FM(\fh,W)}(\cF,\cG) = \coH^{0}_{0}(\uHom_\FM(\cF,\cG)).
\]
\item
For three objects $(\cF,\delta_\cF)$, $(\cG,\delta_\cG)$ and $(\cH,
\delta_\cH)$, one can check that the composition
map~\eqref{eqn:composition-mon} is a morphism of dgg $\bk$-modules; hence it induces a morphism
\[
\Hom_{\FM(\fh,W)}(\cG,\cH) \otimes \Hom_{\FM(\fh,W)}(\cF,\cG) \to \Hom_{\FM(\fh,W)}(\cF,\cH),
\]
which defines the composition in $\FM(\fh,W)$.
\end{itemize}
\end{defn}

The objects of $\FM(\fh,W)$ will be sometimes called \emph{free-monodromic complexes}. To check that this definition makes sense,
let us verify that the differential given above actually makes $\uHom_\FM(\cF,\cG)$ into a dgg $\bk$-module.  The same calculation as in~\eqref{eqn:leftmon-diff} shows that
\[
d(d(f)) = \kappa(\delta_\cG)f + \delta_\cG^2 f  - f \kappa(\delta_\cF) - f \delta_\cF^2
= \Theta_\cG f - f \Theta_\cF,
\]
and this vanishes by~\eqref{eqn:Theta-center}.

As in $\LM(\fh,W)$, if $\cF$ and $\cG$ are free-monodromic complexes, we will use the term \emph{chain map}\index{chain map} to refer to an element $f \in \uHom_\FM(\cF,\cG)^0_0$ such that $d(f)=0$. Such an element defines a morphism in $\FM(\fh,W)$.

As in~\S\S\ref{sec:equ-mixed-complexes}--\ref{sec:lmon-mixed-complexes},
if $\bk'$ is another integral domain and $\varphi : \bk \to \bk'$ is a ring morphism, then we have an ``extension of scalars'' functor
\[
\bk' : \FM(\fh,W) \to \FM(\bk' \otimes_\bk \fh, W),
\]
induced on morphisms by the natural isomorphism of complexes
\[
\bk' \otimes_\bk \uHom_{\FM}(\cF, \cG) \simto \uHom_{\FM}(\bk'(\cF), \bk'(\cG))
\]
for $\cF$, $\cG$ in $\FM(\fh,W)$.

The operation $\langle n \rangle$ on $\DiagBSp$-sequences clearly defines an autoequivalence of the category $\FM(\fh,W)$, which we will also denote by $\langle n \rangle$.
As in the previous section, we also put
\[
\gHom_\FM(\cF,\cG) = \coH^{\bullet}_{\bullet}(\uHom_\FM(\cF,\cG)), \quad \gEnd_\FM(\cF) := \gHom_\FM(\cF, \cF).
\]

For any two $\DiagBSp$-sequences $\cF, \cG$, there is a natural map
\begin{equation}
\label{eqn:uhom-mon-lmon}
\uHom_\FM(\cF,\cG) \to \uHom_\LM(\cF,\cG)
\end{equation}
induced by the counit $\epsilon_{R^\vee} : R^\vee \to \bk$. These maps are compatible with composition. Moreover, the image of $\Theta \in \uEnd_\FM(\cF)$ under~\eqref{eqn:uhom-mon-lmon} is $0$.  Therefore, if $\delta \in \uEnd_\FM(\cF)$ satisfies $\delta \circ \delta + \kappa(\delta) = \Theta$, then its image $\bar\delta \in \uEnd_\LM(\cF)$ under~\eqref{eqn:uhom-mon-lmon} satisfies $\bar\delta \circ \bar \delta + \kappa(\bar\delta) = 0$.  We can therefore define a functor\index{forgetful functor!ForFMLM@$\ForFMLM$}
\[
\ForFMLM: \FM(\fh,W) \to \LM(\fh,W)
\]
that sends $(\cF,\delta)$ to $(\cF,\bar\delta)$.

\begin{rmk}
\label{rmk:FM-triangulated}
Proceeding as in~\S\ref{sec:LM-triangulated}, one can construct a triangulated structure on the category $\FM(\fh,W)$. We will not need this structure, so we will not consider this construction in detail.
\end{rmk}

\section{The left and right monodromy actions}
\label{sec:fm-monodromy}

The contraction maps $\lfrown: \Lambda^\vee \otimes \Lambda \to \Lambda$ and $\rfrown: R^\vee \otimes R^\natural \to R^\vee$ induce maps\index{frownl@$\lfrown$}\index{frownr@$\rfrown$}
\begin{align*}
\lfrown&: \Lambda^\vee \otimes \uHom_\FM(\cF,\cG) \to \uHom_\FM(\cF,\cG), \\
\rfrown&: \uHom_\FM(\cF,\cG) \otimes R^\natural \to \uHom_\FM(\cF,\cG).
\end{align*}

\begin{lem}
\label{lem:kappa-bifrown}
Let $f \in \uHom_\FM(\cF,\cG)$ and $g \in \uHom_\FM(\cG,\cH)$.
\begin{enumerate}
\item
\label{it:kappa-bifrown-left} 
For any $x \in V = (\Lambda^\vee)^{-1}_{-2}$, we have
\begin{gather*}
x \lfrown (g \circ f) = (x \lfrown g) \circ f + (-1)^{|g|} g \circ (x \lfrown f), \\
x \lfrown \kappa(f) = -\kappa(x \lfrown f), \\
x \lfrown \Theta = 0.
\end{gather*}
\item 
\label{it:kappa-bifrown-right}
For any $r \in V^* = (R^\natural)^{0}_{2}$, we have
\begin{gather*}
(g \circ f) \rfrown r= (g \rfrown r) \circ f + g \circ (f \rfrown r), \\
\kappa(f) \rfrown r = \kappa(f \rfrown r), \\
\Theta \rfrown r = 1 \otimes (\id \ustar r) \otimes 1.
\end{gather*}
\end{enumerate}
\end{lem}

\begin{proof}
The first two formulas in~\eqref{it:kappa-bifrown-left} follow from Lemma~\ref{lem:kappa-frown}. The third one is clear. The first formula in~\eqref{it:kappa-bifrown-right} follows from~\eqref{eqn:rfrown-deriv}, and the second and third formulas are clear.
\end{proof}

\begin{thm}
\label{thm:bifrown}
Let $\cF \in \FM(\fh,W)$.  There is a bigraded $\bk$-algebra homomorphism\index{muF@$\mu_\cF$}
\[
\mu_\cF: R^\vee \to \gEnd_\FM(\cF)
\]
determined by setting $\mu_\cF(x) = x \lfrown \delta$ for $x \in V = (R^\vee)^{0}_{-2}$.  More generally, for any $\cF, \cG \in \FM(\fh,W)$, the bigraded module $\gHom_\FM(\cF,\cG)$ is equipped with a canonical structure of a left $R^\vee$-module\index{convolution!starhat@$\hatstar$}
\begin{equation}\label{eqn:left-rv-action}
\hatstar: R^\vee \otimes \gHom_\FM(\cF,\cG) \to \gHom_\FM(\cF,\cG)
\end{equation}
given by $x \hatstar f = \mu_\cG(x) \circ f = f \circ \mu_\cF(x)$.  This action is compatible with composition in the following way: if $f \in \gHom_\FM(\cF,\cG)$, $g \in \gHom_\FM(\cG,\cH)$, and $x \in R^\vee$, then
\[
x \hatstar (g \circ f) = (x \hatstar g) \circ f = g \circ (x \hatstar f).
\]
\end{thm}

\begin{proof}
The same arguments as for Theorem~\ref{thm:frown} apply in this context.
\end{proof}

As in~\S\ref{sec:left-monodromy}, we call the $R^\vee$-action on $\gHom_\FM(\cF,\cG)$ defined above the \emph{left monodromy action}\index{monodromy action!left}.
For later use, we note that as in the proof of Theorem~\ref{thm:frown} we have, for all $f \in \uHom_\FM(\cF,\cG)$ and all $x \in V$, 
\begin{equation}\label{eqn:lfrown-diff}
x \lfrown d(f) = d(-x \lfrown f) + (x \lfrown \delta_\cG) \circ f -  f \circ (x \lfrown \delta_\cF).
\end{equation}

There is an obvious right action of $R^\vee$ on $\uHom_\FM(\cF,\cG)$, and it is easy to see that the differential on this space commutes with the right $R^\vee$-action, so $\gHom_\FM(\cF,\cG)$ inherits the structure of a right $R^\vee$-module.  This action, called the \emph{right monodromy action}\index{monodromy action!right}, will be denoted by
\begin{equation}\label{eqn:right-rv-action}
\hatstar : \gHom_\FM(\cF,\cG) \otimes R^\vee \to \gHom_\FM(\cF,\cG).
\end{equation}
The right monodromy action is compatible with composition: if $f \in \gHom_\FM(\cF,\cG)$, $g \in \gHom_\FM(\cG,\cH)$, and $x \in R^\vee$, then
\begin{equation}\label{eqn:rightmon-interchange}
(g \circ f) \hatstar x = (g \hatstar x) \circ f = g \circ (f \hatstar x).
\end{equation}

From~\eqref{eqn:uhom-mon-lmon}, we see that
\[
\uHom_\LM(\cF,\cG) \cong \uHom_\FM(\cF,\cG) \otimes_{R^\vee} \bk
\]
for any $\DiagBSp$-sequences $\cF$ and $\cG$. If $\cF$ and $\cG$ are free-monodromic complexes, then this isomorphism induces an isomorphism of dgg $\bk$-modules
\begin{equation}
\label{eqn:HomU-Hommon}
\uHom_\LM(\ForFMLM(\cF),\ForFMLM(\cG)) \cong \uHom_\FM(\cF,\cG) \otimes_{R^\vee} \bk.
\end{equation}

\begin{lem}
\label{lem:hom-mon-lmon}
Let $\cF$ and $\cG$ be free-monodromic complexes. Assume that
\[
\Hom_{\LM(\fh,W)}(\ForFMLM(\cF),\ForFMLM(\cG)[i] \langle j \rangle) = 0 \quad \text{unless $i=0$,}
\]
and that each $\Hom_{\LM(\fh,W)}(\ForFMLM(\cF),\ForFMLM(\cG) \langle j \rangle)$ is free over $\bk$.
Then we have
\[
\gHom_{\FM}(\cF,\cG)^i_j = 0 \quad \text{unless $i=0$,}
\]
$\gHom_{\FM}(\cF,\cG)^0_\bullet$ is graded free as a right $R^\vee$-module, and the morphism
\[
\gHom_{\FM}(\cF,\cG)^0_\bullet \otimes_{R^\vee} \bk \to \gHom_{\LM}(\ForFMLM(\cF),\ForFMLM(\cG))^0_\bullet
\]
induced by the functor $\ForFMLM$ is an isomorphism.
\end{lem}

\begin{proof}
Using Remark~\ref{rmk:Hom-shift}, we see that our assumption implies that
\[
\gHom_\LM(\ForFMLM(\cF), \ForFMLM(\cG))^i_j=0 \quad \text{unless $i=0$},
\]
or in other words that the complex $\uHom_\LM(\ForFMLM(\cF), \ForFMLM(\cG))$ is concentrated in cohomological degree $0$, and finally, using~\eqref{eqn:HomU-Hommon}, that the complex $\uHom_\FM(\cF,\cG) \otimes_{R^\vee} \bk$ is concentrated in cohomological degree $0$. Then the desired properties follow from Lemma~\ref{lem:cohomology}.
\end{proof}

Finally, in parallel with Theorem~\ref{thm:bifrown}, one might study the map
\[
V^* \to \uEnd_\FM(\cF)
\qquad\text{given by}\qquad
r \mapsto \delta \rfrown r.
\]
This \emph{does not} extend to a ring homomorphism with values in~$\gEnd_\FM(\cF)$, because it is not true that $d(\delta \rfrown r) = 0$ in general.  Instead, we have
\[
d(\delta \rfrown r) = \kappa(\delta \rfrown r) + \delta(\delta \rfrown r) + (\delta \rfrown r)\delta = \kappa(\delta) \rfrown r + \delta^2 \rfrown r = \Theta \rfrown r = \id \ustar r
\]
by Lemma~\ref{lem:kappa-bifrown}\eqref{it:kappa-bifrown-right}.
More generally, we have the following analogue of~\eqref{eqn:lfrown-diff}: for all $f \in \uHom_\FM(\cF,\cG)$ and all $r \in V^*$, 
\begin{equation}\label{eqn:rfrown-diff}
d(f) \rfrown r = d(f \rfrown r) + (\delta_\cG \rfrown r) \circ f + (-1)^{|f|+1} f \circ (\delta_\cF \rfrown r).
\end{equation}

\section{Examples of free-monodromic complexes}

We will now exhibit some examples of free-monodromic complexes. 

\subsection{The free-monodromic skyscraper sheaf}
\label{sss:t1-freemon}

Let $\Tmon_\varnothing$\index{tilting object!Ttilnoth@$\Tmon_\varnothing$} denote the $\DiagBSp$-sequence consisting of $B_\varnothing$, concentrated in degree~$0$.  We make it into a free-monodromic complex by equipping it with the differential
\[
\delta = \theta \in \uEnd_\FM(\Tmon_\varnothing)^1_0.
\]
We have $\delta^2 = 0$ and $\kappa(\delta) = \Theta$, so this element does indeed make $\Tmon_\varnothing$ into a free-monodromic complex. 
From~\eqref{eqn:R-End-E1} we deduce that we have
\[
\uEnd_\FM(\Tmon_\varnothing) = \Lambda \otimes R \otimes R^\vee.
\]
It is easy to see from the definitions that the differential in this case is $d_{\sA} \otimes \id_{R^\vee}$, so that the bigraded $\bk$-algebra homomorphism $\uEnd_\FM(\Tmon_\varnothing) \to R^\vee$ induced by the counits of $\Lambda$ and $R$ is a quasi-isomorphism of dgg $\bk$-algebras.  In particular, we have
\begin{equation}\label{eqn:endmon-t1}
\gEnd_\FM(\Tmon_\varnothing) \cong R^\vee.
\end{equation}

We now compute the left monodromy action on $\Tmon_\varnothing$. For any $x \in V =
(R^\vee)^{0}_{-2}$ we have
\[
x \lfrown \delta = x \lfrown \theta = \sum_{i = 1}^r \langle x, e_i
\rangle 1 \otimes 1 \otimes \check
e_i = 1 \otimes 1 \otimes x.\]
Thus
\begin{equation}
  \label{eq:1mon=id}
\mu_{\Tmon_\varnothing} = \id_{R^\vee} : 
R^\vee \to  \gEnd_\FM(\Tmon_\varnothing)
\stackrel{\eqref{eqn:endmon-t1}}{=} R^\vee.
\end{equation}

\subsection{The free-monodromic complex associated to a simple reflection}
\label{sss:ts-freemon}

Let $s$ be a simple reflection, and consider the object $\cT_s$ from Example~\ref{ex:ts-leftmon}.  We define the object $\Tmon_s$\index{tilting object!Ttils@$\Tmon_s$} by equipping the same $\DiagBSp$-sequence with the structure of a free-monodromic complex using the differential
\begin{multline*}
\delta = \usebox\lowerdot \oplus \usebox\upperdot \oplus (-\alpha_s \otimes \id) \oplus (\usebox\lowerdot \otimes \alpha_s^\vee) \oplus (-\alpha_s \otimes \id \otimes \alpha_s^\vee) \oplus (-\alpha_s \otimes \id \otimes \alpha_s^\vee) + \theta \\
\in \Hom(B_{\varnothing}(-1),B_s) \oplus \Hom(B_s,B_{\varnothing}(1)) \oplus (\Lambda^{1}_{2} \otimes \Hom(B_{\varnothing}(-1), (B_{\varnothing}(1))(-2))) \\
\oplus (\Hom(B_{\varnothing}(1), B_s(2)) \otimes (R^\vee)^{0}_{-2})
\oplus (\Lambda^{1}_{2} \otimes \End(B_{\varnothing}(1)) \otimes (R^\vee)^{0}_{-2}) \\
\oplus (\Lambda^{1}_{2} \otimes \End(B_s) \otimes (R^\vee)^{0}_{-2})
+ \uEnd_\FM(\Tmon_s)^{1}_{0} 
\subset \uEnd_\FM(\Tmon_s)^{1}_{0}.
\end{multline*}
We depict this with the following picture:
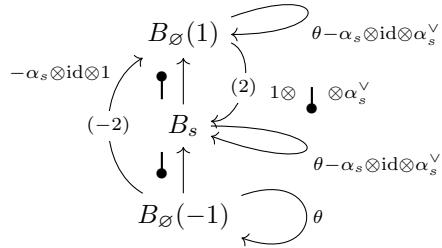
\[
\begin{tikzcd}
B_{\varnothing}(1) \ar[loop right, in=0, out=20, distance=40, "\theta - \alpha_s \otimes \id \otimes \alpha_s^\vee" pos=0.6] \ar[d, bend left=80, "(2)" description, "\ \ 1 \otimes \usebox\lowerdot \otimes \alpha_s^\vee" pos=0.4] \\
B_s \ar[u, "\usebox\upperdot"] \ar[loop right, in=-20, out=0, distance=50, "\theta - \alpha_s \otimes \id \otimes \alpha_s^\vee"] \\
B_{\varnothing}(-1) \ar[u, "\usebox\lowerdot"] \ar[uu, bend left=60, "(-2)" description, "-\alpha_s \otimes \id \otimes 1" near end] \ar[loop right, in=-20, out=20, distance=35, "\theta"]
\end{tikzcd}
\]
To verify that this element is indeed a differential, we first remark that in the $\bk$-algebra $\uEnd_\FM(\Tmon_s)$, by~\eqref{eqn:Theta-center} we have 
\[
\theta \usebox\lowerdot = -\usebox\lowerdot \theta \quad \text{and} \quad \theta \usebox\upperdot = -\usebox\upperdot \theta.
\]
We then compute the various components of $\delta^2 \in \uEnd_\FM(\Tmon_s)$:
\begin{align*}
B_{\varnothing}(-1) \leadsto B_{\varnothing}(-1) &: \theta^2 = 0 \\
B_{\varnothing}(-1) \leadsto B_s &: \usebox\lowerdot \theta + (\theta \usebox\lowerdot - \alpha_s \otimes \usebox\lowerdot \otimes \alpha_s^\vee) + (1 \otimes \usebox\lowerdot \otimes \alpha_s^\vee)(-\alpha_s \otimes \id \otimes 1) \\
&\qquad  = 0 \\
B_{\varnothing}(-1) \leadsto B_{\varnothing}(1) &: \usebox\lowerupperdot - (\alpha_s \otimes \id \otimes 1)\theta - \theta(\alpha_s \otimes \id \otimes 1) + (\alpha_s \wedge \alpha_s) \otimes \id \otimes \alpha_s^\vee \\
&\qquad = \alpha_s \star \id \\
B_s \leadsto B_s &: (\theta - \alpha_s \otimes \id \otimes \alpha_s^\vee)^2 + 1 \otimes \usebox\upperlowerdot \otimes \alpha_s^\vee  = 1 \otimes \usebox\upperlowerdot \otimes \alpha_s^\vee \\
B_s \leadsto B_{\varnothing}(1) &: (\theta-\alpha_s \otimes \id \otimes \alpha_s^\vee) \usebox\upperdot + \usebox\upperdot (\theta-\alpha_s \otimes \id \otimes \alpha_s^\vee)=0  \\
B_{\varnothing}(1) \leadsto B_s &: (\theta - \alpha_s \otimes \id \otimes \alpha_s^\vee)(1 \otimes \usebox\lowerdot \otimes \alpha_s^\vee) \\
&\qquad + (1 \otimes \usebox\lowerdot \otimes \alpha_s^\vee)(\theta - \alpha_s \otimes \id \otimes \alpha_s^\vee) = 0 \\
B_{\varnothing}(1) \leadsto B_{\varnothing}(1) &: 1 \otimes \usebox\lowerupperdot \otimes \alpha_s^\vee + (\theta - \alpha_s \otimes \id \otimes \alpha_s^\vee)^2 = 1 \otimes (\alpha_s \star \id) \otimes \alpha_s^\vee.
\end{align*}
We depict this as follows:
\[
\delta^2 =
\begin{tikzcd}
B_{\varnothing}(1) \ar[loop right, in=-20, out=20, distance=35, "1 \otimes (\alpha_s \star \id) \otimes \alpha_s^\vee"]\\
B_s \ar[loop right, in=-20, out=20, distance=35, "1 \otimes \usebox\upperlowerdot \otimes \alpha_s^\vee"] \\
B_{\varnothing}(-1). \ar[uu, bend left=60, "1 \otimes (\alpha_s \star \id) \otimes 1"]
\end{tikzcd}
\]
On the other hand, we have
\[
\kappa(\delta) =
\begin{tikzcd}
B_{\varnothing}(1) \ar[loop right, in=-20, out=20, distance=35, "\Theta - 1 \otimes (\alpha_s \star \id) \otimes \alpha_s^\vee"]\\
B_s \ar[loop right, in=-20, out=20, distance=35, "\kappa(\theta) - 1 \otimes (\alpha_s \star \id) \otimes \alpha_s^\vee"] \\
B_{\varnothing}(-1). \ar[uu, bend left=60, "-1 \otimes (\alpha_s \star \id) \otimes 1"] \ar[loop right, in=-20, out=20, distance=35, "\Theta"]
\end{tikzcd}
\]
Consider the $0$th term of this $\DiagBSp$-sequence.  In $\uEnd_\BE(B_s)$, by the nil-Hecke relation (see~\S\ref{sec:ew-diagram}) we have
\[
\id \star e_i  = \la\alpha_s^\vee, e_i\ra \usebox\upperlowerdot + s(e_i) \star \id = \la\alpha_s^\vee, e_i\ra \usebox\upperlowerdot + (e_i - \la \alpha_s^\vee,e_i\ra\alpha_s) \star \id.
\]
Since $\sum_i \la \alpha_s^\vee, e_i\ra \check e_i = \alpha_s^\vee$, we deduce that in $\uEnd_\BE(B_s) \otimes R^\vee$, we have
\[
\Theta = \usebox\upperlowerdot \otimes \alpha_s^\vee + \kappa(\theta) - (\alpha_s \star \id) \otimes \alpha_s^\vee.
\]
In particular, we conclude that $\delta^2 + \kappa(\delta) = \Theta$ in $\uEnd_\FM(\Tmon_s)$, as desired.

Let us now compute the left monodromy action on $\Tmon_s$.  For $x \in V$, note that $x \lfrown \theta = \sum e_i(x) \otimes \id \otimes \check e_i = 1 \otimes \id \otimes x$, and that $x \lfrown (\alpha_s \otimes \id \otimes \alpha_s^\vee) = 1 \otimes \id \otimes \alpha_s(x)\alpha_s^\vee$.  We deduce that
\begin{equation*}
\label{eqn:ts-free-leftmon-deg2}
\mu_{\Tmon_s}(x) = 
\begin{tikzcd}[column sep = huge]
B_{\varnothing}(1) \ar[r, "1 \otimes \id \otimes s(x)"]  & B_{\varnothing}(1) \\
B_s  \ar[r, bend left=20, "1\otimes \id \otimes s(x)"] & B_s  \\
B_{\varnothing}(-1) \ar[r, "1\otimes \id \otimes x"] \ar[uur, bend right=20, "(-2)" description, "-\alpha_s(x) \cdot 1 \otimes \id \otimes 1" pos=0.1] & B_{\varnothing}(-1) 
\end{tikzcd}.
\end{equation*}
From this one can deduce that more generally for any $f \in R^\vee$ we have
\begin{equation}
\label{eqn:ts-free-leftmon}
\mu_{\Tmon_s}(f) = 
\begin{tikzcd}[column sep = huge]
B_{\varnothing}(1) \ar[r, "1 \otimes \id \otimes s(f)"]  & B_{\varnothing}(1) \\
B_s  \ar[r, bend left=20, "1\otimes \id \otimes s(f)"] & B_s  \\
B_{\varnothing}(-1) \ar[r, "1\otimes \id \otimes f"] \ar[uur, bend right=20, "(-2)" description, "-\partial_s(f) \cdot 1 \otimes \id \otimes 1" pos=0.1] & B_{\varnothing}(-1) 
\end{tikzcd},
\end{equation}
where $\partial_s$ is the Demazure operator as in~\S\ref{sec:ew-diagram}.

\begin{prop}
\label{prop:End-Tmons}
There is a canonical isomorphism of bigraded $\bk$-alge\-bras
\[
\gEnd_\FM(\Tmon_s) \cong R^\vee \otimes_{(R^\vee)^s} R^\vee.
\]
\end{prop}

Here, $(R^\vee)^s$ is the subring of $R^\vee$ consisting of elements that are fixed by $s$. 

\begin{proof}
In view of Example~\ref{ex:ts-mon-action} and Lemma~\ref{lem:hom-mon-lmon}, as a bigraded right $R^\vee$-module, $\gEnd_\FM(\Tmon_s)$ must be free of rank~$2$, and generated in bidegrees $(0,0)$ and $(0,-2)$.

Let $\phi: R^\vee \otimes R^\vee \to \gEnd_\FM(\Tmon_s)$ be the map given by $\phi(x \otimes y) = \mu_{\Tmon_s}(x)y$.  
By Theorem~\ref{thm:bifrown}, this map is a bigraded $\bk$-algebra morphism. Moreover,
it follows from~\eqref{eqn:ts-free-leftmon} that $\phi$ descends to a map $\phi' : R^\vee \otimes_{(R^\vee)^s} R^\vee \to \gEnd_\FM(\Tmon_s)$.
Now the graded right $R^\vee$-module $R^\vee \otimes_{(R^\vee)^s} R^\vee$ is also free of rank~$2$, and generated in bidegrees $(0,0)$ and $(0,-2)$, see e.g.~\cite[Claim~3.9]{ew}. From the discussion in Example~\ref{ex:ts-mon-action} we know that $\phi' \otimes_{R^\vee} \bk$ is an isomorphism. It follows that $\phi'$ is also an isomorphism.
\end{proof}

\subsection{Another free-monodromic complex associated to $s$}
\label{sss:ts-freemon-2}

The free-monodromic complex $\Tmon_s$ described in \S\ref{sss:ts-freemon} appears to be somewhat ``asymmetric'': indeed, there is another free-mono\-dromic complex $\Tmon'_s$ on the same underlying $\DiagBSp$-sequence given by\index{tilting object!Ttilsp@$\Tmon_s'$}
\[
\delta = \begin{tikzcd}
B_\varnothing(1) \ar[loop right, in=-20, out=20, distance=40, "\theta"] \\
B_s \ar[u, "\usebox\upperdot"] \ar[loop right, in=0, out=20, distance=50, "\theta - \alpha_s \otimes \id \otimes \alpha_s^\vee"] \ar[d, bend left=80, "(2)" description, "\ \ 1 \otimes \usebox\upperdot \otimes \alpha_s^\vee" pos=0.6] \\
B_\varnothing(-1) \ar[u, "\usebox\lowerdot"] \ar[uu, bend left=60, "(-2)" description, "-\alpha_s \otimes \id \otimes 1" near end] \ar[loop right, in=-20, out=0, distance=35, "\theta - \alpha_s \otimes \id \otimes \alpha_s^\vee" pos=0.4]
\end{tikzcd}.
\]
In fact, $\Tmon_s$ and $\Tmon_s'$ are isomorphic objects of $\FM(\fh,W)$.  Consider the elements $f \in \uHom_\FM(\Tmon_s,\Tmon_s')$ and $g \in \uHom_\FM(\Tmon_s',\Tmon_s)$ given by
\[
f =
\begin{tikzcd}[column sep=huge]
B_\varnothing(1) \ar[r, "\id"] \ar[ddr, bend right=20, "(2)" description, "1 \otimes \id \otimes \alpha_s^\vee" pos=0.9] & B_\varnothing(1) \\
B_s  \ar[r, bend left=20, "\id"] & B_s  \\
B_\varnothing(-1) \ar[r, "\id"] & B_\varnothing(-1) 
\end{tikzcd}
\qquad
g =
\begin{tikzcd}[column sep=huge]
B_\varnothing(1) \ar[r, "\id"] \ar[ddr, bend right=20, "(2)" description, "-1 \otimes \id \otimes \alpha_s^\vee" pos=0.9] & B_\varnothing(1) \\
B_s  \ar[r, bend left=20, "\id"] & B_s  \\
B_\varnothing(-1) \ar[r, "\id"] & B_\varnothing(-1) 
\end{tikzcd}.
\]

One can check that $f$ and $g$ both define morphisms in $\FM(\fh,W)$, and that they are isomorphisms, inverse to one another.

In fact, $\Tmon_s$ and $\Tmon'_s$ are \emph{canonically} isomorphic, in the following sense.  Note that the objects $\ForFMLM(\Tmon_s)$ and $\ForFMLM(\Tmon_s')$ are equal. It follows from Lemma~\ref{lem:hom-mon-lmon} that the natural map
\[
\gHom_\FM(\Tmon_s,\Tmon_s') \to \gEnd_\LM(\ForFMLM(\Tmon_s))
\]
is an isomorphism in degree $(0,0)$, so the identity map $\id: \ForFMLM(\Tmon_s) \to \ForFMLM(\Tmon_s')$ lifts to a \emph{unique} morphism in $\FM(\fh,W)$.  This is the morphism induced by the chain map $f$ considered above.

In the remainder of the paper, we will work only with $\Tmon_s$.

\subsection{Free-monodromic unit and counit morphisms}
\label{sss:unit-freemon}

We define a morphism\index{etahats@$\hat\eta_s$}
\[
\hat\eta_s: \Tmon_\varnothing\la -1\ra \to \Tmon_s
\]
in $\FM(\fh,W)$ by
\begin{multline*}
\hat\eta_s = \id_{B_\varnothing(1)} \oplus ( -1 \otimes \id_{B_\varnothing(1)} \otimes \alpha_s^\vee) \\
\in \Hom(B_\varnothing(1), B_\varnothing(1)) \oplus (\Hom(B_\varnothing(1), B_\varnothing(-1)(2)) \otimes (R^\vee)^{0}_{-2}) \\
\subset \uHom_\FM(\Tmon_\varnothing \langle -1 \rangle, \Tmon_s).
\end{multline*}
This morphism is depicted by:
\[
\begin{tikzcd}[column sep=large]
B_\varnothing(1) \ar[loop left,"\theta"]
  \ar[rr, "\id"]
  \ar[ddrr, bend right=20, "(2)" description, "-1 \otimes \id \otimes \alpha_s^\vee\ \ "']
&&B_\varnothing(1) \ar[loop right, in=0, out=20, distance=40, "\theta - \alpha_s \otimes \id \otimes \alpha_s^\vee" pos=0.6] \ar[d, bend left=80, "(2)" description, "\ \ 1 \otimes \usebox\lowerdot \otimes \alpha_s^\vee" pos=0.4] \\
&&B_s \ar[u, "\usebox\upperdot"] \ar[loop right, in=-20, out=0, distance=50, "\theta - \alpha_s \otimes \id \otimes \alpha_s^\vee" pos=0.6] \\
&&B_\varnothing(-1). \ar[u, "\usebox\lowerdot"] \ar[uu, bend left=60, "(-2)" description, "-\alpha_s \otimes \id \otimes 1" near end] \ar[loop right, in=-20, out=20, distance=35, "\theta"]
\end{tikzcd}
\]
To verify that this is indeed a morphism, we must check that $d(\hat\eta_s) = 0$. It is clear that $\kappa(\hat\eta_s) = 0$, so it is enough to check that
\[
\delta \hat\eta_s = \hat\eta_s\delta.
\]
The picture above makes it straightforward to verify this.

Next, we define a morphism\index{epsilonhats@$\hat\epsilon_s$}
\[
\hat\epsilon_s : \Tmon_s \to \Tmon_\varnothing\la 1\ra
\]
in $\FM(\fh,W)$  by
\[
\hat\epsilon_s = - \id_{B_\varnothing(-1)} \in \Hom(B_\varnothing(-1), B_\varnothing(-1)) \subset \uHom_\FM(\Tmon_s,\Tmon_\varnothing\la 1\ra)^{0}_{0}.
\]
This morphism is depicted as follows:
\[
\begin{tikzcd}
B_\varnothing(1) \ar[loop right, in=0, out=20, distance=40, "\theta - \alpha_s \otimes \id \otimes \alpha_s^\vee" pos=0.6] \ar[d, bend left=80, "(2)" description, "\ \ 1 \otimes \usebox\lowerdot \otimes \alpha_s^\vee" pos=0.4] \\
B_s \ar[u, "\usebox\upperdot"] \ar[loop right, in=-20, out=0, distance=50, "\theta - \alpha_s \otimes \id \otimes \alpha_s^\vee" pos=0.6] \\
B_\varnothing(-1) \ar[u, "\usebox\lowerdot"] \ar[uu, bend left=60, "(-2)" description, "-\alpha_s \otimes \id \otimes 1" near end] \ar[loop left, "\theta"]
\ar[rrrr, "-\id"']
&&&& B_\varnothing(-1). \ar[loop right,"\theta"]
\end{tikzcd}.
\]

The following proposition shows that these morphisms satisfy relations similar to the barbell and nil-Hecke relations of~\S\ref{sec:ew-diagram}.

\begin{prop}
\begin{enumerate}
\item 
\label{it:nilHecke-1}
In $\gEnd_\FM(\Tmon_\varnothing)$, we have
\[
\hat\epsilon_s \circ \hat\eta_s = \alpha_s^\vee \hatstar \id = \id \hatstar \alpha_s^\vee.
\]
\item 
\label{it:nilHecke-2}
In $\gEnd_\FM(\Tmon_s)$, for any $x \in V$ we have
\[
\id \hatstar x = \alpha_s(x) \hat\eta_s \hat\epsilon_s + s(x) \hatstar \id.
\]
\end{enumerate}
\end{prop}

\begin{proof} 
We first prove~\eqref{it:nilHecke-1}. We have
\[
\hat\epsilon_s \circ \hat\eta_s = 1 \otimes \id \otimes \alpha_s^\vee
\in \Lambda_0^0 \otimes \uHom_\BE(\Tmon_\varnothing\langle -1 \rangle,\Tmon_\varnothing\la
1\ra)^{0}_{2} \otimes R^0_{-2} \subset \uEnd_\FM(\Tmon_\varnothing)^0_{-2}.
\]
It is immediate from the definitions that (the image of) $1 \otimes \id
\otimes \alpha_s^\vee$ in $\gEnd_\FM(\Tmon_\varnothing)_{-2}^0$ agrees with $\id
\hatstar \alpha_s^\vee$. The equality $\alpha_s^\vee \hatstar \id =\id
\hatstar \alpha_s^\vee$ follows from \eqref{eq:1mon=id}.

We now turn to~\eqref{it:nilHecke-2}. Choose $x \in V$. Using the identities
$\alpha_s(sx) = -\alpha_s(x)$ and $x - s(x) = \alpha_s(x)
\alpha_s^\vee$ together with \eqref{eqn:ts-free-leftmon}, we obtain
\begin{equation*}
\id \hatstar x - s(x) \hatstar \id = 
\begin{tikzcd}[column sep = huge]
B_\varnothing(1) \ar[r, "0"]  & B_\varnothing(1) \\
B_s  \ar[r, bend left=20, "0"] & B_s  \\
B_\varnothing(-1) \ar[r, "1\otimes \id \otimes \alpha_s(x) \alpha_s^\vee" pos=0.45] \ar[uur, bend right=0, "(-2)" description, "-\alpha_s(x) \cdot 1 \otimes \id \otimes 1" pos=0.01] & B_\varnothing(-1) 
\end{tikzcd}.
\end{equation*}
It is easy to check that this morphism agrees with $\alpha_s(x)
\hat\eta_s \hat\epsilon_s$, and (2) follows.
\end{proof}

\chapter{Free-monodromic convolution}
\label{chap:fm-convolution}

\begin{figure}
\[
\begin{tikzcd}
&&&
\DiagBSp(\fh,W), \star \ar[d, hook] \\
\FM(\fh,W) \ar[dr, "\ForFMLM"'] &&&
\BE(\fh,W),\ustar \ar[dl, "\ForBERE"] \\
& \LM(\fh,W) \ar[r, "\sim", "\ForLMRE"'] &
\RE(\fh,W)
\end{tikzcd}
\]
\caption{Various categories of $\DiagBSp$-sequences}\label{fig:diag-categories}
\end{figure}

A summary of the categories defined in the preceding two chapters is shown in Figure~\ref{fig:diag-categories} below.  The categories in the rightmost portion of this diagram are monoidal.  The rest of this paper is guided by the dream that $\FM(\fh,W)$ ``ought to'' be monoidal as well (or at least have a large monoidal subcategory) with respect to some new operation $\hatstar$.

In this chapter, we define $\hatstar$ for a certain class of objects, called \emph{convolutive} objects, as well as for morphisms between them.  Unfortunately, we do not know how to define $\hatstar$ on all of $\FM(\fh,W)$; and when it is defined, it is not obvious whether it is a bifunctor.  Indeed, the question of the bifunctoriality of $\hatstar$ will occupy us for the rest of the paper.  In the meantime, the reader should take care not to be led astray by the fact that the notation ``looks'' like a bifunctor.  

(The notation $\hatstar$ has already been used for the left and right monodromy actions in~\eqref{eqn:left-rv-action} and~\eqref{eqn:right-rv-action}, but those actions will not  appear in this chapter.  This overlap in  notation will be justified later by Lemma~\ref{lem:unitor}.)

\section{Convolutive complexes}
\label{sec:convolutive-UGU}

Let $\cF$ and $\cG$ be free-monodromic complexes.  Throughout this chapter, we will frequently need to refer to the subspaces $\Lambda \otimes \uHom_\BE(\cF,\cG)$ and $\uHom_\BE(\cF,\cG) \otimes R^\vee$ of $\uHom_\FM(\cF,\cG)$.  For brevity, we introduce the following notation:\index{HomuFMlhd@$\uHom_\FM^\lhd$}\index{HomuFMrhd@$\uHom_\FM^\rhd$}
\begin{align*}
\uHom_\FM^\lhd(\cF,\cG) &= \Lambda \otimes \uHom_\BE(\cF,\cG), \\
\uHom_\FM^\rhd(\cF,\cG) &= \uHom_\BE(\cF,\cG) \otimes R^\vee.
\end{align*}
Note that composition respects these subspaces: that is, if $f \in \uHom_\FM^\lhd(\cF,\cG)$ and $g \in \uHom_\FM^\lhd(\cG,\cH)$, then $g \circ f \in \uHom_\FM^\lhd(\cF,\cH)$, and likewise for the $\uHom_\FM^\rhd$ spaces. (On the other hand, the differential does \emph{not} stabilize these subspaces in general.) The convolution map $\ustar$ on $\uHom_\BE$ induces a map\index{convolution!staru@$\ustar$}
\begin{equation}
\label{eqn:ustar-rhd-lhd}
\ustar: \uHom_\FM^\lhd(\cF,\cG) \otimes \uHom_\FM^\rhd(\cF',\cG') 
 \to \uHom_\FM(\cF \ustar \cF', \cG \ustar \cG').
\end{equation}
By~\eqref{eqn:interchange},
this map obeys the signed interchange law:
\begin{equation}
\label{eqn:interchange-law-rhd-lhd}
(g \circ f) \ustar (k \circ h) = (-1)^{|f||k|}(g \ustar k) \circ (f \ustar h).
\end{equation}

\begin{defn}
A free-monodromic complex $(\cF,\delta)$ is said to be \emph{convolutive}\index{convolutive} if
\[
\delta \in (\bk \oplus V^*(-2)[1]) \otimes \uEnd_\BE(\cF) \otimes (\bk \oplus V\la -2\ra) \subset \uEnd_\FM(\cF).
\]
The \emph{category of convolutive complexes}, denoted by $\Conv_\FM(\fh,W)$\index{ConvFM@$\Conv_\FM$}, is defined to be the full (but not strictly full!) subcategory of $\FM(\fh,W)$ consisting of convolutive complexes.
\end{defn}

\begin{lem}
\label{lem:convolutive-mu-nu}
Let $\cF$ be a convolutive free-monodromic complex.
\begin{enumerate}
\item
\label{it:convolutive-mu}
The map $\mu_\cF : V\la -2\ra \to \uEnd_\FM(\cF)$ given by $\mu_\cF(x) = x \lfrown \delta$ extends to a dgg algebra homomorphism\index{muF@$\mu_\cF$}
\[
\mu_\cF : R^\vee \to \uEnd_\FM(\cF)
\]
which takes values in $\uEnd_\FM^\rhd(\cF)$.
\item
\label{it:convolutive-nu}
The map $\nu_\cF : V^*\langle 2\rangle[-1] \to \uEnd_\FM(\cF)$ given by $\nu_\cF(r) = \delta \rfrown r$ extends to a bigraded ring homomorphism\index{nuF@$\nu_\cF$}
\[
\nu_\cF: \Lambda \to \uEnd_\FM^\lhd(\cF).
\]
This map is not a dgg algebra homomorphism in general; instead, for $r \in V^*\langle 2\rangle[-1]$, it satisfies
\[
d(\nu_\cF(r)) = \id \ustar r.
\]
\end{enumerate}
\end{lem}

\begin{proof}
\eqref{it:convolutive-mu}
The claims that $\mu_\cF(V) \subset \uEnd_\FM^\rhd(\cF)$ and $\nu_\cF(V^*) \subset \uEnd_\FM^\lhd(\cF)$ are immediate from the definition of convolutivity.  This definition also implies that
\begin{equation}\label{eqn:convol-frown}
\begin{aligned}
x \lfrown (y \lfrown \delta) &= 0 \qquad \text{for all $x, y \in V$,} \\
(\delta \rfrown s) \rfrown r &= 0 \qquad \text{for all $r,s \in V^*$.}
\end{aligned}
\end{equation}

To show that $\mu_\cF$ extends to a ring homomorphism from $R^\vee$, we must show that $\mu_\cF(x)\mu_\cF(y) = \mu_\cF(y)\mu_\cF(x)$ for all $x,y \in V$.  From~\eqref{eqn:lfrown-diff}, we have
\[
x \lfrown d(y \lfrown \delta) = d(-x \lfrown (y \lfrown \delta)) + (x \lfrown \delta)(y \lfrown \delta) - (y \lfrown \delta)(x \lfrown \delta).
\]
Recall from Theorem~\ref{thm:bifrown} that $d(y \lfrown \delta) = 0$.  The first term on the right-hand side also vanishes by~\eqref{eqn:convol-frown}, so we conclude that $\mu_\cF(x)\mu_\cF(y) - \mu_\cF(y)\mu_\cF(x) = 0$, as desired. The fact that $\mu_\cF$ is compatible with differentials also follows from Theorem~\ref{thm:bifrown}.

\eqref{it:convolutive-nu}
To show that $\nu_\cF$ extends to a ring morphism $\Lambda \to \uEnd_\FM^\lhd(\cF)$ we need to prove that for any $r \in V^*$ we have
\begin{equation}
\label{eqn:condition-nu}
(\delta \rfrown r) \circ (\delta \rfrown r)=0.
\end{equation}
If $e_1, \sdots, e_r$ is a basis of $V^*$ and $\check e_1, \sdots, \check e_r$ is the dual basis for $V$, then we write
\[
\delta = \delta_0 + \sum_i e_i \otimes \delta_i + \sum_i \delta_i' \otimes \check e_i + \sum_{i,j} e_i \otimes \delta_{i,j} \otimes \check e_j
\]
with $\delta_0$, $\delta_i$, $\delta'_i$ and $\delta_{i,j}$ in $\uEnd_\BE(\cF)$. By assumption we have $\delta \circ \delta + \kappa(\delta)=\Theta$. Considering the components of this equality in $\uEnd_\BE(\cF) \otimes \bk \check e_i \check e_j$, we deduce that
\[
\delta'_i \circ \delta'_i = 0
\]
for any $i$, and that
\[
\delta'_i \delta'_j + \delta'_j \delta'_i=0
\]
if $i \neq j$. Similarly, if $i<k$, considering the components in $\bk e_i \wedge e_k \otimes \uEnd_\BE(\cF) \otimes \bk \check e_j \check e_l$ we see that
\[
\delta_{i,j} \delta_{k,j} - \delta_{k,j} \delta_{i,j}=0
\]
for any $j$, and that
\[
\delta_{i,j} \delta_{k,l} + \delta_{i,l} \delta_{k,j} - \delta_{k,j} \delta_{i,l} - \delta_{k,l} \delta_{i,j}=0.
\]
Now we have
\[
\delta \rfrown r = \sum_i \langle \check e_i,r \rangle \delta'_i + \sum_{i,j} \langle \check e_j, r \rangle e_i \otimes \delta_{i,j},
\]
and using the formulas above one easily checks that~\eqref{eqn:condition-nu} is satisfied.

The fact that $d(\nu_\cF(r)) = \id \ustar r$ has already been noticed in the comments after Lemma~\ref{lem:hom-mon-lmon}.
\end{proof}

\begin{rmk}
\label{rmk:mu-nu-commute}
 We have observed in Theorem~\ref{thm:bifrown} that if $\cF,\cG$ are free-monodromic complexes and if $f \in \gHom_\FM(\cF,\cG)$, then $f \circ \mu_\cF(x) = \mu_\cG(x) \circ f$ in $\gHom_\FM(\cF,\cG)$ for any $x \in R^\vee$. If $\cF$ and $\cG$ are convolutive, then the analogous claim in $\uHom_\FM(\cF, \cG)$ is \emph{not} true in general. However, if $f$ and $d(f)$ both belong to $\uHom_\FM^\rhd(\cF,\cG)$, then~\eqref{eqn:lfrown-diff} shows that $f \circ \mu_\cF(x) = \mu_\cG(x) \circ f$ in $\uHom_\FM(\cF,\cG)$.
 
 Similarly, if $f$ and $d(f)$ both belong to $\uHom_\FM^\lhd(\cF,\cG)$, then by~\eqref{eqn:rfrown-diff} we have $\nu_\cG(y) \circ f = (-1)^{|y||f|} f \circ \nu_\cF(y)$ for any $y \in \Lambda$.
\end{rmk}

\begin{ex}
\label{ex:T1-mon-mu-nu}
Consider the special case $\cF=\Tmon_\varnothing$ (see~\S\ref{sss:t1-freemon}). Then we have
\[
\uEnd_\FM(\Tmon_\varnothing) = \Lambda \otimes R \otimes R^\vee.
\]
From the definitions one can check that for $x \in V$ and $y \in V^*$ we have
\[
\mu_{\Tmon_\varnothing}(x) = 1 \otimes 1 \otimes x, \qquad \nu_{\Tmon_\varnothing}(y) = y \otimes 1 \otimes 1.
\]
\end{ex}

As in~\S\ref{sec:fm-mixed-complexes}, let us choose a basis $e_1, \sdots, e_r$ for $V^*$, and let $\check e_1, \sdots, \check e_r$ be the dual basis for $V$.
If $(\cF,\delta)$ is a convolutive complex, we define
\[
\delta^\lhd := \delta - \sum \nu_{\cF}(e_i) \check e_i
\qquad\text{and}\qquad
\delta^\rhd := \delta - \sum e_i \mu_{\cF}(\check e_i).
\]

\begin{lem}
\label{lem:delta-lhd-rhd}
Let $(\cF,\delta)$ be a convolutive free-monodromic complex.
\begin{enumerate}
\item
\label{it:delta-lhd-rhd-1}
We have $\delta^\lhd \in \uEnd_\FM^\lhd(\cF)$.  Moreover, for any $r \in \Lambda$, we have
\[
\delta\nu_{\cF}(r) + (-1)^{|r|+1}\nu_{\cF}(r)\delta = \delta^\lhd \nu_{\cF}(r) + (-1)^{|r|+1}\nu_{\cF}(r)\delta^\lhd.
\]
\item
\label{it:delta-lhd-rhd-2}
We have $\delta^\rhd \in \uEnd_\FM^\rhd(\cF)$.  Moreover, for any $x \in R^\vee$, we have
\[
\delta\mu_{\cF}(x) - \mu_{\cF}(x)\delta = 0 = \delta^\rhd \mu_{\cF}(x) - \mu_{\cF}(x) \delta^\rhd.
\]
\item
\label{it:delta-lhd-rhd-3}
We have
\[
(\delta^\rhd)^2 = \delta^2
\qquad\text{and}\qquad
d(\delta^\rhd) \in \uEnd_\FM^\rhd(\cF).
\]
\end{enumerate}
\end{lem}

\begin{proof}
Part~\eqref{it:delta-lhd-rhd-1} follows from the facts that $\nu_{\cF}$ is a ring morphism and that the elements $\check e_i$ are central in $\uEnd_\FM(\cF)$. 
In~\eqref{it:delta-lhd-rhd-2}, the first equality follows from the fact that $d(\mu_\cF(x))=0$. We deduce the second equality by noting that $\mu_{\cF}$ is an algebra homomorphism, and that
$\mu_{\cF}(x)$ commutes with $e_i$ (because it is of even degree). Let us now prove~\eqref{it:delta-lhd-rhd-3}. We have
\begin{multline*}
(\delta^\rhd)^2 = (\delta - \sum_i e_i \mu_{\cF}(\check e_i))^2 \\
= \delta^2 - \delta \cdot \left( \sum_i e_i \mu_\cF(\check e_i) \right) - \left( \sum_i e_i \mu_\cF(\check e_i) \right) \cdot \delta + \left( \sum_i e_i \mu_\cF(\check e_i) \right)^2.
\end{multline*}
The fourth term on the right-hand side is easily seen to vanish, and then the first formula follows from~\eqref{it:delta-lhd-rhd-2} and the fact that $\delta \cdot e_i = -e_i \cdot \delta$ (because $\delta$ is of odd degree). Finally, we observe that
\[
d(\delta^\rhd) = d(\delta) - \sum_i d(e_i \mu_\cF(\check e_i)) = \Theta_{\cF} + \delta^2 - \sum_i d(e_i) \cdot \mu_\cF(\check e_i) + \sum_i e_i \cdot d(\mu_\cF(\check e_i)).
\]
On the right-hand side, the first term belongs to $\uEnd_\FM^\rhd(\cF)$, as does the second one, since $(\delta^\rhd)^2 = \delta^2$. Since $d(e_i) = 1 \otimes (e_i \ustar \id) \otimes 1$, the third term belongs to $\uEnd_\FM^\rhd(\cF)$. And the fourth one vanishes because $\mu$ is a dgg algebra morphism. The claim follows.
\end{proof}

\section{Convolution on dgg Hom spaces}
\label{sec:fm-convolution-morph}

Let $\cF, \cG, \cF', \cG'$ be convolutive free-monodromic complexes.  We define the map\index{convolution!starhat@$\hatstar$}
\[
\hatstar: \uHom_\FM(\cF,\cG) \otimes \uHom_\FM(\cF',\cG') \to \uHom_\FM(\cF \ustar \cF', \cG \ustar \cG')
\]
to be the following composition:
\begin{multline*}
\uHom_\FM(\cF,\cG) \otimes \uHom_\FM(\cF',\cG') \\
= \Lambda \otimes \uHom_\BE(\cF,\cG) \otimes R^\vee \otimes \Lambda \otimes \uHom_\BE(\cF',\cG') \otimes R^\vee 
\xrightarrow{\id \otimes \id \otimes \mu_{\cG'} \otimes \nu_\cF \otimes \id \otimes \id} \\
\Lambda \otimes \uHom_\BE(\cF,\cG) \otimes \uEnd_\BE(\cG') \otimes R^\vee \otimes \Lambda \otimes \uEnd_\BE(\cF) \otimes \uHom_\BE(\cF',\cG') \otimes R^\vee \\
\xrightarrow{\tau}
\Lambda \otimes \Lambda \otimes \uHom_\BE(\cF,\cG) \otimes \uEnd_\BE(\cF) \otimes  \uEnd_\BE(\cG') \otimes \uHom_\BE(\cF',\cG') \otimes R^\vee \otimes R^\vee\\
\xrightarrow{\circ \otimes \circ \otimes \circ \otimes \circ}
\Lambda \otimes \uHom_\BE(\cF,\cG) \otimes \uHom_\BE(\cF',\cG') \otimes R^\vee \\
\xrightarrow{\id \otimes \ustar \otimes \id}
\Lambda \otimes \uHom_\BE(\cF \ustar \cF', \cG \ustar \cG') \otimes R^\vee
= \uHom_\FM(\cF \ustar \cF', \cG \ustar \cG').
\end{multline*}
It will be useful to express this composition in the following way. Consider some $f \in \uHom_\FM(\cF,\cG)$, and suppose $f = rf'x$, where $r \in \Lambda$, $f' \in \uHom_\BE(\cF,\cG)$, and $x \in R^\vee$.  Similarly, suppose $h \in \uHom_\FM(\cF',\cG')$ can be written as $h = th'z$.  Then we have
\[
f \hatstar h = rf' \nu_{\cF}(t) \ustar \mu_{\cG'}(x) h'z,
\]
where $\ustar$ is the map defined in~\eqref{eqn:ustar-rhd-lhd}.

\begin{lem}
\label{lem:monconv-weak-interchange}
Let $\cF, \cG, \cH, \cF', \cG', \cH'$ be convolutive free-monodromic complexes.
Let $f \in \uHom_\FM(\cF,\cG)$, $g \in \uHom_\FM(\cG,\cH)$, $h \in \uHom_\FM(\cF',\cG')$, $k \in \uHom_\FM(\cG',\cH')$.  Assume that one of the following conditions holds:
\begin{enumerate}
\item
The elements $f$ and $d(f)$ both lie in $\uHom_\FM^\lhd(\cF,\cG)$.
\item
The elements $k$ and $d(k)$ both lie in $\uHom_\FM^\rhd(\cG',\cH')$.
\item
We have $f \in \uHom_\FM^\lhd(\cF,\cG)$ and $k \in \uHom_\FM^\rhd(\cF,\cG)$.
\end{enumerate}
Then we have
\begin{equation}\label{eqn:monconv-weak-interchange}
(g \circ f) \hatstar (k \circ h) = (-1)^{|f||k|}(g \hatstar k) \circ (f \hatstar h).
\end{equation}
\end{lem}

In particular, this lemma can be invoked when either $f$ or $k$ is an identity map.

\begin{proof}
Let us first consider the special case where $\cF = \cG$ and $f = \id_\cF$.  We may assume without loss of generality that $g = s g' y$, where $s \in \Lambda$, $g' \in \uHom_\BE(\cF,\cH)$, and $y \in R^\vee$.  Similarly, assume that $h = t h' z$ and $k = u k' w$.  We have
\[
g \hatstar (k h) = (-1)^{|t||k'|} sg'y \hatstar ut k'h' wz 
=  (-1)^{|t||k'|} sg' \nu_{\cF}(ut) \ustar \mu_{\cH'}(y) k'h'wz
\]
and
\begin{multline*}
(g \hatstar k) \circ (\id_\cF \hatstar h) = (sg' \nu_\cF(u) \ustar \mu_{\cH'}(y) k'w) (\nu_\cF(t) \ustar h'z) \\
= (-1)^{|t||k'|} (sg'\nu_\cF(u) \nu_\cF(t) \ustar \mu_{\cH'}(y) k' h'wz) 
\end{multline*}
by~\eqref{eqn:interchange-law-rhd-lhd}. Hence $g \hatstar (k h)=(g \hatstar k) \circ (\id_\cF \hatstar h)$ as desired.
A similar calculation proves~\eqref{eqn:monconv-weak-interchange} in the special case where $\cG' = \cH'$ and $k = \id_{\cG'}$.

We now turn to the general case.  By the special cases considered above, we have
\begin{align*}
g \hatstar k &= (g \hatstar \id_{\cH'})(\id_{\cG} \hatstar k), \\
f \hatstar h &= (f \hatstar \id_{\cG'})(\id_{\cF} \hatstar h), \\
(gf) \hatstar (kh) &= (g \hatstar \id_{\cH'})(f \hatstar k)(\id_\cF \hatstar h).
\end{align*}
Thus, to prove~\eqref{eqn:monconv-weak-interchange} in general, it suffices to show that
\begin{equation}\label{eqn:monconv-weak-special}
f \hatstar k = (-1)^{|f||k|}(\id_\cG \hatstar k) \circ (f \hatstar \id_{\cG'}).
\end{equation}

Suppose now that $f, d(f) \in \uHom_\FM^\lhd(\cF,\cG)$. Assume without loss of generality that $k = uk'w$ with $u \in \Lambda$, $k' \in \uHom_\BE(\cG',\cH')$, and $w \in R^\vee$.  Our assumption on $f$ implies that
\[
f \hatstar k = f \nu_\cF(u) \ustar k'w
\]
and 
\[
(\id_\cG \hatstar k) \circ (f \hatstar \id_{\cG'}) = (\nu_\cG(u) \ustar k'w)(f \ustar  \id_{\cG'}) = (-1)^{|f||k'|}\nu_\cG(u)f \ustar k'w
\]
by~\eqref{eqn:interchange-law-rhd-lhd}.
Our assumption on $f$ and $d(f)$ implies that 
$\nu_\cG(u)f = (-1)^{|u||f|}f \nu_\cF(u)$, see Remark~\ref{rmk:mu-nu-commute}. The equality~\eqref{eqn:monconv-weak-special} follows.

A similar computation yields~\eqref{eqn:monconv-weak-special} if we instead assume that $k$ and $d(k)$ lie in $\uHom_\FM^\rhd(\cG',\cH')$.  Finally, if $f \in \uHom_\FM^\lhd(\cF,\cG)$ and $k \in \uHom_\FM^\rhd(\cF,\cG)$, then each instance of $\hatstar$ in~\eqref{eqn:monconv-weak-special} reduces to $\ustar$, so the claim follows from~\eqref{eqn:interchange-law-rhd-lhd} in this case as well.
\end{proof}

\section{Convolution of objects}

In this section, we will define the operation $\hatstar$ on free-monodromic complexes. 

\begin{defn}\label{defn:monconv}
Let $\cF$ and $\cF'$ be two convolutive free-monodromic complexes. Their \emph{monodromic convolution product}, denoted by $\cF \hatstar \cF'$, is defined to be the convolutive free-monodromic complex whose underlying $\DiagBSp$-sequence is $\cF \ustar \cF'$, and whose differential is given by\index{convolution!starhat@$\hatstar$}
\[
\delta_{\cF \hatstar \cF'} = \delta_\cF \hatstar \id_{\cF'} + \id_{\cF} \hatstar \delta_{\cF'} - \psi_{\cF \hatstar \cF'},
\]
where
\[
\psi_{\cF \hatstar\cF'} = \sum_i \nu_{\cF}(e_i) \ustar \mu_{\cF'}(\check e_i).
\]
(Here, as usual, $e_1, \sdots, e_r$ and $\check e_1, \sdots, \check e_r$ are dual bases for $V^*$ and $V$ respectively.)
\end{defn}

Of course, we must check that $\delta_{\cF \hatstar \cF'}$ obeys the defining condition for a free-monodromic complex.  This will be done in Lemma~\ref{lem:monconv-diff-ax} below.

\begin{lem}
\label{lem:convdiff-formula}
For any two convolutive free-monodromic complexes $\cF$ and $\cF'$, we have
\[
\delta_{\cF \hatstar \cF'} = \delta_\cF \hatstar \id_{\cF'} + \id_{\cF} \hatstar \delta_{\cF'}^\rhd = \delta_\cF^\lhd \hatstar \id_{\cF'} + \id_{\cF} \hatstar \delta_{\cF'} = \delta_{\cF}^\lhd \hatstar \id_{\cF'} + \id_{\cF} \hatstar \delta_{\cF'}^\rhd + \psi_{\cF \hatstar \cF'}.
\]
\end{lem}

\begin{proof}
We have
\begin{multline*}
\delta_\cF \hatstar \id_{\cF'} = \delta_\cF^\lhd \hatstar \id_{\cF'} + \sum_i \nu_{\cF}(e_i)\check e_i \hatstar \id_{\cF'} \\
= \delta_\cF^\lhd \hatstar \id_{\cF'} + \sum_i \nu_\cF(e_i) \ustar \mu_{\cF'}(\check e_i) = \delta_\cF^\lhd \hatstar \id_{\cF'} + \psi_{\cF \hatstar \cF'}.
\end{multline*}
A similar calculation shows that
\[
\id_{\cF} \hatstar \delta_{\cF'} = \id_{\cF} \hatstar \delta_{\cF'}^\rhd + \psi_{\cF \hatstar \cF'},
\]
and the desired equalities follow.
\end{proof}

\begin{lem}\label{lem:monconv-diff-ax}
Let $\cF$ and $\cF'$ be convolutive free-monodromic complexes.  The element $\delta_{\cF \hatstar \cF'} \in \uEnd_\FM(\cF \ustar \cF')$ satisfies $\delta_{\cF \hatstar \cF'}^2 + \kappa(\delta_{\cF \hatstar \cF'}) = \Theta_{\cF \ustar \cF'}$.
\end{lem}

\begin{proof}
Let us use the formula $\delta_{\cF \hatstar \cF'} = \delta_\cF \hatstar \id + \id \hatstar \delta_{\cF'}^\rhd$ from Lemma~\ref{lem:convdiff-formula}.  Since $\delta_{\cF'}^\rhd$ and $d(\delta_{\cF'}^\rhd)$ both lie in $\uEnd_\FM^\rhd(\cF')$ (see Lemma~\ref{lem:delta-lhd-rhd}\eqref{it:delta-lhd-rhd-3}), we can use Lemma~\ref{lem:monconv-weak-interchange} to see that
\begin{multline*}
\delta_{\cF \hatstar \cF'}^2 = (\delta_\cF \hatstar \id_{\cF'})^2 + (\delta_\cF \hatstar \id_{\cF'})(\id_{\cF} \hatstar \delta_{\cF'}^\rhd) + (\id_{\cF} \hatstar \delta_{\cF'}^\rhd)(\delta_\cF \hatstar \id_{\cF'}) + (\id_{\cF} \hatstar \delta_{\cF'}^\rhd)^2 \\
= \delta_{\cF}^2 \hatstar \id_{\cF'} + \id_{\cF} \hatstar (\delta_{\cF'}^\rhd)^2
= \delta_{\cF}^2 \hatstar \id_{\cF'} + \id_{\cF} \hatstar \delta_{\cF'}^2,
\end{multline*}
where the last equality again comes from Lemma~\ref{lem:delta-lhd-rhd}\eqref{it:delta-lhd-rhd-3}.  Next, it is clear that $\kappa(f) \hatstar \id_{\cF'} = \kappa(f \hatstar \id_{\cF'})$ for any $f$.  On the other hand, we have $\kappa(\id_{\cF} \hatstar \delta_{\cF'}^\rhd) = 0$.  Thus,
\begin{multline*}
\delta_{\cF \hatstar \cF'}^2 + \kappa(\delta_{\cF \hatstar \cF'})
= \delta_{\cF}^2 \hatstar \id_{\cF'} + \id_{\cF} \hatstar \delta_{\cF'}^2 + \kappa(\delta_\cF) \hatstar \id_{\cF'} \\
= \Theta_{\cF} \hatstar \id_{\cF'} + \id_{\cF} \hatstar \Theta_{\cF'} - \id_{\cF} \hatstar \kappa(\delta_{\cF'}).
\end{multline*}
Note that $\Theta_{\cF \ustar \cF'} = \id_{\cF} \hatstar \Theta_{\cF'}$.  Thus, to complete the proof, we must show that
\[
\Theta_{\cF} \hatstar \id_{\cF'} = \id_{\cF} \hatstar \kappa(\delta_{\cF'}).
\]
For the left-hand side, we have
\[
\Theta_\cF \hatstar \id_{\cF'} = \sum_i (\id_\cF \ustar e_i) \ustar \mu_{\cF'}(\check e_i).
\]
For the right-hand side, note that $\kappa(\delta_{\cF'}^\rhd) = 0$, so
\[
\kappa(\delta_{\cF'}) = \kappa \left( \sum_i e_i \mu_{\cF'}(\check e_i) \right) = \sum_i (e_i \ustar \id_{\cF'}) \cdot \mu_{\cF'}(\check e_i).
\]
It follows that $\id_{\cF} \hatstar \kappa(\delta_{\cF'}) = \Theta_\cF \hatstar \id_{\cF'}$, as desired.
\end{proof}

\section{Convolution of morphisms}
\label{sec:convolution-morph}

In this section we check that the operation $\hatstar$ on maps defined in~\S\ref{sec:fm-convolution-morph} is compatible with the appropriate differentials, and hence that it induces an operation on morphisms in $\FM(\fh,W)$.

\begin{prop}
\label{prop:monconv-diff}
Let $\cF, \cG, \cF', \cG'$ be convolutive free-monodromic complexes.
For any $f \in \uHom_\FM(\cF,\cG)$ and $h \in \uHom_\FM(\cF',\cG')$, we have
\[
d(f) \hatstar h + (-1)^{|f|} f \hatstar d(h)
= d(f \hatstar h).
\]
\end{prop}

\begin{proof}
Assume that $f = rf'x$, where $r \in \Lambda$, $f' \in \uHom_\BE(\cF,\cG)$, and $x \in R^\vee$.  We likewise assume that $h$ can be written in the form $h = th'z$.  Assume furthermore that $t = t_1 \cdots t_k$, where $t_1, \sdots, t_k \in V^*$.  Below, we will use the notation
\[
t^{(j)} = t_1 \cdots \widehat{t_{j}} \cdots t_k.
\]

By definition, we have
\begin{multline}
\label{eqn:monconv-diff}
d(f) \hatstar h + (-1)^{|f|} f \hatstar d(h) =
\kappa(f) \hatstar h + (\delta_\cG f) \hatstar h + (-1)^{|f|+1} (f \delta_\cF) \hatstar h \\
+ (-1)^{|f|} f \hatstar \kappa(h) + (-1)^{|f|} f \hatstar (\delta_{\cG'} h) + (-1)^{|f|+|h|+1} f \hatstar (h \delta_{\cF'}).
\end{multline}

We will analyze each term in the right-hand side of~\eqref{eqn:monconv-diff} successively. For the first one,
we have
\begin{multline*}
\kappa(f) \hatstar h = \kappa(rf')x \hatstar th'z = \kappa(rf')\nu_{\cF}(t) \ustar \mu_{\cG'}(x)h' z\\
= \kappa(rf'\nu_{\cF}(t)) \ustar \mu_{\cG'}(x)h'z - (-1)^{|r|+|f'|} rf'\kappa(\nu_{\cF}(t)) \ustar \mu_{\cG'}(x)h'z.
\end{multline*}
Using Lemma~\ref{lem:delta-lhd-rhd}\eqref{it:delta-lhd-rhd-1}, we have
\begin{multline*}
\kappa(\nu_\cF(t)) = d(\nu_\cF(t)) - \delta_\cF \nu_\cF(t) - (-1)^{|t|+1} \nu_\cF(t)\delta_\cF 
\\
= d(\nu_\cF(t)) - \delta_\cF^\lhd \nu_\cF(t) - (-1)^{|t|+1} \nu_\cF(t) \delta_\cF^\lhd.
\end{multline*}
Now, observe that
\begin{multline*}
d(\nu_\cF(t)) = \sum_j (-1)^{j+1} \nu_\cF(t_1 \cdots t_{j-1}) d(\nu_\cF(t_j)) \nu_\cF(t_{j+1} \cdots t_k) \\
= \sum_j (-1)^{j+1} \nu_\cF(t^{(j)}) \ustar t_j
\end{multline*}
by Lemma~\ref{lem:convolutive-mu-nu}\eqref{it:convolutive-nu}.
Since by definition and~\eqref{eqn:interchange-law-rhd-lhd} we have $rf'\nu_\cF(t)\delta_\cF^\lhd \ustar \mu_{\cG'}(x)h'z = (-1)^{|h'|}(f \hatstar h)(\delta_\cF^\lhd \hatstar \id_{\cF'})$, and since $|r|+|f'| \equiv |f| \pmod 2$ and $|t|+|h'| \equiv |h| \pmod 2$, we deduce that
\begin{multline}\label{eqn:monconv-diff1}
\kappa(f) \hatstar h = \kappa(f \hatstar h) - \sum_j (-1)^{|f|+j+1} (rf'\nu_\cF(t^{(j)}) \ustar t_j) \ustar \mu_{\cG'}(x)h'z\\
+ (-1)^{|f|} rf'\delta_\cF^\lhd \nu_{\cF}(t) \ustar \mu_{\cG'}(x)h'z + (-1)^{|f|+|h|+1} (f \hatstar h)(\delta_\cF^\lhd \hatstar \id_{\cF'}).
\end{multline}

For the second term in the right-hand side of~\eqref{eqn:monconv-diff},
by Lemma~\ref{lem:monconv-weak-interchange} we have
\begin{equation}\label{eqn:monconv-diff2}
(\delta_\cG f) \hatstar h = (\delta_\cG \hatstar \id_{\cG'}) (f \hatstar h).
\end{equation}
Next, for the third term we have
\begin{multline}\label{eqn:monconv-diff3}
(f \delta_\cF) \hatstar h = (f \delta_\cF^\lhd) \hatstar h + \sum_i (f \nu_\cF(e_i)\check e_i) \hatstar h \\
= rf'\delta_\cF^\lhd \nu_\cF(t) \ustar \mu_{\cG'}(x) h'z + \sum_i rf' \nu_\cF(e_i)\nu_\cF(t) \ustar \mu_{\cG'}(\check e_i)\mu_{\cG'}(x) h'z.
\end{multline}

To analyze the fourth term in the right-hand side of~\eqref{eqn:monconv-diff}, we
observe that
\[
\kappa(th') 
= \sum_j (-1)^{j+1} t^{(j)} (t_j \ustar h').
\]
Therefore,
\begin{multline}\label{eqn:monconv-diff4}
f \hatstar \kappa(h)  = \sum_j (-1)^{j+1} f \hatstar (t^{(j)}( t_j \ustar h')z) \\
= \sum_j (-1)^{j+1} rf'\nu_{\cF}(t^{(j)}) \ustar \mu_{\cG'}(x)(t_j \ustar h')z \\
= \sum_j (-1)^{j+1} (rf'\nu_{\cF}(t^{(j)}) \ustar t_j) \ustar \mu_{\cG'}(x)h'z.
\end{multline}

We now consider the fifth term in the right-hand side of~\eqref{eqn:monconv-diff}.
We have
\begin{multline*}
f \hatstar (\delta_{\cG'}h) = f \hatstar (\delta_{\cG'}^\rhd h) + \sum_i f \hatstar (e_i\mu_{\cG'}(\check e_i)h) \\
= (-1)^{|t|} f \hatstar t\delta_{\cG'}^\rhd h'z + \sum_i f \hatstar (e_it \mu_{\cG'}(\check e_i) h'z) \\
= (-1)^{|t|}rf' \nu_\cF(t) \ustar \mu_{\cG'}(x) \delta_{\cG'}^\rhd h'z + \sum_i rf' \nu_\cF(e_i)\nu_\cF(t) \ustar \mu_{\cG'}(x)\mu_{\cG'}(\check e_i) h' z.
\end{multline*}
Recall from Lemma~\ref{lem:delta-lhd-rhd}\eqref{it:delta-lhd-rhd-1} that $\mu_{\cG'}(x)\delta_{\cG'}^\rhd = \delta_{\cG'}^\rhd\mu_{\cG'}(x)$.  Since by~\eqref{eqn:interchange-law-rhd-lhd} we have $rf' \nu_\cF(t) \ustar \delta_{\cG'}^\rhd \mu_{\cG'}(x)h'z = (-1)^{|f|+|t|}(\id_{\cG} \hatstar \delta_{\cG'}^\rhd)(f \hatstar h)$, we find that
\begin{equation}\label{eqn:monconv-diff5}
f \hatstar (\delta_{\cG'}h) = (-1)^{|f|}(\id_{\cG} \hatstar \delta_{\cG'}^\rhd)(f \hatstar h) + \sum_i rf' \nu_{\cF}(e_i)\nu_{\cF}(t) \ustar \mu_{\cG'}(x)\mu_{\cG'}(\check e_i) h' z.
\end{equation}

Finally, for the last term in the right-hand side of~\eqref{eqn:monconv-diff}, by Lemma~\ref{lem:monconv-weak-interchange}, we have
\begin{equation}\label{eqn:monconv-diff6}
f \hatstar (h\delta_{\cF'}) = (f \hatstar h)(\id_{\cF} \hatstar \delta_{\cF'}).
\end{equation}

Let us now consolidate our preceding calculations. Combining~\eqref{eqn:monconv-diff1} and \eqref{eqn:monconv-diff4}, we have
\begin{multline}\label{eqn:monconv-diff7}
\kappa(f) \hatstar h + (-1)^{|f|} f \hatstar \kappa(h) \\
= \kappa(f \hatstar h) + (-1)^{|f|+|h|+1} (f \hatstar h)(\delta_\cF^\lhd \hatstar \id_{\cF'}) + (-1)^{|f|} rf'\delta_\cF^\lhd \nu_\cF(t) \ustar \mu_{\cG'}(x)h'z.
\end{multline}
Now combine~\eqref{eqn:monconv-diff7} with~\eqref{eqn:monconv-diff3} and~\eqref{eqn:monconv-diff5}:
\begin{multline}\label{eqn:monconv-diff8}
\kappa(f) \hatstar h + (-1)^{|f|} f \hatstar \kappa(h) + (-1)^{|f|+1} (f\delta_\cF) \hatstar h + (-1)^{|f|} f \hatstar (\delta_{\cG'}h) \\
= \kappa(f \hatstar h) + (-1)^{|f|+|h|+1} (f \hatstar h)(\delta_\cF^\lhd \hatstar \id_{\cF'}) + (\id_{\cG} \hatstar \delta_{\cG'}^\rhd)(f \hatstar h).
\end{multline}
Finally, combine~\eqref{eqn:monconv-diff8} with~\eqref{eqn:monconv-diff2} and~\eqref{eqn:monconv-diff6} and use Lemma~\ref{lem:convdiff-formula} to obtain
\begin{multline*}
d(f) \hatstar h + (-1)^{|f|} f \hatstar d(h) \\
= \kappa(f \hatstar h) + (-1)^{|f|+|h|+1} (f \hatstar h)(\delta_\cF^\lhd \hatstar \id_{\cF'}) + (\id_{\cG} \hatstar \delta_{\cG'}^\rhd)(f \hatstar h) \\
+ (\delta_\cG \hatstar \id_{\cG'})(f \hatstar h) + (-1)^{|f|+|h|+1} (f \hatstar h)(\id_{\cF} \hatstar \delta_{\cF'}) \\
= \kappa(f \hatstar h) + \delta_{\cG \hatstar \cG'}(f \hatstar h) + (-1)^{|f|+|h|+1}(f \hatstar h)\delta_{\cF \hatstar \cF'},
\end{multline*}
as desired.
\end{proof}

If $\cF, \cG, \cF', \cG'$ are convolutive free-monodromic complexes,
Proposition~\ref{prop:monconv-diff} tells us that
\[
\hatstar: \uHom_\FM(\cF,\cG) \otimes \uHom_\FM(\cF',\cG') \to \uHom_\FM(\cF \hatstar \cF', \cG \hatstar \cG')
\]
is a map of chain complexes, so it induces a map on cohomology:
\begin{equation}\label{eqn:hatstar-mor-defn}
\hatstar: \gHom_\FM(\cF,\cG) \otimes \gHom_\FM(\cF',\cG') \to \gHom_\FM(\cF \hatstar \cF', \cG \hatstar \cG').
\end{equation}
We reiterate that we do \emph{not} claim that $\hatstar$ is a bifunctor at this point. However, if $\cF$ is a fixed object of $\Conv_\FM(\fh,W)$, then it follows from Lemma~\ref{lem:monconv-weak-interchange} that the assignment\index{convolution!starhat@$\hatstar$}
\[
 \cG \mapsto \cF \hatstar \cG, \qquad g \mapsto \id_{\cF} \hatstar g
\]
is a functor from $\Conv_\FM(\fh,W)$ to itself, which we will denote $\cF \hatstar (-)$. Similarly, the assignment
\[
 \cG \mapsto \cG \hatstar \cF, \qquad g \mapsto g \hatstar \id_{\cF}
\]
defines a functor, which will be denoted $(-) \hatstar \cF$.

More generally, we have the following result.

\begin{lem}
\label{lem:interchange-rhd-lhd}
 Let $\cF, \cF', \cG, \cG', \cH, \cH'$ be convolutive free-monodromic complexes, and let $f : \cF \to \cG$, $g : \cG \to \cH$, $h : \cF' \to \cG'$, $k : \cG' \to \cH'$ be morphisms in $\FM(\fh,W)$. Assume either that $f$ admits a lift which belongs to $\uHom_\FM^{\lhd}(\cF, \cG)$, or that $k$ admits a lift which belongs to $\uHom^{\rhd}_\FM(\cG',\cH')$. Then we have
 \[
  (g \circ f) \hatstar (k \circ h) = (g \hatstar k) \circ (f \hatstar h).
 \]
\end{lem}

\begin{proof}
 This follows directly from Lemma~\ref{lem:monconv-weak-interchange}.
\end{proof}

\section{The categories \texorpdfstring{$\Conv_\FM^{\rhd}(\fh,W)$}{ConvFM>(h,W)} and \texorpdfstring{$\Conv_\FM^{\lhd}(\fh,W)$}{ConvFM<(h,W)}}
\label{sec:conv-categories}

We can now define the category $\Conv_\FM^{\rhd}(\fh,W)$\index{ConvFMrhd@$\Conv_\FM^\rhd$} as follows: 
\begin{itemize}
 \item 
 the objects of this category are the same as those of $\Conv_\FM(\fh,W)$ (i.e.~the convolutive free-monodromic complexes);
\item
the morphisms in $\Conv_\FM^{\rhd}(\fh,W)$ from $\cF$ to $\cG$ are elements of the $\bk$-module $\gHom_\FM(\cF,\cG)^0_0$ which admit a lift in $\uHom_\FM(\cF,\cG)^0_0$ which belongs to $\uHom_\FM^\rhd(\cF,\cG)$.
 \end{itemize}
 The category $\Conv_\FM^{\lhd}(\fh,W)$\index{ConvFMlhd@$\Conv_\FM^\lhd$} is defined in a similar way. 

By construction $\Conv_\FM^{\rhd}(\fh,W)$ and $\Conv_\FM^{\lhd}(\fh,W)$ are subcategories of the category $\Conv_\FM(\fh,W)$ (but not full subcategories), and it is easily checked that the operation $\hatstar$ sends morphisms in $\Conv_\FM^{\rhd}(\fh,W)$ to morphisms in $\Conv_\FM^{\rhd}(\fh,W)$, and similarly for $\Conv_\FM^{\lhd}(\fh,W)$. It follows from Lemma~\ref{lem:interchange-rhd-lhd} that the operation $\hatstar$ is a bifunctor on $\Conv_\FM^{\rhd}(\fh,W)$ and on $\Conv_\FM^{\lhd}(\fh,W)$.

\section{Action of \texorpdfstring{$\Conv_\FM(\fh,W)$}{ConvFM(h,W)} on \texorpdfstring{$\Conv_\LM(\fh,W)$}{ConvLM(h,W)}}
\label{sec:action-convfm-lm}

Adapting the definition in~\S\ref{sec:convolutive-UGU}, we will say that an object $(\cF, \delta)$ in $\LM(\fh,W)$ is \emph{convolutive}\index{convolutive} if
\[
\delta \in (\bk \oplus V^*(-2)[1]) \otimes \uEnd_\BE(\cF) \subset \uEnd_\LM(\cF).
\]
We will denote by $\Conv_\LM(\fh,W)$\index{ConvLM@$\Conv_\LM$} the full subcategory of $\LM(\fh,W)$ whose objects are convolutive. Then, as in Lemma~\ref{lem:convolutive-mu-nu}, if $(\cG,\delta)$ is convolutive then the morphism $\mu_\cG$ of Theorem~\ref{thm:frown} lifts to a dgg-algebra morphism
\[
\mu_\cG : R^\vee \to \uEnd_\LM(\cG)
\]
which takes values in $\uEnd_\BE(\cG)$. 
Using this construction one can define, for $\cF, \cG$ in $\Conv_\FM(\fh,W)$ and $\cF', \cG'$ in $\Conv_\LM(\fh,W)$, a map
\begin{equation}\label{eqn:hatstar-fmlm-dgg-defn}
\uHom_\FM(\cF,\cG) \otimes \uHom_\LM(\cF',\cG') \to \uHom_\LM(\cF \ustar \cF', \cG \ustar \cG')
\end{equation}
as follows: for $f=rf'x$ (with $r \in \Lambda$, $f' \in \uHom_\BE(\cF,\cG)$, $x \in R^\vee$) and $h=th'$ (with $t \in \Lambda$, $h' \in \uHom_\BE(\cF',\cG')$) we set
\[
f \hatstar h := r f' \nu_\cF(t) \ustar \mu_{\cG'}(x) h',
\]
where
\[
\ustar : \uHom_\FM^\lhd(\cF,\cG) \otimes \uHom_\BE(\cF',\cG') \to \uHom_\LM(\cF \ustar \cF', \cG \ustar \cG')
\]
is the natural convolution morphism.
Finally, if $(\cF,\delta_\cF)$ belongs to $\Conv_\FM(\fh,W)$ and $(\cG,\delta_\cG)$ belongs to $\Conv_\LM(\fh,W)$, we can define an object $\cF \hatstar \cG$\index{convolution!starhat@$\hatstar$} of $\Conv_\LM(\fh,W)$ by setting its underlying $\DiagBSp$-sequence to be $\cF \ustar \cG$, and equipping it with the differential
\[
\delta_{\cF \hatstar \cG} = \delta_\cF \hatstar \id_\cG + \id_\cF \hatstar \delta_\cG - \sum_i \nu_\cF(e_i) \ustar \mu_\cG(\check e_i),
\]
where $e_1, \sdots, e_r$ and $\check e_1, \sdots, \check e_r$ are dual bases for $V^*$ and $V$ respectively. An analogue of Proposition~\ref{prop:monconv-diff} tells us that~\eqref{eqn:hatstar-fmlm-dgg-defn} is a map of chain complexes, so it induces a map on cohomology
\begin{equation}\label{eqn:hatstar-fmlm-mor-defn}
\gHom_\FM(\cF,\cG) \otimes \gHom_\LM(\cF',\cG') \to \gHom_\LM(\cF \ustar \cF', \cG \ustar \cG').
\end{equation}
For any fixed $\cF$, the assignment $\cG \mapsto \cF \hatstar \cG$ is functorial, and for any fixed $\cG$, the assignment $\cF \mapsto \cF \hatstar \cG$ is functorial. However, we do \emph{not} claim that $\hatstar$ is a bifunctor from $\Conv_\FM(\fh,W) \times \Conv_\LM(\fh,W)$ to $\Conv_\LM(\fh,W)$.

It is clear that the functor $\ForFMLM : \FM(\fh,W) \to \LM(\fh,W)$ sends convolutive free-monodromic complexes to convolutive objects of $\LM(\fh,W)$. Moreover, if $\cF$ and $\cG$ belong to $\Conv_\FM(\fh,W)$ then there exists a functorial isomorphism
\begin{equation}
\label{eqn:For-hatstar}
\ForFMLM(\cF \hatstar \cG) \cong \cF \hatstar \ForFMLM(\cG).
\end{equation}

\begin{lem}
\label{lem:convolution-mon-equ}
Let $\cF$ be in $\FM(\fh,W)$ and $\cG$ be in $\BE(\fh,W)$. Then $\ForBELM(\cG)$ belongs to $\Conv_\LM(\fh,W)$, and there exists a canonical isomorphism
\[
\cF \hatstar \ForBELM(\cG) \cong \ForFMLM(\cF) \ustar \cG
\]
in $\LM(\fh,W)$, where $\ustar$ is as in~\eqref{eqn:ustar-UGB}. This isomorphism is functorial in $\cF$ (for fixed $\cG$) and in $\cG$ (for fixed $\cF$).
\end{lem}

\begin{proof}
The fact that $\ForBELM(\cG)$ belongs to $\Conv_\LM(\fh,W)$ is obvious by construction. Then, by the obvious analogue of Lemma~\ref{lem:convdiff-formula}, we have
\[
\delta_{\cF \hatstar \ForBELM(\cG)} = \delta_\cF^\lhd \hatstar \id_\cG + \id_\cF \hatstar \delta_\cG = \delta_\cF^\lhd \ustar \id_\cG + \id_\cF \ustar \delta_\cG.
\]
This is precisely the differential of $\ForFMLM(\cF) \ustar \cG$, so we are done.
\end{proof}

\section{Karoubian envelopes}
\label{sec:karoubian2}

In this section, as in~\S\ref{sec:karoubian1}, we assume that $\bk$ is a field or a complete local ring, and we work with the Karoubian envelope $\Diag(\fh,W)$ of $\DiagBSp(\fh,W)$. We define the bigraded $\bk$-module
\[
\uHom_\FM(\cF,\cG)
\]
for $\cF,\cG$
$\Diag$-sequences,
and the category
\[
\FM(\fh,W)_\Kar,
\]
by the same formulas as for $\DiagBSp(\fh,W)$.  It makes sense to speak of convolutive objects in $\FM(\fh,W)_\Kar$ and to define the operation $\hatstar$ in the various settings we have discussed above.  One can also define the categories
\[
\Conv_\FM^{\rhd}(\fh,W)_\Kar
\qquad\text{and}\qquad
\Conv_\FM^{\lhd}(\fh,W)_\Kar
\]
as in~\S\ref{sec:conv-categories}.

\begin{lem}
\label{lem:idem-equiv-fm}
The obvious functor
\[
\FM(\fh,W) \to \FM(\fh,W)_\Kar
\]
is fully faithful.
\end{lem}
It is likely that this functor is actually an equivalence of categories, but to imitate the proof of Lemma~\ref{lem:idem-equiv}, we would require more information about generating $\FM(\fh,W)_\Kar$ as a triangulated category.  We will not pursue this question in this paper.

\begin{proof}
Since $\DiagBSp(\fh,W)$ is a full subcategory of $\Diag(\fh,W)$, we see that for any two $\DiagBSp$-sequences $\cF$ and $\cG$, the bigraded $\bk$-module $\uHom_\FM(\cF,\cG)$ remains unchanged if we regard $\cF$ and $\cG$ as $\Diag$-sequences.  The lemma follows immediately.
\end{proof}

\chapter{Hints of functoriality}
\label{chap:murmurs}

A monoidal category consists of a category equipped with a bifunctor and some additional data (natural isomorphisms satisfying some compatibility conditions). In the previous chapter, we emphasized the fact that we do not know at the moment whether $\hatstar$ is a bifunctor on $\Conv_\FM(\fh,W)$.  Nevertheless, it is possible to go ahead and construct the additional data associated with being monoidal; indeed, we will need this additional data as we continue to study $\hatstar$ in later chapters.  This task occupies the first three sections of the chapter.  As an application, we will define and study morphisms between free-monodromic analogues of the standard and costandard objects from Example~\ref{ex:std-costd}.

We saw in~\S\ref{sec:convolution-morph} that for a fixed object $\cF \in \Conv_\FM(\fh,W)$, the operations $\cF\hatstar ({-})$ and $({-}) \hatstar \cF$ \emph{are} endofunctors of $\Conv_\FM(\fh,W)$.  In the last two sections of this chapter, we show that these functors extend to all of $\FM(\fh,W)$.

\section{Unitor isomorphisms}

The object $\Tmon_\varnothing$ considered in~\S\ref{sss:t1-freemon} is a unit for the operation $\hatstar$, in the following sense.

\begin{lem}
\label{lem:unitor}
There exist canonical isomorphisms of functors from the category $\Conv_\FM(\fh,W)$ to itself:
\[
\lambda : \Tmon_\varnothing \hatstar (-) \simto \id
\qquad\text{and}\qquad \varrho :
(-) \hatstar \Tmon_\varnothing \simto \id.
\]
Moreover, for any $\cF, \cG \in \Conv_\FM(\fh,W)$, the following diagrams commute:
\begin{gather*}
\begin{tikzcd}[ampersand replacement=\&]
\gEnd_\FM(\Tmon_\varnothing) \otimes \gHom_\FM(\cF,\cG)
\ar[r, "\hatstar", "\text{\eqref{eqn:hatstar-mor-defn}}"'] \ar[d, "\wr"', "\eqref{eqn:endmon-t1}"] \&
\gHom_\FM(\Tmon_\varnothing \hatstar \cF,\Tmon_\varnothing \hatstar \cG)
\ar[d, "\lambda", "\wr"'] \\
R^\vee \otimes \gHom_\FM(\cF,\cG)
\ar[r, "\hatstar", "\eqref{eqn:left-rv-action}"'] \&
\gHom_\FM(\cF,\cG)
\end{tikzcd}
\\
\begin{tikzcd}[ampersand replacement=\&]
\gHom_\FM(\cF,\cG) \otimes \gEnd_\FM(\Tmon_\varnothing) 
\ar[r, "\hatstar", "\eqref{eqn:hatstar-mor-defn}"'] \ar[d, "\wr"', "\eqref{eqn:endmon-t1}"] \&
\gHom_\FM(\cF \hatstar \Tmon_\varnothing,\cG \hatstar \Tmon_\varnothing)
\ar[d, "\varrho", "\wr"'] \\
\gHom_\FM(\cF,\cG) \otimes R^\vee
\ar[r, "\hatstar", "\eqref{eqn:right-rv-action}"'] \&
\gHom_\FM(\cF,\cG)
\end{tikzcd}
\end{gather*}
\end{lem}

This lemma at last justifies the use of $\hatstar$ for the left and right monodromy actions, as promised at the beginning of Chapter~\ref{chap:fm-convolution}.

\begin{proof}
It is clear that we have canonical isomorphisms of $\DiagBSp$-sequences $\cT_\varnothing \ustar \cF \cong \cF \cong \cF \ustar \cT_\varnothing$ (because $B_\varnothing$ is the unit object in the monoidal category $\DiagBS(\fh,W)$). The fact that these isomorphisms of $\DiagBSp$-sequences are also isomorphisms in $\Conv_\FM(\fh,W)$ follows from Lemma~\ref{lem:convdiff-formula}, the fact that $\delta_{\Tmon_\varnothing}^\rhd=\delta_{\Tmon_\varnothing}^\lhd=0$, and the computation of $\mu_{\Tmon_\varnothing}$ and $\nu_{\Tmon_\varnothing}$ in Example~\ref{ex:T1-mon-mu-nu}.

The commutativity of the diagrams also follows from the definitions and the computation of $\mu_{\Tmon_\varnothing}$ and $\nu_{\Tmon_\varnothing}$ in Example~\ref{ex:T1-mon-mu-nu}.
\end{proof}

Our construction is such that for any $\cF$, the isomorphisms $\lambda_\cF$ and $\varrho_\cF$ come from isomorphisms of the underlying $\DiagBSp$-sequences. As a consequence, these isomorphisms belong to both $\Conv_\FM^\rhd(\fh,W)$ and $\Conv_\FM^\lhd(\fh,W)$.

There is a straightforward analogue of Lemma~\ref{lem:unitor} in the context of the action of $\Conv_\FM(\fh,W)$ on $\Conv_\LM(\fh,W)$ discussed in~\S\ref{sec:action-convfm-lm}. We omit the proof.

\begin{lem}
\label{lem:unitor-action}
There exist canonical isomorphisms of functors
\begin{align*}
\lambda' : \Tmon_\varnothing \hatstar (-) \simto \id &: \Conv_\LM(\fh,W) \to \Conv_\LM(\fh,W), \\
\varrho': ({-}) \hatstar \cT_\varnothing \simto \ForFMLM &: \Conv_\FM(\fh,W) \to \Conv_\LM(\fh,W).
\end{align*}
Moreover, for any $\cF,\cG \in \Conv_\FM(\fh,W)$ and $\cF',\cG' \in \Conv_\LM(\fh,W)$, the following diagrams commute:
\begin{gather*}
\begin{tikzcd}[ampersand replacement=\&]
\gEnd_\FM(\Tmon_\varnothing) \otimes \gHom_\LM(\cF',\cG')
\ar[r, "\hatstar", "\text{\eqref{eqn:hatstar-fmlm-mor-defn}}"'] \ar[d, "\wr"', "\eqref{eqn:endmon-t1}"] \&
\gHom_\LM(\Tmon_\varnothing \hatstar \cF',\Tmon_\varnothing \hatstar \cG')
\ar[d, "\lambda'", "\wr"'] \\
R^\vee \otimes \gHom_\LM(\cF',\cG')
\ar[r, "\hatstar", "\eqref{eqn:leftmon-rv-action}"'] \&
\gHom_\LM(\cF',\cG')
\end{tikzcd}
\\
\begin{tikzcd}[ampersand replacement=\&]
\gHom_\FM(\cF,\cG) \otimes \gEnd_\LM(\cT_\varnothing) 
\ar[r, "\hatstar", "\eqref{eqn:hatstar-fmlm-mor-defn}"'] \ar[d, "\wr"', "\eqref{eqn:t1-leftmon}"] \&
\gHom_\LM(\cF \hatstar \cT_\varnothing,\cG \hatstar \cT_\varnothing)
\ar[d, "\varrho'", "\wr"'] \\
\gHom_\FM(\cF,\cG) \otimes_{R^\vee} \bk
\ar[r, "\ForFMLM"] \&
\gHom_\LM(\ForFMLM(\cF),\ForFMLM(\cG))
\end{tikzcd}
\end{gather*}
\end{lem}

\section{Associator isomorphism}

Since the convolution product $\star$ makes the category $\DiagBSp(\fh,W)$ into a monoidal category, for any $\DiagBSp$-sequences $\cF$, $\cG$ and $\cH$ there exists a canonical isomorphism
\begin{equation}
\label{eqn:associativity-parity-sequences}
\cF \ustar (\cG \ustar \cH) \cong (\cF \ustar \cG) \ustar \cH.
\end{equation}
In the following lemma we use this isomorphism to identify these objects.

\begin{lem}
\label{lem:associativity-morphisms}
Let $\cF$, $\cG$ and $\cH$ be convolutive free-monodromic complexes.
\begin{enumerate}
\item
\label{it:associativity-morphisms-1}
If $\cF'$ belongs to $\Conv_\FM(\fh,W)$ and $f \in \uHom_\FM(\cF,\cF')$, then we have
\[
(f \hatstar \id_{\cG}) \hatstar \id_\cH = f \hatstar (\id_\cG \hatstar \id_\cH).
\]
\item
\label{it:associativity-morphisms-2}
If $\cG'$ belongs to $\Conv_\FM(\fh,W)$ and $g \in \uHom_\FM(\cG, \cG')$, then we have
\[
(\id_{\cF} \hatstar g) \hatstar \id_{\cH} = \id_{\cF} \hatstar (g \hatstar \id_{\cH}).
\]
\item
\label{it:associativity-morphisms-3}
If $\cH'$ belongs to $\Conv_\FM(\fh,W)$ and $h \in \uHom_\FM(\cH,\cH')$, then we have
\[
(\id_\cF \hatstar \id_{\cG}) \hatstar h = \id_\cF \hatstar (\id_\cG \hatstar h).
\]
\end{enumerate}
\end{lem}

\begin{proof}
We first prove~\eqref{it:associativity-morphisms-2}. We can assume that $g = r g' x$ with $r \in \Lambda$, $g' \in \uHom_\BE(\cG, \cG')$ and $x \in R^\vee$. Then we have
\[
(\id_{\cF} \hatstar g) \hatstar \id_{\cH} = (\nu_\cF(r) \ustar g'x) \hatstar \id_\cH = (\nu_\cF(r) \ustar g') \ustar \mu_\cH(x).
\]
On the other hand we have
\[
\id_{\cF} \hatstar (g \hatstar \id_{\cH}) = \id_\cF \hatstar (rg' \ustar \mu_\cH(x)) = \nu_\cF(r) \ustar (g' \ustar \mu_\cH(x)).
\]
Hence the desired identification follows from the associativity of $\star$ on $\DiagBSp(\fh,W)$.

Now we prove~\eqref{it:associativity-morphisms-1}. We can assume that $f = r f' x$ with $r \in \Lambda$, $f' \in \uHom_\BE(\cF, \cF')$ and $x \in R^\vee$. We remark that for any $y \in V$, by Lemma~\ref{lem:convdiff-formula} we have
\[
\mu_{\cG \hatstar \cH}(y) = y \lfrown (\delta_\cG \hatstar \id_\cH + \id_\cG \hatstar \delta_\cH^\rhd) = y \lfrown (\delta_\cG \hatstar \id_\cH) = (y \lfrown \delta_\cG) \hatstar \id_\cH = \mu_\cG(y) \hatstar \id_\cH.
\]
Using Lemma~\ref{lem:monconv-weak-interchange} we deduce that this formula holds for any $y \in R^\vee$, and then that
\begin{multline*}
(f \hatstar \id_{\cG}) \hatstar \id_\cH = (rf' \hatstar \mu_{\cG}(x)) \hatstar \id_\cH = ((rf' \hatstar \id_\cG) \circ (\id_\cF \hatstar \mu_{\cG}(x))) \hatstar \id_\cH \\
= ((rf' \hatstar \id_\cG) \hatstar \id_\cH) \circ ((\id_\cF \hatstar \mu_{\cG}(x)) \hatstar \id_\cH) = (rf' \hatstar \id_{\cG \hatstar\cH}) \circ (\id_\cF \hatstar (\mu_{\cG}(x) \hatstar \id_\cH)) \\
= rf' \hatstar (\mu_\cG(x) \hatstar \id_\cH) = rf' \hatstar \mu_{\cG \hatstar \cH}(x) = f \hatstar \id_{\cG \hatstar \cH} = f \hatstar (\id_\cG \hatstar \id_\cH).
\end{multline*}

The proof of~\eqref{it:associativity-morphisms-3} is similar. We can assume that $h=rh'x$ with $r \in \Lambda$, $h' \in \uHom_\BE(\cH, \cH')$ and $x \in R^\vee$. By Lemma~\ref{lem:convdiff-formula},
for any $s \in V^*$ we have
\[
\nu_{\cF \hatstar \cG}(s) = (\delta_\cF^\lhd \hatstar \id_\cG + \id_\cF \hatstar \delta_\cG) \rfrown s = (\id_\cF \hatstar \delta_\cG) \rfrown s = \id_\cF \hatstar (\delta_\cG \rfrown s) = \id_\cF \hatstar \nu_\cG(s).
\]
Using again Lemma~\ref{lem:monconv-weak-interchange} we deduce that this formula holds for any $s \in \Lambda$, and then that
\begin{multline*}
\id_\cF \hatstar (\id_\cG \hatstar h) = \id_\cF \hatstar (\nu_\cG(r) \hatstar h'x) =  \id_\cF \hatstar ((\nu_\cG(r) \hatstar \id_{\cH'}) \circ (\id_\cG \hatstar h'x)) \\
=  (\id_\cF \hatstar (\nu_\cG(r) \hatstar \id_{\cH'})) \circ (\id_\cF \hatstar (\id_\cG \hatstar h'x))
=  ((\id_\cF \hatstar \nu_\cG(r)) \hatstar \id_{\cH'}) \circ (\id_{\cF \hatstar \cG} \hatstar h'x) \\
=  (\id_\cF \hatstar \nu_\cG(r)) \hatstar h'x = \nu_{\cF \hatstar \cG}(r) \hatstar h'x = \id_{\cF \hatstar \cG} \hatstar h = (\id_\cF \hatstar \id_\cG) \hatstar h,
\end{multline*}
as desired.
\end{proof}

\begin{prop}
\label{prop:associator}
For any $\cF, \cG, \cH$ in $\Conv_\FM(\fh,W)$,
the isomorphism~\eqref{eqn:associativity-parity-sequences} induces an isomorphism
\[
\alpha_{\cF,\cG,\cH} : \cF \hatstar (\cG \hatstar \cH) \cong (\cF \hatstar \cG) \hatstar \cH
\]
in $\Conv_\FM(\fh,W)$. Moreover, this identification is compatible with morphisms in the sense that if
$\cF, \cF', \cG, \cG', \cH, \cH'$ are convolutive free-monodromic complexes and if $f : \cF \to \cF'$, $g : \cG \to \cG'$ and $h : \cH \to \cH'$ are morphisms in $\Conv_\FM(\fh,W)$, then the following diagram commutes:
\[
\begin{tikzcd}[row sep=large,column sep=huge]
\cF \hatstar (\cG \hatstar \cH) \ar[r, "\alpha_{\cF,\cG,\cH}"] \ar[d, "f \hatstar (g \hatstar h)" description] & (\cF \hatstar \cG) \hatstar \cH \ar[d, "(f \hatstar g) \hatstar h" description] \\
\cF' \hatstar (\cG' \hatstar \cH') \ar[r, "\alpha_{\cF',\cG',\cH'}"] & (\cF' \hatstar \cG') \hatstar \cH'.
\end{tikzcd}
\]
\end{prop}

\begin{proof}
To prove the first claim we have to check that if $(\cF,\delta_\cF)$, $(\cG,\delta_\cG)$ and $(\cH,\delta_\cH)$ are convolutive free-monodromic complexes, then the isomorphism~\eqref{eqn:associativity-parity-sequences} is an isomorphism of free-monodromic complexes, or in other words that the differentials on the two sides identify. For this we use Lemma~\ref{lem:convdiff-formula} to see that
\[
\delta_{\cF \hatstar (\cG \hatstar \cH)} = \delta_\cF^\lhd \hatstar \id_{\cG \hatstar \cH} + \id_\cF \hatstar \delta_{\cG \hatstar \cH} = \delta_\cF^\lhd \hatstar (\id_{\cG} \hatstar \id_{\cH}) + \id_\cF \hatstar (\delta_{\cG} \hatstar \id_{\cH} + \id_\cG \hatstar \delta_\cH^\rhd). 
\]
On the other hand, for similar reasons we have
\[
\delta_{(\cF \hatstar \cG) \hatstar \cH} = \id_{\cF \hatstar \cG} \hatstar \delta_\cH^\rhd + \delta_{\cF \hatstar \cG} \hatstar \id_\cH = (\id_\cF \hatstar \id_\cG) \hatstar \delta_\cH^\rhd + (\delta_\cF^\lhd \hatstar \id_\cG + \id_\cF \hatstar \delta_\cG) \hatstar \id_\cH.
\]
Hence the desired identification follows from Lemma~\ref{lem:associativity-morphisms}.

Next, we prove the second claim.
In fact, we will prove that this equality holds in $\uHom_\FM(\cF \hatstar (\cG \hatstar \cH), (\cF' \hatstar \cG') \hatstar \cH')$, and for any $f \in \uHom_\FM(\cF, \cF')$, $g \in \uHom_\FM(\cG, \cG')$, $h \in \uHom_\FM(\cH, \cH')$.
For this we use Lemma~\ref{lem:monconv-weak-interchange} repeatedly to see that
\begin{multline*}
f \hatstar (g \hatstar h) = (f \hatstar \id_{\cG' \hatstar \cH'}) \circ (\id_\cF \hatstar (g \hatstar h)) \\
= \bigl( f \hatstar (\id_{\cG'} \hatstar \id_{\cH'}) \bigr) \circ \bigl( \id_\cF \hatstar ((g \hatstar \id_{\cH'}) \circ (\id_\cG \hatstar h)) \bigr) \\
= \bigl( f \hatstar (\id_{\cG'} \hatstar \id_{\cH'}) \bigr) \circ \bigl( \id_\cF \hatstar (g \hatstar \id_{\cH'}) \bigr) \circ \bigl( \id_\cF \hatstar (\id_\cG \hatstar h) \bigr).
\end{multline*}
On the other hand, for similar reasons we have
\begin{multline*}
(f \hatstar g) \hatstar h = ((f \hatstar g) \hatstar \id_{\cH'}) \circ (\id_{\cF \hatstar \cG} \hatstar h) \\
= \bigl( ((f \hatstar \id_{\cG'}) \circ (\id_\cF \hatstar g)) \hatstar \id_{\cH'} \bigr) \circ \bigl( (\id_\cF \hatstar \id_\cG) \hatstar h \bigr) \\
= \bigl( (f \hatstar \id_{\cG'}) \hatstar \id_{\cH'} \bigr) \circ  \bigl((\id_\cF \hatstar g) \hatstar \id_{\cH'} \bigr) \circ \bigl( (\id_\cF \hatstar \id_\cG) \hatstar h \bigr).
\end{multline*}
Hence the desired identification follows once again from Lemma~\ref{lem:associativity-morphisms}.
\end{proof}

Once again, for any $\cF$, $\cG$ and $\cH$ the isomorphism $\alpha_{\cF, \cG, \cH}$ belongs to morphisms both in $\Conv_\FM^{\rhd}(\fh,W)$ and in $\Conv_\FM^{\lhd}(\fh,W)$.

There are analogues of Lemma~\ref{lem:associativity-morphisms} and Proposition~\ref{prop:associator} in the context of the action of $\Conv_\FM(\fh,W)$ on $\Conv_\LM(\fh,W)$, similar to Lemma~\ref{lem:unitor-action}.  We leave it to the reader to formulate these statements.

\section{Coherence conditions and \texorpdfstring{$n$}{n}-fold convolution product}
\label{sec:coherence}

\begin{prop}
\label{prop:hatstar-coherence}
The isomorphisms $\lambda$, $\varrho$, and $\alpha$ satisfy the coherence conditions as spelled out e.g.~in~\cite[\S VII.1]{maclane}. In particular, they equip the categories $\Conv_\FM^{\rhd}(\fh,W)$ and $\Conv^{\lhd}_\FM(\fh,W)$ with structures of monoidal categories.
\end{prop}

\begin{proof}
All our isomorphisms are induced by isomorphisms of $\DiagBSp$-sequences provided by the unitor and associator isomorphisms for the bifunctor $\star$ on the category $\DiagBSp(\fh,W)$. Since these isomorphisms are part of the monoidal structure on $\DiagBSp(\fh,W)$, they satisfy the coherence conditions, and we deduce these conditions for the category $\Conv_\FM(\fh,W)$.
\end{proof}

Since the categories $\Conv^{\rhd}_\FM(\fh,W)$ and $\Conv_\FM^{\lhd}(\fh,W)$ are monoidal, the $n$-fold convolution product $\cF_1 \hatstar \cF_2 \hatstar \cdots \hatstar \cF_n$ is defined up to a unique isomorphism for any convolutive free-monodromic complexes $\cF_1, \sdots, \cF_n$, as in~\cite[\S VII.2, Corollary]{maclane}. (Of course, this object is the same whether considered in $\Conv_\FM^{\rhd}(\fh,W)$ or in $\Conv_\FM^{\lhd}(\fh,W)$.) In fact even more is true: this object is defined canonically. Indeed, since the category $\DiagBSp(\fh,W)$ is a \emph{strict} monoidal category, for any $\cG_1, \sdots, \cG_n$ we can define unambiguously the $n$-fold convolution $\cG_1 \star \cdots \star \cG_n$. This object is canonically isomorphic to the object obtained for any choice of parentheses determining the order in which the convolution products are taken. Hence a similar construction is possible for $\DiagBSp$-sequences. And finally if $\cF_1, \sdots, \cF_n$ are as above, there exists a unique differential on the $\DiagBSp$-sequence $\cF_1 \ustar \cdots \ustar \cF_n$ which identifies with the corresponding differential on the object of $\Conv_\FM(\fh,W)$ obtained for any choice of parentheses.

Since our associator isomorphism is compatible with all morphisms (not only with those in $\Conv^{\rhd}_\FM(\fh,W)$ or in $\Conv_\FM^{\lhd}(\fh,W)$), a similar claim holds for morphisms: if $\cF_1, \sdots, \cF_n$ and $\cG_1, \sdots, \cG_n$ are convolutive free-monodromic complexes, and if $f_i : \cF_i \to \cG_i$ are morphisms in $\Conv_\FM(\fh,W)$, then the canonical identifications between the objects obtained by convolving $\cF_1, \sdots, \cF_n$ and $\cG_1, \sdots, \cG_n$ with any choice of parentheses are compatible with the corresponding morphisms obtained by convolving $f_1, \sdots, f_n$; therefore, there exists a unique morphism
\[
 f_1 \hatstar \cdots \hatstar f_n : \cF_1 \hatstar \cdots \hatstar \cF_n \to \cG_1 \hatstar \cdots \hatstar \cG_n
\]
which identifies with the morphism obtained for any choice of parentheses.

\section{Free-monodromic standard and costandard objects}

For any simple reflection $s$, we define the (convolutive) free-monodromic complexes $\tD_s$ and $\tN_s$ with respective underlying $\DiagBSp$-sequences
\[
(\sdots, 0, B_s, B_\varnothing(1), 0, \sdots) \quad \text{and} \quad (\sdots, 0, B_\varnothing(-1), B_s, 0, \sdots)
\]
(where in each case $B_s$ is in position $0$) and with differential defined as follows:
\[
\tD_s =
 \begin{tikzcd}
B_\varnothing(1) \ar[loop right, in=0, out=20, distance=40, "\theta - \alpha_s \otimes \id \otimes \alpha_s^\vee" pos=0.6] \ar[d, bend left=80, "(2)" description, "\ \ 1 \otimes \usebox\lowerdot \otimes \alpha_s^\vee" pos=0.4] \\
B_s \ar[u, "\usebox\upperdot"] \ar[loop right, in=-20, out=0, distance=50, "\theta - \alpha_s \otimes \id \otimes \alpha_s^\vee"]
\end{tikzcd}
\qquad \quad
\tN_s =
 \begin{tikzcd}
B_s \ar[loop right, in=0, out=20, distance=40, "\theta - \alpha_s \otimes \id \otimes \alpha_s^\vee" pos=0.6] \ar[d, bend left=80, "(2)" description, "\ \ 1 \otimes \usebox\upperdot \otimes \alpha_s^\vee" pos=0.7] \\
B_\varnothing(-1) \ar[u, "\usebox\lowerdot"] \ar[loop right, in=-20, out=0, distance=50, "\theta - \alpha_s \otimes \id \otimes \alpha_s^\vee"]
\end{tikzcd}
\]
We also set $\tD_\varnothing = \tN_\varnothing = \Tmon_\varnothing$.

For any expression $\uw=(t_1, \sdots, t_k)$, we set\index{standard object!Deltatuw@{$\tD_{\uw}$}}\index{costandard object!nablatuw@{$\tN_{\uw}$}}
\[
 \tD_{\uw} := \tD_{t_1} \hatstar \cdots \hatstar \tD_{t_k}, \qquad \tN_{\uw} := \tN_{t_1} \hatstar \cdots \hatstar \tN_{t_k}.
\]

\begin{lem}
\label{lem:mu-tD-tN-new}
 For any expression $\uw$ and any $f \in R^\vee$, in $\uEnd_\FM(\tD_\uw)$ we have
 \[
  \mu_{\tD_{\uw}}(f) = 1 \otimes \id_{\tD_{\uw}} \otimes \pi(\uw)^{-1}(f).
 \]
 Similarly, in $\uEnd_\FM(\tN_\uw)$ we have
 \[
  \mu_{\tN_{\uw}}(f) = 1 \otimes \id_{\tN_{\uw}} \otimes \pi(\uw)^{-1}(f).
 \]
\end{lem}

\begin{proof}
 We treat the case of $\tD_{\uw}$; the case of $\tN_\uw$ is similar.
 We prove the claim by induction on the length of $\uw$. If this length is $0$ or $1$, then the result can be checked directly. If $\uw$ is of positive length, let us denote by $s$ the last simple reflection appearing in $\uw$, and let $\uv$ be the expression obtained from $\uw$ by omitting $s$.  Regard $f$ as an element of $\gEnd(\Tmon_\varnothing)$, via~\eqref{eqn:endmon-t1}.  Using Lemmas~\ref{lem:unitor} and~\ref{lem:associativity-morphisms}, we have 
 \begin{multline*}
  \mu_{\tD_{\uw}}(f) = f \hatstar \id_{\tD_{uw}}
  = f \hatstar (\id_{\tD_{\uv}} \hatstar \id_{\tD_s})
  = (f \hatstar \id_{\tD_{\uv}}) \hatstar \id_{\tD_s}  \\
  = \mu_{\tD_{\uv}}(f) \hatstar \id_{\tD_s} 
  = (\id_{\tD_{\uv}} \cdot \pi(\uv)^{-1}(f)) \hatstar \id_{\tD_s} \\
  = \id_{\tD_{\uv}} \ustar \mu_{\tD_s}(\pi(\uv)^{-1}(f))
  = \id_{\tD_{\uw}} \cdot s(\pi(\uv)^{-1}(f)).
 \end{multline*}
 The claim follows.
\end{proof}

If $s$ is a simple reflection, we can consider the chain map $p_s \in \uHom_\FM^\rhd(\tD_s, \tN_s)$ defined as follows: 
 
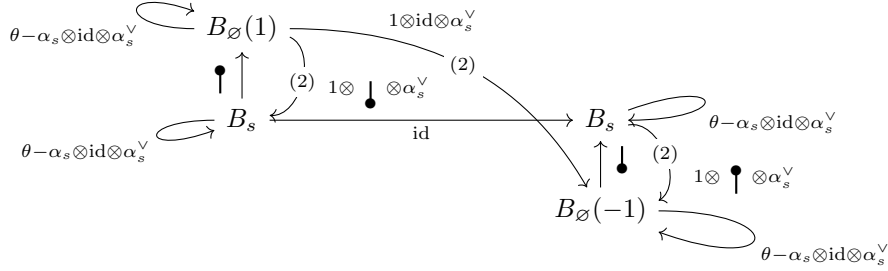
\[
   \begin{tikzcd}[column sep=huge]
B_\varnothing(1) \ar[loop, in=160, out=180, distance=30, "\theta - \alpha_s \otimes \id \otimes \alpha_s^\vee" pos=0.3] \ar[d, bend left=80, "(2)" description, "\ \ 1 \otimes \usebox\lowerdot \otimes \alpha_s^\vee" pos=0.4] \ar[rrdd, in=120, out=0, "1 \otimes \id \otimes \alpha_s^\vee" near start, "(2)" description] \\
B_s \ar[u, "\usebox\upperdot"] \ar[loop, in=200, out=180, distance=30, "\theta - \alpha_s \otimes \id \otimes \alpha_s^\vee", pos=0.5, swap] \ar[rr, "\id"'] &&
B_s \ar[loop, in=0, out=20, distance=40, "\theta - \alpha_s \otimes \id \otimes \alpha_s^\vee" pos=0.6] \ar[d, bend left=80, "(2)" description, "\ \ 1 \otimes \usebox\upperdot \otimes \alpha_s^\vee" pos=0.7] \\
&& B_\varnothing(-1) \ar[u, "\usebox\lowerdot"'] \ar[loop right, in=-20, out=0, distance=50, "\theta - \alpha_s \otimes \id \otimes \alpha_s^\vee"]
\end{tikzcd}
 \]
and the chain map $q_s \in \uHom^\rhd_\FM(\tN_s, \tD_s \langle 2 \rangle)$ defined as follows:
 
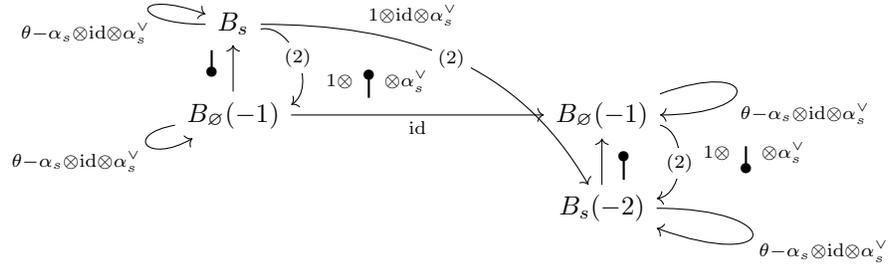
\[
   \begin{tikzcd}[column sep=huge]
B_s \ar[loop, in=160, out=180, distance=30, "\theta - \alpha_s \otimes \id \otimes \alpha_s^\vee" pos=0.3] \ar[d, bend left=80, "(2)" description, "\ \ 1 \otimes \usebox\upperdot \otimes \alpha_s^\vee" pos=0.7] \ar[rrdd, in=120, out=0, "1 \otimes \id \otimes \alpha_s^\vee" near start, "(2)" description] \\
B_\varnothing(-1) \ar[u, "\usebox\lowerdot"] \ar[loop, in=210, out=190, distance=20, "\theta - \alpha_s \otimes \id \otimes \alpha_s^\vee", swap, pos=0.5] \ar[rr, "\id"'] &&
B_\varnothing(-1) \ar[loop, in=0, out=20, distance=40, "\theta - \alpha_s \otimes \id \otimes \alpha_s^\vee" pos=0.6] \ar[d, bend left=80, "(2)" description, "\ \ 1 \otimes \usebox\lowerdot \otimes \alpha_s^\vee" pos=0.5] \\
&& B_s(-2) \ar[u, "\usebox\upperdot"'] \ar[loop right, in=-20, out=0, distance=50, "\theta - \alpha_s \otimes \id \otimes \alpha_s^\vee"]
\end{tikzcd}
 \]
 One can check that we have $q_s \circ p_s = 1 \otimes \id_{\tD_s} \otimes \alpha_s^\vee$ in $\uEnd_\FM(\tD_s)$, and $p_s \circ q_s = 1 \otimes \id_{\tN_s} \otimes \alpha_s^\vee$ in $\uEnd_\FM(\tN_s)$.
 
If now $\uw=(s_1, \sdots, s_k)$ is an expression, we set
\[
 p_\uw := p_{s_1} \hatstar \cdots \hatstar p_{s_k}, \quad q_\uw := q_{s_1} \hatstar \cdots \hatstar q_{s_k}
\]
and define
\[
 a_\uw := \prod_{i=1}^k s_k \cdots s_{i+1} (\alpha_{s_i}^\vee) \quad \in R^\vee.
\]
(When $i=k$, the product $s_k \cdots s_{i+1}$ is interpreted as $1 \in W$.)
Then $p_\uw$ belongs to $\uHom^\rhd_\FM(\tD_\uw, \tN_\uw)$, $q_\uw$ belongs to $\uHom_\FM^\rhd(\tN_\uw, \tD_\uw \langle 2 \ell(\uw) \rangle)$, and both of them are chain maps.

\begin{lem}
\label{lem:tD-tN-uw}
 For any expression $\uw$ we have
 \[
 q_\uw \circ p_\uw = 1 \otimes \id_{\tD_\uw} \otimes a_\uw \quad \text{and} \quad p_\uw \circ q_\uw = 1 \otimes \id_{\tN_\uw} \otimes a_\uw
 \]
 in $\uEnd_\FM(\tD_\uw)$ and $\uEnd_\FM(\tN_\uw)$ respectively.
\end{lem}

\begin{proof}
 The proof proceeds by induction on $k$, using Lemmas~\ref{lem:monconv-weak-interchange} and~\ref{lem:mu-tD-tN-new}; details are left to the reader.
\end{proof}

\section{Convolution of morphisms revisited}

For technical reasons, it is sometimes useful to have a convolution operation defined on the whole of $\FM(\fh,W)$ or $\LM(\fh,W)$, and not only on convolutive objects. In this section and the next one we show that if $\cF$ is a fixed convolutive free-monodromic complex, then indeed the functor $\cF \hatstar (-)$ extends to an endofunctor of $\FM(\fh,W)$ or $\LM(\fh,W)$.

Let $\cF \in \Conv_\FM(\fh,W)$ and $\cF', \cG' \in \FM(\fh,W)$. Define a map
\[
 h \mapsto \id_\cF \hatstar h : \uHom_\FM(\cF', \cG') \to \uHom_\FM(\cF \ustar \cF', \cF \ustar \cG')
\]
by
\[
 \id_\cF \hatstar (th'z) := \nu_\cF(t) \ustar h'z
\]
for $t \in \Lambda$, $h' \in \uHom_\BE(\cF', \cG')$, $z \in R^\vee$.

\begin{lem}
 Let $\cF \in \Conv_\FM(\fh,W)$ and $\cF', \cG', \cH' \in \FM(\fh,W)$. Let $h \in \uHom_\FM(\cF', \cG')$ and $k \in \uHom_\FM(\cG', \cH')$. Then we have
 \begin{equation} \label{eq:convolutive-id-hatstar-composition}
  \id_\cF \hatstar (k \circ h) = (\id_\cF \hatstar k) \circ (\id_\cF \hatstar h);
 \end{equation}
 \begin{equation} \label{eq:hatstar-exchange-kappa}
  (\delta_\cF^\lhd \ustar \id_{\cG'}) \circ (\id_\cF \hatstar h) + (-1)^{|h|+1}(\id_\cF \hatstar h) \circ (\delta_\cF^\lhd \ustar \id_{\cF'}) = \id_\cF \hatstar \kappa(h) - \kappa(\id_\cF \hatstar h);
 \end{equation}
 \begin{equation} \label{eq:id-conv-Theta}
  \id_\cF \hatstar \Theta_{\cF'} = \Theta_{\cF \ustar \cF'}.
 \end{equation}
\end{lem}
\begin{proof}
 The first and third formulas are straightforward, so we focus on the second formula. We may assume that $h = th'z$, where $t \in \Lambda$, $h' \in \uHom_\BE(\cF', \cG')$, $z \in R^\vee$. As in the proof of Proposition~\ref{prop:monconv-diff}, assume further that $t = t_1 \cdots t_k$, where $t_1, \sdots, t_k \in V^*$, and write
 \[
  t^{(j)} = t_1 \cdots \widehat{t_j} \cdots t_k.
 \]
 Then
 \[
  \id_\cF \hatstar \kappa(h) = \id_\cF \hatstar \left( \sum_j(-1)^{j+1}t^{(j)} (t_j \ustar h') z \right)= \sum_j (-1)^{j+1} \nu_\cF(t^{(j)}) \ustar t_j \ustar h'z,
 \]
 and since $\cF$ is convolutive, again as in the proof of Proposition~\ref{prop:monconv-diff},
 \[
  \kappa(\nu_\cF(t)) = \sum_j (-1)^{j+1}\nu_\cF(t^{(j)}) \ustar t_j - \delta_\cF^\lhd \nu_\cF(t) - (-1)^{|t|+1}\nu_\cF(t) \delta_\cF^\lhd.
 \]
 It follows that
 \[
  \id_\cF \hatstar \kappa(h) - \kappa(\id_\cF \hatstar h) = (\delta_\cF^\lhd \nu_\cF(t) + (-1)^{|t|+1}\nu_\cF(t)\delta_\cF^\lhd) \ustar h'z.
 \]
 The left-hand side of the formula also equals
 \begin{multline*}
  (\delta_\cF^\lhd \ustar \id_{\cG'}) \circ (\nu_\cF(t) \ustar h'z) + (-1)^{|h|+1}(\nu_\cF(t) \ustar h'z) \circ (\delta_\cF^\lhd \ustar \id_{\cF'}) \\
  = \delta_\cF^\lhd \nu_\cF(t) \ustar h'z + (-1)^{|t|+1}\nu_\cF(t) \delta_\cF^\lhd \ustar h'z
 \end{multline*}
by~\eqref{eqn:interchange-law-rhd-lhd}.
\end{proof}

\section{Convolution as a functor in one variable}

\begin{defn}
 Given $\cF \in \Conv_\FM(\fh,W)$ and $\cF' \in \FM(\fh,W)$, define $C_\cF(\cF')$ to be the free-monodromic complex whose underlying $\DiagBSp$-sequence is $\cF \ustar \cF'$ and whose differential is given by
 \[
  \delta_{\cF \hatstar \cF'} = \id_\cF \hatstar \delta_{\cF'} + \delta_\cF^\lhd \ustar \id_{\cF'}.
 \]
\end{defn}
Let us check that this defines a free-monodromic complex. By definition,
\begin{multline*}
 \delta_{\cF \hatstar \cF'} \circ \delta_{\cF \hatstar \cF'} + \kappa(\delta_{\cF \hatstar \cF'}) = (\id_\cF \hatstar \delta_{\cF'}) \circ (\id_\cF \hatstar \delta_{\cF'}) +
 (\id_\cF \hatstar \delta_{\cF'}) \circ (\delta_\cF^\lhd \ustar \id_{\cF'}) \\
 + (\delta_\cF^\lhd \ustar \id_{\cF'}) \circ (\id_\cF \hatstar \delta_{\cF'}) + (\delta_\cF^\lhd \ustar \id_{\cF'}) \circ (\delta_\cF^\lhd \ustar \id_{\cF'}) + \kappa(\id_\cF \hatstar \delta_{\cF'}) + \kappa(\delta_\cF^\lhd \ustar \id_{\cF'}).
\end{multline*}
By \eqref{eq:convolutive-id-hatstar-composition}, the first term equals $\id_\cF \hatstar (\delta_{\cF'} \circ \delta_{\cF'})$. By \eqref{eq:hatstar-exchange-kappa}, the second, third, and fifth terms together equal $\id_\cF \hatstar \kappa(\delta_{\cF'})$. By~\eqref{eqn:interchange-law-rhd-lhd} and the definitions, the fourth and six terms together equal $(\delta_\cF^\lhd \circ \delta_\cF^\lhd + \kappa(\delta_\cF^\lhd)) \ustar \id_{\cF'} = 0$. So all together, we get $\id_\cF \hatstar \Theta_{\cF'}$, which equals $\Theta_{\cF \ustar \cF'}$ by \eqref{eq:id-conv-Theta}, as desired.

\begin{lem}
 Let $\cF \in \Conv_\FM(\fh,W)$ and $\cF', \cG' \in \FM(\fh,W)$. Let $h \in \uHom_\FM(\cF', \cG')$. Then we have
 \[
  d(\id_\cF \hatstar h) = \id_\cF \hatstar d(h).
 \]
\end{lem}

\begin{proof}
 By definition,
 \begin{multline*}
  d(\id_\cF \hatstar h) = \kappa(\id_\cF \hatstar h) + \delta_{\cF \hatstar \cG'} \circ (\id_\cF \hatstar h) + (-1)^{|h|+1}(\id_\cF \hatstar h) \circ \delta_{\cF \hatstar \cF'} \\
  = \kappa(\id_\cF \hatstar h) +  (\id_\cF \hatstar \delta_{\cG'} + \delta_\cF^\lhd \hatstar \id_{\cG'}) \circ (\id_\cF \hatstar h) \\
  + (-1)^{|h|+1}(\id_\cF \hatstar h) \circ (\id_\cF \hatstar \delta_{\cF'} + \delta_\cF^\lhd \hatstar \id_{\cF'}).
 \end{multline*}
 By \eqref{eq:convolutive-id-hatstar-composition} and \eqref{eq:hatstar-exchange-kappa}, this equals
 \[
  \id_\cF \hatstar (\kappa(h) + \delta_{\cG'} \circ h + (-1)^{|h|+1}h \circ \delta_{\cF'}) = \id_\cF \hatstar d(h). \qedhere
 \]
\end{proof}

Thus $\id_\cF \hatstar (-)$ induces a map on cohomology, which we also call $C_\cF$. Now, \eqref{eq:convolutive-id-hatstar-composition} shows that $C_\cF$ defines an endofunctor of $\FM(\fh,W)$. We have thus proved the following result.

\begin{prop}
\label{prop:convolution-Dmix}
 For any fixed object $\cF$ in $\Conv_\FM(\fh,W)$, there is an additive functor
 \[
  C_{\cF} : \FM(\fh,W) \to \FM(\fh,W)
 \]
 whose restriction to $\Conv_\FM(\fh,W)$ is the functor $\cF \hatstar (-)$.
\end{prop}

The same argument also yields for every $\cF \in \Conv_\FM(\fh,W)$ an endofunctor $C_\cF$ of $\LM(\fh,W)$ whose restriction to $\Conv_\LM(\fh,W)$ is $\cF \hatstar (-)$. In subsequent sections, in both settings we will sometimes write $\cF \hatstar \cG$ for $C_{\cF}(\cG)$ if no confusion is likely.

We will need the following property of this construction.

\begin{prop} \label{prop:convolutive-hatstar-exact}
 For any $\cF \in \Conv_\FM(\fh,W)$, the functor
 \[
  C_\cF : \LM(\fh,W) \to \LM(\fh,W)
 \]
is triangulated.
\end{prop}
\begin{proof}
 By Proposition~\ref{prop:left-vs-right-triangulated-structure}, we may check this using the right triangulated structure on $\LM(\fh,W)$. 
 The advantage of this convention is that we have an equality $C_\cF \circ \Sigma_r = \Sigma_r \circ C_\cF$ on the nose, so that 
we do not need to keep track of natural isomorphisms.
 
 We must check that $C_\cF$ takes a standard triangle to a distinguished triangle. In fact for any $\cF', \cG' \in \LM(\fh,W)$ and any chain map $h : \cF' \to \cG'$, we claim that there is an isomorphism of triangles
 \[
  \begin{tikzcd}[column sep=large]
   C_\cF(\cF') \ar[r, "\id_\cF \hatstar h"] \ar[d, equal] & C_\cF(\cG') \ar[r, "\id_\cF \hatstar \alpha_r(h)"] \ar[d, equal]
    & C_\cF(\cone_r(h)) \ar[r, "\id_\cF \hatstar \beta_r(h)"] \ar[d, "\sigma"] & \Sigma_r C_\cF(\cF') \ar[d, equal] \\
   C_\cF(\cF') \ar[r, "\id_\cF \hatstar h"] & C_\cF(\cG') \ar[r, "\alpha_r(\id_\cF \hatstar h)"]
    & \cone_r(\id_\cF \hatstar h) \ar[r, "\beta_r(\id_\cF \hatstar h)"] & \Sigma_r C_\cF(\cF')
  \end{tikzcd}
 \]
 in $(\LM(\fh,W), \Sigma_r)$. Here, $\sigma$ is the evident isomorphism of the underlying $\DiagBSp$-sequences
 \[
  \cF \ustar (\cF'[1] \oplus \cG') \simto (\cF \ustar \cF')[1] \oplus (\cF \ustar \cG').
 \]
 In terms of these decompositions,
 \[
  \id_\cF \hatstar \alpha_r(h) = \id_\cF \hatstar \begin{bmatrix} 0 \\ \id_{\cG'} \end{bmatrix},
  \qquad
  \id_\cF \hatstar \beta_r(h) = \id_\cF \hatstar \begin{bmatrix} \id_{\cF'[1]} & 0 \end{bmatrix},
 \]
 \[
  \alpha_r(\id_\cF \hatstar h) = \begin{bmatrix} 0 \\ \id_{\cF \ustar \cG'} \end{bmatrix},
  \qquad
  \beta_r(\id_\cF \hatstar h) = \begin{bmatrix} \id_{(\cF \ustar \cF')[1]} & 0 \end{bmatrix},
 \]
 so the diagram above commutes. It remains to check that $\sigma$ is a chain map, i.e.~that it identifies the differentials
 \begin{multline*}
  \delta_{C_\cF ( \cone_r(h))} = \id_\cF \hatstar \delta_{\cone_r(h)} + \delta_\cF^\lhd \ustar \id_{\cone_r(h)}
  \\
  = \id_\cF \hatstar
  \begin{bmatrix}
   \delta_{\cF'} & \\
   \bt^{-1}(\bu_{\cF',\cG'}(h)) & \delta_{\cG'}
  \end{bmatrix}
  + \delta_\cF^\lhd \ustar
  \begin{bmatrix}
   \id_{\cF'} & \\
   & \id_{\cG'}
  \end{bmatrix}
 \end{multline*}
 and
 \[
  \delta_{\cone_r(\id_\cF \hatstar h)} =
   \begin{bmatrix}
    \delta_{C_\cF(\cF')} & \\
    \bt^{-1}(\bu_{\cF \ustar \cF', \cF \ustar \cG'}(\id_\cF \hatstar h)) & \delta_{C_\cF(\cG')}
   \end{bmatrix}.
 \]
 This is clear for the diagonal entries. For the lower left entry, we may assume that $h = s \otimes h'$ for $s \in \Lambda$ and $h' \in \uHom_\BE(\cF', \cG')$. Then
 \[
  \id_\cF \hatstar \bt^{-1}(\bu_{\cF',\cG'}(h)) = \id_\cF \hatstar (s \otimes \bt^{-1}(\bu_{\cF',\cG'}(h'))) = \nu_\cF(s) \ustar \bt^{-1}(\bu_{\cF',\cG'}(h'))
 \]
 and
 \[
  \bt^{-1}(\bu_{\cF \ustar \cF',\cF \ustar \cG'}(\id_\cF \hatstar h)) = \bt^{-1}(\bu_{\cF \ustar \cF',\cF \ustar \cG'}(\nu_\cF(s) \ustar h')),
 \]
 which agree by \eqref{eq:btbu-conv}.
\end{proof}

\section{Tilting complexes and the functoriality conjecture}
\label{sec:functoriality-conjecture}

\begin{defn}
Let $\uw = (s_1,\sdots,s_r)$ be an expression. The \emph{free-monodromic Bott--Samelson tilting complex} associated to $\uw$ is the convolutive free-monodromic complex given by\index{tilting object!Tmonuw@$\Tmon_\uw$}
\[
\Tmon_\uw := \Tmon_{s_1} \hatstar \cdots \hatstar \Tmon_{s_r}.
\]
(This complex is well defined by the remarks in~\S\ref{sec:coherence}.)
We denote by
\[
\TiltBSp(\fh,W)
\]
the full additive subcategory of $\Conv_\FM(\fh,W)$ whose objects are direct sums of objects of the form $\Tmon_\uw\la n\ra$.
\end{defn}

We expect that, under suitable assumptions on $\fh$, the operation $\hatstar$ equips $\TiltBSp(\fh,W)$ with the structure of a monoidal category. Since we have already discussed the existence of unitor and associator isomorphisms, as well as the coherence conditions (see~\S\ref{sec:coherence}), the content of this claim really comes down to the question of whether $\hatstar$ defines a bifunctor on $\TiltBSp(\fh,W)$. Therefore, this conjecture can be rephrased as follows:
\begin{equation}
\label{eqn:conj-interchange}
\begin{array}{c}
\text{\emph{Let $f : \cF \to \cG$, $g : \cG \to \cH$, $h : \cF' \to \cG'$, $k : \cG' \to \cH'$ be morphisms}} \\
\text{\emph{in $\TiltBSp(\fh,W)$. Then we have
$(g \circ f) \hatstar (k \circ h) = (g \hatstar k) \circ (f \hatstar h)$.}}
\end{array}
\end{equation}

The main result of this paper is the proof of this property in the case of Cartan realizations of crystallographic Coxeter groups, when $\bk$ is Noetherian and of finite global dimension: see Theorem~\ref{thm:functoriality-hatstar}.

We also expect that, again under suitable technical conditions, the monoidal category $(\TiltBSp(\fh,W), \hatstar)$ is closely related to the category $\DiagBSp(\fh^\vee,W)$, where $\fh^\vee = (V^*, \{\alpha_s : s \in S\}, \{\alpha_s^\vee : s \in S\})$. This statement will be proved, in the case of Cartan realizations of crystallographic Coxeter groups, in~\cite{mkdkm}.

\chapter{Finite dihedral groups}
\label{chap:dihedral}

Our strategy for proving~\eqref{eqn:conj-interchange} is to reduce it to the case of certain specific morphisms for which explicit computations are possible.  Those computations will take place in a finite dihedral group.  This chapter and the next lay the groundwork for those computations.

Thus, in this chapter, we assume that $W$ is a finite dihedral group, generated by simple reflections $s$ and $t$. We denote by $m_{st}$ the order of $st$. We also assume that $\bk$ is a field, and we fix a balanced realization $\fh$ of $W$ over $\bk$.
Recall that in this setting we have defined $[n]_s, [n]_t \in \bk$ in~\S\ref{sec:realizations}. We also set $a_{st}=\langle \alpha_s^\vee, \alpha_t \rangle$ and $a_{ts} = \langle \alpha_t^\vee, \alpha_s \rangle$.

We will make the following additional assumptions:
\begin{enumerate}
 \item
 \label{it:dihassump-quantum-numbers}
 for $1 \leq n < m_{st}$, we have $[n]_s \in \bk^\times$ and $[n]_t \in \bk^\times$;
 \item
 \label{it:dihassump-faithful}
 the $W$-action on $V$ is faithful;
 \item
  \label{it:dihassump-locnondeg}
 we have $4 - a_{st} a_{ts} \in \bk^\times$.
\end{enumerate}
In the language of~\cite[\S 3.3]{elias}, we assume ``local non-degeneracy'' and ``lesser invertibility.''
Note that~\eqref{it:dihassump-locnondeg} implies that Demazure surjectivity is satisfied.

Starting from a Cartan realization $\fh$ of a crystallographic Coxeter system $(W',S')$ (see~\S\ref{sec:Cartan-realizations} below), given $s,t \in S'$ distinct and generating a finite subgroup $W$ of $W'$, one obtains a realization of $W$ by restriction. In this case we have $m_{st} \in \{2,3,4,6\}$,~\eqref{it:dihassump-locnondeg} is always satisfied, and~\eqref{it:dihassump-quantum-numbers} is automatic if $m_{st} \in \{2,3\}$, equivalent to $\mathrm{char}(\bk) \neq 2$ if $m_{st}=4$, and equivalent to $\mathrm{char}(\bk) \notin \{2,3\}$ if $m_{st}=6$.


\section{Roots for dihedral groups}
\label{sec:roots}

Thanks to assumption~\eqref{it:dihassump-locnondeg}, we can define elements $\varpi_s, \varpi_t \in V$ by the following formulas:\index{pis@{$\varpi_s$}}\index{pit@{$\varpi_t$}}
\begin{equation}\label{eqn:varpi-defn}
\varpi_s = \frac{1}{4 - a_{st}a_{ts}}( 2 \alpha_s^\vee - a_{st} \alpha_t^\vee),  
\qquad
\varpi_t = \frac{1}{4 - a_{st}a_{ts}}( -a_{ts} \alpha_s^\vee + 2 \alpha_t^\vee).
\end{equation}
These satisfy
\[
\langle \alpha_u, \varpi_{v} \rangle = \delta_{u,v} \quad \text{for $u,v \in \{s,t\}$.}
\]

Define $\alpha_{s, n}, \alpha_{t, n} \in V^*$ for $n \ge 1$ as follows:
\begin{equation}\label{eq:alpha-s-t-n-pre}
 \begin{array}{rclcccrcl}
  \alpha_{s, 1} &=& \alpha_s      & & & & \alpha_{t, 1} &=& \alpha_t \\
  \alpha_{s, 2} &=& s(\alpha_t)   & & & & \alpha_{t, 2} &=& t(\alpha_s) \\
  \alpha_{s, 3} &=& st(\alpha_s)  & & & & \alpha_{t, 3} &=& ts(\alpha_t) \\
  \alpha_{s, 4} &=& sts(\alpha_t) & & & & \alpha_{t, 4} &=& tst(\alpha_s) \\
                &\vdots&          & & & &               &\vdots&
 \end{array}
\end{equation}
Induction using that $\alpha_{s,n} = s(\alpha_{t,n-1})$ and $\alpha_{t,n} = t(\alpha_{s,n-1})$ together with~\eqref{eq:2q-recursion-alt} show that
\begin{equation} \label{eq:alpha-s-t-n}
 \alpha_{s, n} = [n]_s\alpha_s + [n-1]_t\alpha_t, \qquad \alpha_{t, n} = [n-1]_s\alpha_s + [n]_t\alpha_t.
\end{equation}
Generalizing~\eqref{eqn:varpi-defn}, we define 
\begin{equation} \label{eq:varpi-s-t-n}
 \begin{aligned}
  \varpi_t^{s,n} &= -\frac{[n-1]_t}{[n]_s}\varpi_s + \varpi_t,   & \varpi_{s,n}^t &= \frac{1}{[n]_s}\varpi_s, \\
  \varpi_s^{t,n} &= \varpi_s - \frac{[n-1]_s}{[n]_t}\varpi_t, & \varpi_{t,n}^s &= \frac{1}{[n]_t}\varpi_t.
 \end{aligned}
\end{equation}
The pair $\{\varpi^{t,n}_s, \varpi^s_{t,n}\}$ is ``dual'' to the set $\{\alpha_s, \alpha_{t,n}\}$, in the sense that
\begin{align*}
\la \alpha_s, \varpi^{t,n}_s \ra &= 1, & \la \alpha_{t,n}, \varpi^{t,n}_s\ra &= 0, \\ 
\la \alpha_s, \varpi_{t,n}^s \ra &= 0, & \la \alpha_{t,n}, \varpi_{t,n}^s\ra &= 1.
\end{align*}
(If $\dim V = 2$, then $\{\varpi^{t,n}_s, \varpi^s_{t,n}\}$ is simply the dual basis to the basis $\{\alpha_s, \alpha_{t,n}\}$ for $V^*$.)  The pair $\{\varpi_t^{s,n}, \varpi^t_{s,n}\}$ is dual in the same sense to $\{\alpha_t, \alpha_{s,n}\}$.  

The following lemma shows that condition~\eqref{it:dihassump-faithful} follows from~\eqref{it:dihassump-quantum-numbers} and~\eqref{it:dihassump-locnondeg} in most cases.

\begin{lem}
\label{lem:dihedral-faithfulness}
Assume that~\eqref{it:dihassump-quantum-numbers} and~\eqref{it:dihassump-locnondeg} are satisfied, and that $\mathrm{char}(\bk) \neq 2$ if $m_{st}$ is even. Then the $W$-action on $V$ is faithful.
\end{lem}

\begin{proof}
Assume for a contradiction that $x \in W \smallsetminus \{1\}$ acts trivially on $V$. Swapping $s$ and $t$ if necessary, we can assume that $x$ admits a reduced expression $\ux$ starting with $s$. We set $n=\ell(x)$.

First, assume that the last simple reflection in $\ux$ is $t$. Then $x(\alpha_t)=-\alpha_{s,n}=\alpha_t$. By~\eqref{eq:alpha-s-t-n}, this implies that $[n]_s=0$ and $[n-1]_t=-1$. The first condition implies that $n=m_{st}$ (so that $m_{st}$ is even), and since our realization is balanced the second condition implies that $\mathrm{char}(\bk)=2$. This contradicts our assumptions.

Now, assume that the last simple reflection in $\ux$ is $s$. Then $x(\alpha_t) = \alpha_{s,n+1}=\alpha_t$. By~\eqref{eq:alpha-s-t-n}, this implies that $[n+1]_s=0$ and $[n]_t=1$, hence that $n=m_{st}-1$. On the other hand $x(\alpha_s) = -\alpha_{s,n}=\alpha_s$, which implies that $[n]_s=-1$ and $[n-1]_t=0$. Hence $n=1$, $m_{st}=2$, and again the fact that our realization is balanced implies that $\mathrm{char}(\bk)=2$, a contradiction.
\end{proof}

\section{Jones--Wenzl projectors}
\label{sec:JW-projectors}

For $u \in \{s,t\}$, we denote by $\check{u}$ the element of $\{s,t\}$ which is different from $u$.

For $\uw \in \hW$, recall that we have a \emph{Jones--Wenzl projector}\index{JW@{$\mathrm{JW}_\uw$}}\index{Jones--Wenzl projector}
\[
 \mathrm{JW}_{\uw} \in \End(B_{\uw}),
\]
see Sections~\ref{sec:JW}--\ref{sec:ew-diagram},~\cite[\S A.6]{elias},~\cite[\S 5.2]{ew},~\cite[\S 6]{ew2} and references therein. 
The value of these endomorphisms on expressions of length $\leq 3$ is as follows:
\begin{gather*}
\mathrm{JW}_{(u)} = \id_{B_u}, 
\quad \mathrm{JW}_{(u,\check{u})} = \id_{B_{(u,\check{u})}}, \quad
\mathrm{JW}_{(u,\check{u},u)} =
\begin{array}{c}
\begin{tikzpicture}[xscale=0.3, yscale=0.2,thick]
\draw (-1.5,3) -- (-1.5,-3);
\draw (0,3) -- (0,-3);
\draw (1.5,3) -- (1.5,-3);
\node at (-1.5,-3.6) {\tiny $u$};
\node at (-1.5,3.6) {\tiny $u$};
\node at (1.5,-3.6) {\tiny $u$};
\node at (1.5,3.6) {\tiny $u$};
\node at (0,-3.6) {\tiny $\check{u}$};
\node at (0,3.6) {\tiny $\check{u}$};
\end{tikzpicture}
\end{array}
+
\frac{[1]}{[2]_u}
\begin{array}{c}
\begin{tikzpicture}[xscale=0.3, yscale=0.2,thick]
\draw (-1.5,3) -- (-1.5,-3);
\draw (0,3) -- (0,1.5);
\draw (0,-3) -- (0,-1.5);
\node at (0,1.5) {\small $\bullet$};
\node at (0,-1.5) {\small $\bullet$};
\draw (-1.5,0) -- (1.5,0);
\draw (1.5,3) -- (1.5,-3);
\node at (-1.5,-3.6) {\tiny $u$};
\node at (-1.5,3.6) {\tiny $u$};
\node at (1.5,-3.6) {\tiny $u$};
\node at (1.5,3.6) {\tiny $u$};
\node at (0,-3.6) {\tiny $\check{u}$};
\node at (0,3.6) {\tiny $\check{u}$};
\end{tikzpicture}
\end{array}.
\end{gather*}
They satisfy the following properties.
\begin{enumerate}[label=(JW\arabic*)]
 \item
 \label{it:JW-projector}
 $\mathrm{JW}_{\uw}$ is an idempotent.  According to~\cite[Theorem~6.24]{elias}, its image can be identified with the indecomposable object
 \[
 B_w \in \Diag(\fh,W),
 \]
 where $w = \pi(\uw)$. (In particular, this image depends only on $w$, and not on $\uw$.)
 \item
 \label{it:JW'}
 There exists a diagram $\mathrm{JW}'_{\uw}$ such that
\[
\begin{array}{c}
\begin{tikzpicture}[xscale=0.3, yscale=0.2,thick]
 \draw (-3,-2) rectangle (3,2);
 \node at (0,0) {$\mathrm{JW}_{\uw}$};
 \draw (-2, -4) -- (-2, -2);
 \draw (2,-4) -- (2, -2);
  \draw (-2, 4) -- (-2, 2);
 \draw (2,4) -- (2, 2);
 \node at (0,3) {$\cdots$};
 \node at (0,-3) {$\cdots$};
\end{tikzpicture}
\end{array} =
\begin{array}{c}
\begin{tikzpicture}[xscale=0.3, yscale=0.2,thick]
 \draw (-3,-2) rectangle (3,2);
 \node at (0,0) {$\mathrm{JW}'_{\uw}$};
 \draw (-2, -4) -- (-2, -2);
 \draw (2,-4) -- (2, -2);
  \draw (-2, 4) -- (-2, 2);
 \draw (2,4) -- (2, 2);
 \draw (-4,-4) -- (-4,4);
 \draw (4,-4) -- (4,4);
 \draw (-4,0) -- (-3,0);
 \draw (4,0) -- (3,0);
 \node at (0,3) {$\cdots$};
 \node at (0,-3) {$\cdots$};
\end{tikzpicture}
\end{array}.
\]
 \item
 \label{it:death-pitchfork}
 Adding a pitchfork (see~\cite[\S 5.3.4]{elias}) on top or bottom of $\mathrm{JW}_{\uw}$ produces $0$.
 \item
 \label{it:JW-horizontal-reflection}
 $\mathrm{JW}_{\uw}$ is invariant under horizontal reflection.
 \item
 \label{it:JW-recurrence-1}
 If $2 \leq r \leq m_{st}-1$ and $\uw = (s_1, \sdots, s_{r+1})$, then setting $\uv = (s_1, \sdots, s_{r})$ and $\underline{u} = (s_1, \sdots, s_{r-1})$ we have (see e.g.~\cite[Theorem~6.10]{ew2}\footnote{Applying the formula from~\cite[Theorem~6.10]{ew2} one obtains a formula similar to that in~\ref{it:JW-recurrence-1}, but without the box $\mathrm{JW}_{\underline{u}}$. This box can be added thanks to the ``death by pitchfork'' property, see~\ref{it:death-pitchfork}.})
 \[
  \begin{array}{c}
\begin{tikzpicture}[xscale=0.3, yscale=0.2,thick]
 \draw (-3,-2) rectangle (3,2);
 \node at (0,0) {$\mathrm{JW}_{\uw}$};
 \draw (-2, -4) -- (-2, -2);
 \draw (1,-4) -- (1,-2);
 \draw (2,-4) -- (2, -2);
  \draw (-2, 4) -- (-2, 2);
 \draw (1,4) -- (1,2);
 \draw (2,4) -- (2, 2);
 \node at (-0.5,3) {$\cdots$};
 \node at (-0.5,-3) {$\cdots$};
 \node at (1,-4.6) {\tiny $s_{r}$};
 \node at (2.7,-4.6) {\tiny $s_{r+1}$};
 \node at (1,4.6) {\tiny $s_{r}$};
 \node at (2.7,4.6) {\tiny $s_{r+1}$};
\end{tikzpicture}
\end{array}
=
\begin{array}{c}
\begin{tikzpicture}[xscale=0.3, yscale=0.2,thick]
 \draw (-3,-2) rectangle (2,2);
 \node at (-0.5,0) {$\mathrm{JW}_{\uv}$};
 \draw (-2, -4) -- (-2, -2);
 \draw (1,-4) -- (1,-2);
 \draw (3,-4) -- (3, 4);
  \draw (-2, 4) -- (-2, 2);
 \draw (1,4) -- (1,2);
 \node at (-0.5,3) {$\cdots$};
 \node at (-0.5,-3) {$\cdots$};
 \node at (1,-4.6) {\tiny $s_{r}$};
 \node at (3,-4.6) {\tiny $s_{r+1}$};
 \node at (1,4.6) {\tiny $s_{r}$};
 \node at (3,4.6) {\tiny $s_{r+1}$};
\end{tikzpicture}
\end{array}
+
\frac{[r-1]_{s_{r}}}{[r]_{s_{r+1}}}
\begin{array}{c}
\begin{tikzpicture}[xscale=0.3,yscale=0.25,thick]
 \draw (-3,-1) rectangle (1,1);
 \node at (-1,0) {$\mathrm{JW}_{\underline{u}}$};
 \draw (-3,-5) rectangle (2,-3);
 \node at (-1,-4) {$\mathrm{JW}_{\underline{v}}$};
 \draw (-3,3) rectangle (2,5);
 \node at (-1,4) {$\mathrm{JW}_{\underline{v}}$}; 
 \draw (-2, -6) -- (-2, -5);
 \draw (1, -6) -- (1, -5);
 \draw (-2, 6) -- (-2, 5);
 \draw (1, 6) -- (1, 5);
 \draw (-2,-3) -- (-2,-1);
 \draw (0,-3) -- (0,-1);
 \draw (-2,3) -- (-2,1);
 \draw (0,3) -- (0,1);
 \draw (3,-6) -- (3,-1.5);
 \draw (3,1.5) -- (3,6);
 \draw (0,-1.5) -- (3,-1.5);
 \draw (0,1.5) -- (3,1.5);
 \draw (1,-3) -- (1,-2.2);
 \node at (1,-2.2) {\small $\bullet$};
 \draw (1,3) -- (1,2.2);
 \node at (1,2.2) {\small $\bullet$};
 \node at (-2,-6.6) {\tiny $s_1$};
 \node at (1,-6.6) {\tiny $s_{r}$};
 \node at (3,-6.6) {\tiny $s_{r+1}$};
 \node at (-2,6.6) {\tiny $s_1$};
 \node at (1,6.6) {\tiny $s_{r}$};
 \node at (3,6.6) {\tiny $s_{r+1}$};
  \node at (-0.9,-2) {$\cdots$};
  \node at (-0.9,2) {$\cdots$};
  \node at (-0.5,-5.5) {$\cdots$};
  \node at (-0.5,5.5) {$\cdots$};
\end{tikzpicture}
\end{array}.
 \]
 \item
 \label{it:JW-recurrence-2}
 If $3 \leq r \leq m_{st}-1$ and $\uw = (s_1, \sdots, s_{r+1})$, then setting $\uv = (s_1, \sdots, s_{r})$ we have (see e.g.~\cite[Remark~6.11]{ew2})
 \begin{multline*}
\begin{array}{c}
\begin{tikzpicture}[xscale=0.3,yscale=0.25,thick]
 \draw (-3,-2) rectangle (3,2);
 \node at (0,0) {$\mathrm{JW}_{\uw}$};
 \draw (-2, -4) -- (-2, -2);
 \draw (1,-4) -- (1,-2);
 \draw (2,-4) -- (2, -2);
  \draw (-2, 4) -- (-2, 2);
 \draw (1,4) -- (1,2);
 \draw (2,4) -- (2, 2);
 \node at (-0.5,3) {$\cdots$};
 \node at (-0.5,-3) {$\cdots$};
 \node at (1,-4.6) {\tiny $s_{r}$};
 \node at (2.7,-4.6) {\tiny $s_{r+1}$};
 \node at (1,4.6) {\tiny $s_{r}$};
 \node at (2.7,4.6) {\tiny $s_{r+1}$};
\end{tikzpicture}
\end{array}
=
\begin{array}{c}
\begin{tikzpicture}[xscale=0.3,yscale=0.25,thick]
 \draw (-3,-2) rectangle (2,2);
 \node at (-0.5,0) {$\mathrm{JW}_{\uv}$};
 \draw (-2, -4) -- (-2, -2);
 \draw (1,-4) -- (1,-2);
 \draw (3,-4) -- (3, 4);
  \draw (-2, 4) -- (-2, 2);
 \draw (1,4) -- (1,2);
 \node at (-0.5,3) {$\cdots$};
 \node at (-0.5,-3) {$\cdots$};
 \node at (1,-4.6) {\tiny $s_{r}$};
 \node at (3,-4.6) {\tiny $s_{r+1}$};
 \node at (1,4.6) {\tiny $s_{r}$};
 \node at (3,4.6) {\tiny $s_{r+1}$};
\end{tikzpicture}
\end{array}
+ \frac{1}{[r]_{s_{r+1}}}
\begin{array}{c}
\begin{tikzpicture}[xscale=0.3,yscale=0.25,thick]
 \draw (-3,-2) rectangle (3,2);
 \node at (0,0) {$\mathrm{JW}_{\uv}$};
 \draw (-2, -3) -- (-2, -2);
 \node at (-2,-3) {\small $\bullet$};
 \draw (-1,-4) -- (-1,-2);
 \draw (1,-4) -- (1,-2);
 \draw (2,-4) -- (2, -2);
  \draw (-2, 4) -- (-2, 2);
 \draw (1,3) -- (1,2);
 \draw (2,2.5) -- (2, 2);
 \node at (-0.5,3) {$\cdots$};
 \node at (0,-3) {$\cdots$};
 \draw (-4,-4) -- (-4,4);
 \node at (-4,-4.6) {\tiny $s_1$};
 \node at (-4.3,4.6) {\tiny $s_1$};
 \node at (-3,4.6) {\tiny $\check{s_1}$};
 \node at (-1.8,4.6) {\tiny $s_1$};
 \draw (-4,3) -- (-2,3);
 \draw (-3,3.5) -- (-3,4);
 \node at (-3,3.5) {\small $\bullet$};
 \node at (2,2.5) {\small $\bullet$};
 \draw (1,3) -- (4,3);
 \draw (4,4) -- (4,-4);
  \node at (2,-4.6) {\tiny $s_r$};
  \node at (4,-4.6) {\tiny $s_{r+1}$};
  \node at (4,4.6) {\tiny $s_{r+1}$};
\end{tikzpicture}
\end{array}
\\
+
\sum_{a=2}^{r-2}
\frac{[r-a]_{\check{s_a}}}{[r]_{s_{r+1}}}
\begin{array}{c}
\begin{tikzpicture}[xscale=0.3,yscale=0.25,thick]
 \draw (-3,-2) rectangle (3,2);
 \node at (0,0) {$\mathrm{JW}_{\uv}$};
 \draw (-2, -4) -- (-2, -2);
 \draw (1,-4) -- (1,-2);
 \draw (2,-4) -- (2, -2);
  \draw (-2, 4) -- (-2, 2);
 \draw (1,3) -- (1,2);
 \draw (1,3) -- (4,3);
 \draw (2,2.5) -- (2, 2);
 \node at (2,2.5) {\small $\bullet$};
 \draw (4,-4) -- (4,4);
 \draw (-0.5,2) -- (-0.5,3);
 \draw (-0.5,4) -- (-0.5,3.5);
 \node at (-0.5,3.5) {\small $\bullet$};
 \node at (-0.5,4.5) {\tiny $\check{s_a}$};
 \node at (-1.4,4.5) {\tiny $s_a$};
  \node at (0.4,4.5) {\tiny $s_a$};
 \draw (-1,3) -- (0,3);
 \draw (-1,3) -- (-1,4);
 \draw (0,3) -- (0,4);
 \node at (-1.2,2.5) {\tiny $\cdots$};
  \node at (0.3,2.5) {\tiny $\cdots$};
 \node at (-0.5,-3) {$\cdots$};
 \node at (4,4.5) {\tiny $s_{r+1}$};
 \node at (4,-4.5) {\tiny $s_{r+1}$};
 \node at (2,-4.5) {\tiny $s_r$};
\end{tikzpicture}
\end{array}
+
\frac{[r-1]_{s_r}}{[r]_{s_{r+1}}}
\begin{array}{c}
\begin{tikzpicture}[xscale=0.3,yscale=0.25,thick]
 \draw (-3,-2) rectangle (2,2);
 \node at (-0.5,0) {$\mathrm{JW}_{\uv}$};
 \draw (-2, -4) -- (-2, -2);
 \draw (1,-4) -- (1,-2);
 \draw (3,-4) -- (3, 4);
 \draw (-2, 4) -- (-2, 2);
 \draw (1,2.5) -- (1,2);
 \node at (1,2.5) {\small $\bullet$};
 \draw (0,4) -- (0,2);
 \draw (0,3) -- (3,3);
 \draw (1,4) -- (1,3.5);
 \node at (1,3.5) {$\bullet$};
 \node at (-1,3) {$\cdots$};
 \node at (-0.5,-3) {$\cdots$};
 \node at (1,-4.6) {\tiny $s_r$};
 \node at (3,-4.6) {\tiny $s_{r+1}$};
 \node at (1,4.6) {\tiny $s_{r}$};
 \node at (3,4.6) {\tiny $s_{r+1}$};
\end{tikzpicture}
\end{array}.
\end{multline*}
\end{enumerate}

Note that using~\ref{it:JW-horizontal-reflection}
one can reflect the relations in~\ref{it:JW-recurrence-1} and~\ref{it:JW-recurrence-2} to produce new relations. The relations~\ref{it:JW-recurrence-1}--\ref{it:JW-recurrence-2} also have ``left'' versions corresponding to adding a strand on the left.

\section{Further properties of Jones--Wenzl projectors}

In this section we collect a number of consequences of the properties of Jones--Wenzl projectors stated in~\S\ref{sec:JW-projectors}, to be used later.

\begin{lem}
\label{lem:JW-dot-trivalent}
Let $\uw \in \hW$. Then we have
\[
\begin{array}{c}
\begin{tikzpicture}[xscale=0.3, yscale=0.2,thick]
 \draw (-3,-2) rectangle (3,2);
 \node at (0,0) {$\mathrm{JW}_{\uw}$};
 \draw (-2, -3) -- (-2, -2);
 \node at (-2,-3) {\small $\bullet$};
 \draw (-1,-4) -- (-1,-2);
 \draw (2,-4) -- (2, -2);
  \draw (-2, 4) -- (-2, 2);
 \draw (2,4) -- (2, 2);
 \node at (0,3) {$\cdots$};
 \node at (0.5,-3) {$\cdots$};
 \draw (-4,-4) -- (-4,3) -- (-2,3);
\end{tikzpicture}
\end{array}
=
\begin{array}{c}
\begin{tikzpicture}[xscale=0.3, yscale=0.2,thick]
 \draw (-3,-2) rectangle (3,2);
 \node at (0,0) {$\mathrm{JW}_{\uw}$};
 \draw (-2, -4) -- (-2, -2);
 \draw (2,-4) -- (2, -2);
  \draw (-2, 4) -- (-2, 2);
 \draw (2,4) -- (2, 2);
 \node at (0,3) {$\cdots$};
 \node at (0,-3) {$\cdots$};
\end{tikzpicture}
\end{array}.
\]
\end{lem}

\begin{proof}
This property is a direct consequence of~\ref{it:JW'} and the Frobenius relations.
\end{proof}

\begin{lem}
\label{lemJW-handle}
Let $\uw \in \hW$. We have
\[
  \begin{array}{c}\begin{tikzpicture}[xscale=0.3, yscale=0.3,thick]
   \draw (-3,1) -- (-3,3); \node at (-1.4,2) {$\cdots$}; \draw (0,1) -- (0,3);
   \draw (-3.5,-1) rectangle (0.5,1); \node at (-1.4,0) {\small $\JW_\uw$};
   \draw (-3,-3) -- (-3,-1); \node at (-1.4,-2) {$\cdots$}; \draw (0,-3) -- (0,-1); \draw (0,-2) -- (1.5,-2) -- (1.5,2) -- (0,2);
  \end{tikzpicture}\end{array}
  = 0.
\]
\end{lem}

\begin{proof}
This property is a consequence of~\ref{it:JW'} and the needle relation.
\end{proof}

\begin{lem}
\label{lem:JW-2dots}
Let $\uw \in \hW$, and
 assume that $\ell(\uw) \geq 2$. Then we have
 \[
    \begin{array}{c}
\begin{tikzpicture}[xscale=0.3, yscale=0.2,thick]
 \draw (-3,-2) rectangle (4,2);
 \node at (0.5,0) {$\mathrm{JW}_{\uw}$};
 \draw (-2, -4) -- (-2, -2);
 \draw (1,-4) -- (1,-2);
 \draw (2,-3) -- (2, -2);
 \node at (3,-3) {\small $\bullet$};
  \node at (2,-3) {\small $\bullet$};
 \draw (3,-3) -- (3,-2);
  \draw (-2, 4) -- (-2, 2);
 \draw (3,4) -- (3,2);
 \node at (0.5,3) {$\cdots$};
 \node at (-0.5,-3) {$\cdots$};
 \end{tikzpicture}
\end{array}
=
    \begin{array}{c}
\begin{tikzpicture}[yscale=0.2, xscale=-0.3,thick]
 \draw (-3,-2) rectangle (4,2);
 \node at (0.5,0) {$\mathrm{JW}_{\uw}$};
 \draw (-2, -4) -- (-2, -2);
 \draw (1,-4) -- (1,-2);
 \draw (2,-3) -- (2, -2);
 \node at (3,-3) {\small $\bullet$};
  \node at (2,-3) {\small $\bullet$};
 \draw (3,-3) -- (3,-2);
  \draw (-2, 4) -- (-2, 2);
 \draw (3,4) -- (3,2);
 \node at (0.5,3) {$\cdots$};
 \node at (-0.5,-3) {$\cdots$};
 \end{tikzpicture}
\end{array}.
 \]
\end{lem}

\begin{proof}
Let $u \in \{s,t\}$, and let $\delta \in V^*$ be such that $\langle \delta, \alpha_u^\vee \rangle = 1$. (Such an element exists by the Demazure surjectivity assumption) Then we have
\[
    \begin{array}{c}
\begin{tikzpicture}[xscale=0.3,yscale=0.5,thick]
 \draw (-1,-1) -- (-1,1);
 \draw (0,0) -- (0,1);
 \draw (1,0) -- (1,1);
 \node at (0,0) {\small $\bullet$};
 \node at (1,0) {\small $\bullet$};
 \node at (-1,-1.3) {\tiny $u$};
\end{tikzpicture}
\end{array}
=
    \begin{array}{c}
\begin{tikzpicture}[xscale=0.3,yscale=0.5,thick]
 \draw (-1,-1) -- (-1,1);
 \draw (0,0.5) -- (0,1);
 \draw (-1,-0.5) -- (1,-0.5) -- (1,1);
 \node at (0,0.5) {\small $\bullet$};
 \node at (-1,-1.3) {\tiny $u$};
 \node at (0.2,-0.1) {\small $\delta$};
\end{tikzpicture}
\end{array}
-
    \begin{array}{c}
\begin{tikzpicture}[xscale=0.3,yscale=0.5,thick]
 \draw (-1,-1) -- (-1,1);
 \draw (0,0.5) -- (0,1);
 \draw (-1,0) -- (1,0) -- (1,1);
 \node at (0,0.5) {\small $\bullet$};
 \node at (-1,-1.3) {\tiny $u$};
 \node at (0.4,-0.4) {\small $u(\delta)$};
\end{tikzpicture}
\end{array}
=
    \begin{array}{c}
\begin{tikzpicture}[xscale=0.3,yscale=0.5,thick]
 \draw (-1,0) -- (-1,1);
 \draw (0,0) -- (0,1);
 \draw (1,-1) -- (1,1);
 \node at (0,0) {\small $\bullet$};
 \node at (-1,0) {\small $\bullet$};
 \node at (1,-1.3) {\tiny $u$};
\end{tikzpicture}
\end{array}
+
    \begin{array}{c}
\begin{tikzpicture}[xscale=0.3,yscale=0.5,thick]
 \draw (-1,-1) -- (-1,1);
 \draw (0,0.5) -- (0,1);
 \draw (-1,0) -- (1,0) -- (1,1);
 \node at (0,0.5) {\small $\bullet$};
 \node at (-1,-1.3) {\tiny $u$};
 \node at (-2.3,0) {\small $u(\delta)$};
\end{tikzpicture}
\end{array}
-
    \begin{array}{c}
\begin{tikzpicture}[xscale=0.3,yscale=0.5,thick]
 \draw (-1,-1) -- (-1,1);
 \draw (0,0.5) -- (0,1);
 \draw (-1,0) -- (1,0) -- (1,1);
 \node at (0,0.5) {\small $\bullet$};
 \node at (-1,-1.3) {\tiny $u$};
 \node at (0.4,-0.4) {\small $u(\delta)$};
\end{tikzpicture}
\end{array}.
\]
Using this equation and the ``death by pitchfork'' property~\ref{it:death-pitchfork}, we see that when adding dots on two consecutive strands at the bottom of a Jones--Wenzl morphism, the result does not depend on the choice of strands; in particular, we obtain the equality of the lemma.
\end{proof}

\begin{lem}
\label{lem:JW-dots-JW}
Let $\uw \in \hW$, and assume that $\ell(\uw) \geq 3$. Then we have
\[
\begin{array}{c}
\begin{tikzpicture}[yscale=0.3,xscale=0.4,thick]
 \draw (-3,-3) rectangle (3,-1);
 \node at (0,-2) {$\mathrm{JW}_{\uw}$};
  \draw (-3,1) rectangle (3,3);
 \node at (0,2) {$\mathrm{JW}_{\uw}$};
 \draw (-2,-1)  -- (-2,1);
 \draw (2,-1) -- (2,1);
 \draw (0,-1) -- (0,-0.4);
 \draw (0,0.4) -- (0,1);
 \node at (0,-0.4) {$\bullet$};
 \node at (0,0.4) {$\bullet$};
 \node at (-1,0) {\small $\cdots$};
 \node at (1,0) {\small $\cdots$};
 \draw (-2, -4) -- (-2, -3);
 \draw (2,-4) -- (2, -3);
  \draw (-2, 4) -- (-2, 3);
 \draw (2,4) -- (2, 3);
 \node at (0,3.5) {$\cdots$};
 \node at (0,-3.5) {$\cdots$};
\end{tikzpicture}
\end{array}
=0
\]
(where the dots are on the $i$-th strands for some $i \in \{2, \sdots, \ell(\uw)-1\}$).
\end{lem}

\begin{proof}
This fact follows from the relation
\[
\begin{array}{c}
\begin{tikzpicture}[yscale=0.6,xscale=0.8,thick]
\draw (-1,-1) -- (-1,1);
\draw (0,-1) -- (0,1);
\node at (-1,-1.3) {\tiny $u$};
\node at (0,-1.3) {\tiny $u$};
\end{tikzpicture}
\end{array}
=
\begin{array}{c}
\begin{tikzpicture}[yscale=0.6,xscale=0.8,thick]
\draw (-1,-1) -- (-1,1);
\draw (0,-1) -- (0,1);
\draw (-1,0) -- (0,0);
\node at (-1,-1.3) {\tiny $u$};
\node at (0,-1.3) {\tiny $u$};
\node at (-0.5,0.6) {\small $\delta$};
\end{tikzpicture}
\end{array}
-
\begin{array}{c}
\begin{tikzpicture}[yscale=0.6, xscale=0.8,thick]
\draw (-1,-1) -- (-1,1);
\draw (0,-1) -- (0,1);
\draw (-1,0) -- (0,0);
\node at (-1,-1.3) {\tiny $u$};
\node at (0,-1.3) {\tiny $u$};
\node at (-0.5,-0.6) {\small $u(\delta)$};
\end{tikzpicture}
\end{array}
\]
(where $u \in \{s,t\}$ and $\delta \in V^*$ are such that $\langle \delta, \alpha_u^\vee \rangle = 1$) and the ``death by pitchfork'' property.
\end{proof}

\begin{lem}
\label{lem:composition-JW}
Let $\uv_1, \uv_2, \uv_3 \in \hW$, and assume that $\uw=\uv_1 \uv_2 \uv_3 \in \hW$. Then
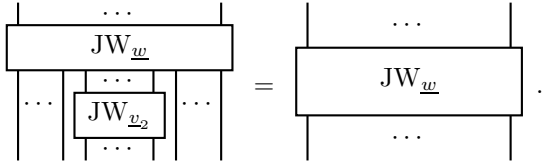
\[
\begin{array}{c}
\begin{tikzpicture}[scale=0.3,thick]
\draw (-5,0.5) rectangle (5,2.5);
\draw (-4.5,2.5) -- (-4.5,3.5);
\draw (4.5,2.5) -- (4.5,3.5);
\node at (0,3) {$\cdots$};
\draw (-2,-0.5) rectangle (2,-2.5);
\draw (-1.5,-2.5) -- (-1.5,-3.5);
\draw (1.5,-2.5) -- (1.5,-3.5);
\node at (0,-3) {$\cdots$};
\draw (-1.5,-0.5) -- (-1.5,0.5);
\draw (1.5,-0.5) -- (1.5,0.5);
\node at (0,0) {$\cdots$};
\draw (-4.5,-3.5) -- (-4.5,0.5);
\draw (-2.5,-3.5) -- (-2.5,0.5);
\node at (-3.5,-1) {$\cdots$};
\draw (4.5,-3.5) -- (4.5,0.5);
\draw (2.5,-3.5) -- (2.5,0.5);
\node at (3.5,-1) {$\cdots$};
\node at (0,1.5) {$\mathrm{JW}_{\uw}$};
\node at (0,-1.5) {$\mathrm{JW}_{\uv_2}$};
\end{tikzpicture}
\end{array}
=
\begin{array}{c}
\begin{tikzpicture}[scale=0.3,thick]
\draw (-5,-1.5) rectangle (5,1.5);
\draw (-4.5,-3.5) -- (-4.5,-1.5);
\draw (4.5,-3.5) -- (4.5,-1.5);
\draw (-4.5,3.5) -- (-4.5,1.5);
\draw (4.5,3.5) -- (4.5,1.5);
\node at (0,0) {$\mathrm{JW}_{\uw}$};
\node at (0,-2.5) {$\cdots$};
\node at (0,2.5) {$\cdots$};
\end{tikzpicture}
\end{array}.
\]
\end{lem}

\begin{proof}
We prove this property by induction on $\ell(\uw)-\ell(\uv_2)$. If this difference is $0$, then the desired formula follows from~\ref{it:JW-projector}. Then the induction step is proved using~\ref{it:JW-recurrence-2} 
and the ``death by pitchfork'' property~\ref{it:death-pitchfork}.
\end{proof}

Let now $\uw, \uv \in \hW$ with $\pi(\uv) < \pi(\uw)$, and write $\uw = (s_1, \sdots, s_r)$. Consider the ``light leaves morphisms'' $B_{\uw} \to B_{\uv} (n)$ constructed in~\cite[\S 6]{ew}, for all $n \in\Z$. (Note that in our present setting no ``rex move'' is required to construct these morphisms, so that the construction does not involve any choice.) These morphisms are parametrized by the sequences $\underline{e}=(e_1, \sdots, e_r)$ of elements of $\{0,1\}$ such that $\pi(\uv) = s_1^{e_1} \cdots s_r^{e_r}$; we write $f_{\underline{e}}$ for the morphism associated with $\underline{e}$. (Note that our assumptions imply that $\underline{e} \neq (1, \sdots, 1)$.)

\begin{lem}
\label{lem:LL-JW}
If $\underline{e}$ is not of the form $(0, \sdots, 0, 1, \sdots, 1)$ (where the sequence of $1$'s is allowed to be empty) or $(0, \sdots, 0, 1, \sdots, 1, 0)$ (where the sequence of $0$'s is allowed to be empty), then $f_{\underline{e}} \circ \mathrm{JW}_{\uw}=0$.
\end{lem}

\begin{proof}
Assume that $\underline{e}$ is as in the statement. Then $\underline{e}$ starts with a (possibly empty) series of $0$'s, followed by a series of $1$'s. Let $i$ be the index of the last element in this series. Then $i \leq r-2$, $e_{i+1}=0$ and $e_{i+2}$ is either $0$ or $1$. Hence $f_{\underline{e}}$ is either of the form
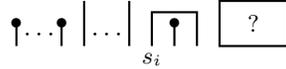
\[
\begin{array}{c}
\begin{tikzpicture}[scale=0.3,thick]
\draw (-2, -1) -- (-2,0);
\node at (-2,0) {\small $\bullet$};
\node at (-.9, -0.5) {$\cdots$};
\draw (0, -1) -- (0,0);
\node at (0,0) {\small $\bullet$};
\draw (1,-1) -- (1,1);
\node at (2.1,-0.5) {$\cdots$};
\draw (3,-1) -- (3,1);
\draw (4,-1) -- (4,0.5) -- (6,0.5) -- (6,-1);
\draw (5,-1)-- (5,0);
\node at (5,0) {\small $\bullet$};
\node at (4,-1.6) {\small $s_i$};
\draw (7,-1) rectangle (10,1);
\node at (8.5,0) {\small $?$};
\end{tikzpicture}
\end{array}
\]
or of the form
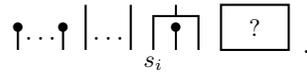
\[
\begin{array}{c}
\begin{tikzpicture}[scale=0.3,thick]
\draw (-2, -1) -- (-2,0);
\node at (-2,0) {\small $\bullet$};
\node at (-.9, -0.5) {$\cdots$};
\draw (0, -1) -- (0,0);
\node at (0,0) {\small $\bullet$};
\draw (1,-1) -- (1,1);
\node at (2.1,-0.5) {$\cdots$};
\draw (3,-1) -- (3,1);
\draw (4,-1) -- (4,0.5) -- (6,0.5) -- (6,-1);
\draw (5,-1)-- (5,0);
\draw (5,0.5) -- (5,1);
\node at (5,0) {\small $\bullet$};
\node at (4,-1.6) {\small $s_i$};
\draw (7,-1) rectangle (10,1);
\node at (8.5,0) {\small $?$};
\end{tikzpicture}
\end{array}.
\]
In any case $f_{\underline{e}}$ contains a pitchfork, hence satisfies $f_{\underline{e}} \circ \mathrm{JW}_{\uw}=0$ by~\ref{it:death-pitchfork}.
\end{proof}

\begin{cor}
\label{cor:dih-indecomp-hom}
If $w,w' \in W$ and $n \in \Z$, then
\[
\Hom_{\Diag(\fh,W)}(B_w, B_{w'}(n)) \neq \{0\} \quad \Rightarrow \quad n\geq0.
\]
Moreover,
\begin{align*}
\Hom_{\Diag(\fh,W)}(B_w, B_{w'}) \neq \{0\} \quad &\Rightarrow \quad w=w'; \\
  \Hom_{\Diag(\fh,W)}(B_w, B_{w'}(1)) \neq \{0\} \quad &\Rightarrow \quad |\ell(w) - \ell(w')| = 1.
\end{align*}
\end{cor}

\begin{proof}
Choose reduced expressions $\uw$ and $\uw'$ for $w$ and $w'$, respectively.  If $w = w'$, assume that $\uw = \uw'$. Consider the ``double leaves basis'' of the $R$-module
\[
\bigoplus_{n \in \Z} \Hom(B_\uw, B_{\uw'} (n)),
\]
see~\cite[Theorem~6.11]{ew}. This basis is parametrized by certain pairs of sequences $(\underline{e}, \underline{e}')$ of elements of $\{0,1\}$, where $\underline{e}$ (resp.~$\underline{e}'$) has length $\ell(\uw)$ (resp.~$\ell(\uw')$).  Explicitly, the elements of the basis are of the form $f^\mathrm{op}_{\underline{e}'} \circ f_{\underline{e}}$, where $f_{\underline{e}}$ is as in Lemma~\ref{lem:LL-JW}, and $f^{\mathrm{op}}_{\underline{e}'}$ is given by the same construction flipped vertically.  Thus, all elements of
\[
\bigoplus_{n \in \Z} \Hom(B_w, B_{w'} (n))
\]
are given by $R$-linear combinations of elements of the form
\begin{equation}\label{eqn:dih-indecomp-hom}
\JW_{\uw'} \circ f^{\mathrm{op}}_{\underline{e}'} \circ f_{\underline{e}} \circ \JW_\uw.
\end{equation}
If $\underline{e} = (1, \sdots, 1)$, then $f_{\underline{e}} = \id_{B_\uw}$, and likewise for $f^{\mathrm{op}}_{\underline{e}'}$.  Thus, the case $\underline{e} = (1, \sdots, 1)$, $\underline{e}' = (1, \sdots, 1)$ can occur only when $\uw = \uw'$; in this case, $f^{\mathrm{op}}_{\underline{e}'} \circ f_{\underline{e}} = \id_{B_{\uw}}$.

Suppose now that $\underline{e}$ is not of the form $(1,\sdots,1)$.  If it is of the form considered in Lemma~\ref{lem:LL-JW}, then $f_{\underline{e}}$ has strictly positive degree, and hence so does $f_{\underline{e}} \circ \JW_\uw$.  Otherwise, that same Lemma tells us that $f_{\underline{e}} \circ \JW_\uw = 0$.  The same reasoning tells us that if $\underline{e}'$ is not of the form $(1,\sdots,1)$, then $\JW_{\uw'} \circ f^{\mathrm{op}}_{\underline{e}'}$ is either zero or of strictly positive degree.

We conclude that~\eqref{eqn:dih-indecomp-hom} is either of strictly positive degree or the identity map of $B_w$, as desired. The special cases of morphisms of degree $0$ or $1$ also follow from these considerations.
\end{proof}

\section{``Inductive'' description of indecomposable objects}
\label{sec:Bw-convolution}

The following result is a diagrammatic version of the description of indecomposable Soergel bimodules for dihedral groups, see~\cite[\S 4]{soergel-bim}.

\begin{lem} \label{lem:BwBu}
 Let $w \in W$ and $u \in \{s, t\}$.
 Then
 \[
  B_w \star B_u \cong
  \begin{cases}
   B_{wu}                       & \text{if $w \in \{ 1, \check u \}$;} \\
   B_{wu} \oplus B_{w\check{u}} & \text{if $wu > w$ and $w \notin \{ 1, \check u \}$;} \\
   B_w(-1) \oplus B_w(1)        & \text{if $wu < w$.}
  \end{cases}
 \]
\end{lem}
\begin{proof}
 The claim is obvious if $w = 1$, and follows from $\JW_{(\check u, u)} = \id_{B_{(\check u, u)}}$ if $w = \check u$.
 
 Suppose $wu < w$. Then there exists $\uw \in \hW$ ending in $u$ so that $\pi(\uw) = w$. Let $\delta \in V^*$ be such that $\langle \alpha_u^\vee, \delta \rangle = 1$, and set $\delta' = -u(\delta)$. Consider the following morphisms:
 \[
  p_1 =
  \begin{array}{c}\begin{tikzpicture}[xscale=0.3, yscale=-0.2,thick]
   \draw (-3,4) -- (-3,5); \node at (-1.4,4.5) {$\cdots$}; \draw (0,4) -- (0,5);
   \draw (-3.5,2) rectangle (0.5,4); \node at (-1.4,3) {\small $\JW_\uw$};
   \draw (-3,-2) -- (-3,2); \node at (-1.4,0) {$\cdots$}; \draw (0,-2) -- (0,2); \node at (0.75,1) {\footnotesize $\delta$}; \draw (0,0) -- (1.5,0) -- (1.5,5);
   \draw (-3.5,-4) rectangle (0.5,-2); \node at (-1.4,-3) {\small $\JW_\uw$};
   \draw (-3,-5) -- (-3,-4); \node at (-1.4,-4.5) {$\cdots$}; \draw (0,-5) -- (0,-4);
  \end{tikzpicture}\end{array}
   : B_{\uw u} \to B_\uw(1),
  \qquad
  p_{-1} =
  \begin{array}{c}\begin{tikzpicture}[xscale=0.3, yscale=-0.2,thick]
   \draw (-3,4) -- (-3,5); \node at (-1.4,4.5) {$\cdots$}; \draw (0,4) -- (0,5);
   \draw (-3.5,2) rectangle (0.5,4); \node at (-1.4,3) {\small $\JW_\uw$};
   \draw (-3,-2) -- (-3,2); \node at (-1.4,0) {$\cdots$}; \draw (0,-2) -- (0,2); \draw (0,0) -- (1.5,0) -- (1.5,5);
   \draw (-3.5,-4) rectangle (0.5,-2); \node at (-1.4,-3) {\small $\JW_\uw$};
   \draw (-3,-5) -- (-3,-4); \node at (-1.4,-4.5) {$\cdots$}; \draw (0,-5) -- (0,-4);
  \end{tikzpicture}\end{array}
   : B_{\uw u} \to B_\uw(-1),
 \]
 \[
  i_1 = 
  \begin{array}{c}\begin{tikzpicture}[xscale=0.3, yscale=0.2,thick]
   \draw (-3,4) -- (-3,5); \node at (-1.4,4.5) {$\cdots$}; \draw (0,4) -- (0,5);
   \draw (-3.5,2) rectangle (0.5,4); \node at (-1.4,3) {\small $\JW_\uw$};
   \draw (-3,-2) -- (-3,2); \node at (-1.4,0) {$\cdots$}; \draw (0,-2) -- (0,2); \draw (0,0) -- (1.5,0) -- (1.5,5);
   \draw (-3.5,-4) rectangle (0.5,-2); \node at (-1.4,-3) {\small $\JW_\uw$};
   \draw (-3,-5) -- (-3,-4); \node at (-1.4,-4.5) {$\cdots$}; \draw (0,-5) -- (0,-4);
  \end{tikzpicture}\end{array}
   : B_\uw(1) \to B_{\uw u},
  \qquad
  i_{-1} = 
  \begin{array}{c}\begin{tikzpicture}[xscale=0.3, yscale=0.2,thick]
   \draw (-3,4) -- (-3,5); \node at (-1.4,4.5) {$\cdots$}; \draw (0,4) -- (0,5);
   \draw (-3.5,2) rectangle (0.5,4); \node at (-1.4,3) {\small $\JW_\uw$};
   \draw (-3,-2) -- (-3,2); \node at (-1.4,0) {$\cdots$}; \draw (0,-2) -- (0,2); \node at (0.75,1) {\footnotesize $\delta'$}; \draw (0,0) -- (1.5,0) -- (1.5,5);
   \draw (-3.5,-4) rectangle (0.5,-2); \node at (-1.4,-3) {\small $\JW_\uw$};
   \draw (-3,-5) -- (-3,-4); \node at (-1.4,-4.5) {$\cdots$}; \draw (0,-5) -- (0,-4);
  \end{tikzpicture}\end{array}
   : B_\uw(-1) \to B_{\uw u}.
 \]
 Since $B_w \star B_u$, resp.~$B_w$, is the image of $\JW_\uw \star \id_{B_u}$, resp.~$\JW_\uw$, the claim follows from the following formulas:
 \[
  \begin{gathered}
   p_{-1} \circ i_1 = 0, \quad p_1 \circ i_{-1} = 0, \quad p_{-1} \circ i_{-1} = \JW_\uw, \quad p_1 \circ i_1 = \JW_\uw, \\
   i_1 \circ p_1 + i_{-1} \circ p_{-1} = \JW_\uw \star \id_{B_u}.
  \end{gathered}
 \]
 
The first formula follows from Lemma~\ref{lemJW-handle}. For the second formula, pull $\delta\delta'$ to the right; since $\partial_u(\delta\delta') = 0$, the claim again follows from Lemma~\ref{lemJW-handle}.
The third and fourth formulas are proved in the same way, using $\alpha_u(\delta) = \alpha_u(\delta') = 1$. The last formula follows from the formula
 \[
  \begin{array}{c}\begin{tikzpicture}[xscale=0.3, yscale=0.2,thick]
   \draw (-3,1) -- (-3,4); \node at (-1.4,2.5) {$\cdots$}; \draw (0,1) -- (0,4); \draw (0,2) -- (1.5,2) -- (1.5,4);
   \draw (-3.5,-1) rectangle (0.5,1); \node at (-1.4,0) {\small $\JW_\uw$};
   \draw (-3,-4) -- (-3,-1); \node at (-1.4,-2.5) {$\cdots$}; \draw (0,-4) -- (0,-1); \node at (0.75,-3) {\footnotesize $\delta$}; \draw (0,-2) -- (1.5,-2) -- (1.5,-4);
  \end{tikzpicture}\end{array}
  +
  \begin{array}{c}\begin{tikzpicture}[xscale=0.3, yscale=0.2,thick]
   \draw (-3,1) -- (-3,4); \node at (-1.4,2.5) {$\cdots$}; \draw (0,1) -- (0,4); \node at (0.75,3) {\footnotesize $\delta'$}; \draw (0,2) -- (1.5,2) -- (1.5,4);
   \draw (-3.5,-1) rectangle (0.5,1); \node at (-1.4,0) {\small $\JW_\uw$};
   \draw (-3,-4) -- (-3,-1); \node at (-1.4,-2.5) {$\cdots$}; \draw (0,-4) -- (0,-1); \draw (0,-2) -- (1.5,-2) -- (1.5,-4);
  \end{tikzpicture}\end{array}
  =
  \begin{array}{c}\begin{tikzpicture}[xscale=0.3, yscale=0.2,thick]
   \draw (-3,1) -- (-3,4); \node at (-1.4,2.5) {$\cdots$}; \draw (0,1) -- (0,4);
   \draw (-3.5,-1) rectangle (0.5,1); \node at (-1.4,0) {\small $\JW_\uw$};
   \draw (-3,-4) -- (-3,-1); \node at (-1.4,-2.5) {$\cdots$}; \draw (0,-4) -- (0,-1);
   \draw (1.5,-4) -- (1.5,4);
  \end{tikzpicture}\end{array},
 \]
 which in turn follows from property~\ref{it:JW'} and the following computation:
 \[
   \begin{array}{c}\begin{tikzpicture}[xscale=0.3, yscale=0.2,thick]
   \draw (-3,-4) rectangle (-1,4); \node at (-2,0) {\small ?};
   \draw (-1,0) -- (0,0); \draw (0,-4) -- (0,4); \draw (0,2) -- (1.5,2) -- (1.5,4); \draw (0,-2) -- (1.5,-2) -- (1.5,-4);
   \node at (0.75,-3) {\footnotesize $\delta$};
  \end{tikzpicture}\end{array}
  +
  \begin{array}{c}\begin{tikzpicture}[xscale=0.3, yscale=0.2,thick]
   \draw (-3,-4) rectangle (-1,4); \node at (-2,0) {\small ?};
   \draw (-1,0) -- (0,0); \draw (0,-4) -- (0,4); \draw (0,2) -- (1.5,2) -- (1.5,4); \draw (0,-2) -- (1.5,-2) -- (1.5,-4);
   \node at (0.75,3) {\footnotesize $\delta'$};
  \end{tikzpicture}\end{array}
  =
  \begin{array}{c}\begin{tikzpicture}[xscale=0.3, yscale=0.2,thick]
   \draw (-3,-4) rectangle (-1,4); \node at (-2,0) {\small ?};
   \draw (-1,0) -- (0,0); \draw (0,-4) -- (0,4); \draw (0,2) -- (1.5,2); \draw (1.5,-4) -- (1.5,4);
   \node at (0.75,-3) {\footnotesize $\delta$};
  \end{tikzpicture}\end{array}
  +
  \begin{array}{c}\begin{tikzpicture}[xscale=0.3, yscale=0.2,thick]
   \draw (-3,-4) rectangle (-1,4); \node at (-2,0) {\small ?};
   \draw (-1,0) -- (0,0); \draw (0,-4) -- (0,4); \draw (0,2) -- (1.5,2); \draw (1.5,-4) -- (1.5,4);
   \node at (0.75,3) {\footnotesize $\delta'$};
  \end{tikzpicture}\end{array}
  =
  \begin{array}{c}\begin{tikzpicture}[xscale=0.3, yscale=0.2,thick]
   \draw (-3,-4) rectangle (-1,4); \node at (-2,0) {\small ?};
   \draw (-1,0) -- (0,0); \draw (0,-4) -- (0,4); \draw (1.5,-4) -- (1.5,4);
  \end{tikzpicture}\end{array}.
 \]

 Finally, suppose that $wu > w$ and $w \notin \{1, \check u \}$. Let $n = \ell(w)$. Then $2 \le n \le m_{st}-1$, so $w$ has a unique reduced expression $\uw \in \hW$; it is of the form $\uw = \ux \check u$, where $\ux \in \hW$ satisfies $\pi(\ux) = w{\check u}$. 
 Consider the following morphisms:
 \[
  p =
  \begin{array}{c}\begin{tikzpicture}[xscale=0.3, yscale=0.2,thick]
\draw (-3,-1.5) rectangle (3,1.5);
\draw (-2.5,-5) -- (-2.5,-1.5);
\draw (2.5,-5) -- (2.5,-1.5);
\draw (-2.5,5) -- (-2.5,1.5);
\draw (2.5,5) -- (2.5,1.5);
\node at (0,0) {$\mathrm{JW}_{\uw u}$};
\node at (0,-3) {$\cdots$};
\node at (0,3) {$\cdots$};
  \end{tikzpicture}\end{array}
   : B_{\uw u} \to B_{\uw u},
  \qquad
  p'= -\frac{[n-1]_{\check u}}{[n]_u}
  \begin{array}{c}\begin{tikzpicture}[xscale=0.3, yscale=-0.2,thick]
   \draw (-3,3) -- (-3,5); \node at (-1.4,4) {$\cdots$}; \draw (0,3) -- (0,5); \draw (1,3) -- (1,5); 
   \draw (-3.5,1) rectangle (1.5,3); \node at (-0.8,2) {\small $\JW_\uw$};
   \draw (-3,-1) -- (-3,1); \node at (-1.4,0) {$\cdots$}; \draw (0,-1) -- (0,1); \draw (1,0.3) -- (1,1); \node at (1,0.3) {\small $\bullet$}; \draw (0,-0.4) -- (2,-0.4) -- (2,5);
   \draw (-3.5,-3) rectangle (0.5,-1); \node at (-1.4,-2) {\small $\JW_\ux$};
   \draw (-3,-5) -- (-3,-3); \node at (-1.4,-4) {$\cdots$}; \draw (0,-5) -- (0,-3);
  \end{tikzpicture}\end{array}
   : B_{\uw u} \to B_\ux,
 \]
 \[
  i =
  \begin{array}{c}\begin{tikzpicture}[xscale=0.3, yscale=-0.2,thick]
\draw (-3,-1.5) rectangle (3,1.5);
\draw (-2.5,-5) -- (-2.5,-1.5);
\draw (2.5,-5) -- (2.5,-1.5);
\draw (-2.5,5) -- (-2.5,1.5);
\draw (2.5,5) -- (2.5,1.5);
\node at (0,0) {$\mathrm{JW}_{\uw u}$};
\node at (0,-3) {$\cdots$};
\node at (0,3) {$\cdots$};
  \end{tikzpicture}\end{array}
   : B_{\uw u} \to B_{\uw u},
  \qquad
  i' =
  \begin{array}{c}\begin{tikzpicture}[xscale=0.3, yscale=0.2,thick]
   \draw (-3,3) -- (-3,5); \node at (-1.4,4) {$\cdots$}; \draw (0,3) -- (0,5); \draw (1,3) -- (1,5); 
   \draw (-3.5,1) rectangle (1.5,3); \node at (-0.8,2) {\small $\JW_\uw$};
   \draw (-3,-1) -- (-3,1); \node at (-1.4,0) {$\cdots$}; \draw (0,-1) -- (0,1); \draw (1,0.3) -- (1,1); \node at (1,0.3) {\small $\bullet$}; \draw (0,-0.4) -- (2,-0.4) -- (2,5);
   \draw (-3.5,-3) rectangle (0.5,-1); \node at (-1.4,-2) {\small $\JW_\ux$};
   \draw (-3,-5) -- (-3,-3); \node at (-1.4,-4) {$\cdots$}; \draw (0,-5) -- (0,-3);
  \end{tikzpicture}\end{array}
   : B_\ux \to B_{\uw u}.
 \]
 As before, it suffices prove the following formulas:
 \[
  \begin{gathered}
   p' \circ i = 0, \quad p \circ i' = 0, \quad p' \circ i' = \JW_\ux, \quad p \circ i = \JW_{\uw u}, \\
   i \circ p + i' \circ p' = \JW_\uw \star \id_{B_u}.
  \end{gathered}
 \]
 The first two formulas follow from Lemma~\ref{lem:composition-JW} and the ``death by pitchfork'' property~\ref{it:death-pitchfork}. The fourth is clear. The fifth 
 follows from property~\ref{it:JW-recurrence-1}.
 Finally, we see that $p' \circ i'$ is equal to
 \begin{multline*}
  -\frac{[n-1]_{\check u}}{[n]_u}
  \begin{array}{c}\begin{tikzpicture}[xscale=0.3, yscale=0.2,thick]
   \draw (-3,5) -- (-3,7); \node at (-1.4,6) {$\cdots$}; \draw (0,5) -- (0,7);
   \draw (-3.5,3) rectangle (0.5,5); \node at (-1.4,4) {\small $\JW_\ux$};
   \draw (-3,1) -- (-3,3); \node at (-1.4,2) {$\cdots$}; \draw (0,1) -- (0,3); \draw (0,2.4) -- (2,2.4);
   \draw (1,1) -- (1,1.7); \node at (1,1.7) {\small $\bullet$};
   \draw (-3.5,-1) rectangle (1.5,1); \node at (-0.8,0) {\small $\JW_\uw$};
   \draw (1,-1.7) -- (1,-1); \node at (1,-1.7) {\small $\bullet$};
   \draw (-3,-3) -- (-3,-1); \node at (-1.4,-2) {$\cdots$}; \draw (0,-3) -- (0,-1); \draw (0,-2.4) -- (2,-2.4) -- (2,2.4);
   \draw (-3.5,-5) rectangle (0.5,-3); \node at (-1.4,-4) {\small $\JW_\ux$};
   \draw (-3,-7) -- (-3,-5); \node at (-1.4,-6) {$\cdots$}; \draw (0,-7) -- (0,-5);
  \end{tikzpicture}\end{array}
 \overset{(\dag)}{=}
  -\frac{[n-1]_{\check u}}{[n]_u}\left(
  \begin{array}{c}\begin{tikzpicture}[xscale=0.3, yscale=0.2,thick]
   \draw (-3,5) -- (-3,7); \node at (-1.4,6) {$\cdots$}; \draw (0,5) -- (0,7);
   \draw (-3.5,3) rectangle (0.5,5); \node at (-1.4,4) {\small $\JW_\ux$};
   \draw (-3,-3) -- (-3,3); \node at (-1.4,0) {$\cdots$}; \draw (0,-3) -- (0,3); \node at (0.75,0) {\small $\alpha_{\check u}$}; \draw (0,-2.4) -- (1.5,-2.4) -- (1.5,2.4) -- (0,2.4);
   \draw (-3.5,-5) rectangle (0.5,-3); \node at (-1.4,-4) {\small $\JW_\ux$};
   \draw (-3,-7) -- (-3,-5); \node at (-1.4,-6) {$\cdots$}; \draw (0,-7) -- (0,-5);
  \end{tikzpicture}\end{array}
  + \frac{[n-2]_u}{[n-1]_{\check u}}
  \begin{array}{c}\begin{tikzpicture}[xscale=0.3, yscale=0.2,thick]
   \draw (-3,7) -- (-3,9); \node at (-1.4,8) {$\cdots$}; \draw (0,7) -- (0,9);
   \draw (-3.5,5) rectangle (0.5,7); \node at (-1.4,6) {\small $\JW_\ux$};
   \draw (-3,3) -- (-3,5); \node at (-1.4,4) {$\cdots$}; \draw (0,3) -- (0,5); \draw (0,4) -- (1.5,4) -- (1.5,0);
   \draw (-3.5,1) rectangle (0.5,3); \node at (-1.4,2) {\small $\JW_\ux$};
   \draw (0,0.4) -- (0,1); \node at (0,0.4) {\small $\bullet$};
   \draw (-3,-1) -- (-3,1); \node at (-1.8,0) {$\cdots$}; \draw (-0.8,-1) -- (-0.8,1);
   \draw (0,-1) -- (0,-0.4); \node at (0,-0.4) {\small $\bullet$};
   \draw (-3.5,-3) rectangle (0.5,-1); \node at (-1.4,-2) {\small $\JW_\ux$};
   \draw (-3,-5) -- (-3,-3); \node at (-1.4,-4) {$\cdots$}; \draw (0,-5) -- (0,-3); \draw (0,-4) -- (1.5,-4) -- (1.5,0);
   \draw (-3.5,-7) rectangle (0.5,-5); \node at (-1.4,-6) {\small $\JW_\ux$};
   \draw (-3,-9) -- (-3,-7); \node at (-1.4,-8) {$\cdots$}; \draw (0,-9) -- (0,-7);
  \end{tikzpicture}\end{array}
  \right) \\
  \overset{(\ast)}{=} \frac{[n-1]_{\check u}[2]_u - [n-2]_u}{[n]_u} \ \JW_\ux.
 \end{multline*}
 Here $(\dag)$ follows from~\ref{it:JW-recurrence-1}, and $(\ast)$ is proved by pulling $\alpha_{\check u}$ to the right in the first diagram, and using Lemma~\ref{lem:JW-dot-trivalent} (and its horizontal reflection) in the second diagram. Finally, using~\eqref{eq:2q-recursion} we see that the coefficient in the last term is $1$, which finishes the proof.
\end{proof}

\section{Morphisms between indecomposables}
\label{sec:llmorphisms}

This section owes much to~\cite{sauerwein} (where, however, only symmetric realizations are considered). Recall the notation $\s{n}$, $\t{n}$, 
$\ss{n}$, $\tt{n}$, 
introduced in Section~\ref{sec:JW}. We will denote the images of these expressions in $W$ by $\sov{n}$\index{nsov@{$\sov{n}$}}, $\tov{n}$\index{ntov@{$\tov{n}$}}, $\ssov{n}$ and $\ttov{n}$ respectively. 


For any $x, y \in W$ with $|\ell(y) - \ell(x)|= 1$ there exists a canonical morphism
$\LL_x^y \in \Hom(B_x, B_y(1))$ obtained by
composing a single dot map with the projection and inclusion
maps. For example, if $x = \sov{n}$ and $y = \tov{n+1}$ then $\LL_x^y$\index{LL@{$\LL_x^y$}}
(resp.~$\LL_y^x$) is given by the composition
\begin{gather*}
  B_{\sov{n}} \hookrightarrow B_{\s{n}}
  \xrightarrow{ \tikz{
      \draw (0,0) -- (0,-.2); \node (a) at (0,-.2) {$\bullet$}; \node
      (b)
      at (0.35,-.1) {\small $\star \id$}} } B_{\t{n+1}} \twoheadrightarrow B_{\tov{n+1}} \\
( \text{resp.} \quad   B_{\tov{n+1}} \hookrightarrow B_{\t{n+1}}
  \xrightarrow{ \tikz{
      \draw (0,0) -- (0,.2); \node (a) at (0,.2) {$\bullet$}; \node
      (b)
      at (0.35,.1) {\small $\star \id$}} }
B_{\s{n}}
\twoheadrightarrow B_{\sov{n}}).
\end{gather*}
Corollary~\ref{cor:dih-indecomp-hom} and its proof imply that we have
\begin{equation}
\label{eqn:Hom-dih-length-diff-1}
\Hom_{\Diag(\fh,W)}(B_x, B_y(1)) = \bk \cdot \LL_x^y.
\end{equation}

These morphisms satisfy a number of relations that we explain in the following lemmas.

\begin{lem}
If $x, y, y', z$ are such that $\ell(z) = \ell(y) + 1 =
  \ell(y') + 1 = \ell(x) + 2$, then
  \begin{gather}
    \label{eq:upup}
    \LL_y^z \circ \LL_x^y =     \LL_{y'}^z \circ \LL_x^{y'}, \\
    \label{eq:downdown}
    \LL_y^x \circ \LL_z^y =     \LL_{y'}^x \circ \LL_z^{y'}.
  \end{gather}
\end{lem}

\begin{proof}
We prove~\eqref{eq:upup}; then~\eqref{eq:downdown} follows by horizontal reflection. Assume first that there exists $u \in \{s,t\}$ and $n \geq 0$ such that
\[
x =\uov{n}, \quad y=\uov{n+1}, \quad y'=\vov{n+1}, \quad z=\uov{n+2},
\]
where $v=\check{u}$.
Then, setting
\[
\underline{x} ={}_u \widehat{n}, \quad \underline{y}={}_u \widehat{n+1}, \quad \underline{y}'={}_v \widehat{n+1}, \quad \underline{z}={}_u \widehat{n+2},
\]
as a morphism from $B_{\underline{x}}$ to $B_{\underline{z}}$ we have
\[
\LL_y^z \circ \LL_x^y = 
\begin{array}{c}
\begin{tikzpicture}[scale=0.3,thick]
\draw (-3,-4) rectangle (0,-2);
\node at (-1.5,-3) {$\mathrm{JW}_{\underline{x}}$};
\draw (-3,-1) rectangle (1,1);
\node at (-1,0) {$\mathrm{JW}_{\underline{y}}$};
\draw (-3,2) rectangle (2,4);
\node at (-0.5,3) {$\mathrm{JW}_{\underline{z}}$};
\draw (-2.5,-5) -- (-2.5,-4);
\draw (-2.5, -2) -- (-2.5,-1);
\draw (-2.5,1) -- (-2.5,2);
\draw (-2.5,4) -- (-2.5,5);
\draw (-0.5,-5) -- (-0.5,-4);
\draw (-0.5,-2) -- (-0.5,-1);
\draw (0.5,1) -- (0.5,2);
\draw (1.5,4) -- (1.5,5);
\draw (0.5,-1.5) -- (0.5,-1);
\node at (0.5,-1.5) {\small $\bullet$};
\draw (1.5,1.5) -- (1.5,2);
\node at (1.5,1.5) {\small $\bullet$};
\node at (-1.4,-4.5) {$\cdots$};
\node at (-1.4,-1.5) {$\cdots$};
\node at (-1,1.5) {$\cdots$};
\node at (-0.5,4.5) {$\cdots$};
\end{tikzpicture}
\end{array}.
\]
Using Lemma~\ref{lem:composition-JW}
we deduce that
\[
\LL_y^z \circ \LL_x^y = 
    \begin{array}{c}
\begin{tikzpicture}[xscale=0.3, yscale=0.2,thick]
 \draw (-3,-2) rectangle (4,2);
 \node at (0.5,0) {$\mathrm{JW}_{\underline{z}}$};
 \draw (-2, -4) -- (-2, -2);
 \draw (1,-4) -- (1,-2);
 \draw (2,-3) -- (2, -2);
 \node at (3,-3) {\small $\bullet$};
  \node at (2,-3) {\small $\bullet$};
 \draw (3,-3) -- (3,-2);
  \draw (-2, 4) -- (-2, 2);
 \draw (3,4) -- (3,2);
 \node at (0.5,3) {$\cdots$};
 \node at (-0.5,-3) {$\cdots$};
\end{tikzpicture}
\end{array}.
\]
Similarly, we have
\[
\LL_{y'}^z \circ \LL_x^{y'} = 
    \begin{array}{c}
\begin{tikzpicture}[xscale=-0.3, yscale=0.2,thick]
 \draw (-3,-2) rectangle (4,2);
 \node at (0.5,0) {$\mathrm{JW}_{\underline{z}}$};
 \draw (-2, -4) -- (-2, -2);
 \draw (1,-4) -- (1,-2);
 \draw (2,-3) -- (2, -2);
 \node at (3,-3) {\small $\bullet$};
  \node at (2,-3) {\small $\bullet$};
 \draw (3,-3) -- (3,-2);
  \draw (-2, 4) -- (-2, 2);
 \draw (3,4) -- (3,2);
 \node at (0.5,3) {$\cdots$};
 \node at (-0.5,-3) {$\cdots$};
\end{tikzpicture}
\end{array}.
\]
Hence the desired equality follows from Lemma~\ref{lem:JW-2dots}.

Next, we have to consider the situation when
\[
x =\uov{n}, \quad y=\uov{n+1}, \quad y'=\vov{n+1}, \quad z=\vov{n+2}
\]
for some $u \in \{s,t\}$ and $n \geq 0$, where again $v=\check{u}$. In this situation, setting $\underline{z}={}_v \widehat{n+2}$ we see that
\[
\LL_y^z \circ \LL_x^y = 
    \begin{array}{c}
\begin{tikzpicture}[xscale=0.3, yscale=0.2,thick]
 \draw (-3,-2) rectangle (4,2);
 \node at (0.5,0) {$\mathrm{JW}_{\underline{z}}$};
 \draw (-2, -3) -- (-2, -2);
 \draw (-1,-4) -- (-1,-2);
 \draw (2,-4) -- (2, -2);
 \node at (3,-3) {\small $\bullet$};
  \node at (-2,-3) {\small $\bullet$};
 \draw (3,-3) -- (3,-2);
  \draw (-2, 4) -- (-2, 2);
 \draw (3,4) -- (3,2);
 \node at (0.5,3) {$\cdots$};
 \node at (0.5,-3) {$\cdots$};
\end{tikzpicture}
\end{array}
=
\LL_{y'}^z \circ \LL_x^{y'},
\]
and the claim follows.
\end{proof}

\begin{lem}
If $n \geq 0$, $u,v \in \{s,t\}$ and if
$x = \uov{n}$, $y = \uov{n+1} = u \cdots v$ then
  \begin{gather}
    \label{eq:updown2}
    \LL_y^x \circ \LL_x^y =     \begin{cases}
\id_{B_x} \star \alpha_v + \frac{[n-1]_{\check v}}{[n]_v}\LL_w^x \circ \LL_x^w & \text{if $n \ge 1$
  and $w := \uov{n-1}$}, \\
\alpha_v & \text{if $n = 0$.}\end{cases}
  \end{gather}
\end{lem}

\begin{proof}
 This relation follows from Lemma~\ref{lem:composition-JW} and property~\ref{it:JW-recurrence-1}.
\end{proof}

\begin{lem}
If $n \geq 0$, $u \in \{s,t\}$ and if
$x = \uov{n}$, $y = \uvov{n+1}$
then
  \begin{gather}
    \label{eq:updown1}
    \LL_y^x \circ \LL_x^y =     \begin{cases}
\alpha_{\check u} \star \id_{B_x} + \frac{[n-1]_u}{[n]_{\check u}}\LL_w^x \circ \LL_x^w & \text{if $n \ge 1$
  and $w := \uvov{n-1}$}, \\
\alpha_{\check u} & \text{if $n = 0$.}\end{cases}
  \end{gather}
\end{lem}

\begin{proof}
 This relation is a direct consequence of Lemma~\ref{lem:composition-JW} and the ``left'' version of the relation in~\ref{it:JW-recurrence-1}.
\end{proof}

\begin{lem}
\label{lem:updown34}
 Let $n \geq 1$, let $u \in \{s,t\}$, and let $v=\check{u}$.
 If $x = \uov{n}, x' = \vov{n}, y = \uov{n+1}, y' = \vov{n+1}$ then
  \begin{gather}
    \label{eq:updown3}
    \LL_y^{x'} \circ \LL_x^y =   \begin{cases}\LL_1^{x'} \circ \LL_x^1  &
    \text{if $n = 1$,}\\
\frac{1}{[n]_u} \LL_w^{x'} \circ \LL_x^w +
    \LL_{w'}^{x'} \circ \LL_x^{w'} & 
    \text{if $n \ge 2$,}  \end{cases}
\\
    \label{eq:updown4}
    \LL_{y'}^{x'} \circ \LL_x^{y'} =   \begin{cases}\LL_1^{x'} \circ \LL_x^1  &
    \text{if $n = 1$,}\\
\LL_w^{x'} \circ \LL_x^w +\frac{1}{[n]_v} 
    \LL_{w'}^{x'} \circ \LL_x^{w'} & 
    \text{if $n \ge 2$}
\end{cases}
  \end{gather}
(where in the cases $n \ge 2$ we set $w := \uov{n-1}, w' := \vov{n-1}$).
\end{lem}

\begin{proof}
We prove~\eqref{eq:updown3}; the proof of~\eqref{eq:updown4} is similar. If $n=1$, the relation is obvious. If $n \geq 2$, what we have to consider is the composition
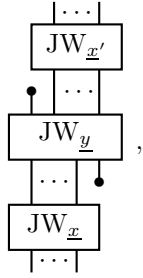
\[
\begin{array}{c}
\begin{tikzpicture}[scale=0.3,thick]
 \draw (-3,-1) rectangle (2,1);
 \node at (-0.5,0) {$\mathrm{JW}_{\underline{y}}$};
 \draw (-3,-5) rectangle (1,-3);
 \node at (-1,-4) {$\mathrm{JW}_{\underline{x}}$};
 \draw (-2,3) rectangle (2,5);
 \node at (0,4) {$\mathrm{JW}_{\underline{x}'}$}; 
 \draw (-2, -6) -- (-2, -5);
 \draw (0, -6) -- (0, -5);
 \draw (-1, 6) -- (-1, 5);
 \draw (1, 6) -- (1, 5);
 \draw (-2,-3) -- (-2,-1);
 \draw (0,-3) -- (0,-1);
 \draw (-1,3) -- (-1,1);
 \draw (-2,1) -- (-2,2);
 \node at (-2,2) {\small $\bullet$};
 \draw (1,3) -- (1,1);
 \draw (1,-1) -- (1,-2);
 \node at (1,-2) {\small $\bullet$};
  \node at (-0.9,-2) {$\cdots$};
  \node at (0.2,2) {$\cdots$};
  \node at (-0.9,-5.5) {$\cdots$};
  \node at (0.1,5.5) {$\cdots$};
\end{tikzpicture}
\end{array},
\]
where $\underline{x}={}_u \widehat{n}$, $\underline{y} = {}_u \widehat{n+1}$ and $\underline{x}'={}_v \widehat{n}$.
Consider the relation stated in~\ref{it:JW-recurrence-2} (for $\uw=\underline{y}$). In this relation, only the first two terms contribute to the composition under consideration (since the other ones contain pitchforks). Therefore we have
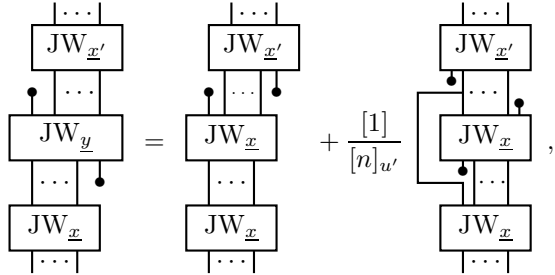
\[
\begin{array}{c}
\begin{tikzpicture}[scale=0.3,thick]
 \draw (-3,-1) rectangle (2,1);
 \node at (-0.5,0) {$\mathrm{JW}_{\underline{y}}$};
 \draw (-3,-5) rectangle (1,-3);
 \node at (-1,-4) {$\mathrm{JW}_{\underline{x}}$};
 \draw (-2,3) rectangle (2,5);
 \node at (0,4) {$\mathrm{JW}_{\underline{x}'}$}; 
 \draw (-2, -6) -- (-2, -5);
 \draw (0, -6) -- (0, -5);
 \draw (-1, 6) -- (-1, 5);
 \draw (1, 6) -- (1, 5);
 \draw (-2,-3) -- (-2,-1);
 \draw (0,-3) -- (0,-1);
 \draw (-1,3) -- (-1,1);
 \draw (-2,1) -- (-2,2);
 \node at (-2,2) {\small $\bullet$};
 \draw (1,3) -- (1,1);
 \draw (1,-1) -- (1,-2);
 \node at (1,-2) {\small $\bullet$};
  \node at (-0.9,-2) {$\cdots$};
  \node at (0.2,2) {$\cdots$};
  \node at (-0.9,-5.5) {$\cdots$};
  \node at (0.1,5.5) {$\cdots$};
\end{tikzpicture}
\end{array}
=
\begin{array}{c}
\begin{tikzpicture}[scale=0.3,thick]
 \draw (-3,-1) rectangle (1,1);
 \node at (-1,0) {$\mathrm{JW}_{\underline{x}}$};
 \draw (-3,-5) rectangle (1,-3);
 \node at (-1,-4) {$\mathrm{JW}_{\underline{x}}$};
 \draw (-2,3) rectangle (2,5);
 \node at (0,4) {$\mathrm{JW}_{\underline{x}'}$}; 
 \draw (-2, -6) -- (-2, -5);
 \draw (0, -6) -- (0, -5);
 \draw (-1, 6) -- (-1, 5);
 \draw (1, 6) -- (1, 5);
 \draw (-2,-3) -- (-2,-1);
 \draw (0,-3) -- (0,-1);
 \draw (-1.3,3) -- (-1.3,1);
 \draw (-2,1) -- (-2,2);
 \node at (-2,2) {\small $\bullet$};
 \draw (1,3) -- (1,2);
 \draw (0.3,1) -- (0.3,3);
 \node at (1,2) {\small $\bullet$};
  \node at (-0.9,-2) {$\cdots$};
  \node at (-0.4,2) {\tiny $\cdots$};
  \node at (-0.9,-5.5) {$\cdots$};
  \node at (0.1,5.5) {$\cdots$};
\end{tikzpicture}
\end{array}
+
\frac{1}{[n]_{u'}}
\begin{array}{c}
\begin{tikzpicture}[scale=0.3,thick]
 \draw (-3,-1) rectangle (1,1);
 \node at (-1,0) {$\mathrm{JW}_{\underline{x}}$};
 \draw (-3,-5) rectangle (1,-3);
 \node at (-1,-4) {$\mathrm{JW}_{\underline{x}}$};
 \draw (-3,3) rectangle (1,5);
 \node at (-1,4) {$\mathrm{JW}_{\underline{x}'}$}; 
 \draw (-2, -6) -- (-2, -5);
 \draw (0, -6) -- (0, -5);
 \draw (-2, 6) -- (-2, 5);
 \draw (0, 6) -- (0, 5);
   \node at (-0.9,-5.5) {$\cdots$};
  \node at (-0.9,5.5) {$\cdots$};
  \draw (-2,-3) -- (-2,-2) -- (-4,-2) -- (-4,2) -- (-2,2);
  \draw (-2,1) -- (-2,3);
  \draw (-2.5,3) -- (-2.5,2.5);
  \node at (-2.5,2.5) {\small $\bullet$};
  \draw (-2,-1) -- (-2,-1.5);
  \node at (-2,-1.5) {\small $\bullet$};
  \draw (-1.5,-3) -- (-1.5,-1);
  \draw (0,-3) -- (0,-1);
  \node at (-0.6,-2) {\small $\cdots$};
  \draw (0.5,1) -- (0.5,1.5);
  \node at (0.5,1.5) {\small $\bullet$};
  \draw (0,1) -- (0,3);
  \node at (-1,2) {$\cdots$};
 \end{tikzpicture}
\end{array},
\]
where $u'$ is the last entry in $\underline{y}$. Using Lemma~\ref{lem:JW-dot-trivalent} and the fact that $[n]_s = [n]_t$ if $n$ is odd we deduce that
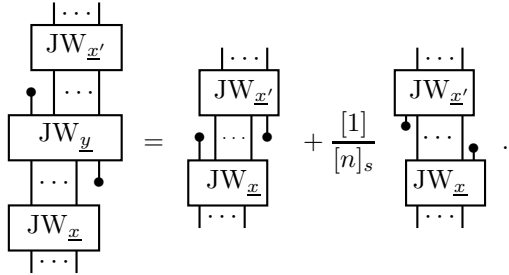
\[
\begin{array}{c}
\begin{tikzpicture}[scale=0.3,thick]
 \draw (-3,-1) rectangle (2,1);
 \node at (-0.5,0) {$\mathrm{JW}_{\underline{y}}$};
 \draw (-3,-5) rectangle (1,-3);
 \node at (-1,-4) {$\mathrm{JW}_{\underline{x}}$};
 \draw (-2,3) rectangle (2,5);
 \node at (0,4) {$\mathrm{JW}_{\underline{x}'}$}; 
 \draw (-2, -6) -- (-2, -5);
 \draw (0, -6) -- (0, -5);
 \draw (-1, 6) -- (-1, 5);
 \draw (1, 6) -- (1, 5);
 \draw (-2,-3) -- (-2,-1);
 \draw (0,-3) -- (0,-1);
 \draw (-1,3) -- (-1,1);
 \draw (-2,1) -- (-2,2);
 \node at (-2,2) {\small $\bullet$};
 \draw (1,3) -- (1,1);
 \draw (1,-1) -- (1,-2);
 \node at (1,-2) {\small $\bullet$};
  \node at (-0.9,-2) {$\cdots$};
  \node at (0.2,2) {$\cdots$};
  \node at (-0.9,-5.5) {$\cdots$};
  \node at (0.1,5.5) {$\cdots$};
\end{tikzpicture}
\end{array}
=
\begin{array}{c}
\begin{tikzpicture}[scale=0.3,thick]
 \draw (-2.5,-1) rectangle (1,1);
 \node at (-0.5,0) {$\mathrm{JW}_{\underline{x}}$};
 \draw (-2,3) rectangle (1.5,5);
 \node at (0,4) {$\mathrm{JW}_{\underline{x}'}$}; 
 \draw (-1, 6) -- (-1, 5);
 \draw (1, 6) -- (1, 5);
 \draw (-2,-2) -- (-2,-1);
 \draw (0,-2) -- (0,-1);
 \draw (-1.3,3) -- (-1.3,1);
 \draw (-2,1) -- (-2,2);
 \node at (-2,2) {\small $\bullet$};
 \draw (1,3) -- (1,2);
 \draw (0.3,1) -- (0.3,3);
 \node at (1,2) {\small $\bullet$};
  \node at (-0.9,-1.5) {$\cdots$};
  \node at (-0.4,2) {\tiny $\cdots$};
  \node at (0.1,5.5) {$\cdots$};
\end{tikzpicture}
\end{array}
+
\frac{1}{[n]_{u}}
\begin{array}{c}
\begin{tikzpicture}[scale=0.3,thick]
 \draw (-2.5,-1) rectangle (1,1);
 \node at (-1,0) {$\mathrm{JW}_{\underline{x}}$};
 \draw (-3,3) rectangle (0.5,5);
 \node at (-1,4) {$\mathrm{JW}_{\underline{x}'}$}; 
 \draw (-2, 6) -- (-2, 5);
 \draw (0, 6) -- (0, 5);
  \node at (-0.9,5.5) {$\cdots$};
  \draw (-2,-2) -- (-2,-1);
  \draw (-2,1) -- (-2,3);
  \draw (-2.5,3) -- (-2.5,2.5);
  \node at (-2.5,2.5) {\small $\bullet$};
  \draw (0,-2) -- (0,-1);
  \node at (-0.9,-1.5) {$\cdots$};
  \draw (0.5,1) -- (0.5,1.5);
  \node at (0.5,1.5) {\small $\bullet$};
  \draw (0,1) -- (0,3);
  \node at (-1,2) {$\cdots$};
 \end{tikzpicture}
\end{array}.
\]
By Lemma~\ref{lem:composition-JW},
the first term identifies with $\mathrm{LL}_{w'}^{x'} \circ \mathrm{LL}_x^{w'}$. And for similar reasons the second term identifies with $\mathrm{LL}_w^{x'} \circ \mathrm{LL}_x^w$, which finishes the proof.
\end{proof}

\begin{rmk} As seen in the proof of Lemma~\ref{lem:updown34}, the $n \ge 2$ cases of \eqref{eq:updown3} and \eqref{eq:updown4}
  can be rewritten as
   \begin{gather}
    \label{eq:updown3p}
    \LL_y^{x'} \circ \LL_x^y =   
\frac{1}{[n]_{u'}} \LL_w^{x'} \circ \LL_x^w +
    \LL_{w'}^{x'} \circ \LL_x^{w'},
\\
    \label{eq:updown4p}
    \LL_{y'}^{x'} \circ \LL_x^{y'} =   
\LL_w^{x'} \circ \LL_x^w +\frac{1}{[n]_{\check u'}} 
    \LL_{w'}^{x'} \circ \LL_x^{w'}
 \end{gather}
where $u' \in \{ s, t\}$ is such that $y = u\cdots u'$.
\end{rmk}

\begin{lem}
If $x = u \cdots v$ with $u,v \in \{ s, t \}$ and $f \in R$ then
  \begin{gather}
    \label{eq:polys}
    \id_{B_x} \star f - x(f) \star \id_{B_x} =
(\LL_w^x \circ \LL_x^w) \star \partial_v(f) - \partial_u(xf) \star (
\LL_{w'}^x \circ \LL_x^{w'})
  \end{gather}
where $w = xv$ and $w' = ux$.
\end{lem}

\begin{proof}
Let $\ux=(u, \sdots, v)$, and
consider the morphism
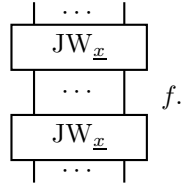
\[
\begin{array}{c}
\begin{tikzpicture}[scale=0.3,thick]
 \draw (-3,-3) rectangle (3,-1);
 \node at (0,-2) {$\mathrm{JW}_{\ux}$};
  \draw (-3,3) rectangle (3,1);
 \node at (0,2) {$\mathrm{JW}_{\ux}$};
 \draw (-2, -4) -- (-2, -3);
  \draw (-2, -1) -- (-2, 1);
 \draw (2,-4) -- (2, -3);
  \draw (2,-1) -- (2, 1);
  \draw (-2, 3) -- (-2, 4);
 \draw (2,3) -- (2, 4);
 \node at (0,3.5) {$\cdots$};
  \node at (0,0) {$\cdots$};
 \node at (0,-3.5) {$\cdots$};
\end{tikzpicture}
\end{array} f.
\]
Pulling $f$ through the middle strands and using Lemma~\ref{lem:JW-dots-JW} we obtain that this morphism equals
\[
\left(
x(f)
\begin{array}{c}
\begin{tikzpicture}[scale=0.3,thick]
 \draw (-3,-3) rectangle (3,-1);
 \node at (0,-2) {$\mathrm{JW}_{\ux}$};
  \draw (-3,3) rectangle (3,1);
 \node at (0,2) {$\mathrm{JW}_{\ux}$};
 \draw (-2, -4) -- (-2, -3);
  \draw (-2, -1) -- (-2, 1);
 \draw (2,-4) -- (2, -3);
  \draw (2,-1) -- (2, 1);
  \draw (-2, 3) -- (-2, 4);
 \draw (2,3) -- (2, 4);
 \node at (0,3.5) {$\cdots$};
  \node at (0,0) {$\cdots$};
 \node at (0,-3.5) {$\cdots$};
\end{tikzpicture}
\end{array}
\right)
+
\left(
\begin{array}{c}
\begin{tikzpicture}[scale=0.3,thick]
 \draw (-3,-3) rectangle (3,-1);
 \node at (0,-2) {$\mathrm{JW}_{\ux}$};
  \draw (-3,3) rectangle (3,1);
 \node at (0,2) {$\mathrm{JW}_{\ux}$};
 \draw (-2, -4) -- (-2, -3);
  \draw (-2, -1) -- (-2, 1);
 \draw (2,-4) -- (2, -3);
  \draw (1.5,-1) -- (1.5, 1);
  \draw (2,-1) -- (2,-0.4);
  \draw (2,0.4) -- (2,1);
  \node at (2,-0.4) {\small $\bullet$};
  \node at (2,0.4) {\small $\bullet$};
  \draw (-2, 3) -- (-2, 4);
 \draw (2,3) -- (2, 4);
 \node at (0,3.5) {$\cdots$};
  \node at (0,0) {$\cdots$};
 \node at (0,-3.5) {$\cdots$};
\end{tikzpicture}
\end{array}
\partial_v(f)
\right)
-
\left(
\partial_u(x(f))
\begin{array}{c}
\begin{tikzpicture}[yscale=0.3,xscale=-0.3,thick]
 \draw (-3,-3) rectangle (3,-1);
 \node at (0,-2) {$\mathrm{JW}_{\ux}$};
  \draw (-3,3) rectangle (3,1);
 \node at (0,2) {$\mathrm{JW}_{\ux}$};
 \draw (-2, -4) -- (-2, -3);
  \draw (-2, -1) -- (-2, 1);
 \draw (2,-4) -- (2, -3);
  \draw (1.5,-1) -- (1.5, 1);
  \draw (2,-1) -- (2,-0.4);
  \draw (2,0.4) -- (2,1);
  \node at (2,-0.4) {\small $\bullet$};
  \node at (2,0.4) {\small $\bullet$};
  \draw (-2, 3) -- (-2, 4);
 \draw (2,3) -- (2, 4);
 \node at (0,3.5) {$\cdots$};
  \node at (0,0) {$\cdots$};
 \node at (0,-3.5) {$\cdots$};
\end{tikzpicture}
\end{array}
\right).
\]
Here the second term identifies with $(\LL_w^x \circ \LL_x^w) \star \partial_v(f)$, and the third one with $\partial_u(xf) \star (\LL_{w'}^x \circ \LL_x^{w'})$
(see Lemma~\ref{lem:composition-JW}).
The claim follows.
\end{proof}

\section{Breaking and two-dot morphisms}

In Chapter~\ref{chap:rouquier} it will be convenient to have abbreviations for certain commonly
occurring morphisms. Given $x \in W$ with $\ell(x) > 1$ we consider
``left and right breaking'' morphisms\index{Brl@{$\Br_l$}}\index{Brr@{$\Br_r$}}
\[
\Br_l := \LL^x_w \circ \LL_x^w \quad \text{and} \quad \Br_r := \LL^x_{w'}
  \circ \LL_x^{w'} 
\]
where $w$ (resp.~$w'$) is the unique element of length $\ell(x) - 1$
with the same right (resp.~left) descent set as $w$. If $\ell(x) = 1$
we set
\[
\Br_l := \LL_1^x \circ \LL_x^1 =: \Br_r.
\]
If $\underline{x}$ is a reduced expression for $x$, then as morphisms from $B_{\underline{x}}$ to itself we have
\[
\Br_l =
\begin{array}{c}
\begin{tikzpicture}[scale=-0.3,thick]
 \draw (-3,-3) rectangle (3,-1);
 \node at (0,-2) {$\mathrm{JW}_{\underline{x}}$};
  \draw (-3,3) rectangle (3,1);
 \node at (0,2) {$\mathrm{JW}_{\underline{x}}$};
 \draw (-2, -4) -- (-2, -3);
  \draw (-2, -1) -- (-2, 1);
 \draw (2,-4) -- (2, -3);
  \draw (1.5,-1) -- (1.5, 1);
  \draw (2,-1) -- (2,-0.4);
  \draw (2,0.4) -- (2,1);
  \node at (2,-0.4) {\small $\bullet$};
  \node at (2,0.4) {\small $\bullet$};
  \draw (-2, 3) -- (-2, 4);
 \draw (2,3) -- (2, 4);
 \node at (0,3.5) {$\cdots$};
  \node at (0,0) {$\cdots$};
 \node at (0,-3.5) {$\cdots$};
\end{tikzpicture}
\end{array} \quad \text{and} \quad \Br_r =
\begin{array}{c}
\begin{tikzpicture}[scale=0.3,thick]
 \draw (-3,-3) rectangle (3,-1);
 \node at (0,-2) {$\mathrm{JW}_{\underline{x}}$};
  \draw (-3,3) rectangle (3,1);
 \node at (0,2) {$\mathrm{JW}_{\underline{x}}$};
 \draw (-2, -4) -- (-2, -3);
  \draw (-2, -1) -- (-2, 1);
 \draw (2,-4) -- (2, -3);
  \draw (1.5,-1) -- (1.5, 1);
  \draw (2,-1) -- (2,-0.4);
  \draw (2,0.4) -- (2,1);
  \node at (2,-0.4) {\small $\bullet$};
  \node at (2,0.4) {\small $\bullet$};
  \draw (-2, 3) -- (-2, 4);
 \draw (2,3) -- (2, 4);
 \node at (0,3.5) {$\cdots$};
  \node at (0,0) {$\cdots$};
 \node at (0,-3.5) {$\cdots$};
\end{tikzpicture}
\end{array}.
\]

Given $x \in W$ with $1 \le \ell(x) < m_{st}$, consider the ``left and
right two-dot'' morphisms\index{TDl@{$\TD_l$}}\index{TDr@{$\TD_r$}}
\[
\TD_l := \LL^x_y \circ \LL_x^y
 \quad \text{and} \quad \TD_r := \LL^x_{y'}
  \circ \LL_x^{y'},
\]
where $y$ (resp.~$y'$) is the unique element of length $\ell(x) + 1$ whose
right (resp.~left) descent set contains that of $x$. If $\underline{x}$ is a reduced expression for $x$, and $\underline{y}$ (resp.~$\underline{y}'$) is the reduced expression obtained by adding one simple reflection to the left (resp.~right) of $\underline{x}$, then as a morphism from $B_{\underline{x}}$ to itself we have
\[
\TD_l = \begin{array}{c}
\begin{tikzpicture}[xscale=0.3,yscale=0.4,thick]
 \draw (-3,-1) rectangle (3,1);
 \node at (0,0) {$\mathrm{JW}_{\underline{y}}$};
 \draw (-2,-1.5) -- (-2,-1);
 \node at (-2,-1.5) {\small $\bullet$};
 \draw (-1.5,-2) -- (-1.5,-1);
  \draw (-1.5,2) -- (-1.5,1);
 \draw (2,-2) -- (2,-1);
 \draw (-2,1.5) -- (-2,1);
 \node at (-2,1.5) {\small $\bullet$};
 \draw (2,2) -- (2,1);
 \node at (0.2,-1.5) {$\cdots$};
 \node at (0.2,1.5) {$\cdots$};
\end{tikzpicture}
\end{array}
\quad \text{and} \quad
\TD_r = \begin{array}{c}
\begin{tikzpicture}[xscale=-0.3,yscale=0.4,thick]
 \draw (-3,-1) rectangle (3,1);
 \node at (0,0) {$\mathrm{JW}_{\underline{y}'}$};
 \draw (-2,-1.5) -- (-2,-1);
 \node at (-2,-1.5) {\small $\bullet$};
 \draw (-1.5,-2) -- (-1.5,-1);
  \draw (-1.5,2) -- (-1.5,1);
 \draw (2,-2) -- (2,-1);
 \draw (-2,1.5) -- (-2,1);
 \node at (-2,1.5) {\small $\bullet$};
 \draw (2,2) -- (2,1);
 \node at (0.2,-1.5) {$\cdots$};
 \node at (0.2,1.5) {$\cdots$};
\end{tikzpicture}
\end{array}.
\]

\chapter{Rouquier complexes}
\label{chap:rouquier}

In~\cite[\S9]{rouquier}, Rouquier obtained a categorification of the braid group via certain chain complexes of Soergel bimodules, often called \emph{minimal Rouquier complexes}.  This chapter is devoted to analogues of Rouquier's construction for finite dihedral groups.  We begin with a biequivariant version (Proposition~\ref{prop:Rouquier-convolution-new}) that is quite similar to Rouquier's own result. But the bulk of the work in this chapter goes toward the free-monodromic version (Theorem~\ref{thm:Rouquier-convolution-fm-new}). An important property of the free-monodromic minimal Rouquier complexes is that they are convolutive.

The assumptions and notations of Chapter~\ref{chap:dihedral} remain in force in this chapter.

\section{Biequivariant minimal Rouquier complexes}
\label{sec:BE-rouquier}

Let $m \in \{0, \sdots, m_{st}\}$ and $u \in \{s,t\}$, and set
\[
v=\check{u}, \quad \uw = {}_u \widehat{m}, \quad w=\pi(\uw) = \uov{m}.
\]
The \emph{minimal Rouquier complex}\index{minimal Rouquier complex!biequivariant} $(\Delta_{\uw, \min}, \delta_c)$ (where ``$c$'' stands for ``classical'') is the following biequivariant complex:\index{standard object!Deltauwmin@$\Delta_{\uw,\min}$}
\[
 \begin{tikzcd}
  & B_1(m) \\
  B_u(m-1) \ar[ur] & \oplus & B_v(m-1) \ar[ul] \\
  B_{uv}(m-2) \ar[u, "-"] \ar[urr] & \oplus & B_{vu}(m-2) \ar[ull] \ar[u, "-"'] \\
  \vdots \ar[u] &  & \vdots \ar[u] \\
  B_{\uov{m-1}}(1) \ar[u, "(-1)^m"] & \oplus & B_{\vov{m-1}}(1) \ar[u, "(-1)^m"'] \\
  & B_w \ar[ul, "(-1)^{m+1}"] \ar[ur]
 \end{tikzcd}
\]
Here, $(\Delta_{\uw, \min})^0 = B_w$ and the components of $\delta_c$ are the $\LL$ maps from~\S\ref{sec:llmorphisms}, with a minus sign on each ``vertical'' component whose source is labeled by an even-length element. If $m$ is even, there is also a minus sign on the component $B_w \leadsto B_{\uov{m-1}}(1)$. The fact that this indeed defines a complex follows from~\eqref{eq:upup}. 

The following result is a diagrammatic version of a result of Rouquier~\cite{rouquier}.

\begin{prop}
\label{prop:Rouquier-convolution-new}
Let $\uw \in \hW$, and suppose $\uw = (s_1, \sdots, s_r)$.  Then there is an isomorphism
\[
\Delta_{s_1} \ustar \Delta_{s_2} \ustar \cdots \ustar \Delta_{s_r} \simto \Delta_{\uw,\min}
\]
in $\BE(\fh,W)$.
\end{prop}

To prove this proposition we need some preparatory lemmas.
For $w, w' \in W$ such that $|\ell(w) - \ell(w')| = 1$, define $A_w^{w'}$ to be the biequivariant complex
\[
 B_w \xrightarrow{\LL_w^{w'}} B_{w'}(1),
\]
where $B_w$ is in position 0.

\begin{lem} \label{lem:FwBu}
 Let $1 \le m \le m_{st} - 1$ and $u \in \{s, t\}$. There exists an isomorphism
 \[
  \Delta_{\huv{m}, \min} \ustar B_u \cong A_{\ovu{m+1}}^{\ovu{m}}.
 \]
\end{lem}

\begin{proof}
 Arrange the terms of $\Delta_{\huv{m}, \min}$ as follows:
 \[
  \begin{tikzcd}[row sep=0]
                      &        & B_1(m) \\
  B_u(m-1)            & \oplus & B_{\check u}(m-1) \\
  B_{\check u u}(m-2) & \oplus & B_{u \check u}(m-2) \\
  \vdots              &        & \vdots \\
  B_{\ovu{m-2}}(2)    & \oplus & B_{\ovuv{m-2}}(2) \\
  B_{\ovu{m-1}}(1)    & \oplus & B_{\ovuv{m-1}}(1) \\
                      &        & B_{\ovuv{m}} \\
  \end{tikzcd}
 \]
 Consider the complex $\Delta_{\huv{m}, \min} \ustar B_u$. By composing with the explicit direct summand decompositions from the proof of Lemma~\ref{lem:BwBu}, we find that this complex is isomorphic to a complex denoted $\Delta_{\huv{m}, \min, u}$ with the following terms, where for clarity we place a prime on terms in the right-hand column:
 \[
  \begin{tikzcd}[row sep=0]
                                                  &        & B'_u(m) \\
   B_u(m) \oplus B_u(m-2)                         & \oplus & B'_{\check u u}(m-1) \\
   B_{\check u u}(m-1) \oplus B_{\check u u}(m-3) & \oplus & B'_{\ovu{3}}(m-2) \oplus B'_u(m-2) \\
   \vdots                                         &        & \vdots \\
   B_{\ovu{m-2}}(3) \oplus B_{\ovu{m-2}}(1)       & \oplus & B'_{\ovu{m-1}}(2) \oplus B'_{\ovu{m-3}}(2) \\
   B_{\ovu{m-1}}(2) \oplus B_{\ovu{m-1}}          & \oplus & B'_{\ovu{m}}(1) \oplus B'_{\ovu{m-2}}(1) \\
                                                  &        & B'_{\ovu{m+1}} \oplus B'_{\ovu{m-1}}
  \end{tikzcd}
 \]
 To prove the desired isomorphism, it suffices to show that the following components of $\Delta_{\huv{m}, \min, u}$ are isomorphisms:
 \begin{enumerate}
  \item
  \label{it:Rouquier-1}
  $B_u(m) \leadsto B'_u(m)$;
  \item 
  \label{it:Rouquier-2}
  $B_{\ovu{m-n}}(n+1) \oplus B'_{\ovu{m-n-1}}(n) \leadsto B_{\ovu{m-n-1}}(n) \oplus B'_{\ovu{m-n}}(n+1)$ for $1 \le n \le m-2$; 
  \item
  \label{it:Rouquier-3}
  $B'_{\ovu{m-1}} \leadsto B_{\ovu{m-1}}$.
 \end{enumerate}
 Indeed, it would then follow by repeated use of Gaussian elimination that $\Delta_{\huv{m}, \min, u}$ is isomorphic to a complex consisting of the remaining terms $B'_{\ovu{m+1}}$ and $B'_{\ovu{m}}(1)$. Moreover, the component $B'_{\ovu{m+1}} \leadsto B'_{\ovu{m}}(1)$ clearly remains unchanged while eliminating the terms in~\eqref{it:Rouquier-1} and~\eqref{it:Rouquier-2}, and even while eliminating the terms in~\eqref{it:Rouquier-3} since $\Hom(B'_{\ovu{m+1}}, B_{\ovu{m-1}}) = 0$ (by Corollary~\ref{cor:dih-indecomp-hom}).
 
 We begin with~\eqref{it:Rouquier-2}. 
 We will show that the component in question has the form
 \[
  \begin{bmatrix}
   0 & -\frac{[m-n]_u}{[m-n-1]_u} \cdot \id_{B_{\ovu{m-n-1}}} \\
   \id_{B_{\ovu{m-n}}} & ?
  \end{bmatrix},
 \]
 and is therefore invertible. 
 We need to compute the following composition:
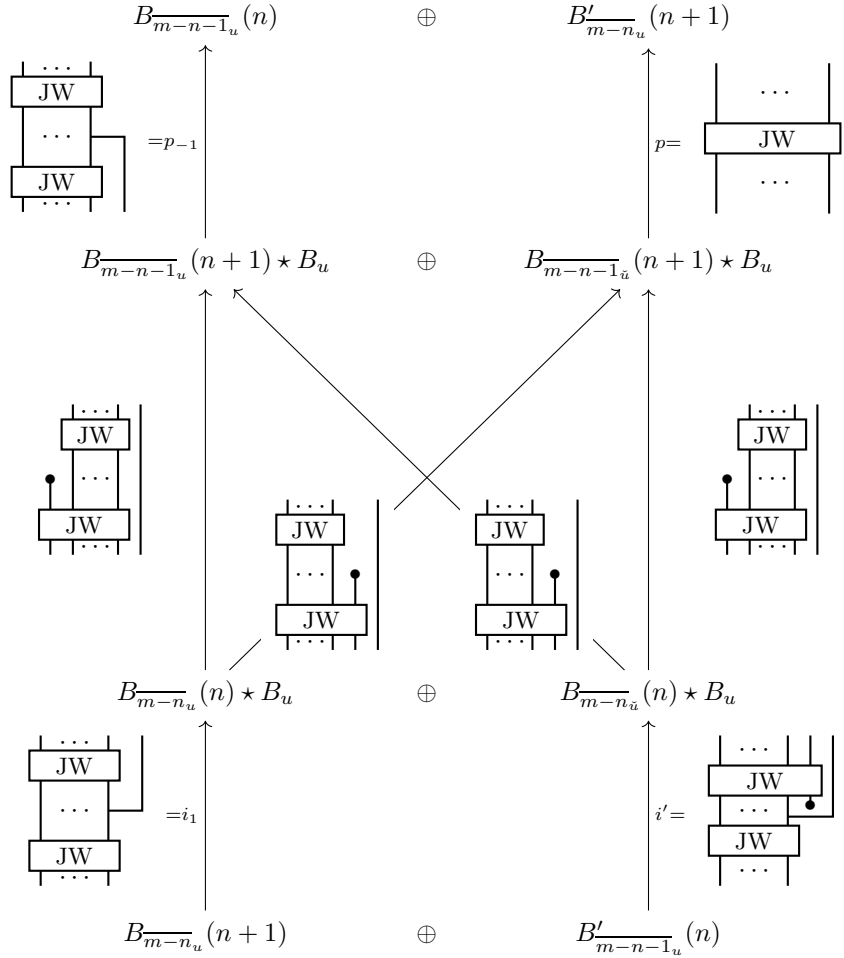
\newsavebox\dLLid
\savebox\dLLid{
 \begin{tikzpicture}[xscale=0.3, yscale=0.2,thick]
  \draw (-2,4) -- (-2,5); \node at (-0.9,4.5) {$\cdots$}; \draw (0,4) -- (0,5);
  \draw (-2.5,2) rectangle (0.5,4);
  \node at (-1,3) {\small $\JW$};
  \draw (-3,-2) -- (-3,0); \node at (-3,0) {\small $\bullet$}; \draw (-2,-2) -- (-2,2); \node at (-0.9,0) {$\cdots$}; \draw (0,-2) -- (0,2);
  \draw (-3.5,-4) rectangle (0.5,-2);
  \node at (-1.5,-3) {\small $\JW$};
  \draw (-3,-5) -- (-3,-4); \draw (-2,-5) -- (-2,-4); \node at (-0.9,-4.5) {$\cdots$}; \draw (0,-5) -- (0,-4);
  \draw (1,-5) -- (1,5);
 \end{tikzpicture}
}

\newsavebox\LLdid
\savebox\LLdid{
 \begin{tikzpicture}[xscale=0.3, yscale=0.2,thick]
  \draw (-3,4) -- (-3,5); \node at (-1.9,4.5) {$\cdots$}; \draw (-1,4) -- (-1,5);
  \draw (-3.5,2) rectangle (-0.5,4);
  \node at (-2,3) {\small $\JW$};
  \draw (-3,-2) -- (-3,2); \node at (-1.9,0) {$\cdots$}; \draw (-1,-2) -- (-1,2); \draw (0,-2) -- (0,0); \node at (0,0) {\small $\bullet$};
  \draw (-3.5,-4) rectangle (0.5,-2);
  \node at (-1.5,-3) {\small $\JW$};
  \draw (-3,-5) -- (-3,-4); \node at (-1.9,-4.5) {$\cdots$}; \draw (-1,-5) -- (-1,-4); \draw (0,-5) -- (0,-4);
  \draw (1,-5) -- (1,5);
 \end{tikzpicture}
}

\newsavebox\BwBuione
\savebox\BwBuione{
 \begin{tikzpicture}[xscale=0.3, yscale=0.2,thick]
  \draw (-3,4) -- (-3,5); \node at (-1.4,4.5) {$\cdots$}; \draw (0,4) -- (0,5);
  \draw (-3.5,2) rectangle (0.5,4);
  \node at (-1.5,3) {\small $\JW$};
  \draw (-3,-2) -- (-3,2); \node at (-1.4,0) {$\cdots$}; \draw (0,-2) -- (0,2); \draw (0,0) -- (1.5,0) -- (1.5,5);
  \draw (-3.5,-4) rectangle (0.5,-2);
  \node at (-1.5,-3) {\small $\JW$};
  \draw (-3,-5) -- (-3,-4); \node at (-1.4,-4.5) {$\cdots$}; \draw (0,-5) -- (0,-4);
 \end{tikzpicture}
}

\newsavebox\BwBuiprime
\savebox\BwBuiprime{
 \begin{tikzpicture}[xscale=0.3, yscale=0.2,thick]
  \draw (-3,3) -- (-3,5); \node at (-1.4,4) {$\cdots$}; \draw (0,3) -- (0,5); \draw (1,3) -- (1,5); 
  \draw (-3.5,1) rectangle (1.5,3);
  \node at (-1,2) {\small $\JW$};
  \draw (-3,-1) -- (-3,1); \node at (-1.4,0) {$\cdots$}; \draw (0,-1) -- (0,1); \draw (1,0.3) -- (1,1); \node at (1,0.3) {\small $\bullet$}; \draw (0,-0.4) -- (2,-0.4) -- (2,5);
  \draw (-3.5,-3) rectangle (0.5,-1);
  \node at (-1.5,-2) {\small $\JW$};
  \draw (-3,-5) -- (-3,-3); \node at (-1.4,-4) {$\cdots$}; \draw (0,-5) -- (0,-3);
 \end{tikzpicture}
}

\newsavebox\BwBupminusone
\savebox\BwBupminusone{
 \begin{tikzpicture}[xscale=0.3, yscale=-0.2,thick]
  \draw (-3,4) -- (-3,5); \node at (-1.4,4.5) {$\cdots$}; \draw (0,4) -- (0,5);
  \draw (-3.5,2) rectangle (0.5,4);
  \node at (-1.5,3) {\small $\JW$};
  \draw (-3,-2) -- (-3,2); \node at (-1.4,0) {$\cdots$}; \draw (0,-2) -- (0,2); \draw (0,0) -- (1.5,0) -- (1.5,5);
  \draw (-3.5,-4) rectangle (0.5,-2);
  \node at (-1.5,-3) {\small $\JW$};
  \draw (-3,-5) -- (-3,-4); \node at (-1.4,-4.5) {$\cdots$}; \draw (0,-5) -- (0,-4);
 \end{tikzpicture}
}

\newsavebox\BwBup
\savebox\BwBup{
 \begin{tikzpicture}[xscale=0.3, yscale=0.2,thick]
\draw (-3.5,-1) rectangle (2.5,1);
\draw (-3,-5) -- (-3,-1);
\draw (-3,1) -- (-3,5);
\draw (2,-5) -- (2,-1);
\draw (2,1) -- (2,5);
 \node at (-0.3,-3) {$\cdots$};
 \node at (-0.3,3) {$\cdots$};
 \node at (-0.3,0) {\small $\JW$};
 \end{tikzpicture}
}

\[
 \begin{tikzcd}[row sep=1in]
  B_{\ovu{m-n-1}}(n) & \oplus & B'_{\ovu{m-n}}(n+1) \\
  B_{\ovu{m-n-1}}(n+1) \star B_u \ar[u, "\begin{array}{c}\usebox\BwBupminusone\end{array} = p_{-1}"] & \oplus & B_{\ovuv{m-n-1}}(n+1) \star B_u \ar[u, "p = \begin{array}{c}\usebox\BwBup\end{array}"']\\
  & \\
  B_{\ovu{m-n}}(n) \star B_u \ar[uu, "\usebox\dLLid\phantom{===}"] \ar[uurr, "\usebox\LLdid" description, near start] & \oplus & B_{\ovuv{m-n}}(n) \star B_u \ar[uull, "\usebox\LLdid" description, near start] \ar[uu, "\phantom{===}\usebox\dLLid"'] \\
  B_{\ovu{m-n}}(n+1) \ar[u, "\begin{array}{c}\usebox\BwBuione\end{array} = i_1"] & \oplus & B'_{\ovu{m-n-1}}(n) \ar[u, "i' = \begin{array}{c}\usebox\BwBuiprime\end{array}"']
 \end{tikzcd}
\]
 The component $B_{\ovu{m-n}}(n+1) \leadsto B_{\ovu{m-n-1}}(n)$ vanishes by the same argument as in the proof of Lemma~\ref{lemJW-handle}.
 The component $B_{\ovu{m-n}}(n+1) \leadsto B'_{\ovu{m-n}}(n+1)$ is the identity map by Lemma~\ref{lem:JW-dot-trivalent} and Lemma~\ref{lem:composition-JW}. A calculation very similar the one we did for $p' \circ i'$ in the proof of Lemma~\ref{lem:BwBu} shows that the component $B'_{\ovu{m-n-1}}(n) \leadsto B_{\ovu{m-n-1}}(n)$ equals
 \[
  \begin{array}{c}\begin{tikzpicture}[xscale=0.3, yscale=0.2,thick]
   \draw (-3,7) -- (-3,9); \node at (-1.4,8) {$\cdots$}; \draw (0,7) -- (0,9);
   \draw (-3.5,5) rectangle (0.5,7);
   \node at (-1.5,6) {\small $\JW$};
   \draw (-3,3) -- (-3,5); \node at (-1.4,4) {$\cdots$}; \draw (0,3) -- (0,5); \draw (1,3) -- (1,3.7); \node at (1,3.7) {\small $\bullet$}; \draw (0,4.3) -- (2,4.3);
   \draw (-3.5,1) rectangle (1.5,3);
   \node at (-1,2) {\small $\JW$};
   \draw (-3,-1) -- (-3,1); \node at (-1.4,0) {$\cdots$}; \draw (0,-1) -- (0,1); \draw (1,0.3) -- (1,1); \node at (1,0.3) {\small $\bullet$}; \draw (0,-0.4) -- (2,-0.4) -- (2,4.3);
   \draw (-3.5,-3) rectangle (0.5,-1);
   \node at (-1.5,-2) {\small $\JW$};
   \draw (-3,-5) -- (-3,-3); \node at (-1.4,-4) {$\cdots$}; \draw (0,-5) -- (0,-3);
  \end{tikzpicture}\end{array}
  = -\frac{[m-n]_u}{[m-n-1]_u} \cdot \id_{B_{\ovu{m-n-1}}}.
 \]
 This proves the claim.
 
 Similarly, the component~\eqref{it:Rouquier-1}, resp.~\eqref{it:Rouquier-3}, may be viewed as a special case of the component $p \circ (\LL \star \id) \circ i_1$, resp.~$p_{-1} \circ (\LL \star \id) \circ i'$, already calculated when treating~\eqref{it:Rouquier-2}.
\end{proof}

\begin{lem}
\label{lem:conv-Rouquier-BE}
 Let $1 \le m \le m_{st}-1$ and $u \in \{s, t\}$. Then there exists an isomorphism
 \[
  \Delta_{\huv{m}, \min} \ustar \Delta_u \cong \Delta_{\widehat{m+1}_u, \min}.
 \]
\end{lem}
\begin{proof}
 We use induction on $m$. The case $m = 1$ is clear. Let $m > 1$. Note that
 \[
  \Delta_{\widehat{m+1}_u, \min}[1] \cong \cone(A_{\ovu{m+1}}^{\ovu{m}} \xrightarrow{\lambda} \Delta_{\huv{m}, \min}(1)),
 \]
 where $\lambda$ is represented by the obvious chain map (with components the $\LL$ maps with appropriate signs). Meanwhile,
 \begin{multline*}
  (\Delta_{\huv{m}, \min} \ustar \Delta_u)[1] \cong \cone(\Delta_{\huv{m}, \min} \ustar B_u \to \Delta_{\huv{m}, \min}(1)) \\
  \overset{\text{Lemma~\ref{lem:FwBu}}}{\cong} \cone(A_{\ovu{m+1}}^{\ovu{m}} \xrightarrow{\lambda'} \Delta_{\huv{m}, \min}(1))
 \end{multline*}
 for some $\lambda'$. By induction, $\Delta_{\huv{m}, \min}$ is isomorphic to a product of copies of $\Delta_s$ and $\Delta_t$; since both are invertible (by Lemma~\ref{lem:convolution-DN-new}), it follows that $\Delta_{\huv{m}, \min} \ustar \Delta_u$ is indecomposable. Hence $\lambda' \neq 0$. Moreover, it follows easily from~\eqref{eqn:Hom-dih-length-diff-1} that
 \[
  \Hom(A_{\ovu{m+1}}^{\ovu{m}}, \Delta_{\huv{m}, \min}(1)) = \bk \cdot \lambda,
 \]
 so $\lambda'$ is a nonzero multiple of $\lambda$. The claim follows.
\end{proof}

\begin{proof}[Proof of Proposition~{\rm \ref{prop:Rouquier-convolution-new}}]
The claim follows from Lemma~\ref{lem:conv-Rouquier-BE} and induction on $r$.
\end{proof}

We finish this section with the following consequence of Proposition~\ref{prop:Rouquier-convolution-new}.

\begin{cor}
\label{cor:End-Delta-min}
Let $\uw \in \hW$. Then we have
\[
\gEnd_\RE \bigl( \ForBERE(\Delta_{\uw, \min}) \bigr) \cong \bk.
\]
\end{cor}
\begin{proof}
This follows from Proposition~\ref{prop:Rouquier-convolution-new} and Lemma~\ref{lem:End-Delta}.
\end{proof}

\section{Lifting minimal Rouquier complexes}

In this section we fix $m$, $u$, $v$, $\uw$, $w$ as in the beginning of~\S\ref{sec:BE-rouquier}. Our goal is to explicitly describe a convolutive free-monodromic lift $(\tD_{\uw, \min}, \tdelta_\uw)$\index{minimal Rouquier complex!free-monodromic}\index{standard object!Deltatuwmin@$\tD_{\uw,\min}$} of $\Delta_{\uw, \min}$. First, define the following auxiliary elements $\delta'$, $\delta''_x$ (for $x \in W$), and $\delta_\bot$ of $\uEnd^\rhd_\FM(\Delta_{\uw, \min})^{1}_{0}$.  Each component of these elements is of the form $\pm(\text{an $\LL$ map}) \otimes (\text{an element of $V$})$.  To reduce clutter, we suppress the $\LL$ maps and write down only the sign and the element of $V$.  

In addition, here, and in most calculations in this chapter, we suppress internal shifts (i.e., we write ``$B_{\uov{n}}$'' instead of ``$B_{\uov{n}}(n)$''), but they are always present.
\begin{multline*}
 \delta'
 = 
 \begin{tikzcd}[ampersand replacement=\&,row sep=scriptsize]
       \& B_1 \ar[dl, "\otimes \varpi_u"'] \ar[dr, "\otimes \varpi_v"] \\
 B_u \& \& B_v
 \end{tikzcd} \\
 + \sum_{\substack{1 \le n \le m-2 \\ n \text{ odd}}} \left(
 \begin{tikzcd}[ampersand replacement=\&]
  B_{\uov{n}} \ar[d, "- \otimes \varpi_v^{u,n+1}"'] \ar[dr, "\otimes \varpi_{u,n+1}^v" near end] \\
  B_{\uov{n+1}} \& B_{\vov{n+1}}
 \end{tikzcd}
 +
 \begin{tikzcd}[ampersand replacement=\&]
  \& B_{\vov{n}} \ar[dl, "\otimes \varpi_{v,n+1}^u"' near end] \ar[d, "- \otimes \varpi_u^{v,n+1}"] \\
  B_{\uov{n+1}} \& B_{\vov{n+1}}
 \end{tikzcd} 
 \right) \\
 + \sum_{\substack{1 \le n \le m-2 \\ n \text{ even}}} \left(
 \begin{tikzcd}[ampersand replacement=\&]
  B_{\uov{n}} \ar[d, "\otimes \varpi_u^{v,n+1}"'] \ar[dr, "\otimes \varpi_{v,n+1}^u" near end] \\
  B_{\uov{n+1}} \& B_{\vov{n+1}}
 \end{tikzcd}
 +
 \begin{tikzcd}[ampersand replacement=\&]
  \& B_{\vov{n}} \ar[dl, "\otimes \varpi_{u,n+1}^v"' near end] \ar[d, "\otimes \varpi_v^{u,n+1}"] \\
  B_{\uov{n+1}} \& B_{\vov{n+1}}
 \end{tikzcd}
 \right),
\end{multline*}
\begin{multline*}
 \delta''_x =
 \begin{tikzcd}[ampersand replacement=\&]
       \& B_1 \ar[dl, "\otimes x^{-1}(\varpi_u)"'] \ar[dr, "\otimes x^{-1}(\varpi_v)"] \\
 B_u \& \& B_v
 \end{tikzcd} \\
 + \sum_{1 \le n \le m-2} \left(
 \begin{tikzcd}[ampersand replacement=\&]
  B_{\uov{n}} \ar[d, "(-1)^n \otimes x^{-1}(\varpi_{u,n+1}^v)"'] \ar[dr, "\otimes x^{-1}(\varpi_v^{u,n+1})" near end] \\
  B_{\uov{n+1}} \& B_{\vov{n+1}}
 \end{tikzcd}
 \right.
 \\
 +
 \left.
 \begin{tikzcd}[ampersand replacement=\&]
  \& B_{\vov{n}} \ar[dl, "\otimes x^{-1}(\varpi_u^{v,n+1})"' near end] \ar[d, "(-1)^n \otimes x^{-1}(\varpi_{v,n+1}^u)"] \\
  B_{\uov{n+1}} \& B_{\vov{n+1}}
 \end{tikzcd}
 \right),
\end{multline*}
\[
 \delta_\bot =
 \begin{cases}
 \begin{tikzcd}
   B_{\uov{m-1}} \ar[dr, "\otimes \alpha_u^\vee"' near start] & & B_{\vov{m-1}} \ar[dl, "\otimes (-w^{-1}(\alpha_u^\vee))" near start] \\
   & B_w
  \end{tikzcd}
  &\text{if $m$ is odd;} \\
 \begin{tikzcd}
  B_{\uov{m-1}} \ar[dr, "- \otimes \alpha_v^\vee"' near start] & & B_{\vov{m-1}} \ar[dl, "\otimes (-w^{-1}(\alpha_u^\vee))" near start] \\
   & B_w
  \end{tikzcd}
  &\text{if $m$ is even.}
 \end{cases}
\]
Note that in both $\delta'$ and $\delta''_x$, the components with source $B_1$ can be considered as special cases (when $n = 0$) of components with source either $B_{\uov{n}}$ or $B_{\vov{n}}$. 

Now, set
\[
 \delta_{-1} = \delta' - \delta''_w + \delta_\bot, \qquad \tdelta_\uw = \delta_c + \theta_w + \delta_{-1},
\]
where $\theta_x$ is the sum of self-loop components
\[
\sum_i x(e_i) \otimes \id \otimes \check e_i = \sum_i e_i \otimes \id \otimes x^{-1}(\check e_i)
\]
at every term, for any $x \in W$.

\begin{prop}
\label{prop:free-monodromic-minimal-Rouquier-new}
 With the definitions above, $(\Delta_{\uw, \min}, \tdelta_\uw)$ is a (convolutive) free-monodromic complex, i.e.~$\tdelta_\uw$ satisfies
 \begin{equation}
 \label{eq:tdelta-differential-new}
  \tdelta_\uw \circ \tdelta_\uw + \kappa(\tdelta_\uw) = \Theta.
 \end{equation}
 This object will be denoted $\tD_{\uw,\min}$; it satisfies
\begin{equation}
\label{eqn:Rouquier-For-new}
\ForFMLM(\tD_{\uw,\min}) \cong \ForBELM(\Delta_{\uw,\min}).
\end{equation}
\end{prop}

The proof of this proposition occupies the rest of this section. It is very technical, but none of its details will be needed later on. Of course,~\eqref{eqn:Rouquier-For-new} is obvious once we have proved that~$\tD_{\uw,\min}$ makes sense as a free-monodromic complex. So, we only have to check~\eqref{eq:tdelta-differential-new}.

We define the \emph{chain degree} of an element in $\uEnd_\FM(\tD_{\uw,\min})$ to be
\begin{multline*}
(\text{cohomological degree of the $\uEnd_\BE$ component})\\ - (\text{internal degree of the $\uEnd_\BE$ component}).
\end{multline*}
For example, $\delta_c$, $\theta_w$, and $\delta_{-1}$ have chain degrees $1$, $0$, and $-1$ respectively.  

We will look at~\eqref{eq:tdelta-differential-new} in each chain degree $d$ separately. For $|d| > 2$, \eqref{eq:tdelta-differential-new} reads $0 = 0$. For $d = 2$, it reads $\delta_c \circ \delta_c = 0$, which holds since $(\Delta_{\uw, \min}, \delta_c)$ is a complex. For $d = 1$, it reads $\theta_w \circ \delta_c + \delta_c \circ \theta_w = 0$, which holds since $\theta_w$ anticommutes with elements of odd cohomological degree, by~\eqref{eqn:Theta-center}. The relation holds for $d = -1$ for the same reason.

For $d \in \{0, -2\}$, we will first consider the ``generic components'' of $\tdelta_\uw \circ \tdelta_\uw + \kappa(\tdelta_\uw)$, i.e.,~those not involving $\delta_\bot$. For generic components, \eqref{eq:tdelta-differential-new} will follow from separate relations involving $\delta'$ and $\delta''_x$, for any $x \in W$. Since these components are symmetric in $u$ and $v$, it suffices to consider those with source $B_{\uov{n}}$.

\subsection{$d = 0$, generic components}

In this case, the generic components of $\tdelta_\uw \circ \tdelta_\uw + \kappa(\tdelta_\uw)$ are those with source $B_{\uov{n}}$ or $B_{\vov{n}}$, $0 \le n \le m-2$. Here, \eqref{eq:tdelta-differential-new} follows from the two lemmas below (with $x = w$), which address self-loop (e.g.~$B_{\uov{n}} \leadsto B_{\uov{n}}$) and cross-term (e.g.~ $B_{\uov{n}} \leadsto B_{\vov{n}}$) components, respectively.

\begin{lem}
\label{lem:d=0-1-new}
 We have
 \begin{align*}
  \delta' \circ \delta_c + \delta_c \circ \delta' &= \Theta \\
  \delta''_x \circ \delta_c + \delta_c \circ \delta''_x &= \kappa(\theta_x), \mbox{ all } x \in W
 \end{align*}
 for components $B_{\uov{n}} \leadsto B_{\uov{n}}$ and $B_{\vov{n}} \leadsto B_{\vov{n}}$, $0 \le n \le m-2$.
\end{lem}

\begin{proof}
 We compute each side of the desired equalities, in the case of the components $B_{\uov{n}} \leadsto B_{\uov{n}}$ with $n>0$. (As mentioned above, the case $n=0$ and the case of components $B_{\vov{n}} \leadsto B_{\vov{n}}$ are similar.)
 
 \emph{First formula, $n$ odd}: The relevant components (of $\delta_c$ and $\delta'$) are as follows:
 \[
  \begin{tikzcd}
   B_{\uov{n-1}} \ar[d, "\otimes \varpi_u^{v,n}"', bend right=20] & B_{\vov{n-1}} \ar[dl, "\otimes \varpi_{u,n}^v" near start, bend left=20] \\
   B_{\uov{n}} \ar[u] \ar[ur] \ar[d, "\otimes \varpi_v^{u,n+1}"', bend right=20] \ar[dr, "\otimes \varpi_{u,n+1}^v" near end, bend left=20] \\
   B_{\uov{n+1}} \ar[u] & B_{\vov{n+1}}. \ar[ul]
  \end{tikzcd}
 \]
 (The signs are omitted here, as they always cancel.) So on $B_{\uov{n}}$ we have
 \[
  \delta' \circ \delta_c + \delta_c \circ \delta' = \Br_r \otimes \varpi_u^{v,n} + \Br_l \otimes \varpi_{u,n}^v + \TD_r \otimes \varpi_v^{u,n+1} + \TD_l \otimes \varpi_{u,n+1}^v.
 \]
 By \eqref{eq:updown2}--\eqref{eq:updown1}, this equals
 \begin{multline*}
  (\id \star \alpha_v) \otimes \varpi_v^{u,n+1} + (\alpha_v \star \id) \otimes \varpi_{u,n+1}^v \\
  + \Br_r \otimes \left( \frac{[n-1]_u}{[n]} \varpi_v^{u,n+1} + \varpi_u^{v,n} \right) + \Br_l \otimes \left( \frac{[n-1]_u}{[n]} \varpi_{u,n+1}^v + \varpi_{u,n}^v \right).
 \end{multline*}
 Meanwhile, using the basis $\{ \alpha_v, \alpha_{u,n+1} \}$,
 \begin{multline*}
  \Theta_{B_{\uov{n}}} = (\id \star \alpha_v) \otimes \varpi_v^{u,n+1}
  + (\id \star \alpha_{u,n+1}) \otimes \varpi_{u,n+1}^v \\
  = (\id \star \alpha_v) \otimes \varpi_v^{u,n+1} + (\alpha_v \star \id) \otimes \varpi_{u,n+1}^v \\
  + \Br_r \otimes \la \alpha_u^\vee, \alpha_{u,n+1} \ra \varpi_{u,n+1}^v + \Br_l \otimes [2]_u \varpi_{u,n+1}^v.
 \end{multline*}
 Here, the second equality follows from \eqref{eq:polys}. We will see below that the two sides agree.
 
 \emph{First formula, $n$ even}: The relevant components (of $\delta_c$ and $\delta'$) are as follows:
 \[
  \begin{tikzcd}
   B_{\uov{n-1}} \ar[d, "\otimes \varpi_v^{u,n}"', bend right=20] & B_{\vov{n-1}} \ar[dl, "\otimes \varpi_{v,n}^u" near start, bend left=20] \\
   B_{\uov{n}} \ar[u] \ar[ur] \ar[d, "\otimes \varpi_u^{v,n+1}"', bend right=20] \ar[dr, "\otimes \varpi_{v,n+1}^u" near end, bend left=20] \\
   B_{\uov{n+1}} \ar[u] & B_{\vov{n+1}}. \ar[ul]
  \end{tikzcd}
 \]
 Again using \eqref{eq:updown2}--\eqref{eq:updown1} we see that on $B_{\uov{n}}$ we have
 \begin{multline*}
  \delta' \circ \delta_c + \delta_c \circ \delta' = \Br_r \otimes \varpi_v^{u,n} + \Br_l \otimes \varpi_{v,n}^u + \TD_r \otimes \varpi_u^{v,n+1} + \TD_l \otimes \varpi_{v,n+1}^u \\
  = (\id \star \alpha_u) \otimes \varpi_u^{v,n+1} + (\alpha_v \star \id) \otimes \varpi_{v,n+1}^u \\
  + \Br_r \otimes \left( \frac{[n-1]}{[n]_u} \varpi_u^{v,n+1} + \varpi_v^{u,n} \right) + \Br_l \otimes \left( \frac{[n-1]}{[n]_v} \varpi_{v,n+1}^u + \varpi_{v,n}^u \right).
 \end{multline*}
 Meanwhile, using the basis $\{ \alpha_u, \alpha_{v,n+1} \}$,
 \begin{multline*}
  \Theta_{B_{\uov{n}}} = (\id \star \alpha_u) \otimes \varpi_u^{v,n+1} + (\id \star \alpha_{v,n+1}) \otimes \varpi_{v,n+1}^u \\
  = (\id \star \alpha_u) \otimes \varpi_u^{v,n+1} + (\alpha_v \star \id) \otimes \varpi_{v,n+1}^u \\
  + \Br_r \otimes \la \alpha_v^\vee, \alpha_{v,n+1} \ra \varpi_{v,n+1}^u + \Br_l \otimes [2]_u \varpi_{v,n+1}^u.
 \end{multline*}
 
 In either parity, the main terms (involving $\id \star \alpha_u$ or $\alpha_v \star \id$) agree on the two sides, and the desired equality of lower-order terms (involving $\Br_l$ or $\Br_r$) follows from formulas involving $2$-colored quantum numbers. Up to switching $u$ and $v$, they can be summarized as follows:
 \begin{equation} \label{eq:diagonal-2-q-identity-1-new}
  \la \alpha_u^\vee, \alpha_{u,n+1} \ra \varpi_{u,n+1}^v
  =
   \frac{[n-1]_u}{[n]_v} \varpi_v^{u,n+1} + \varpi_u^{v,n}.
 \end{equation}
 \begin{equation} \label{eq:diagonal-2-q-identity-2-new}
  \frac{[n-1]_u}{[n]_u} \varpi_{u,n+1}^v + \varpi_{u,n}^v
  =
  \begin{cases}
   [2]_u \varpi_{u,n+1}^v &\text{if $n$ is odd;} \\
   [2]_v \varpi_{u,n+1}^v &\text{if $n$ is even.}
  \end{cases}
 \end{equation}
 By \eqref{eq:alpha-s-t-n} and \eqref{eq:2q-recursion-alt},
 \[
  \la \alpha_u^\vee, \alpha_{u, n+1} \ra = [n+1]_u \cdot 2 - [n]_v \cdot [2]_u = [n+1]_u - [n-1]_u.
 \]
 This and \eqref{eq:varpi-s-t-n} show \eqref{eq:diagonal-2-q-identity-1-new}. Using \eqref{eq:varpi-s-t-n}, \eqref{eq:diagonal-2-q-identity-2-new} reduces immediately to \eqref{eq:2q-recursion}.
 
 \emph{Second formula}: The relevant components (of $\delta_c$ and $\delta''_x$) are as follows:
 \[
  \begin{tikzcd}
   B_{\uov{n-1}} \ar[d, "\otimes x^{-1}(\varpi_{u,n}^v)"', bend right=20] & B_{\vov{n-1}} \ar[dl, "\otimes x^{-1}(\varpi_u^{v,n})" near start, bend left=20] \\
   B_{\uov{n}} \ar[u] \ar[ur] \ar[d, "\otimes x^{-1}(\varpi_{u,n+1}^v)"', bend right=20] \ar[dr, "\otimes x^{-1}(\varpi_v^{u,n+1})" near end, bend left=20] \\
   B_{\uov{n+1}} \ar[u] & B_{\vov{n+1}}. \ar[ul]
  \end{tikzcd}
 \]
 So
 \begin{multline*}
  \delta''_x \circ \delta_c + \delta_c \circ \delta''_x = \Br_r \otimes x^{-1}(\varpi_{u,n}^v) + \Br_l \otimes x^{-1}(\varpi_u^{v,n}) \\
  + \TD_r \otimes x^{-1}(\varpi_{u,n+1}^v) + \TD_l \otimes x^{-1}(\varpi_v^{u,n+1}).
 \end{multline*}
 By \eqref{eq:updown2}--\eqref{eq:updown1},
this equals
 \begin{multline*}
  (\id \star \alpha_u) \otimes x^{-1}(\varpi_{u,n+1}^v) + (\alpha_v \star \id) \otimes x^{-1}(\varpi_v^{u,n+1}) \\
  + \Br_r \otimes \left( \frac{[n-1]}{[n]_u} x^{-1}(\varpi_{u,n+1}^v) + x^{-1}(\varpi_{u,n}^v) \right) \\
  +
   \Br_l \otimes \left( \frac{[n-1]}{[n]_v}x^{-1}(\varpi_v^{u,n+1}) + x^{-1}(\varpi_u^{v,n}) \right) 
 \end{multline*}
 if $n$ is even, and
  \begin{multline*}
  (\id \star \alpha_v) \otimes x^{-1}(\varpi_{u,n+1}^v) + (\alpha_v \star \id) \otimes x^{-1}(\varpi_v^{u,n+1}) \\
  + \Br_r \otimes \left( \frac{[n-1]_u}{[n]} x^{-1}(\varpi_{u,n+1}^v) + x^{-1}(\varpi_{u,n}^v) \right) \\
  +
   \Br_l \otimes \left( \frac{[n-1]_u}{[n]}x^{-1}(\varpi_v^{u,n+1}) + x^{-1}(\varpi_u^{v,n}) \right) 
 \end{multline*}
 if $n$ is odd.
 Meanwhile, using the basis $\{ \alpha_{u,n+1}, \alpha_v \}$, we see that
\[
  \kappa(\theta_x) = (\alpha_{u,n+1} \star \id) \otimes x^{-1}(\varpi_{u,n+1}^v) + (\alpha_v \star \id) \otimes x^{-1}(\varpi_v^{u,n+1}),
\]
which equals
\begin{multline*}
(\id \star \alpha_u) \otimes x^{-1}(\varpi_{u,n+1}^v) + (\alpha_v \star \id) \otimes x^{-1}(\varpi_v^{u,n+1}) \\
 + \Br_l \otimes \la \alpha_u^\vee, \alpha_{u,n+1} \ra x^{-1}(\varpi_{u,n+1}^v) + \Br_r \otimes [2]_v x^{-1}(\varpi_{u,n+1}^v)
\end{multline*}
if $n$ is even and
\begin{multline*}
(\id \star \alpha_v) \otimes x^{-1}(\varpi_{u,n+1}^v) + (\alpha_v \star \id) \otimes x^{-1}(\varpi_v^{u,n+1}) \\
 + \Br_l \otimes \la \alpha_u^\vee, \alpha_{u,n+1} \ra x^{-1}(\varpi_{u,n+1}^v) +  \Br_r \otimes [2]_u x^{-1}(\varpi_{u,n+1}^v)
\end{multline*}
if $n$ is odd.
 Again the main terms agree on the two sides. The equality of lower-order terms reduces to formulas independent of $x$; in fact, they are exactly \eqref{eq:diagonal-2-q-identity-1-new} and \eqref{eq:diagonal-2-q-identity-2-new}.
\end{proof}

\begin{lem}
\label{lem:d=0-2-new}
 We have
 \begin{align*}
  \delta' \circ \delta_c + \delta_c \circ \delta' &= 0 \\
  \delta''_x \circ \delta_c + \delta_c \circ \delta''_x &= 0, \mbox{ all } x \in W
 \end{align*}
 for components $B_{\uov{n}} \leadsto B_{\vov{n}}$ and $B_{\vov{n}} \leadsto B_{\uov{n}}$, $1 \le n \le m-2$.
\end{lem}

The proof of Lemma~\ref{lem:d=0-2-new} will be similar to that of Lemma~\ref{lem:d=0-1-new}, now using \eqref{eq:updown3}--\eqref{eq:updown4} instead of \eqref{eq:updown2}--\eqref{eq:updown1}. 
We simplify the notation further compared to
the proof of Lemma~\ref{lem:d=0-1-new}: we indicate only the $R^\vee$ parts of our morphisms and possible signs, and omit the components belonging to $\uEnd_\BE(\tD_{\uw,\min})$ (which should be clear from context).

\begin{proof}
We treat the case $n>1$, and leave the modifications in the case $n=1$ to the reader.

 \emph{First formula, $n$ odd}: The relevant components (of $\delta_c$ and $\delta'$) are as follows:
 \[
  \begin{tikzcd}
   B_{\uov{n-1}} \ar[dr, "\otimes \varpi_{v,n}^u"] \\
   B_{\uov{n}} \ar[u] \ar[dr, "\otimes \varpi_{u,n+1}^v"'] & B_{\vov{n}} \\
   & B_{\vov{n+1}} \ar[u, "-"']
  \end{tikzcd}
  +
  \begin{tikzcd}
    & B_{\vov{n-1}} \ar[d, "\otimes \varpi_v^{u,n}"] \\
   B_{\uov{n}} \ar[ur] \ar[d, "- \otimes \varpi_v^{u,n+1}"'] & B_{\vov{n}}. \\
   B_{\uov{n+1}} \ar[ur] &
  \end{tikzcd}
 \]
 So
 \[
  \delta' \circ \delta_c + \delta_c \circ \delta' = \otimes \varpi_{v,n}^u + \otimes \varpi_v^{u,n} + \otimes (-\varpi_{u,n+1}^v) + \otimes (-\varpi_v^{u,n+1}).
 \]
 By \eqref{eq:updown3}--\eqref{eq:updown4}, this equals
 \[
   \otimes \left( \varpi_{v,n}^u - \varpi_{u,n+1}^v - \frac{1}{[n]}\varpi_v^{u,n+1} \right) + \otimes \left( \varpi_v^{u,n} - \varpi_v^{u,n+1} - \frac{1}{[n]}\varpi_{u,n+1}^v \right).
 \]
 We will see below that this vanishes.
 
 \emph{First formula, $n$ even}: The relevant components (of $\delta_c$ and $\delta'$) are as follows:
 \[
  \begin{tikzcd}
   B_{\uov{n-1}} \ar[dr, "\otimes \varpi_{u,n}^v"] \\
   B_{\uov{n}} \ar[u, "-"] \ar[dr, "\otimes \varpi_{v,n+1}^u"'] & B_{\vov{n}} \\
   & B_{\vov{n+1}} \ar[u]
  \end{tikzcd}
  +
  \begin{tikzcd}
    & B_{\vov{n-1}} \ar[d, "- \otimes \varpi_u^{v,n}"] \\
   B_{\uov{n}} \ar[ur] \ar[d, "\otimes \varpi_u^{v,n+1}"'] & B_{\vov{n}}. \\
   B_{\uov{n+1}} \ar[ur] &
  \end{tikzcd}
 \]
 So
 \[
  \delta' \circ \delta_c + \delta_c \circ \delta' = \otimes (-\varpi_{u,n}^v) + \otimes (-\varpi_u^{v,n}) + \otimes \varpi_{v,n+1}^u + \otimes \varpi_u^{v,n+1}.
 \]
 By \eqref{eq:updown3}--\eqref{eq:updown4}, this equals
 \[
   \otimes \left(-\varpi_{u,n}^v + \varpi_{v,n+1}^u + \frac{1}{[n]_u}\varpi_u^{v,n+1} \right) + \otimes \left( -\varpi_u^{v,n} + \varpi_u^{v,n+1} + \frac{1}{[n]_v}\varpi_{v,n+1}^u \right).
 \]
 
 Thus in either parity, the first formula follows from the following formulas, up to switching $u$ and $v$:
 \begin{gather} \label{eq:cross-term-2q-identity-one-new}
  \varpi_{v,n}^u - \varpi_{u,n+1}^v = \frac{1}{[n]_v}\varpi_v^{u,n+1}, \\
 \label{eq:cross-term-2q-identity-two-new}
  \varpi_v^{u,n} - \varpi_v^{u,n+1} = \frac{1}{[n]_u}\varpi_{u,n+1}^v.
 \end{gather}
 Using \eqref{eq:varpi-s-t-n}, we get \eqref{eq:cross-term-2q-identity-one-new} immediately, while \eqref{eq:cross-term-2q-identity-two-new} reduces to \eqref{eq:2q-square-identity}.
 
 \emph{Second formula}: The relevant components (of $\delta_c$ and $\delta''_x$) are as follows:
 \[
  \begin{tikzcd}
   B_{\uov{n-1}} \ar[dr, "\otimes x^{-1}(\varpi_v^{u,n})"] \\
   B_{\uov{n}} \ar[u, "(-1)^{n+1}"] \ar[dr, "\otimes x^{-1}(\varpi_v^{u,n+1})"'] & B_{\vov{n}} \\
   & B_{\vov{n+1}} \ar[u, "(-1)^n"']
  \end{tikzcd}
  +
  \begin{tikzcd}
    & B_{\vov{n-1}} \ar[d, "(-1)^{n+1} \otimes x^{-1}(\varpi_{v,n}^u)"] \\
   B_{\uov{n}} \ar[ur] \ar[d, "(-1)^n \otimes x^{-1}(\varpi_{u,n+1}^v)"'] & B_{\vov{n}}. \\
   B_{\uov{n+1}} \ar[ur] &
  \end{tikzcd}
 \]
 So
 \begin{multline*}
  (-1)^n(\delta''_x \circ \delta_c + \delta_c \circ \delta''_x) = \otimes (-x^{-1}(\varpi_v^{u,n})) + \otimes (-x^{-1}(\varpi_{v,n}^u)) \\
  + \otimes x^{-1}(\varpi_v^{u,n+1}) + \otimes x^{-1}(\varpi_{u,n+1}^v).
 \end{multline*}
 By \eqref{eq:updown3}--\eqref{eq:updown4}, this equals
 \begin{multline*}
   \otimes \left( -x^{-1}(\varpi_v^{u,n}) + x^{-1}(\varpi_v^{u,n+1}) + \frac{1}{[n]_u}x^{-1}(\varpi_{u,n+1}^v) \right) \\
   + \otimes \left( -x^{-1}(\varpi_{v,n}^u) + x^{-1}(\varpi_{u,n+1}^v) + \frac{1}{[n]_v}x^{-1}(\varpi_v^{u,n+1}) \right).
 \end{multline*}
 The two terms vanish by~\eqref{eq:cross-term-2q-identity-two-new} and~\eqref{eq:cross-term-2q-identity-one-new}, respectively.
\end{proof}

\subsection{$d = -2$, generic components}

Here we fix $n<m-2$.
Below we will use the following fact, which can be checked explicitly:
 \begin{equation}
 \label{eqn:Winv-new}
 \text{the element $\frac{1}{[2]_s}\varpi_s^2 - \varpi_s\varpi_t + \frac{1}{[2]_t}\varpi_t^2 \in \mathrm{Sym}^2(V)$ is $W$-invariant.}
 \end{equation}

In the present setting, \eqref{eq:tdelta-differential-new} reads $\delta_{-1} \circ \delta_{-1} = 0$. The relevant generic components of $\delta_{-1} \circ \delta_{-1}$ come in two shapes:

\begin{enumerate}
\item
\emph{``Triangle''}:
\[
 \delta' = 
 \begin{array}{cc}
 \begin{tikzcd}
  B_{\uov{n}} \ar[d, "- \otimes \varpi_v^{u,n+1}"'] \ar[dr, "\otimes \varpi_{u,n+1}^v"] \\
  B_{\uov{n+1}} \ar[d, "\otimes \varpi_u^{v,n+2}"'] & B_{\vov{n+1}} \ar[dl, "\otimes \varpi_{u,n+2}^v"] \\
  B_{\uov{n+2}}
 \end{tikzcd}
 &
 \begin{tikzcd}
  B_{\uov{n}} \ar[d, "\otimes \varpi_u^{v,n+1}"'] \ar[dr, "\otimes \varpi_{v,n+1}^u"] \\
  B_{\uov{n+1}} \ar[d, "- \otimes \varpi_v^{u,n+2}"'] & B_{\vov{n+1}} \ar[dl, "\otimes \varpi_{v,n+2}^u"] \\
  B_{\uov{n+2}}
 \end{tikzcd} \\
 \text{if $n$ is odd,} & \text{if $n$ is even.}
 \end{array}
\]
\[
 \delta''_x = 
 \begin{tikzcd}
  B_{\uov{n}} \ar[d, "(-1)^n \otimes x^{-1}(\varpi_{u,n+1}^v)"'] \ar[dr, "\otimes x^{-1}(\varpi_v^{u,n+1})"] \\
  B_{\uov{n+1}} \ar[d, "(-1)^{n+1} \otimes x^{-1}(\varpi_{u,n+2}^v)"'] & B_{\vov{n+1}} \ar[dl, "\otimes x^{-1}(\varpi_u^{v,n+2})"] \\
  B_{\uov{n+2}};
 \end{tikzcd}
\]
\item
\emph{``Parallelogram''}:
\[
 \delta' =
 \begin{array}{cc}
 \begin{tikzcd}
  B_{\uov{n}} \ar[d, "- \otimes \varpi_v^{u,n+1}"'] \ar[dr, "\otimes \varpi_{u,n+1}^v"] \\
  B_{\uov{n+1}} \ar[dr, "\otimes \varpi_{v,n+2}^u"'] & B_{\vov{n+1}} \ar[d, "\otimes \varpi_v^{u,n+2}"] \\
  & B_{\vov{n+2}}
 \end{tikzcd}
 &
 \begin{tikzcd}
  B_{\uov{n}} \ar[d, "\otimes \varpi_u^{v,n+1}"'] \ar[dr, "\otimes \varpi_{v,n+1}^u"] \\
  B_{\uov{n+1}} \ar[dr, "\otimes \varpi_{u,n+2}^v"'] & B_{\vov{n+1}} \ar[d, "- \otimes \varpi_u^{v,n+2}"] \\
  & B_{\vov{n+2}}
 \end{tikzcd} \\
 \text{if $n$ is odd,} & \text{if $n$ is even.}
 \end{array}
\]
\[
 \delta''_x =
 \begin{tikzcd}
  B_{\uov{n}} \ar[d, "(-1)^n \otimes x^{-1}(\varpi_{u,n+1}^v)"'] \ar[dr, "\otimes x^{-1}(\varpi_v^{u,n+1})"] \\
  B_{\uov{n+1}} \ar[dr, "\otimes x^{-1}(\varpi_v^{u,n+2})"'] & B_{\vov{n+1}} \ar[d, "(-1)^{n+1} \otimes x^{-1}(\varpi_{v,n+2}^u)"] \\
  & B_{\vov{n+2}}.
 \end{tikzcd}
\]
\end{enumerate}

Since
\[
 \delta_{-1} \circ \delta_{-1} = \delta' \circ \delta' + \delta''_w \circ \delta''_w - (\delta' \circ \delta''_w + \delta''_w \circ \delta'),
\]
we see that~\eqref{eq:tdelta-differential-new} follows from Lemmas~\ref{lem:tdelta-differential-minus-two-W-invt-new} and~\ref{lem:tdelta-differential-minus-two-cross-terms-new} below.

\begin{lem}
\label{lem:tdelta-differential-minus-two-W-invt-new}
 Let $W$ act on $\uEnd_\BE(\Delta_{\uw,\min}) \otimes R^\vee$ by acting on $R^\vee$. Then $\delta' \circ \delta'$ is $W$-invariant, and $\delta''_x \circ \delta''_x =
 -\delta' \circ \delta'$ for all $x \in W$.
\end{lem}

\begin{proof}
Using~\eqref{eq:downdown}, the lemma reduces to the analogous claims inside $(R^\vee)^{0}_{-4}$. 
 In the calculations below, we only show the $R^\vee$ part.
 
 \emph{``Triangle''}: For $n$ odd, by \eqref{eq:varpi-s-t-n} we have
 \begin{multline*}
  \delta' \circ \delta' = -\varpi_v^{u,n+1}\varpi_u^{v,n+2} + \varpi_{u,n+1}^v\varpi_{u,n+2}^v \\
  = -\left(-\frac{[n]_v}{[n+1]_u}\varpi_u + \varpi_v\right) \left(\varpi_u - \frac{[n+1]_u}{[n+2]_v}\varpi_v\right) + \left(\frac{1}{[n+1]_u}\varpi_u\right) \left(\frac{1}{[n+2]_u}\varpi_u\right) \\
  = \left(\frac{[n]_v}{[n+1]_u} + \frac{1}{[n+1]_u[n+2]_u}\right)\varpi_u^2 - \left(1 + \frac{[n]_v}{[n+2]_v}\right)\varpi_u\varpi_v + \frac{[n+1]_u}{[n+2]_v}\varpi_v^2.
 \end{multline*}
 By~\eqref{eq:2q-recursion-alt},
 \[
  1 + \frac{[n]_v}{[n+2]_v} = [2]_v \frac{[n+1]_u}{[n+2]_v}.
 \]
 By \eqref{eq:2q-square-identity} and \eqref{eq:2q-s-t-odd}--\eqref{eq:2q-s-t-even}, we have
 \[
  [2]_u \left(\frac{[n]_v}{[n+1]_u} + \frac{1}{[n+1]_u[n+2]_u}\right) = [2]_v \frac{[n+1]_u}{[n+2]_v}.
 \]
 Hence $\delta' \circ \delta'$ is $W$-invariant by~\eqref{eqn:Winv-new}.
 The last claim of the lemma is clear:
 \[
  \delta''_x \circ \delta''_x = -x^{-1}(\omega_{u,n+1}^v\omega_{u,n+2}^v - \varpi_v^{u,n+1}\varpi_u^{v,n+2}) = -x^{-1}(\delta' \circ \delta') = -\delta' \circ \delta'.
 \]
 
 For $n$ even, $u$ and $v$ are switched in the components of $\delta'$. 
A very similar calculation shows that $\delta' \circ \delta'$ is again $W$-invariant. Moreover, since 
$1 + \frac{[n]_u}{[n+2]_u} = 1 + \frac{[n]_v}{[n+2]_v}$ (see~\eqref{eq:2q-product-0}), it follows that $\delta' \circ \delta'$ is given by the same formula as in the case $n$ is odd. Then the last statement is again immediate.
 
 \emph{``Parallelogram''}: For $n$ odd, by \eqref{eq:varpi-s-t-n} we have
 \begin{multline*}
  \delta' \circ \delta' = -\varpi_v^{u,n+1}\varpi_{v,n+2}^u + \varpi_{u,n+1}^v\varpi_v^{u,n+2} \\
  = -\left(-\frac{[n]_v}{[n+1]_u}\varpi_u + \varpi_v\right) \left(\frac{1}{[n+2]_v}\varpi_v\right) + \left(\frac{1}{[n+1]_u}\varpi_u\right) \left(-\frac{[n+1]_v}{[n+2]_u}\varpi_u + \varpi_v\right) \\
  = -\frac{[n+1]_v}{[n+1]_u[n+2]_u}\varpi_u^2 + \left(\frac{[n]_v}{[n+1]_u[n+2]_v} + \frac{1}{[n+1]_u}\right)\varpi_u\varpi_v - \frac{1}{[n+2]_v}\varpi_v^2.
 \end{multline*}
 By \eqref{eq:2q-recursion-alt} and~\eqref{eq:2q-s-t-odd}--\eqref{eq:2q-s-t-even} respectively, we have
 \[
  \frac{[n]_v}{[n+1]_u[n+2]_v} + \frac{1}{[n+1]_u} = [2]_u\frac{[n+1]_v}{[n+1]_u[n+2]_u} = [2]_v\frac{1}{[n+2]_v}.
 \]
 Hence $\delta' \circ \delta'$ is $W$-invariant by~\eqref{eqn:Winv-new}.
 The rest of the argument is the same as for the ``triangle'' components.
\end{proof}

\begin{lem} \label{lem:tdelta-differential-minus-two-cross-terms-new}
 We have $\delta' \circ \delta''_x + \delta''_x \circ \delta' = 0$ for all $x \in W$.
\end{lem}
\begin{proof}
 We simply compute the left-hand side using \eqref{eq:varpi-s-t-n}. There are four cases.
 
 \emph{$B_{\uov{n}} \leadsto B_{\uov{n+2}}$, $n$ odd}:
 {\footnotesize
 \begin{multline*}
  -\varpi_v^{u,n+1}x^{-1}(\varpi_{u,n+2}^v) - x^{-1}(\varpi_{u,n+1}^v)\varpi_u^{v,n+2} + \varpi_{u,n+1}^vx^{-1}(\varpi_u^{v,n+2}) + x^{-1}(\varpi_v^{u,n+1})\varpi_{u,n+2}^v \\
  = -\left(-\frac{[n]_v}{[n+1]_u}\varpi_u + \varpi_v\right) x^{-1}\left(\frac{1}{[n+2]_u}\varpi_u\right) 
   - x^{-1}\left(\frac{1}{[n+1]_u}\varpi_u\right) \left(\varpi_u - \frac{[n+1]_u}{[n+2]_v}\varpi_v\right) \\
   + \left(\frac{1}{[n+1]_u}\varpi_u\right) x^{-1}\left(\varpi_u - \frac{[n+1]_u}{[n+2]_v}\varpi_v\right)
   + x^{-1}\left(-\frac{[n]_v}{[n+1]_u}\varpi_u + \varpi_v\right) \left(\frac{1}{[n+2]_u}\varpi_u\right).
 \end{multline*}
 }Collecting terms of the form $\varpi_ux^{-1}(\varpi_u)$, $\varpi_vx^{-1}(\varpi_u)$, and $\varpi_ux^{-1}(\varpi_v)$, and using~\eqref{eq:2q-s-t-odd}, we see that each coefficient vanishes.

 \emph{$B_{\uov{n}} \leadsto B_{\uov{n+2}}$, $n$ even}:
{\footnotesize
 \begin{multline*}
  -\varpi_u^{v,n+1}x^{-1}(\varpi_{u,n+2}^v) - x^{-1}(\varpi_{u,n+1}^v)\varpi_v^{u,n+2}
  + \varpi_{v,n+1}^ux^{-1}(\varpi_u^{v,n+2}) + x^{-1}(\varpi_v^{u,n+1})\varpi_{v,n+2}^u \\
  = -\left(\varpi_u - \frac{[n]_u}{[n+1]_v}\varpi_v\right) x^{-1}\left(\frac{1}{[n+2]_u}\varpi_u\right)
   - x^{-1}\left(\frac{1}{[n+1]_u}\varpi_u\right) \left(-\frac{[n+1]_v}{[n+2]_u}\varpi_u + \varpi_v\right) \\
   + \left(\frac{1}{[n+1]_v}\varpi_v\right) x^{-1}\left(\varpi_u - \frac{[n+1]_u}{[n+2]_v}\varpi_v\right)
   + x^{-1}\left(-\frac{[n]_v}{[n+1]_u}\varpi_u + \varpi_v\right) \left(\frac{1}{[n+2]_v}\varpi_v\right).
 \end{multline*}
 }Collecting terms and using~\eqref{eq:2q-s-t-odd}, the coefficients of $\varpi_ux^{-1}(\varpi_u)$ and $\varpi_vx^{-1}(\varpi_v)$ clearly vanish, while the coefficient
 \[
  \frac{[n]_u}{[n+1]_v[n+2]_u} - \frac{1}{[n+1]_u} + \frac{1}{[n+1]_v} - \frac{[n]_v}{[n+1]_u[n+2]_v}
 \]
 of $\varpi_vx^{-1}(\varpi_u)$ vanishes by \eqref{eq:2q-product-0} and~\eqref{eq:2q-s-t-odd}.
 
 \emph{$B_{\uov{n}} \leadsto B_{\vov{n+2}}$, $n$ odd}:
 {\tiny
 \begin{multline*}
  -\varpi_v^{u,n+1}x^{-1}(\varpi_v^{u,n+2}) -\varpi_{v,n+2}^ux^{-1}(\varpi_{u,n+1}^v) \\
  + \varpi_{u,n+1}^vx^{-1}(\varpi_{v,n+2}^u) + x^{-1}(\varpi_v^{u,n+1})\varpi_v^{u,n+2} \\
  = -\left(-\frac{[n]_v}{[n+1]_u}\varpi_u + \varpi_v\right) x^{-1}\left(-\frac{[n+1]_v}{[n+2]_u}\varpi_u + \varpi_v\right)
   -\left(\frac{1}{[n+2]_v}\varpi_v\right) x^{-1}\left(\frac{1}{[n+1]_u}\varpi_u\right) \\
   + \left(\frac{1}{[n+1]_u}\varpi_u\right) x^{-1}\left(\frac{1}{[n+2]_v}\varpi_v\right)
   + x^{-1}\left(-\frac{[n]_v}{[n+1]_u}\varpi_u + \varpi_v\right) \left(-\frac{[n+1]_v}{[n+2]_u}\varpi_u + \varpi_v\right).
 \end{multline*}
 }Collecting terms, the coefficients of $\varpi_ux^{-1}(\varpi_u)$ and $\varpi_vx^{-1}(\varpi_v)$ clearly vanish, while the coefficients
 \[
  \pm\left( \frac{[n]_v}{[n+1]_u} + \frac{1}{[n+1]_u[n+2]_v} - \frac{[n+1]_v}{[n+2]_u} \right)
 \]
 of $\varpi_ux^{-1}(\varpi_v)$ and $\varpi_vx^{-1}(\varpi_u)$ vanish by \eqref{eq:2q-square-identity}.

 \emph{$B_{\uov{n}} \leadsto B_{\vov{n+2}}$, $n$ even}:
 {\tiny
 \begin{multline*}
  \varpi_u^{v,n+1}x^{-1}(\varpi_v^{u,n+2}) + x^{-1}(\varpi_{u,n+1}^v)\varpi_{u,n+2}^v \\
  -\varpi_{v,n+1}^ux^{-1}(\varpi_{v,n+2}^u) - x^{-1}(\varpi_v^{u,n+1})\varpi_u^{v,n+2} \\
  = \left(\varpi_u - \frac{[n]_u}{[n+1]_v}\varpi_v\right) x^{-1}\left(-\frac{[n+1]_v}{[n+2]_u}\varpi_u + \varpi_v\right)
   + x^{-1}\left(\frac{1}{[n+1]_u}\varpi_u\right) \left(\frac{1}{[n+2]_u}\varpi_u\right) \\
   -\left(\frac{1}{[n+1]_v}\varpi_v\right) x^{-1}\left(\frac{1}{[n+2]_v}\varpi_v\right) - x^{-1}\left(-\frac{[n]_v}{[n+1]_u}\varpi_u + \varpi_v\right) \left(\varpi_u - \frac{[n+1]_u}{[n+2]_v}\varpi_v\right) \\
  = \left(-\frac{[n+1]_v}{[n+2]_u} + \frac{1}{[n+1]_u[n+2]_u} + \frac{[n]_v}{[n+1]_u}\right)\varpi_ux^{-1}(\varpi_u) \\
   + \left(-\frac{[n]_u}{[n+1]_v} - \frac{1}{[n+1]_v[n+2]_v} + \frac{[n+1]_u}{[n+2]_v}\right)\varpi_vx^{-1}(\varpi_v) \\
   + (1-1) \cdot \varpi_ux^{-1}(\varpi_v) + \left(\frac{[n]_u}{[n+2]_u} - \frac{[n]_v}{[n+2]_v}\right)\varpi_vx^{-1}(\varpi_u).
 \end{multline*}
 }The coefficients of $\varpi_ux^{-1}(\varpi_u)$ and $\varpi_vx^{-1}(\varpi_v)$ vanish by \eqref{eq:2q-square-identity}, and the coefficient of $\varpi_vx^{-1}(\varpi_u)$ vanishes by \eqref{eq:2q-product-0}.
\end{proof}

\subsection{Relations involving $\delta_\bot$}

For these components, we cannot treat $\delta'$ and $\delta''_w$ separately, and \eqref{eq:tdelta-differential-new} relies on $x = w$. We begin by finding alternative expressions for components of $\delta' - \delta''_w$ with source $B_{\uov{m-1}}$ and $B_{\vov{m-1}}$.

For $a \in \{s, t\}$,
\begin{multline*}
 \uov{n}(\varpi_a) = \la \uov{n}(\varpi_a), \alpha_u \ra \varpi_u + \la \uov{n}(\varpi_a), \alpha_v \ra \varpi_v \\
 = \begin{cases}
    -\la \varpi_a, \alpha_{u,n} \ra \varpi_u + \la \varpi_a, \alpha_{u,n+1} \ra \varpi_v &\text{if $n$ is odd;} \\
    -\la \varpi_a, \alpha_{v,n} \ra \varpi_u + \la \varpi_a, \alpha_{v,n+1} \ra \varpi_v &\text{if $n$ is even.}
    \end{cases}
\end{multline*}
So by \eqref{eq:alpha-s-t-n},
\begin{equation} \label{eq:n-varpi-new}
 \begin{split}
 \uov{n}(\varpi_u) = 
  \begin{cases}
   -[n]_u\varpi_u + [n+1]_u\varpi_v &\text{if $n$ is odd;} \\
   -[n-1]_u\varpi_u + [n]_u\varpi_v &\text{if $n$ is even;}
  \end{cases} \\
 \uov{n}(\varpi_v) = 
  \begin{cases}
   -[n-1]_v\varpi_u + [n]_v\varpi_v &\text{if $n$ is odd;} \\
   -[n]_v\varpi_u + [n+1]_v\varpi_v &\text{if $n$ is even.}
  \end{cases}
 \end{split}
\end{equation}
Of course, the same formulas with $u$ and $v$ swapped hold.

\begin{lem} \label{lem:ell-minus-one-components-new}
If $m$ is odd, we have
 \begin{align*}
  \varpi_u^{v,m-1} - w^{-1}(\varpi_{v,m-1}^u) &= \alpha_u^\vee; \\
  \varpi_{v,m-1}^u - w^{-1}(\varpi_u^{v,m-1}) &= -w^{-1}(\alpha_u^\vee); \\
  -\varpi_v^{u,m-1} + w^{-1}(\varpi_{u,m-1}^v) &= \frac{1}{[m-1]_u}w^{-1}(\alpha_u^\vee); \\
  \varpi_{u,m-1}^v - w^{-1}(\varpi_v^{u,m-1}) &= \frac{1}{[m-1]_u}\alpha_u^\vee.
 \end{align*}
 If $m$ is even, we have
 \begin{align*} 
  \varpi_v^{u,m-1} - w^{-1}(\varpi_{v,m-1}^u) &= \alpha_v^\vee; \\
  \varpi_{u,m-1}^v - w^{-1}(\varpi_u^{v,m-1}) &= -w^{-1}(\alpha_u^\vee); \\
  \varpi_u^{v,m-1} - w^{-1}(\varpi_{u,m-1}^v) &= -\frac{1}{[m-1]}w^{-1}(\alpha_u^\vee); \\
  \varpi_{v,m-1}^u - w^{-1}(\varpi_v^{u,m-1}) &= \frac{1}{[m-1]}\alpha_v^\vee.
 \end{align*}
\end{lem}

\begin{proof}
 Suppose $m$ is odd. By \eqref{eq:varpi-s-t-n} and \eqref{eq:n-varpi-new}, 
 \[
  w^{-1}(\varpi_{v,m-1}^u) = \uov{m} \left( \frac{1}{[m-1]_v}\varpi_v \right) = -\varpi_u + \frac{[m]_v}{[m-1]_v}\varpi_v,
 \]
 so by \eqref{eq:varpi-s-t-n} and \eqref{eq:2q-recursion-alt},
 \begin{multline*}
  \varpi_u^{v,m-1} - w^{-1}(\varpi_{v,m-1}^u) = \left( \varpi_u - \frac{[m-2]_u}{[m-1]_v}\varpi_v \right) - \left( -\varpi_u + \frac{[m]_v}{[m-1]_v}\varpi_v \right) \\
  = 2\varpi_u - \frac{[m-2]_u + [m]_v}{[m-1]_v}\varpi_v = 2\varpi_u - [2]_u\varpi_v = \alpha_u^\vee.
 \end{multline*}
 This shows the first formula. Similarly,
 \[
  w^{-1}(\varpi_{u,m-1}^v) = \uov{m} \left( \frac{1}{[m-1]_u}\varpi_u \right) = \frac{1}{[m-1]_u}(-[m]_u\varpi_u + [m+1]_u\varpi_v),
 \]
 so
 {\small
 \begin{multline*}
  -\varpi_v^{u,m-1} + w^{-1}(\varpi_{u,m-1}^v) = \left( \frac{[m-2]_v}{[m-1]_u}\varpi_u - \varpi_v \right) + \frac{1}{[m-1]_u}(-[m]_u\varpi_u + [m+1]_u\varpi_v) \\
  = \frac{1}{[m-1]_u} \bigl( ([m-2]_v - [m]_u)\varpi_u + (-[m-1]_u + [m+1]_u)\varpi_v \bigr).
 \end{multline*}
 }Meanwhile,
 {\small
 \begin{multline*}
  w^{-1}(\alpha_u^\vee) = \uov{m}(2\varpi_u - [2]_u\varpi_v) = 2(-[m]_u\varpi_u + [m+1]_u\varpi_v) - [2]_u(-[m-1]_v\varpi_u + [m]_v\varpi_v) \\
  = (-2[m]_u + [2]_u[m-1]_v)\varpi_u + (2[m+1]_u - [2]_u[m]_v)\varpi_v \\
  = (-[m]_u + [m-2]_u)\varpi_u + ([m+1]_u - [m-1]_u)\varpi_v
 \end{multline*}
 }by~\eqref{eq:2q-recursion-alt}. This shows the third formula. The second (resp.~fourth) formula is obtained by applying $-w$ (resp.~$w$) to the first (resp.~third) formula (using the fact that $w=w^{-1}$ because $m$ is odd).
 
 If $m$ is even,
 \[
  w^{-1}(\varpi_{v,m-1}^u) = \vov{m} \left( \frac{1}{[m-1]}\varpi_v \right) = \frac{[m]_v}{[m-1]}\varpi_u - \varpi_v,
 \]
 so
 \begin{multline*}
  \varpi_v^{u,m-1} - w^{-1}(\varpi_{v,m-1}^u) = \left( -\frac{[m-2]_v}{[m-1]}\varpi_u + \varpi_v \right) - \left( \frac{[m]_v}{[m-1]}\varpi_u - \varpi_v \right) \\
  = \frac{-[m-2]_v - [m]_v}{[m-1]}\varpi_u + 2\varpi_v = -[2]_v\varpi_u + 2\varpi_v = \alpha_v^\vee
 \end{multline*}
 by~\eqref{eq:2q-recursion-alt},
 which shows the fifth formula. Similarly,
 \[
  w^{-1}(\varpi_{u,m-1}^v) = \vov{m} \left( \frac{1}{[m-1]}\varpi_u \right) = \frac{1}{[m-1]}([m+1]\varpi_u - [m]_u\varpi_v),
 \]
 so
 {\small
 \begin{multline*}
  \varpi_u^{v,m-1} - w^{-1}(\varpi_{u,m-1}^v) = \left( \varpi_u - \frac{[m-2]_u}{[m-1]}\varpi_v \right) - \frac{1}{[m-1]} \left( [m+1]\varpi_u - [m]_u\varpi_v \right) \\
  = \frac{1}{[m-1]} \bigl( ([m-1] - [m+1])\varpi_u + ([m]_u - [m-2]_u)\varpi_v \bigr).
 \end{multline*}
 }Meanwhile,
 \begin{multline*}
  w^{-1}(\alpha_u^\vee) = \vov{m}(2\varpi_u - [2]_u\varpi_v) \\
  = 2([m+1]_u\varpi_u - [m]_u\varpi_v) - [2]_u([m]_v\varpi_u - [m-1]_v\varpi_v) \\
  = (2[m+1]_u - [2]_u[m]_v)\varpi_u + (-2[m]_u + [2]_u[m-1]_v)\varpi_v \\
  = ([m+1] - [m-1])\varpi_u + (-[m]_u + [m-2]_u)\varpi_v.
 \end{multline*}
 This shows the seventh formula. The sixth (resp.~eighth) formula is obtained by applying $-w$ to the fifth (resp.~seventh) formula, followed by exchanging $u$ and $v$, noting that this exchanges $w$ and $w^{-1}$.
\end{proof}

By Lemma~\ref{lem:ell-minus-one-components-new}, the relevant components are as follows: if $m$ is odd,
\[
 \delta' - \delta''_w =
 \begin{tikzcd}
  B_{\uov{m-2}} \ar[d, "- \otimes \left( -\frac{1}{[m-1]_u}w^{-1}(\alpha_u^\vee) \right)"'] \ar[dr, "\otimes \frac{1}{[m-1]_u}\alpha_u^\vee" near end] \\
  B_{\uov{m-1}} & B_{\vov{m-1}}
 \end{tikzcd}
 +
 \begin{tikzcd}
  & B_{\vov{m-2}} \ar[dl, "\otimes (-w^{-1}(\alpha_u^\vee))"' near end] \ar[d, "- \otimes \alpha_u^\vee"] \\
  B_{\uov{m-1}} & B_{\vov{m-1}}
 \end{tikzcd}
\]
\[
 \delta_\bot =
 \begin{tikzcd}
  B_{\uov{m-1}} \ar[d, "\otimes \alpha_u^\vee"'] & B_{\vov{m-1}} \ar[dl, "\otimes (-w^{-1}(\alpha_u^\vee))" near start]\\
  B_{\uov{m}}
 \end{tikzcd}
\]
and if $m$ is even,
\[
 \delta' - \delta''_w =
 \begin{tikzcd}
  B_{\uov{m-2}} \ar[d, "\otimes \left( -\frac{1}{[m-1]}w^{-1}(\alpha_u^\vee) \right)"'] \ar[dr, "\otimes \frac{1}{[m-1]}\alpha_v^\vee" near end] \\
  B_{\uov{m-1}} & B_{\vov{m-1}}
 \end{tikzcd}
 +
 \begin{tikzcd}
  & B_{\vov{m-2}} \ar[dl, "\otimes (-w^{-1}(\alpha_u^\vee))"' near end] \ar[d, "\otimes \alpha_v^\vee"] \\
  B_{\uov{m-1}} & B_{\vov{m-1}}
 \end{tikzcd}
\]
\[
 \delta_\bot =
 \begin{tikzcd}
  B_{\uov{m-1}} \ar[d, "- \otimes \alpha_v^\vee"'] & B_{\vov{m-1}} \ar[dl, "\otimes (-w^{-1}(\alpha_u^\vee))" near start]\\
  B_{\uov{m}}
 \end{tikzcd}
\]
It is now clear that $\delta_\bot \circ (\delta' - \delta''_w) = 0$, which is the desired relation for $d = -2$. The cross-term components for $d = 0$ (see the relevant components below) are also easily seen to vanish using \eqref{eq:updown3}--\eqref{eq:updown4} as in the analogous computation for the generic components.

$B_{\uov{m-1}} \leadsto B_{\vov{m-1}}$:
 \[
  \begin{tikzcd}
   B_{\uov{m-2}} \ar[dr, "\otimes \frac{1}{[m-1]_u}\alpha_u^\vee"] \\
   B_{\uov{m-1}} \ar[u, "-"] & B_{\vov{m-1}} \\
   \phantom{B_{\uov{m}}}
  \end{tikzcd}
  +
  \begin{tikzcd}
    & B_{\vov{m-2}} \ar[d, "- \otimes \alpha_u^\vee"] \\
   B_{\uov{m-1}} \ar[ur] \ar[d, "\otimes \alpha_u^\vee"'] & B_{\vov{m-1}} \\
   B_{\uov{m}} \ar[ur] &
  \end{tikzcd}
  \quad \text{if $m$ is odd,}
 \]
 \[
  \begin{tikzcd}
   B_{\uov{m-2}} \ar[dr, "\otimes \frac{1}{[m-1]}\alpha_v^\vee"] \\
   B_{\uov{m-1}} \ar[u] & B_{\vov{m-1}} \\
   \phantom{B_{\uov{m}}}
  \end{tikzcd}
  +
  \begin{tikzcd}
    & B_{\vov{m-2}} \ar[d, "\otimes \alpha_v^\vee"] \\
   B_{\uov{m-1}} \ar[ur] \ar[d, "- \otimes \alpha_v^\vee"'] & B_{\vov{m-1}} \\
   B_{\uov{m}} \ar[ur] &
  \end{tikzcd}
  \quad \text{if $m$ is even.}
 \]

$B_{\vov{m-1}} \leadsto B_{\uov{m-1}}$:
 \[
  \begin{tikzcd}
   & B_{\vov{m-2}} \ar[dl, "\otimes (-w^{-1}(\alpha_u^\vee))"'] \\
   B_{\uov{m-1}} & B_{\vov{m-1}} \ar[u, "-"'] \ar[dl, "\otimes (-w^{-1}(\alpha_u^\vee))"] \\
   B_{\uov{m}} \ar[u]
  \end{tikzcd}
  +
  \begin{tikzcd}
   B_{\uov{m-2}} \ar[d, "- \otimes \left( -\frac{1}{[m-1]_u}w^{-1}(\alpha_u^\vee) \right)"'] \\
   B_{\uov{m-1}} & B_{\vov{m-1}} \ar[ul] \\
   \phantom{B_{\uov{m}}}
  \end{tikzcd}
  \quad \text{if $m$ is odd,}
 \]
 \[
  \begin{tikzcd}
   & B_{\vov{m-2}} \ar[dl, "\otimes (-w^{-1}(\alpha_u^\vee))"'] \\
   B_{\uov{m-1}} & B_{\vov{m-1}} \ar[u] \ar[dl, "\otimes (-w^{-1}(\alpha_u^\vee))"] \\
   B_{\uov{m}} \ar[u, "-"]
  \end{tikzcd}
  +
  \begin{tikzcd}
   B_{\uov{m-2}} \ar[d, "\otimes \left( -\frac{1}{[m-1]}w^{-1}(\alpha_u^\vee) \right)"'] \\
   B_{\uov{m-1}} & B_{\vov{m-1}} \ar[ul] \\
   \phantom{B_{\uov{m}}}
   \end{tikzcd}
   \quad \text{if $m$ is even.}
 \]

For the self-loop components, we will use the following computation. Define $\alpha_{s, n}^\vee, \alpha_{t, n}^\vee \in V$ for $n \ge 1$ as follows:
\[
 \begin{array}{rclcccrcl}
  \alpha_{s, 1}^\vee &=& \alpha_s^\vee      & & & & \alpha_{t, 1}^\vee &=& \alpha_t^\vee \\
  \alpha_{s, 2}^\vee &=& s(\alpha_t^\vee)   & & & & \alpha_{t, 2}^\vee &=& t(\alpha_s^\vee) \\
  \alpha_{s, 3}^\vee &=& st(\alpha_s^\vee)  & & & & \alpha_{t, 3}^\vee &=& ts(\alpha_t^\vee) \\
  \alpha_{s, 4}^\vee &=& sts(\alpha_t^\vee) & & & & \alpha_{t, 4}^\vee &=& tst(\alpha_s^\vee) \\
                &\vdots&          & & & &               &\vdots&
 \end{array}
\]
An easy induction shows that for any $n \geq 1$ we have
\begin{equation} \label{eq:alpha-s-t-n-check-new}
 \alpha_{s, n}^\vee = [n]_t\alpha_s^\vee + [n-1]_s\alpha_t^\vee, \qquad \alpha_{t, n}^\vee = [n-1]_t\alpha_s^\vee + [n]_s\alpha_t^\vee.
\end{equation}

\begin{lem}
 If $m$ is odd, we have
 \begin{gather} \label{eq:delta-bot-self-loop-one-new}
  \frac{[m-2]}{[m-1]_u}\alpha_u^\vee - \frac{1}{[m-1]_u}w^{-1}(\alpha_u^\vee) = u(\alpha_v^\vee);
 \\ \label{eq:delta-bot-self-loop-two-new}
  \frac{[m-2]}{[m-1]_u}(-w^{-1}(\alpha_u^\vee)) + \frac{1}{[m-1]_u}\alpha_u^\vee = -w^{-1}u(\alpha_v^\vee).
 \end{gather}
 If $m$ is even, we have
 \begin{gather} \label{eq:delta-bot-self-loop-three-new}
  \frac{[m-2]_u}{[m-1]}\alpha_v^\vee -\frac{1}{[m-1]}w^{-1}(\alpha_u^\vee) = v(\alpha_u^\vee);
 \\ \label{eq:delta-bot-self-loop-four-new}
  \frac{[m-2]_v}{[m-1]}(-w^{-1}(\alpha_u^\vee)) + \frac{1}{[m-1]}\alpha_v^\vee = -w^{-1}u(\alpha_v^\vee).
 \end{gather}
\end{lem}

\begin{proof}
If $m$ is odd, $w^{-1}(\alpha_u^\vee) = -\alpha_{u,m}^\vee$, so by \eqref{eq:alpha-s-t-n-check-new} and \eqref{eq:2q-recursion-alt},
 \[
  \frac{[m-2]}{[m-1]_u}\alpha_u^\vee - \frac{1}{[m-1]_u}w^{-1}(\alpha_u^\vee) = \frac{[m-2] + [m]}{[m-1]_u}\alpha_u^\vee + \alpha_v^\vee = [2]_v\alpha_u^\vee + \alpha_v^\vee = u(\alpha_v^\vee),
 \]
 proving \eqref{eq:delta-bot-self-loop-one-new}. Then \eqref{eq:delta-bot-self-loop-two-new} is obtained from this by applying $-w$ and using that $w = w^{-1}$.

If $m$ is even, $w^{-1}(\alpha_u^\vee) = -\alpha_{v,m}^\vee$, so by \eqref{eq:alpha-s-t-n-check-new} and \eqref{eq:2q-recursion-alt},
{\small
 \begin{multline*}
  \frac{[m-2]_u}{[m-1]}\alpha_v^\vee -\frac{1}{[m-1]}w^{-1}(\alpha_u^\vee) = \alpha_u^\vee + \frac{[m-2]_u + [m]_u}{[m-1]}\alpha_v^\vee = \alpha_u^\vee + [2]_u\alpha_v^\vee = v(\alpha_u^\vee),
 \end{multline*}
 }proving \eqref{eq:delta-bot-self-loop-three-new}. Then \eqref{eq:delta-bot-self-loop-four-new} is obtained from this by applying $-w$ then switching $u$ and $v$, observing that this switches $w$ and $w^{-1}$.
\end{proof}

We compute each side of \eqref{eq:tdelta-differential-new} for the remaining self-loop components, using \eqref{eq:updown2}--\eqref{eq:updown1} as we did for the generic components.

\emph{Self-loop at $B_{\uov{m-1}}$, $m$ odd:}
 \[
  \begin{tikzcd}
   B_{\uov{m-2}} \ar[d, "-\otimes \left( -\frac{1}{[m-1]_u}w^{-1}(\alpha_u^\vee) \right)"', bend right=20] & B_{\vov{m-2}} \ar[dl, "\otimes (-w^{-1}(\alpha_u^\vee))" near start, bend left=20] \\
   B_{\uov{m-1}} \ar[u, "-" swap] \ar[ur] \ar[d, "\otimes \alpha_u^\vee"', bend right=20] \\
   B_{\uov{m}}. \ar[u]
  \end{tikzcd}
 \]
 \begin{multline*}
  \delta_{-1} \circ \delta_c + \delta_c \circ \delta_{-1} = \Br_r \otimes \left( -\frac{1}{[m-1]_u}w^{-1}(\alpha_u^\vee) \right) + \Br_l \otimes (-w^{-1}(\alpha_u^\vee)) + \TD_r \otimes \alpha_u^\vee \\
  = (\id \star \alpha_u) \otimes \alpha_u^\vee + \Br_r \otimes \left( \frac{[m-2]}{[m-1]_u}\alpha_u^\vee -\frac{1}{[m-1]_u}w^{-1}(\alpha_u^\vee) \right) + \Br_l \otimes (-w^{-1}(\alpha_u^\vee)).
 \end{multline*}
 Meanwhile,
 \[
  \Theta_{B_{\uov{m-1}}} = \sum (\id \star e_i) \otimes \check e_i
 \]
 and
 \begin{multline*}
  \kappa(\theta_w) = \sum (w(e_i) \star \id) \otimes \check e_i = \sum (\id \star u(e_i)) \otimes \check e_i \\
  - \sum \Br_r \otimes \la \alpha_v^\vee, u(e_i) \ra \check e_i + \sum \Br_l \otimes \la \alpha_u^\vee, w(e_i) \ra \check e_i
 \end{multline*}
 by~\eqref{eq:polys},
 so
 \[
  \Theta_{B_{\uov{m-1}}} - \kappa(\theta_w) = (\id \star \alpha_u) \otimes \alpha_u^\vee \\
  + \Br_r \otimes u(\alpha_v^\vee) + \Br_l \otimes (-w^{-1}(\alpha_u^\vee)).
 \]
 Now the desired equality follows from~\eqref{eq:delta-bot-self-loop-one-new}.

\emph{Self-loop at $B_{\vov{m-1}}$, $m$ odd:}
 \[
  \begin{tikzcd}
   B_{\uov{m-2}} \ar[dr, "\otimes \frac{1}{[m-1]_u}\alpha_u^\vee"' near start, bend right=20] & B_{\vov{m-2}} \ar[d, "- \otimes \alpha_u^\vee", bend left=20] \\
   & B_{\vov{m-1}} \ar[ul] \ar[u, "-"] \ar[dl, "\otimes (-w^{-1}(\alpha_u^\vee))"' near end, bend right=20] \\
   B_{\uov{m}}. \ar[ur]
  \end{tikzcd}
 \]
 {\small
 \begin{multline*}
  \delta_{-1} \circ \delta_c + \delta_c \circ \delta_{-1} = \Br_l \otimes \frac{1}{[m-1]_u}\alpha_u^\vee + \Br_r \otimes \alpha_u^\vee + \TD_l \otimes (-w^{-1}(\alpha_u^\vee)) \\
  = (\alpha_u \star \id) \otimes (-w^{-1}(\alpha_u^\vee)) + \Br_l \otimes \left( \frac{[m-2]}{[m-1]_u}(-w^{-1}(\alpha_u^\vee)) + \frac{1}{[m-1]_u}\alpha_u^\vee \right) + \Br_r \otimes \alpha_u^\vee.
 \end{multline*}
 }Meanwhile,
 \begin{multline*}
  \Theta_{B_{\vov{m-1}}} = \sum (\id \star e_i) \otimes \check e_i \\
  = \sum (uw(e_i) \star \id) \otimes \check e_i - \sum \Br_l \otimes \la \alpha_v^\vee, uw(e_i) \ra \check e_i + \sum \Br_r \otimes \la \alpha_u^\vee, e_i \ra \check e_i \\
  = \sum (e_i \star \id) \otimes w^{-1}u(\check e_i) + \Br_l \otimes (-w^{-1}u(\alpha_v^\vee)) + \Br_r \otimes \alpha_u^\vee
  \end{multline*}
  by~\eqref{eq:polys},
 and
 \[
  \kappa(\theta_w) = \sum (u(e_i) \star \id) \otimes w^{-1}u(\check e_i),
 \]
 so
 \[
  \Theta_{B_{\vov{m-1}}} - \kappa(\theta_w) = (\alpha_u \star \id) \otimes (-w^{-1}(\alpha_u^\vee)) + \Br_l \otimes (-w^{-1}u(\alpha_v^\vee)) + \Br_r \otimes \alpha_u^\vee.
 \]
 Now the desired equality follows from~\eqref{eq:delta-bot-self-loop-two-new}.

\emph{Self-loop at $B_{\uov{m-1}}$, $m$ even:}
 \[
  \begin{tikzcd}
   B_{\uov{m-2}} \ar[d, "\otimes \left( -\frac{1}{[m-1]}w^{-1}(\alpha_u^\vee) \right)"', bend right=20] & B_{\vov{m-2}} \ar[dl, "\otimes (-w^{-1}(\alpha_u^\vee))" near start, bend left=20] \\
   B_{\uov{m-1}} \ar[u] \ar[ur] \ar[d, "- \otimes \alpha_v^\vee"', bend right=20] \\
   B_{\uov{m}}. \ar[u, "-" swap]
  \end{tikzcd}
 \]
 \begin{multline*}
  \delta_{-1} \circ \delta_c + \delta_c \circ \delta_{-1} = \Br_r \otimes \left( -\frac{1}{[m-1]}w^{-1}(\alpha_u^\vee) \right) + \Br_l \otimes (-w^{-1}(\alpha_u^\vee)) + \TD_r \otimes \alpha_v^\vee \\
  = (\id \star \alpha_v) \otimes \alpha_v^\vee + \Br_r \otimes \left( \frac{[m-2]_u}{[m-1]}\alpha_v^\vee -\frac{1}{[m-1]}w^{-1}(\alpha_u^\vee) \right) + \Br_l \otimes (-w^{-1}(\alpha_u^\vee)).
 \end{multline*}
 Meanwhile,
 \[
  \Theta_{B_{\uov{m-1}}} = \sum (\id \star e_i) \otimes \check e_i
 \]
 and
 \begin{multline*}
  \kappa(\theta_w) = \sum (w(e_i) \star \id) \otimes \check e_i \\
  = \sum (\id \star v(e_i)) \otimes \check e_i - \sum \Br_r \otimes \la \alpha_u^\vee, v(e_i) \ra \check e_i + \sum \Br_l \otimes \la \alpha_u^\vee, w(e_i) \ra \check e_i
 \end{multline*}
 by~\eqref{eq:polys},
 so
 \[
  \Theta_{B_{\uov{m-1}}} - \kappa(\theta_w) = (\id \star \alpha_v) \otimes \alpha_v^\vee \\
  + \Br_r \otimes v(\alpha_u^\vee) + \Br_l \otimes (-w^{-1}(\alpha_u^\vee)).
 \]
 Now the desired equality follows from~\eqref{eq:delta-bot-self-loop-three-new}.

\emph{Self-loop at $B_{\vov{m-1}}$, $m$ even:}
 \[
  \begin{tikzcd}
   B_{\uov{m-2}} \ar[dr, "\otimes \frac{1}{[m-1]}\alpha_v^\vee"' near start, bend right=20] & B_{\vov{m-2}} \ar[d, "\otimes \alpha_v^\vee", bend left=20] \\
   & B_{\vov{m-1}} \ar[ul] \ar[u] \ar[dl, "\otimes (-w^{-1}(\alpha_u^\vee))"' near end, bend right=20] \\
   B_{\uov{m}}. \ar[ur]
  \end{tikzcd}
 \]
 \begin{multline*}
  \delta_{-1} \circ \delta_c + \delta_c \circ \delta_{-1} = \Br_l \otimes \frac{1}{[m-1]}\alpha_v^\vee + \Br_r \otimes \alpha_v^\vee + \TD_l \otimes (-w^{-1}(\alpha_u^\vee)) \\
  =  (\alpha_u \star \id) \otimes (-w^{-1}(\alpha_u^\vee)) + \Br_l \otimes \left( \frac{[m-2]_v}{[m-1]}(-w^{-1}(\alpha_u^\vee)) + \frac{1}{[m-1]}\alpha_v^\vee \right) + \Br_r \otimes \alpha_v^\vee.
 \end{multline*}
 Meanwhile,
 \begin{multline*}
  \Theta_{B_{\vov{m-1}}} = \sum (\id \star e_i) \otimes \check e_i \\
  = \sum (uw(e_i) \star \id) \otimes \check e_i - \sum \Br_l \otimes \la \alpha_v^\vee, uw(e_i) \ra \check e_i + \sum \Br_r \otimes \la \alpha_v^\vee, e_i \ra \check e_i \\
  = \sum (e_i \star \id) \otimes w^{-1}u(\check e_i) + \Br_l \otimes (-w^{-1}u(\alpha_v^\vee)) + \Br_r \otimes \alpha_v^\vee
  \end{multline*}
  by~\eqref{eq:polys},
 and
 \[
  \kappa(\theta_w) = \sum (u(e_i) \star \id) \otimes w^{-1}u(\check e_i),
 \]
 so
 \[
  \Theta_{B_{\vov{m-1}}} - \kappa(\theta_w) = (\alpha_u \star \id) \otimes (-w^{-1}(\alpha_u^\vee)) + \Br_l \otimes (-w^{-1}u(\alpha_v^\vee)) + \Br_r \otimes \alpha_v^\vee.
 \]
 Now the desired equality follows from~\eqref{eq:delta-bot-self-loop-four-new}.
 
\emph{Self-loop at $B_{\uov{m}}$:} by~\eqref{eq:polys} we have
\begin{multline*}
 \Theta_{B_{\uov{m}}} = \sum (\id \star e_i) \otimes \check e_i \\
 = \sum (w(e_i) \star \id) \otimes \check e_i - \sum \Br_l \otimes \la \alpha_u^\vee, w(e_i) \ra \check e_i
 + \begin{cases}
    \Br_r \otimes \alpha_u^\vee &\text{if $m$ is odd;} \\
    \Br_r \otimes \alpha_v^\vee &\text{if $m$ is even}
   \end{cases} \\
 = \kappa(\theta_w) + \Br_l \otimes (-w^{-1}(\alpha_u^\vee))
 + \begin{cases}
    \Br_r \otimes \alpha_u^\vee &\text{if $m$ is odd;} \\
    \Br_r \otimes \alpha_v^\vee &\text{if $m$ is even}
   \end{cases} \\
 = \kappa(\theta_w) + \delta_\bot \circ \delta_c.
\end{multline*}

This concludes the proof of Proposition~\ref{prop:free-monodromic-minimal-Rouquier-new}.

\section{Convolution of free-monodromic minimal Rouquier complexes}

The goal of this section is to prove the following free-monodromic analogue of Proposition~\ref{prop:Rouquier-convolution-new}.

\begin{thm}
\label{thm:Rouquier-convolution-fm-new}
If $\uw \in \hW$ and $\uw = (s_1, \sdots, s_r)$, then there exists a morphism
\[
\tD_{s_1} \hatstar \tD_{s_2} \hatstar \dots \hatstar \tD_{s_r} \to \tD_{\uw,\min}
\]
in $\Conv_\FM^\rhd(\fh,W)_\Kar$ whose image in $\FM(\fh,W)_\Kar$ is an isomorphism.
\end{thm}

The proof of this proposition will use the following lemma, which is a direct consequence of Lemma~\ref{lem:BwBu}.

\begin{lem}
\label{lem:components-conv-Rouquier}
Let $\uw \in \hW$ and $u \in \{s,t\}$ be such that $\uw u \in \hW$. Then the $p$-th component of the $\Diag$-sequence underlying $\tD_{\uw, \min} \hatstar \tD_{u}$ is a direct sum of objects of the form $B_v(n)$ with 
\[
(\ell(v),n) \in \{(\ell(\uw)-p+1,p), (\ell(\uw)-p-1,p), (\ell(\uw)-p,p+1), (\ell(\uw)-p,p-1)\}.
\]
\end{lem}

\begin{proof}[Proof of Theorem~{\rm\ref{thm:Rouquier-convolution-fm-new}}]
As in the proof of Proposition~\ref{prop:Rouquier-convolution-new}, it suffices to prove that if $\uv \in \hW$ and $u \in \{s,t\}$ are such that $\uv u \in \hW$, then there exists a morphism
\[
\tD_{\uv, \min} \hatstar \tD_u \to \tD_{\uv u, \min}
\]
in $\Conv_\FM^\rhd(\fh,W)_\Kar$ whose image in $\FM(\fh,W)_\Kar$ is an isomorphism.

By~\eqref{eqn:Rouquier-For-new} we have $\ForFMLM(\tD_{\uv u,\min}) \cong \ForBELM(\Delta_{\uv u,\min})$. On the other hand, using~\eqref{eqn:For-hatstar} and Lemma~\ref{lem:convolution-mon-equ}, we see that
\[
\ForFMLM(\tD_{\uv, \min} \hatstar \tD_u) \cong \ForFMLM(\tD_{\uv, \min}) \ustar \Delta_u \cong \ForBELM(\Delta_{\uv, \min}) \ustar \Delta_u \cong \ForBELM(\Delta_{\uv u,\min})
\]
by Proposition~\ref{prop:Rouquier-convolution-new}.
Using these isomorphisms, Corollary~\ref{cor:End-Delta-min}, Theorem~\ref{thm:leftmon-con}, and Lemma~\ref{lem:hom-mon-lmon}, we obtain that if $M$ and $N$ are either $\tD_{\uv u,\min}$ or $\tD_{\uv, \min} \hatstar \tD_u$, the right $R^\vee$-module $\gHom_\FM(M,N)$ is free, and that
\[
\gHom_\FM(M, N) \otimes_{R^\vee} \bk \cong \gEnd_\LM \bigl( \ForBELM(\Delta_{\uv u, \min}) \bigr) \cong \bk.
\]
In particular, it follows that the functor $\ForFMLM$ induces an isomorphism
\[
\Hom_{\FM(\fh,W)_{\Kar}}(M, N) \simto \Hom_{\LM(\fh,W)_{\Kar}}(\ForFMLM(M), \ForFMLM(N)).
\]
We deduce that any isomorphism
\[
\ForFMLM(\tD_{\uv u,\min}) \simto \ForFMLM(\tD_{\uv, \min} \hatstar \tD_u), \ \text{resp.} \
\ForFMLM(\tD_{\uv, \min} \hatstar \tD_u) \simto \ForFMLM(\tD_{\uv u,\min}),
\]
can be lifted uniquely to an isomorphism
\[
\tD_{\uv u,\min} \simto \tD_{\uv, \min} \hatstar \tD_u, \quad \text{resp.} \quad \tD_{\uv, \min} \hatstar \tD_u \simto \tD_{\uv u,\min}. 
\]
What remains to be proved is that the isomorphism $\tD_{\uv, \min} \hatstar \tD_u \simto \tD_{\uv u,\min}$ automatically belongs to $\Conv_\FM^\rhd(\fh,W)_\Kar$.

So, let us fix a chain map $f : \tD_{\uv, \min} \hatstar \tD_u \to \tD_{\uv u,\min}$ which induces an isomorphism in $\FM(\fh,W)_\Kar$. To simplify notation we set $\cF:=\tD_{\uv, \min} \hatstar \tD_u$ and $\cG:=\tD_{\uv u,\min}$. Then $f$ belongs to
\begin{multline*}
\uHom_\FM(\cF,\cG)^0_0 = \bigoplus_{\substack{n \geq 0, \ m \geq 0 \\ n+i=0, \\ 2n+j-2m=0}} \Lambda^n_{2n} \otimes \uHom_\BE(\cF,\cG)^i_j \otimes (R^\vee)^0_{-2m} \\
= \bigoplus_{\substack{n \geq 0, \ m \geq 0 \\ n+i=0, \\ 2n+j-2m=0}} \Lambda^n_{2n} \otimes \prod_{q-p=i-j} \Hom_\BE(\cF^p(-p),\cG^q(-q)(i)) \otimes (R^\vee)^0_{-2m}.
\end{multline*}
In these factors, if $i \geq 0$ then $n=0$, so the corresponding component belongs to $\uHom^\rhd_\FM(\cF,\cG)$. From Corollary~\ref{cor:dih-indecomp-hom} and Lemma~\ref{lem:components-conv-Rouquier} we see that the other components vanish except maybe when $i=-1$.

We claim that $f$ does not have any nonzero component for which $i=-1$. Indeed, for such a component to be nonzero, by Lemma~\ref{lem:components-conv-Rouquier} again we must have $\ell(\uv)-p = \ell(\uv u)-q$, and hence $j=-1+p-q =-2$. Since $j=2m-2$, we deduce that $m=0$. Thus, if $f$ has a nonzero component in $\Lambda^1_2 \otimes \uHom_\BE(\cF,\cG) \otimes R^\vee$, then $\ForFMLM(f) \in \uHom_\LM(\cF,\cG)^0_0$ has a nonzero component in $\Lambda^1_2 \otimes \uHom_\BE(\cF,\cG)^{-1}_{-2}$.

On the other hand, by Corollary~\ref{cor:End-Delta-min}, the image of $\ForFMLM(f)$ in $\gHom_\LM(\cF,\cG)$ must be a multiple of the image under $\ForBELM$ of the isomorphism of Proposition~\ref{prop:Rouquier-convolution-new}. Therefore, there is a $k \in \uHom_\LM(\cF,\cG)^{-1}_0$ such that $\ForFMLM(f)-d_{\uHom_\LM(\cF,\cG)}(k) \in \uHom_\BE(\cF,\cG)$. However, as above we have
\begin{multline*}
\uHom_\LM(\cF,\cG)^{-1}_0 = \bigoplus_{n \geq 0} \Lambda^n_{2n} \otimes \uHom_\BE(\cF,\cG)^{-n-1}_{-2n} \\
= \bigoplus_{n \geq 0} \Lambda^n_{2n} \otimes \prod_{q-p=n-1} \Hom_\BE(\cF^p(-p), \cG^q(-q)(-n-1)).
\end{multline*}
Hence, again for the same reason as above, $k$ must belong to $\uHom_\BE(\cF,\cG)$. Since $\ForFMLM(\cF)$ and $\ForFMLM(\cG)$ belong to the essential image of $\ForBELM$, we deduce that $d_{\uHom_\LM(\cF,\cG)}(k) \in \uHom_\BE(\cF,\cG)$ also.
This implies that $f \in \uHom_\BE(\cF,\cG)$, a contradiction.
\end{proof}

\chapter{Flag varieties for Kac--Moody groups}
\label{chap:kac-moody}

For the remainder of the paper, we will restrict our attention to Cartan realizations of crystallographic Coxeter groups. For such realizations, the category $\DiagBSp(\fh,W)$ is equivalent to the category of parity complexes on a Kac--Moody flag variety. This equivalence lets us invoke powerful tools from sheaf theory and geometry in our study of free-monodromic complexes.  

This chapter consists mostly of preliminaries and definitions related to the Kac--Moody setting. We will define sheaf-theoretic analogues of the four categories $\BE(\fh,W)$, $\RE(\fh,W)$, $\LM(\fh,W)$, and $\FM(\fh,W)$.  At the end of the chapter, we prove a few new statements on (co)standard and tilting sheaves.

\section{Cartan realizations of crystallographic Coxeter groups}
\label{sec:Cartan-realizations}

Let $A = (a_{ij})_{i,j \in I}$ be a generalized Cartan matrix, with rows and columns parametrized by a finite set $I$.  Let $(I,\bX, \{\alpha_i : i \in I\}, \{\alpha_i^\vee : i \in I\})$ be an associated Kac--Moody root datum in the sense of~\cite[\S1.2]{tits}.  Thus, $\bX$ is a finitely generated free abelian group, and the $\alpha_i$ and the $\alpha_i^\vee$ are elements of $\bX$ and $\bX^*$, respectively, satisfying $\alpha_i^\vee(\alpha_j) = a_{ij}$.  (Here, as usual, we put $\bX^* = \Hom_\Z(\bX,\Z)$.)

The matrix $A$ determines a crystallographic Coxeter group $(W,S)$ by a well-known recipe.  The set of simple reflections $S$ is equipped with a fixed bijection $S \simto I$ (denoted by $s \mapsto i_s$), and for distinct $s, t \in S$, the order of $st$ (necessarily $2$, $3$, $4$, $6$, or $\infty$) is determined by the integer $a_{i_si_t}a_{i_ti_s}$. (See, for instance,~\cite[\S3.1]{tits} for details.)

Let $\bk$ be an integral domain. Using the Kac--Moody root datum, we can construct a realization of $(W,S)$ over $\bk$ as follows: we set $V := \bk \otimes_\Z \bX^*$ (so that $V^*$ is identified with $\bk \otimes \bX$), and for $s \in S$, we define $\alpha_s$, resp.~$\alpha_s^\vee$, to be the image of $\alpha_{i_s}$, resp.~$\alpha^\vee_{i_s}$, in $V^*$, resp.~in $V$. This realization is always balanced, but it might not satisfy Demazure surjectivity.

More precisely, let us define $\Zdem$\index{Zdem@$\Zdem$} to be $\Z$ if the maps $\alpha_i : \bX^* \to \Z$ and $\alpha_i^\vee : \bX \to \Z$ are surjective for all $i \in I$, and as $\Z[\frac{1}{2}]$ otherwise. Then Demazure surjectivity holds if there exists a ring morphism $\Zdem \to \bk$.

A realization of $(W,S)$ obtained in this way is called a \emph{Cartan realization}\index{Cartan realization}\index{realization!Cartan}.  From now on, we assume that our Coxeter group is crystallographic, that we are working with a Cartan realization, and that there exists a ring morphism $\Zdem \to \bk$. Note that the discussion in Section~\ref{sec:JW} shows that the assumptions ensuring that the Elias--Williamson category is well behaved are always satisfied for such realizations.

\section{Parity complexes on flag varieties}
\label{sec:parity-complexes}

To $A$ and the given root datum $(I,\bX,\{\alpha_i : i \in I\}, \{\alpha_i^\vee : i \in I\})$, one can associate (following Mathieu~\cite{mathieu-KM}) a Kac--Moody group $\GKM_\Z$ with a canonical subgroup $\BKM_\Z$. Here, $\GKM_\Z$ is a group ind-scheme over $\Z$, and $\BKM_\Z$ is a subgroup scheme. Let $\UKM_\Z$ be the pro-unipotent radical of $\BKM_\Z$. We denote by $\GKM$, resp.~$\BKM$, resp.~$\UKM$, the base-change of $\GKM_\Z$, resp.~$\BKM_\Z$, resp.~$\UKM_\Z$, to $\C$. Then the quotient $\GKM/\BKM$ has a natural structure of complex ind-projective ind-variety, and we have a Bruhat decomposition
\[
\GKM/\BKM = \bigsqcup_{w \in W} \mathscr{X}_w
\]
where each $\mathscr{X}_w$ is a $\BKM$-orbit isomorphic to an affine space of dimension $\ell(w)$.  See~\cite[\S 9.1]{rw} for a brief description of this construction, and~\cite{mathieu, mathieu-KM} for full details.

\begin{rmk}
\begin{enumerate}
\item If $A$ is of finite type, then the Kac--Moody root datum is equivalent to a root datum in the ordinary sense, and $\GKM$ is the associated complex connected reductive group.
\item In~\cite{mathieu-KM}, the root datum is assumed to satisfy additional conditions (which, in the finite-type case, amount to looking only at semisimple, simply connected groups).  An explanation of why these conditions can be dropped can be found in~\cite[\S6]{tits} (see also the remarks in~\cite[\S9.1]{rw}).
\end{enumerate}
\end{rmk} 

Let us fix a Noetherian commutative ring $\bk$ of finite global dimension. We will work with the $\BKM$-equivariant derived category of $\bk$-sheaves on $\GKM/\BKM$, usually denoted by $\Db_\BKM(\GKM/\BKM,\bk)$.  In this paper, we will instead use the ``stacky'' notation
\[
\Db(\BGB,\bk).
\]
This category is equipped with a convolution product $\star$, making it into a monoidal category. Following the convention of~\cite{modrap2}, the cohomological shift functor on $\Db(\BGB,\bk)$ will be denoted by
\[
\cF \mapsto \cF\{1\}.
\]

For any expression $\uw$ we have a ``Bott--Samelson resolution''
\[
\nu_{\uw} : \mathsf{BS}(\uw) \to \GKM/\BKM.
\]
We set\index{Euw@{$\cE(\uw)$}}
\[
\cE(\uw) := (\nu_{\uw})_* \underline{\bk}_{\mathsf{BS}(\uw)} \{\ell(\uw)\},
\]
and denote by $\ParityBS(\BGB, \bk)$\index{categories of parity complexes!ParityBSBGB@{$\ParityBS(\BGB, \bk)$}} the full subcategory of $\Db(\BGB,\bk)$ whose objects are of the form $\cE(\uw)\{n\}$ for some expression $\uw$ and some $n \in \Z$.
(These objects are parity complexes in the sense of~\cite{jmw}; see~\cite[Part~3]{rw} for details.) This category is a monoidal subcategory of $\Db(\BGB,\bk)$.  Next, let $\ParityBS^\oplus(\BGB, \bk)$\index{categories of parity complexes!ParityBSBGBp@{$\ParityBS^\oplus(\BGB, \bk)$}} be its additive envelope (see~\S\ref{sec:additive-hull-Diag}). This category inherits from $\ParityBS(\BGB, \bk)$ the structure of a monoidal category.

One can consider similar constructions in the derived category of constructible complexes for the Bruhat stratification, or, equivalently, the $\UKM$-equivariant derived category of $\GKM/\BKM$, denoted by $\Db(\UGBold,\bk)$.  (For a discussion of this equivalence, see, for instance,~\cite[\S5.3]{modrap1}.)  The resulting categories will be denoted\index{categories of parity complexes!ParityBSUGB@{$\ParityBS(\UGBold, \bk)$}}\index{categories of parity complexes!ParityBSUGBp@{$\ParityBS^\oplus(\UGBold, \bk)$}}
\[
\ParityBS(\UGBold, \bk) \quad \text{and} \quad \ParityBS^\oplus(\UGBold, \bk).
\]
The convolution construction provides a right action of the monoidal category $\ParityBS(\BGB, \bk)$ on $\ParityBS(\UGBold, \bk)$, and hence also a right action of the monoidal category $\ParityBS^\oplus(\BGB, \bk)$ on $\ParityBS^\oplus(\UGBold, \bk)$.

In the case when
$\bk$ is in addition an integral domain admitting a ring morphism $\Zdem \to \bk$, the results of~\cite[Part~3]{rw} explain the relation between these categories and the Elias--Williamson category of~\S\ref{sec:ew-diagram} associated with the realization considered in~\S\ref{sec:Cartan-realizations}. More precisely, by~\cite[Theorem~10.6]{rw} there exists a canonical equivalence of monoidal categories
\[
\DiagBS(\fh,W) \simto \ParityBS(\BGB, \bk)
\]
which intertwines the shift functor $(1)$ on $\DiagBS(\fh,W)$ and the cohomological shift $\{1\}$ on $\ParityBS(\BGB, \bk)$ and sends $B_\uw$ to $\cE(\uw)$. This equivalence induces an equivalence of additive monoidal categories
\[
\DiagBSp(\fh,W) \simto \ParityBS^\oplus(\BGB, \bk),
\]
and equivalences of module categories
\[
\oDiagBS(\fh,W) \simto \ParityBS(\UGBold, \bk), \quad \oDiagBSp(\fh,W) \simto \ParityBS^\oplus(\UGBold, \bk).
\]

If $\bk$ is a field or a complete local ring, we will also denote by\index{categories of parity complexes!ParityBGB@{$\Parity(\BGB, \bk)$}}\index{categories of parity complexes!ParityUGB@{$\Parity(\UGBold, \bk)$}}
\[
\Parity(\BGB, \bk), \quad \text{resp.} \quad \Parity(\UGBold, \bk),
\]
the category of all parity complexes in $\Db(\BGB,\bk)$, resp.~in $\Db(\UGBold,\bk)$. The results of~\cite{jmw} show that $\Parity(\BGB, \bk)$ identifies with the Karoubian envelope of $\ParityBS(\BGB, \bk)$, and that $\Parity(\UGBold, \bk)$ identifies with the Karoubian envelope of $\ParityBS(\UGBold, \bk)$.
As a consequence, the equivalences considered above induce equivalences
\[
\Diag(\fh,W) \simto \Parity(\BGB, \bk), \qquad \oDiag(\fh,W) \simto \Parity(\UGBold, \bk).
\]

\section{Parity sequences}
\label{sec:parity-sequences}

Via the equivalences recalled in~\S\ref{sec:parity-complexes}, all the work carried out in the preceding chapters can be transferred to the setting of (sequences of) parity complexes on $\GKM/\BKM$.  In contrast, the results in this chapter and the last chapter are proved here only for realizations that come from the Kac--Moody setting.  To avoid confusion about the applicability of these results, we now introduce new notation for the main categories of interest.

The $\ParityBS^\oplus(\BGB, \bk)$-sequences will be called \emph{Bott--Samelson parity sequences}. Similarly,
if $\bk$ is a field or a complete local ring, 
the $\Parity(\BGB,\bk)$-sequences will be called
\emph{parity sequences}.\index{parity sequence}

In both settings, we define three shift-of-grading functors as follows:\index{shifts!d@{$\{1\}$}}\index{shifts!b@{$\langle 1\rangle$}}\index{shifts!c@{$[1]$}}
\[
\cF[n]^i = \cF^{i+n},
\qquad
\cF\la n\ra^i = \cF^{i+n}\{-n\},
\qquad
\cF\{n\}^i = \cF^i\{n\}.
\]
These shift functors commute with each other and are related by the formula $\la 1 \ra = [1]\{-1\}$. Following~\cite{modrap2}, in this setting the functor $\langle 1 \rangle$ will be called \emph{Tate twist}\index{Tate twist}.

In a slight abuse of language, when no confusion is likely, we will sometimes omit the modifier ``Bott--Samelson,'' and simply use the term ``parity sequence'' in both of the above situations.

Given two (Bott--Samelson) parity sequences $\cF = (\cF^i)_{i \in \Z}$ and $\cG = (\cG^i)_{i \in \Z}$, we define a bigraded $\bk$-module $\uHom_{\BE}(\cF,\cG)$\index{HomuBE@$\uHom_\BE$} by
\[
\uHom_{\BE}(\cF,\cG)^i_j = \prod_{q-p=i-j} \Hom_{\ParityBS(\BGB, \bk)}(\cF^p,\cG^q\{j\}).
\]
We also define\index{HomuRE@$\uHom_\RE$}\index{HomuLM@$\uHom_\LM$}\index{HomuFM@$\uHom_\FM$}
\begin{align*}
\uHom_\RE(\cF,\cG) &:= \bk \otimes_R \uHom_\BE(\cF,\cG), \\
\uHom_\LM(\cF,\cG) &:= \Lambda \otimes \uHom_\BE(\cF,\cG), \\
\uHom_\FM(\cF,\cG) &:= \Lambda \otimes \uHom_\BE(\cF,\cG) \otimes R^\vee.
\end{align*}
In the case where $\cF = \cG$, we also write $\uEnd_\BE(\cF)$, $\uEnd_\RE(\cF)$, etc., for these bigraded $\bk$-modules.

We can now define a number of categories analogous to those we studied in Chapters~\ref{chap:soergel-diagrams}--\ref{chap:fm-complexes}.  

\begin{defn}
The \emph{biequivariant mixed derived category}\index{biequivariant} associated to $\GKM$, denoted by\index{categories of parity sequences!DmixBGB@{$\Dmix(\BGB,\bk)$}}
\[
\Dmix(\BGB,\bk),
\]
is the category whose objects are pairs $(\cF,\delta)$, where $\cF$ is a Bott--Samelson parity sequence, and $\delta \in \uEnd_\BE(\cF)^1_0$ is an element such that $\delta \circ \delta = 0$.

The \emph{right-equivariant mixed derived category}\index{right-equivariant} associated to $\GKM$, denoted by\index{categories of parity sequences!DmixUGBold@{$\Dmix(\UGBold,\bk)$}}
\[
\Dmix(\UGBold,\bk),
\]
is the category whose objects are pairs $(\cF,\delta)$, where $\cF$ is a Bott--Samelson parity sequence, and $\delta \in \uEnd_\RE(\cF)^1_0$ is an element such that $\delta \circ \delta = 0$.

The \emph{left-monodromic mixed derived category}\index{left-monodromic} associated to $\GKM$, denoted by\index{categories of parity sequences!DmixUGB@{$\Dmix(\UGB,\bk)$}}
\[
\Dmix(\UGB,\bk),
\]
is the category whose objects are pairs $(\cF,\delta)$, where $\cF$ is a Bott--Samelson parity sequence, and $\delta \in \uEnd_\LM(\cF)^1_0$ is an element such that $\delta \circ \delta + \kappa(\delta) = 0$.

The \emph{free-monodromic mixed derived category}\index{free-monodromic} associated to $\GKM$, denoted by\index{categories of parity sequences!DmixUGU@{$\Dmix(\UGU,\bk)$}}
\[
\Dmix(\UGU,\bk),
\]
is the category whose objects are pairs $(\cF,\delta)$, where $\cF$ is a Bott--Samelson parity sequence, and $\delta \in \uEnd_\FM(\cF)^1_0$ is an element such that $\delta \circ \delta + \kappa(\delta) = \Theta_\cF$.

In each case, the morphisms and composition maps are defined in terms of certain complexes, just like in Chapters~\ref{chap:soergel-diagrams}--\ref{chap:fm-complexes}.
\end{defn}

These categories come equipped with canonical equivalences to their counterparts $\BE(\fh,W)$, $\RE(\fh,W)$, etc., considered in Chapters~\ref{chap:soergel-diagrams}--\ref{chap:fm-complexes}.  For simplicity, we will use much of the same notation as in those chapters, e.g.~for the extension-of-scalars functors $\bk'$ or for the forgetful functors $\ForFMLM$, etc.  The category of convolutive\index{convolutive} objects in $\Dmix(\UGU,\bk)$, resp.~$\Dmix(\UGB,\bk)$, is denoted by $\Conv(\UGU, \bk)$,\index{ConvUGU@{$\Conv(\UGU,\bk)$}} resp.~$\Conv(\UGB,\bk)$,\index{ConvUGB@{$\Conv(\UGB,\bk)$}}.

When $\bk$ is a field or a complete local ring, one can repeat these definitions using objects of $\Parity(\BGB,\bk)$ rather than $\ParityBS^\oplus(\BGB,\bk)$.  As in~\S\ref{sec:karoubian1} and~\S\ref{sec:karoubian2}, the resulting categories are denoted by
\begin{gather*}
\Dmix(\BGB,\bk)_\Kar,
\qquad
\Dmix(\UGBold,\bk)_\Kar,
\\
\Dmix(\UGB,\bk)_\Kar,
\qquad
\Dmix(\UGU,\bk)_\Kar,
\end{gather*}
respectively.  The proofs of Lemmas~\ref{lem:idem-equiv} and~\ref{lem:idem-equiv-fm} yield the following.

\begin{lem}
\label{lem:idem-equiv-km}
The obvious functors
\begin{align*}
\Dmix(\BGB,\bk ) &\to \Dmix(\BGB,\bk)_\Kar, \\
\Dmix(\UGBold,\bk) &\to \Dmix(\UGBold,\bk)_\Kar, \\
\Dmix(\UGB,\bk) &\to \Dmix(\UGB,\bk)_\Kar
\end{align*}
are equivalences of categories.  Similarly, the obvious functor
\[
\Dmix(\UGU,\bk) \to \Dmix(\UGU,\bk)_\Kar
\]
is at least fully faithful.
\end{lem}

(As with Lemma~\ref{lem:idem-equiv-fm}, the last of these functors is expected to be an equivalence as well.)

\section{Mixed perverse sheaves}

When $\bk$ is a field or a complete local ring, the category $\Dmix(\BGB,\bk)_\Kar$ coincides with the one that was denoted $\Dmix_\BKM(\GKM/\BKM, \bk)$ in~\cite[\S3.5 and~\S4.3]{modrap2}.  Likewise, the category $\Dmix(\UGBold,\bk)_\Kar$ defined here coincides with the category denoted by $\Dmix_{(\BKM)}(\GKM/\BKM, \bk)$ in~\cite[\S2.1 and~\S4.1]{modrap2}.

\begin{rmk}
In~\cite{modrap2} the Tate twist $\langle 1 \rangle$ on $\Dmix_{\BKM}(\GKM/\BKM, \bk)$ is defined by $\{-1\}[1]$, and hence sends (in the present notation) a pair $(\cF, \delta)$ to $(\cF \langle 1 \rangle, -\delta)$. This change of convention is harmless since the two functors are easily seen to be isomorphic. 
\end{rmk}

One of the main topics of~\cite{modrap2} was the definition and study of the \emph{perverse t-structure}\index{perverse t-structure} on these two categories.  In this section, we review some facts about these t-structures, whose hearts are denoted by
\[
\Perv^\mix(\BGB,\bk)
\qquad\text{and}\qquad
\Perv^\mix(\UGBold,\bk),
\]
respectively.  We will make implicit use of the equivalences in Lemma~\ref{lem:idem-equiv-km}, and regard these as subcategories of $\Dmix(\BGB,\bk)$ and $\Dmix(\UGBold,\bk)$, respectively.  We assume throughout this section that $\bk$ is a field or a complete local ring.

For any $w \in W$, there are objects
\begin{align*}
\Delta_w, \nabla_w &\in \Perv^\mix(\BGB,\bk),  \\
\ForBERE(\Delta_w), \ForBERE(\nabla_w) &\in \Perv^\mix(\UGBold,\bk),
\end{align*}
called \emph{standard} and \emph{costandard perverse sheaves}, respectively.\index{standard object!Deltaw@$\Delta_w$}\index{costandard object!nablaw@$\nabla_w$}  For an explicit definition of the latter, see~\cite[\S3.1]{modrap2}.  (In~\cite{modrap2}, the same notation is used in both $\Perv^\mix(\BGB,\bk)$ and $\Perv^\mix(\UGBold,\bk)$; but in this paper, it will be convenient to maintain the distinction between them.)  When $w$ is a simple reflection, the proof of~\cite[Lemma~2.4]{modrap2} shows that these objects agree with those defined in Example~\ref{ex:std-costd}.

The adjunction properties established in~\cite[\S2]{modrap2} imply that
\begin{equation}
\label{eqn:Delta-gEnd}
\gEnd_\RE(\ForBERE(\Delta_w)) \cong \gEnd_\RE(\ForBERE(\nabla_w)) \cong \bk.
\end{equation}
In addition, according to~\cite[Lemma~3.2]{modrap2}, we have
\begin{equation}
\label{eqn:Delta-nabla-gHom}
\gHom_\RE(\ForBERE(\Delta_w), \ForBERE(\nabla_v)) \cong
\begin{cases}
\bk & \text{if $v = w$,} \\
0 & \text{otherwise.}
\end{cases}
\end{equation}

The fact that the objects $\Delta_w$ and $\nabla_w$ (and therefore also the objects $\ForBERE(\Delta_w)$ and $\ForBERE(\nabla_w)$) lie in the heart of the perverse t-structure is proved in~\cite[Proposition~4.6]{modrap2}.  Moreover, according to~\cite[Proposition~4.4(1)]{modrap2}, if $\uw$ is any reduced expression for $w$, then 
\begin{equation}
\label{eqn:D-N-independence}
\Delta_\uw \cong \Delta_w
\qquad\text{and}\qquad
\nabla_\uw \cong \nabla_w
\end{equation}
(where $\Delta_\uw$ and $\nabla_\uw$ are defined in Example~\ref{ex:std-costd}).
In particular, for any $\uw \in \hW$, the objects $\Delta_\uw$ and $\nabla_\uw$ depend only on $\pi(\uw)$.  Similarly, for $\uw \in \hW$, the objects $\ForBERE(\Delta_\uw)$ and $\ForBERE(\nabla_\uw)$ depend only on $\pi(\uw)$.

\begin{lem}
\label{lem:tD-tN-ForUB}
For any 
$\uw \in \hW$, there are isomorphisms
\[
\ForFMLM(\tD_{\uw}) \cong \ForBELM(\Delta_{\pi(\uw)})
\qquad\text{and}\qquad
\ForFMLM(\tN_{\uw}) \cong \ForBELM(\nabla_{\pi(\uw)}).
\]
\end{lem}
\begin{proof}
It is easy to see by induction on the length of $\uw$, using \eqref{eqn:For-hatstar} and Lemma~\ref{lem:convolution-mon-equ}, that $\ForFMLM(\tD_\uw) \cong \ForBELM(\Delta_\uw)$.  Then, 
the discussion of~\cite[Proposition~4.4(1)]{modrap2} above implies that $\ForBELM(\Delta_\uw) \cong \ForBELM(\Delta_{\pi(w)})$.  The same arguments apply to $\tN_\uw$.
\end{proof}

\begin{prop}
\label{prop:tD-tN}
Let $\uv, \uw \in \hW$.
 \begin{enumerate}
  \item
  \label{it:tD-tN-independence}
The objects $\tD_{\uw}$ and $\tN_{\uw}$ depend only on $\pi(\uw)$ (up to isomorphism).
  \item
  \label{it:tD-tN-Hom}
  We have
  \[
   \gHom_\FM(\tD_{\uv}, \tN_{\uw}) \cong \begin{cases}
                                                R^\vee & \text{if $\pi(\uv)=\pi(\uw)$;} \\
                                                0 & \text{otherwise.}
                                               \end{cases}
  \]
 \end{enumerate}
\end{prop}

\begin{proof}
 \eqref{it:tD-tN-independence}
 We treat the case of $\tD_{\uw}$; the case of $\tN_{\uw}$ is similar. Let $\uv$ and $\uw$ be two reduced expressions such that $\pi(\uv)=\pi(\uw)=:w$. We fix isomorphisms $\ForFMLM(\tD_{\uv}) \cong \ForBELM(\Delta_w)$ and $\ForFMLM(\tD_{\uw}) \cong \ForBELM(\Delta_w)$ as in Lemma~\ref{lem:tD-tN-ForUB}. Then we have
 \[
  \gHom_\LM(\ForFMLM(\tD_{\uv}), \ForFMLM(\tD_{\uw})) \cong \gHom_\LM(\ForBELM(\Delta_w), \ForBELM(\Delta_w)) = \bk,
 \]
by~\eqref{eqn:Delta-gEnd}.  Using Lemma~\ref{lem:hom-mon-lmon}, we deduce that
\[
\Hom_{\Dmix(\UGU,\bk)}(\tD_{\uv}, \tD_{\uw})\cong\bk,
\]
and that the functor $\ForFMLM$ induces an isomorphism
\[
 \Hom_{\Dmix(\UGU,\bk)}(\tD_{\uv}, \tD_{\uw}) \simto \End_{\Dmix(\UGB,\bk)}(\ForBELM(\Delta_w)).
\]
Similar comments apply to each of
\begin{multline*}
\Hom_{\Dmix(\UGU,\bk)}(\tD_{\uv}, \tD_{\uv}), \quad
\Hom_{\Dmix(\UGU,\bk)}(\tD_{\uw}, \tD_{\uw}), \\
\Hom_{\Dmix(\UGU,\bk)}(\tD_{\uw}, \tD_{\uv}).
\end{multline*}
Then, if $\varphi : \tD_{\uv} \to \tD_{\uw}$ and $\psi : \tD_{\uw} \to \tD_{\uv}$ are generators of the $\bk$-modules $\Hom_{\Dmix(\UGU,\bk)}(\tD_{\uv}, \tD_{\uw})$ and $\Hom_{\Dmix(\UGU,\bk)}(\tD_{\uw}, \tD_{\uv})$ respectively, we see that $\psi \circ \varphi$ and $\varphi \circ \psi$ are invertible, so that $\varphi$ and $\psi$ are isomorphisms.

Part~\eqref{it:tD-tN-Hom} follows similarly from Lemma~\ref{lem:hom-mon-lmon} and~\eqref{eqn:Delta-nabla-gHom}.
\end{proof}

\begin{rmk}
\label{rmk:generator-Hom-tD-tN}
Let $\uw$ be a reduced expression. The morphism $\tD_\uw \to \tN_\uw$ defined by $p_\uw$ is nonzero by Lemma~\ref{lem:tD-tN-uw}. Hence it is a generator of the free $R^\vee$-module $\gHom_\FM(\tD_\uw, \tN_\uw)$.
\end{rmk}

\section{Tilting perverse sheaves}

According to~\cite[Proposition~3.11]{modrap2}, when $\bk$ is a field, $\Perv^\mix(\UGBold,\bk)$ has a natural structure of a graded highest weight category, with the objects $\Delta_w$ (resp.~$\nabla_w$) as the standard (resp.~costandard) objects. As in~\cite[Proposition~3.14]{modrap2}, one can consider the collection of tilting objects therein.  Let
\[
\Tilt^\mix(\UGBold,\bk)
\]
be the full additive subcategory of $\Perv^\mix(\UGBold,\bk)$ consisting of tilting objects, i.e.~objects which admit a filtration with subquotients of the form $\Delta_w \langle n \rangle$, and a filtration with subquotients of the form $\nabla_w \langle n \rangle$. This category plays a central role in the main results of~\cite{modrap2}.

In this section, we will consider certain subcategories of $\Dmix(\UGU,\bk)$ and $\Dmix(\UGB,\bk)$ that behave similarly to $\Tilt^\mix(\UGBold,\bk)$.  The reader should beware that for $\Dmix(\UGB,\bk)$, the definition does \emph{not} simply invoke Theorem~\ref{thm:leftmon-con} to transfer the notion from $\Dmix(\UGBold,\bk)$.  Instead, the relationship with $\Tilt^\mix(\UGBold,\bk)$ will be a statement that requires proof; see Proposition~\ref{prop:tiltBS-field} below.

We allow $\bk$ to be any Noetherian integral domain of finite global dimension.  Recall that for any expression $\uw=(s_1, \sdots, s_r)$, we have defined in~\S\ref{sec:functoriality-conjecture} the object
\[
\Tmon_\uw := \Tmon_{s_1} \hatstar \cdots \hatstar \Tmon_{s_r} \quad \in \Conv(\UGU,\bk).
\]
We also set\index{tilting object!Tuw@$\cT_\uw$}
\[
\cT_\uw := \ForFMLM(\Tmon_\uw) \quad \in \Conv(\UGB,\bk).
\]
By~\eqref{eqn:For-hatstar}, we have a canonical isomorphism
\begin{equation}
\label{eqn:Tuw}
\cT_\uw \cong \Tmon_{s_1} \hatstar \cdots \hatstar \Tmon_{s_n} \hatstar \cT_1.
\end{equation}
When we want to emphasize the ring of coefficients, we will write $\Tmon_\uw^\bk$ instead of $\Tmon_\uw$, and $\cT_\uw^\bk$ instead of $\cT_\uw$. Then, for any Noetherian integral domain $\bk'$ of finite global dimension and any ring morphism $\bk \to \bk'$, we have canonical isomorphisms
\begin{equation}
\label{eqn:k'-tilting}
\bk'(\Tmon_\uw^\bk) \cong \Tmon_\uw^{\bk'}, \qquad \bk'(\cT_\uw^\bk) \cong \cT_\uw^{\bk'}.
\end{equation}

Copying the definition in~\S\ref{sec:functoriality-conjecture},
we will denote by\index{categories of tilting sheaves!TiltBSUGU@{$\TiltBSp(\UGU, \bk)$}}
\[
\TiltBSp(\UGU, \bk)
\]
the full subcategory of $\Dmix(\UGU,\bk)$ whose objects are direct sums of objects of the form $\Tmon_\uw \langle n \rangle$.
Similarly, we will denote by\index{categories of tilting sheaves!TiltBSUGB@{$\TiltBSp(\UGB, \bk)$}}
\[
\TiltBSp(\UGB, \bk)
\]
the full subcategory of $\Dmix(\UGB,\bk)$ whose objects are direct sums of objects of the form $\cT_\uw \langle n \rangle$.
When $\bk$ is a field or a complete local ring, we will also work with the Karoubian envelopes of these categories, denoted by
\begin{align*}
\Tilt(\UGU,\bk) &:= \Kar(\TiltBSp(\UGU,\bk)), \\
\Tilt(\UGB,\bk) &:= \Kar(\TiltBSp(\UGB,\bk)),
\end{align*}
respectively.  By construction, the functor $\ForFMLM$ restricts to a functor from the category $\TiltBSp(\UGU,\bk)$ to $\TiltBSp(\UGB,\bk)$, and from $\Tilt(\UGU,\bk)$ to $\Tilt(\UGB,\bk)$.

Objects of the form $\Tmon_\uw$ or $\cT_\uw$ will be called \emph{Bott--Samelson tilting perverse sheaves}, and more generally, objects of $\Tilt(\UGU,\bk)$ or $\Tilt(\UGB,\bk)$ will be called (free-monodromic or left-monodromic) \emph{tilting perverse sheaves}.  (Note, however, that we have \emph{not} defined a perverse t-structure on $\Dmix(\UGU,\bk)$, or even on $\Dmix(\UGB,\bk)$ for general $\bk$.)

The main result of this section is the following.

\begin{prop}
\label{prop:tiltBS-field}
Assume that $\bk$ is a field. Then the functor $\ForLMRE$ induces an equivalence of additive categories
\[
\ForLMRE: \Tilt(\UGB,\bk) \simto \Tilt^\mix(\UGBold,\bk).
\]
\end{prop}

\begin{rmk}
Proposition~\ref{prop:tiltBS-field} holds more generally when $\bk$ is a complete local ring. But for simplicity we restrict to the case of fields.
\end{rmk}

Before proving Proposition~\ref{prop:tiltBS-field}, we require several lemmas.  For brevity, given $s \in S$, we denote by $C_s$ the functor $C_{\Tmon_s}$ from Proposition~\ref{prop:convolution-Dmix}.

\begin{lem}
\label{lem:Cs-Delta-nabla}
Assume that $\bk$ is a field, and
let $w \in W$ and $s \in S$.
\begin{enumerate}
\item
\label{it:ses-Cs-delta-nabla-1}
If $sw>w$, then there exist distinguished triangles
\begin{multline*}
\ForBELM(\Delta_{sw}) \to C_s(\ForBELM(\Delta_w)) \to \ForBELM(\Delta_w) \langle 1 \rangle \xrightarrow{[1]} \\
\text{and} \quad
\ForBELM(\nabla_{sw}) \langle -1 \rangle \to C_s(\ForBELM(\nabla_w)) \to \ForBELM(\nabla_{sw}) \xrightarrow{[1]}
\end{multline*}
in $\Dmix(\UGB,\bk)$.
\item
\label{it:ses-Cs-delta-nabla-2}
If $sw<w$, then there exist distinguished triangles
\begin{multline*}
\ForBELM(\Delta_w) \langle -1 \rangle \to C_s(\ForBELM(\Delta_w)) \to \ForBELM(\Delta_{sw}) \xrightarrow{[1]} \\
\text{and} \quad
\ForBELM(\nabla_{sw}) \to C_s(\ForBELM(\nabla_w)) \to \ForBELM(\nabla_w) \langle 1 \rangle \xrightarrow{[1]}
\end{multline*}
in $\Dmix(\UGB,\bk)$.
\end{enumerate}
\end{lem}

\begin{proof}
Since $\ForBELM(\Delta_w)$ and $\ForBELM(\nabla_w)$ belong to $\Conv(\UGB)$, we have
\[
C_s(\ForBELM(\Delta_w)) \cong \Tmon_s \hatstar \ForBELM(\Delta_w), \quad C_s(\ForBELM(\nabla_w)) \cong \Tmon_s \hatstar \ForBELM(\nabla_w).
\]
Then using Lemma~\ref{lem:convolution-mon-equ} we deduce isomorphisms
\[
C_s(\ForBELM(\Delta_w)) \cong \cT_s \ustar \Delta_w, \quad C_s(\ForBELM(\nabla_w)) \cong \cT_s \ustar \nabla_w.
\]
Now it is well known (and easy to check by hand) that there exist exact sequences
\[
\ForBERE(\Delta_s) \hookrightarrow \cT'_s \twoheadrightarrow \ForBERE(\Delta_1) \langle 1 \rangle \quad \text{and} \quad \ForBERE(\nabla_1) \langle -1 \rangle \hookrightarrow \cT'_s \twoheadrightarrow \ForBERE(\nabla_s)
\]
in $\Perv^\mix(\UGBold,\bk)$, and hence distinguished triangles
\[
\ForBELM(\Delta_s) \to \cT_s \to \ForBELM(\Delta_1) \langle 1 \rangle \xrightarrow{[1]} \ \text{and} \ \ForBELM(\nabla_1) \langle -1 \rangle \to \cT_s \to \ForBELM(\nabla_s) \xrightarrow{[1]}
\]
in $\Dmix(\UGB,\bk)$. The wished-for triangles are obtained by applying the functor $(-) \ustar \Delta_w$ or $(-) \ustar \nabla_w$ and using~\cite[Proposition~4.4]{modrap2}.
\end{proof}

\begin{lem}
\label{lem:properties-Cs}
Assume that $\bk$ is a field.
If $\cF$ belongs to $\Dmix(\UGB,\bk)$ and $\ForLMRE(\cF)$ is a tilting mixed perverse sheaf, then $\ForLMRE(C_s(\cF))$ is a tilting mixed perverse sheaf.
\end{lem}

\begin{proof}
If $\ForLMRE(\cF)$ is a tilting mixed perverse sheaf, then there exist $k \in \Z_{\geq 0}$, $w_1, \sdots, w_k \in W$, $n_1, \sdots, n_k \in \Z$, $\cF_1, \sdots, \cF_{k+1}$ in $\Dmix(\UGB,\bk)$, and distinguished triangles
\[
\ForBELM(\Delta_{w_i}) \langle n_i \rangle \to \cF_i \to \cF_{i+1} \xrightarrow{[1]}
\]
for $i \in \{1, \sdots, k\}$,
with $\cF_1 = \cF$ and $\cF_{k+1} = 0$. Applying $\ForLMRE \circ C_s$ to these triangles, and using Lemma~\ref{lem:Cs-Delta-nabla}, we see (by decreasing induction on $i$) that each $\ForLMRE(C_s(\cF_i))$ is a mixed perverse sheaf which admits a standard filtration. In particular, this is the case for $\ForLMRE(C_s(\cF))$. Similar arguments show that $\ForLMRE(C_s(\cF))$ also admits a costandard filtration; hence it is tilting.
\end{proof}

\begin{proof}[{Proof of Proposition~\ref{prop:tiltBS-field}}]
Lemma~\ref{lem:properties-Cs} and~\eqref{eqn:Tuw} imply that $\ForLMRE$ sends objects of the form $\Tmon_\uw\la n\ra$ to objects of $\Tilt^\mix(\UGBold,\bk)$.  Since the latter is Karoubian, we see that $\ForLMRE$ does indeed induce a functor
\[
\Tilt(\UGB,\bk) \to \Tilt^\mix(\UGBold,\bk).
\]
It is immediate from Theorem~\ref{thm:leftmon-con} that this functor is fully faithful.  A routine support argument shows that it is essentially surjective.
\end{proof}

Transferring the known properties of the category $\Tilt^\mix(\UGBold,\bk)$ (proved in~\cite{modrap2}) to the category $\Tilt(\UGB,\bk)$, we deduce the following result.

\begin{cor}
\label{cor:lift-tilt-lm}
Assume that $\bk$ is a field.  The category $\Tilt(\UGB,\bk)$ is Krull--Schmidt. For any $w \in W$, there exists a unique (up to isomorphism) indecomposable object $\cT_w$ characterized by the following properties:
\begin{enumerate}
\item for any $\uw \in \hW$ with $\pi(\uw) = w$, $\cT_w$ occurs as a direct summand of $\cT_\uw$ with multiplicity $1$;
\item $\cT_w$ does not occur as a direct summand of any $\cT_\uv \langle n \rangle$ with $\ell(\uv) < \ell(w)$.
\end{enumerate}
Moreover, the assignment $(w,n) \mapsto \cT_w \langle n \rangle$ induces a bijection between $W \times \Z$ and the set of isomorphism classes of indecomposable objects in $\Tilt(\UGB,\bk)$.
\end{cor}

\section{Morphisms between Bott--Samelson tilting perverse sheaves}

In this section, $\bk$ is again an arbitrary Noetherian integral domain
of finite global dimension such that there exists a ring morphism $\Zdem \to \bk$.

\begin{prop}
\label{prop:Hom-Tilt-lmon}
For any expressions $\uv$, $\uw$ and any $i,j \in \Z$, we have
\begin{equation}
\label{eqn:Hom-Tilt-lmon}
\Hom_{\Dmix(\UGB,\bk)}(\cT_\uv, \cT_\uw [i] \langle j \rangle) = 0 \quad \text{unless $i=0$,}
\end{equation}
and each $\Hom_{\Dmix(\UGB,\bk)}(\cT_\uv, \cT_\uw \langle j \rangle)$ is free over $\bk$. Moreover, for any Noetherian integral domain $\bk'$ of finite global dimension and any ring morphism $\bk \to \bk'$, the functor $\bk'$ induces an isomorphism
\[
\bk' \otimes_\bk \Hom_{\Dmix(\UGB, \bk)}(\cT^\bk_\uv, \cT^\bk_\uw \langle j \rangle) \simto \Hom_{\Dmix(\UGB, \bk')}(\cT^{\bk'}_\uv, \cT^{\bk'}_\uw \langle j \rangle).
\]
\end{prop}

\begin{proof}
We start by proving~\eqref{eqn:Hom-Tilt-lmon}.
First, assume that $\bk$ is a field. Then the claim follows from Theorem~\ref{thm:leftmon-con}, Proposition~\ref{prop:tiltBS-field}, and well-known properties of tilting mixed perverse sheaves (see~\cite[\S\S3.2--3.4]{modrap2}).

Now we consider the case $\bk=\Zdem$. The complex $\uHom_\LM(\cT_\uv^\Zdem, \cT_\uw^\Zdem)$ is a complex of free $\Zdem$-modules of finite rank, with bounded cohomology (by Theorem~\ref{thm:leftmon-con}). We claim that this complex is K-flat in the sense of~\cite{spaltenstein}, so that for any ring $\bk'$ we have
\begin{equation}
\label{eqn:uHom-Z-bk'}
\bk' \mathbin{\mathop{\otimes}^L}_\Zdem \uHom_\LM(\cT_\uv^\Zdem, \cT_\uw^\Zdem) \cong \bk' \otimes_\Zdem \uHom_\LM(\cT_\uv^\Zdem, \cT_\uw^\Zdem)
\end{equation}
in the derived category of $\bk'$-modules. 
In fact, we will show more generally that any complex of free $\Zdem$-modules with bounded cohomology is K-flat. Choose an integer $n$ such that $\coH^i(M)=0$ for any $i \geq n$. Since $M^{n+1}$ is free of finite rank, $M^n / \ker(d^n) \cong \mathrm{im}(d^n) \subset M^{n+1}$ is free, so that $\ker(d^n)$ is a direct summand of $M^n$. Let $N^n \subset M^n$ be a complement of $\ker(d^n)$. Then
\[
M \cong (\cdots \to M^{n-1} \to \ker(d^n) \to 0 \to \cdots) \oplus (\cdots \to 0 \to N^n \to M^{n+1} \to \cdots)
\]
as complexes of $\Zdem$-modules. The first term on the right-hand side is a bounded above complex of free $\Zdem$-modules, so it is K-flat. On the other hand, the same considerations as above show that the second term is a direct sum of complexes of the form
\[
\cdots \to 0 \to \Zdem \xrightarrow{\id} \Zdem \to 0 \to \cdots,
\]
which are obviously K-flat. Since any direct sum of K-flat complexes is K-flat, this proves the claim.

Any object in the bounded derived category of $\Zdem$-modules is isomorphic to its cohomology, so that we can rewrite~\eqref{eqn:uHom-Z-bk'} as an isomorphism
\[
\bk' \mathbin{\mathop{\otimes}^L}_\Zdem \gHom_\LM(\cT_\uv^\Zdem, \cT_\uw^\Zdem) \cong \bk' \otimes_\Zdem \uHom_\LM(\cT_\uv^\Zdem, \cT_\uw^\Zdem)
\]
in the derived category of $\bk'$-modules.
Moreover, by~\eqref{eqn:ext-scalars-uHomU} and~\eqref{eqn:k'-tilting}, the right-hand side is isomorphic to $\uHom_\LM(\cT_\uv^{\bk'}, \cT_\uw^{\bk'})$. In particular, if $\bk'$ is a field, the proposition in the case of $\bk'$ proved above and Remark~\ref{rmk:Hom-shift} imply that the complex $\bk' \mathbin{\mathop{\otimes}^L_\Zdem} \gHom_\LM(\cT_\uv^\Zdem, \cT_\uw^\Zdem)$ is concentrated in degree $0$. Since this property is satisfied for any field, we deduce that $\gHom_\LM(\cT_\uv^\Zdem, \cT_\uw^\Zdem)$ is concentrated in degree $0$, and free over $\Zdem$, proving the proposition over $\Zdem$.

Finally we treat the general case. As above we have an isomorphism
\[
\bk \mathbin{\mathop{\otimes}^L}_\Zdem \gHom_\LM(\cT_\uv^\Zdem, \cT_\uw^\Zdem) \cong \uHom_\LM(\cT_\uv^\bk, \cT_\uw^\bk)
\]
in the derived category of $\bk$-modules. Since $\gHom_\LM(\cT_\uv^\Zdem, \cT_\uw^\Zdem)$ is concentrated in degree $0$ and free over $\Zdem$, this says that the complex $\uHom_\LM(\cT_\uv^\bk, \cT_\uw^\bk)$ is quasi-isomorphic to one that is concentrated in degree $0$, and free over $\bk$. Using Remark~\ref{rmk:Hom-shift} again, we deduce~\eqref{eqn:Hom-Tilt-lmon}.

We also obtain a canonical isomorphism
\[
\bk \otimes_\Zdem \Hom_{\Dmix(\UGB, \Zdem)}(\cT_\uv^\Zdem, \cT_\uw^\Zdem \langle j \rangle) \simto \Hom_{\Dmix(\UGB, \bk)}(\cT_\uv^\bk, \cT_\uw^\bk \langle j \rangle)
\]
for any $j \in \Z$. The claim about the functor $\bk'$ follows.
\end{proof}

\begin{cor}
\label{cor:Hom-Tilt-mon}
For any expressions $\uv$, $\uw$ and any $i,j \in \Z$, we have
\[
\gHom_\FM(\Tmon_\uv, \Tmon_\uw)^i_j = 0 \quad \text{unless $i=0$,}
\]
$\gHom_\FM(\Tmon_\uv, \Tmon_\uw)^0_\bullet$ is graded free as a right $R^\vee$-module, and the morphism
\[
\gHom_\FM(\Tmon_\uv, \Tmon_\uw)^0_\bullet \otimes_{R^\vee} \bk \to \gHom_\LM(\cT_\uv, \cT_\uw)^0_\bullet
\]
induced by the functor $\ForFMLM$ is an isomorphism.
Finally, for any Noetherian integral domain $\bk'$ of finite global dimension and any ring morphism $\bk \to \bk'$, the functor $\bk'$ induces an isomorphism
\[
\bk' \otimes_\bk \Hom_{\Dmix(\UGU, \bk)}(\Tmon^\bk_\uv, \Tmon^\bk_\uw \langle j \rangle) \simto \Hom_{\Dmix(\UGU, \bk')}(\Tmon^{\bk'}_\uv, \Tmon^{\bk'}_\uw \langle j \rangle)
\]
for any $j \in \Z$.
\end{cor}

\begin{proof}
All the claims except the last one
follow from Lemma~\ref{lem:hom-mon-lmon} and Proposition~\ref{prop:Hom-Tilt-lmon}. The last claim can be deduced from the other ones and the corresponding property in Proposition~\ref{prop:Hom-Tilt-lmon}.
\end{proof}

\section{Lifting indecomposable tilting perverse sheaves}
\label{sec:lifting-tilting}

In this section, we upgrade Corollary~\ref{cor:lift-tilt-lm} to the following statement about $\Tilt(\UGU,\bk)$. (This result is not needed for the present paper, but will be used in~\cite{mkdkm}.)

\begin{thm}
\label{thm:lift-tilt-mon}
Assume that $\bk$ is a field.  The category $\Tilt(\UGU,\bk)$ is Krull--Schmidt. For any $w \in W$, there exists a unique (up to isomorphism) indecomposable object $\Tmon_w$ such that $\ForFMLM(\Tmon_w) \cong \cT_w$.  In addition, $\Tmon_w$ is characterized by the following properties:
\begin{enumerate}
\item for any $\uw \in \hW$ with $\pi(\uw) = w$, $\Tmon_w$ occurs as a direct summand of $\Tmon_\uw$ with multiplicity $1$;
\item $\Tmon_w$ does not occur as a direct summand of any $\Tmon_\uv \langle n \rangle$ with $\ell(\uv) < \ell(w)$.
\end{enumerate}
Moreover, the assignment $(w,n) \mapsto \Tmon_w \langle n \rangle$ induces a bijection between $W \times \Z$ and the set of isomorphism classes of indecomposable objects in $\Tilt(\UGU,\bk)$.
\end{thm}

\begin{proof}
The fact that $\Tilt(\UGU,\bk)$ is Krull--Schmidt follows from~\cite[Corollary~A.2]{cyz}. We claim that for any $\cF, \cG$ in $\Tilt(\UGU, \bk)$, the morphism
\[
\Hom_{\Tilt(\UGU, \bk)}(\cF, \cG) \to \Hom_{\Tilt(\UGB, \bk)}(\ForFMLM(\cF), \ForFMLM(\cG))
\]
induced by $\ForFMLM$ is surjective. Indeed, if $\cF$ and $\cG$ are of the form $\Tmon_\uw$,
then this claim follows from Corollary~\ref{cor:Hom-Tilt-mon}. The general case follows, since any object of the category $\Tilt(\UGU,\bk)$ is isomorphic to a direct summand of a direct sum of Tate twists of such objects.
It follows in particular from this claim that if $\cF$ is an indecomposable object of $\Tilt(\UGU,\bk)$, then $\ForFMLM(\cF)$ is an indecomposable object of $\Tilt(\UGB,\bk)$ (since its endomorphism ring is local, as a quotient of a local ring).

Now let $w \in W$, choose some reduced expression $\uw$ for $w$, and consider the decomposition of $\Tmon_\uw$ into indecomposable direct summands. The image under $\ForFMLM$ of this decomposition is the decomposition of $\cT_\uw$ into indecomposable direct summands, so that $\Tmon_\uw$ admits an indecomposable direct summand whose image under $\ForFMLM$ is $\cT_w$. This shows the existence of $\Tmon_w$.

We fix a choice of $\Tmon_w$ for any $w \in W$. Then to conclude the proof we only need to prove that any indecomposable object in $\Tilt(\UGU, \bk)$ is isomorphic to $\Tmon_w \langle n \rangle$ for some $(w,n) \in W \times \Z$. Let $\cF$ be such an indecomposable object. Then by the classification of indecomposable objects in $\Tilt(\UGB, \bk)$ (see Corollary~\ref{cor:lift-tilt-lm}), we know that there exists $(w,n) \in W \times \Z$ such that $\ForFMLM(\cF) \cong \cT_w \langle n \rangle$. Fix some inverse isomorphisms $f : \ForFMLM(\cF) \simto \cT_w \langle n \rangle$ and $g : \cT_w \langle n \rangle \simto \ForFMLM(\cF)$. By the claim at the beginning of the proof, there exist morphisms $\tilde f : \cF \to \Tmon_w \langle n \rangle$ and $\tilde g : \Tmon_w \langle n \rangle \to \cF$ such that $f = \ForFMLM(\tilde f)$ and $g = \ForFMLM(\tilde g)$. Then $\tilde g \circ \tilde f$ is a degree $(0,0)$ element of $\gEnd_\FM(\cF)$ which is equal to $\id_\cF$ modulo $\ker(\epsilon_{R^\vee})$. 
Since $\gEnd_\FM(\cF)$ is bounded above for the internal grading, degree considerations show that any element in $\gEnd_\FM(\cF)^0_0$ which belongs to $\ker(\epsilon_{R^\vee}) \cdot \gEnd_\FM(\cF)$ is nilpotent. Hence $\tilde g \circ \tilde f - \id_\cF$ is nilpotent, which implies that $\tilde g \circ \tilde f$ is invertible. Similarly $\tilde f \circ \tilde g$ is invertible, and so $\tilde f$ and $\tilde g$ are invertible. In particular, $\cF$ is isomorphic to $\Tmon_w \langle n \rangle$.
\end{proof}

\begin{rmk}
The object $\Tmon_s$ that was defined in~\S\ref{sss:ts-freemon} is clearly a special case of the object $\Tmon_w$ introduced in Theorem~\ref{thm:lift-tilt-mon}, so that there is no conflict of notation. (Unlike $\Tmon_s$, however, the object $\Tmon_w$ cannot, in general, be chosen in a canonical way.)
\end{rmk}

\chapter{Proof of functoriality in the Kac--Moody case}
\label{chap:functoriality}

In this chapter, we finally prove that, if $\bk$ is Noetherian and of finite global dimension, the operation $\hatstar$ defines a bifunctor on $\TiltBSp(\UGU,\bk)$, making it into a monoidal category.

\section{Localization}
\label{sec:localization}

In this section we assume that $\bk$ is a field. Note that the Demazure surjectivity assumption implies that the images of simple coroots in $R^\vee$ are nonzero.
We denote by $Q^\vee$ the localization of the ring $R^\vee$ at the multiplicative subset generated by the $W$-conjugates of the images of all simple coroots $\alpha_s^\vee$ ($s \in S$). This ring has a natural bigrading induced by that of $R^\vee$.

If $\cF$ and $\cG$ are parity sequences, we set\index{Homuloc@$\uHom_\loc$}
\[
 \uHom_\loc(\cF, \cG) := \uHom_\FM(\cF, \cG) \otimes_{R^\vee} Q^\vee = \Lambda \otimes \uHom_\BE(\cF,\cG) \otimes Q^\vee.
\]
If $\cF$ and $\cG$ are in $\Dmix(\UGU, \bk)$,
then $\uHom_\loc(\cF, \cG)$ admits a differential induced by that of $\uHom_\FM(\cF,\cG)$, and we set
\[
 \gHom_\loc(\cF,\cG) := \coH^{\bullet}_{\bullet} \bigl( \uHom_\loc(\cF, \cG) \bigr).
\]
Since $Q^\vee$ is flat over $R^\vee$, we have a natural isomorphism of bigraded $Q^\vee$-modules
\begin{equation}
\label{eqn:Hom-mon-loc}
  \gHom_\loc(\cF,\cG) \cong \gHom_\FM(\cF,\cG) \otimes_{R^\vee} Q^\vee.
\end{equation}

It is clear that for $\cF, \cG, \cH$ in $\Dmix(\UGU, \bk)$,
the composition map~\eqref{eqn:composition-mon} induces a similar map
\begin{equation}
\label{eqn:composition-loc}
 \uHom_\loc(\cG,\cH) \otimes \uHom_\loc(\cF,\cG) \to \uHom_\loc(\cF,\cH),
\end{equation}
which is also a morphism of dgg $\bk$-modules. Therefore we can define a category $\Dmix_\loc(\UGU,\bk)$\index{categories of parity sequences!DmixUGUloc@{$\Dmix_\loc(\UGU,\bk)$}} whose objects are the same as those of $\Dmix(\UGU,\bk)$, whose morphism spaces are given by
\[
 \Hom_{\Dmix_\loc(\UGU,\bk)}(\cF,\cG):=\coH^0_0 \bigl( \uHom_\loc(\cF,\cG) \bigr),
\]
and whose composition law is induced by~\eqref{eqn:composition-loc}. The Tate twist $\langle 1 \rangle$ induces a functor on $\Dmix_\loc(\UGU,\bk)$, which will be denoted similarly. If $I \neq \varnothing$, this functor satisfies $\langle 1 \rangle \circ \langle 1 \rangle \cong \id$ (since multiplication by a simple coroot is an isomorphism).

We denote by $\Conv_\loc(\UGU,\bk)$ the full subcategory of $\Dmix_\loc(\UGU,\bk)$ whose objects are the convolutive free-monodromic complexes.

Of course, the same construction can be carried out starting from the category $\Dmix(\UGU,\bk)_\Kar$ instead of $\Dmix(\UGU,\bk)$.  We obtain a new category
\[
\Dmix_\loc(\UGU,\bk)_\Kar,
\]
along with a full subcategory $\Conv_\loc(\UGU,\bk)_\Kar$ of convolutive objects.

If $\cF$ and $\cG$ are in $\Dmix(\UGU, \bk)$, by construction there exists a natural (injective) morphism of dgg $\bk$-modules
\[
 \uHom_\FM(\cF,\cG) \to \uHom_\loc(\cF,\cG).
\]
Taking the induced map on cohomology allows us to define a functor\index{Loc@{$\Loc$}}
\[
 \Loc : \Dmix(\UGU,\bk) \to \Dmix_\loc(\UGU,\bk).
\]
We also obtain a similar functor for the category $\Dmix(\UGU,\bk)_\Kar$, along with the following commutative diagram:
\begin{equation}
\label{eqn:loc-karoubian}
\begin{tikzcd}
\Dmix(\UGU,\bk) \ar[r, "\Loc"] \ar[d] & \Dmix_\loc(\UGU,\bk) \ar[d] \\
\Dmix(\UGU,\bk)_\Kar \ar[r, "\Loc"] & \Dmix_\loc(\UGU,\bk)_\Kar,
\end{tikzcd}
\end{equation}
in which the vertical arrows are fully faithful. (For the left arrow, see Lemma~\ref{lem:idem-equiv-km}; the case of the right arrow can be treated similarly.)

\begin{lem}
\label{lem:Loc-Tilt}
 The restriction of the functor $\Loc$ to $\TiltBSp(\UGU,\bk)$ is faithful.
\end{lem}

\begin{proof}
 The claim follows from Corollary~\ref{cor:Hom-Tilt-mon} and~\eqref{eqn:Hom-mon-loc}.
\end{proof}

\begin{lem}
\label{lem:conv-Loc-isom}
Let $\cF, \cG, \cH$ be convolutive objects in either $\Dmix(\UGU,\bk)$ or $\Dmix(\UGU, \bk)_\Kar$, and let $f : \cG \to \cH$ be a morphism. 
If $\Loc(f)$ is an isomorphism, then $\Loc(\id_\cF \hatstar f)$ is an isomorphism.
Similarly, if $\Loc(f)=0$, then $\Loc(\id_\cF \hatstar f)=0$.
\end{lem}

\begin{proof}
We treat the case of $\Dmix(\UGU, \bk)$; the case of the category $\Dmix(\UGU, \bk)_\Kar$ is similar.
For any 
$\cG'$, $\cH'$ in $\Conv(\UGU, \bk)$, the map
\[
\uHom_\FM(\cG', \cH') \to \uHom_\FM(\cF \hatstar \cG', \cF \hatstar \cH')
\]
defined by $g \mapsto \id_\cF \hatstar g$ is $R^\vee$-linear (for the natural right action). Therefore, it induces a map
\[
\uHom_\loc(\cG', \cH') \to \uHom_\loc(\cF \hatstar \cG', \cF \hatstar \cH').
\]
This operation is compatible with composition (by Lemma~\ref{lem:monconv-weak-interchange}) and with differentials (by Proposition~\ref{prop:monconv-diff}), so it defines a functor
\[
\Conv_\loc(\UGU,\bk) \to \Conv_\loc(\UGU,\bk),
\]
which is compatible with the functor
\[
\cF \hatstar (-) : \Conv(\UGU,\bk) \to \Conv(\UGU,\bk)
\]
in the natural sense. Therefore, the latter functor must send morphisms whose image under $\Loc$ are isomorphisms, resp.~$0$, to morphisms which have the same property.
\end{proof}

\begin{lem}
\label{lem:conv-Loc-isom-tD-tN}
Let $\cF$ and $\cG$ be convolutive objects in either $\Dmix(\UGU, \bk)$ or $\Dmix(\UGU, \bk)_\Kar$,
and let $f : \cF \to \cG$ be a morphism such that $\Loc(f)$ is an isomorphism. 
For any expression $\uw$, the morphisms $\Loc(f \hatstar \id_{\tD_\uw})$ and $\Loc(f \hatstar \id_{\tN_\uw})$ are isomorphisms.
\end{lem}

\begin{proof}
We observe using Lem\-ma~\ref{lem:mu-tD-tN-new} that the map
\[
\uHom_\FM(\cF,\cG) \to \uHom_\FM(\cF \hatstar \tD_\uw, \cG \hatstar \tD_\uw)
\]
defined by $g \mapsto g \hatstar \id_{\tD_\uw}$ is ``twisted $R^\vee$-linear'' for the action of $\pi(\uw)^{-1} \in W$ on $R^\vee$.  That is, for $x \in R^\vee$, this map sends $g \hatstar x$ to $(g \hatstar \id_{\tD_\uw}) \hatstar \pi(\uw)^{-1}(x)$.  (For an analogous statement in the biequivariant setting, see~\cite[Proposition~A.17]{ar:agsr}.) Since the map $R^\vee \to Q^\vee$ is $W$-equivariant, there is still an induced map
\[
\uHom_\loc(\cF,\cG) \to \uHom_\loc(\cF \hatstar \tD_\uw, \cG \hatstar \tD_\uw),
\]
and hence a functor $({-}) \hatstar \tD_\uw: \Conv_\loc(\UGU,\bk) \to \Conv_\loc(\UGU,\bk)$.  Similar considerations apply to $({-}) \hatstar \tN_\uw$.  The rest of the argument proceeds as in the proof of Lemma~\ref{lem:conv-Loc-isom}.
\end{proof}

\section{Images of (co)standard and tilting objects in \texorpdfstring{$\Dmix_\loc(\UGU,\bk)$}{Dmixloc(UGU,k)}}

In this section, we will show that after localization, Bott--Samelson tilting objects split as direct sums of standard (or costandard) objects.  We retain the assumption that $\bk$ is a field.

\begin{prop}
\label{prop:loc-tD-tN-tilt}
Let $s$ be a simple reflection.
\begin{enumerate}
 \item
 \label{it:tilti-tD-loc}
   There exist chain maps
\[
\vartheta_s^1 \in \uHom^\rhd_\FM(\tD_\varnothing \langle -1 \rangle \oplus \tD_s, \Tmon_s) \quad \text{and} \quad
 \vartheta_s^2 \in \uHom_\FM^\rhd(\Tmon_s, \tD_\varnothing \langle 1 \rangle \oplus \tD_s \langle 2 \rangle)
 \]
such that
\[
\vartheta_s^2 \circ \vartheta_s^1 = 1 \otimes \id_{\tD_\varnothing \langle -1 \rangle \oplus \tD_s} \otimes \alpha_s^\vee \quad \text{and} \quad \vartheta_s^1 \circ \vartheta_s^2 = 1 \otimes \id_{\Tmon_s} \otimes \alpha_s^\vee
\]
in $\uEnd_\FM(\tD_\varnothing \langle -1 \rangle \oplus \tD_s)$ and $\uEnd_\FM(\Tmon_s)$, respectively.
  \item
  \label{it:tilti-tN-loc}
     There exist chain maps
 \[
 \varkappa_s^1 \in \uHom^\rhd_\FM(\tN_\varnothing \langle -1 \rangle \oplus \tN_s \langle -2 \rangle, \Tmon_s) \quad \text{and} \quad
 \varkappa_s^2 \in \uHom^\rhd_\FM(\Tmon_s, \tN_\varnothing \langle 1 \rangle \oplus \tN_s)
 \]
such that
\[
\varkappa_s^2 \circ \varkappa_s^1 = 1 \otimes \id_{\tN_\varnothing \langle -1 \rangle \oplus \tN_s \langle -2 \rangle} \otimes \alpha_s^\vee \quad \text{and} \quad \varkappa_s^1 \circ \varkappa_s^2 = 1 \otimes\id_{\Tmon_s} \otimes \alpha_s^\vee
\]
in $\uEnd_\FM(\tN_\varnothing \langle -1 \rangle \oplus \tN_s \langle -2 \rangle)$ and $\uEnd_\FM(\Tmon_s)$, respectively. 
\end{enumerate}
\end{prop}

\begin{proof}
\eqref{it:tilti-tD-loc}
 Consider the chain map $f_s \in \uHom_\FM^\rhd(\tD_s, \Tmon_s)$ depicted below:
   \[
   \begin{tikzcd}[column sep=huge]
\cE_\varnothing\{1\} \ar[loop, in=160, out=180, distance=30, "\theta - \alpha_s \otimes \id \otimes \alpha_s^\vee" pos=0.3] \ar[d, bend left=80, "\{2\}" description, "\ \ 1 \otimes \usebox\lowerdot \otimes \alpha_s^\vee" pos=0.4] \ar[rr, "\id"] && \cE_\varnothing\{1\} \ar[loop right, in=0, out=20, distance=40, "\theta - \alpha_s \otimes \id \otimes \alpha_s^\vee" pos=0.6] \ar[d, bend left=80, "\{2\}" description, "\ \ 1 \otimes \usebox\lowerdot \otimes \alpha_s^\vee" pos=0.4] \\
\cE_s \ar[u, "\usebox\upperdot"] \ar[loop, in=200, out=180, distance=40, "\theta - \alpha_s \otimes \id \otimes \alpha_s^\vee", pos=0.5, swap] \ar[rr, "\id"'] && \cE_s \ar[u, "\usebox\upperdot"'] \ar[loop right, in=-20, out=0, distance=50, "\theta - \alpha_s \otimes \id \otimes \alpha_s^\vee"] \\
&& \cE_\varnothing\{-1\}. \ar[u, "\usebox\lowerdot"'] \ar[uu, bend left=40, "\{-2\}" description, near start, "-\alpha_s \otimes \id \otimes 1" near end] \ar[loop right, in=-20, out=20, distance=35, "\theta"]
    \end{tikzcd}
 \]
Consider also the chain map $g_s \in \uHom_\FM^\rhd(\Tmon_s, \tD_s \langle 2 \rangle)$ shown here:
   
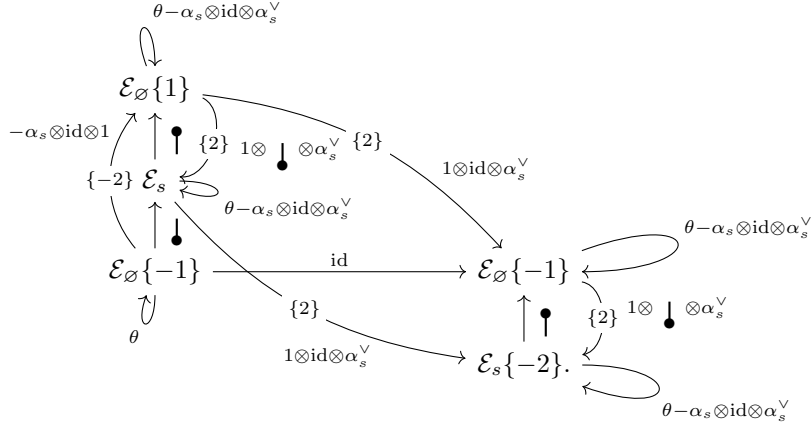
\[
   \begin{tikzcd}[column sep=huge]
   \cE_\varnothing\{1\} \ar[loop, in=90, out=110, distance=20, "\theta - \alpha_s \otimes \id \otimes \alpha_s^\vee" pos=0.6] \ar[d, bend left=80, "\{2\}" description, "\ \ 1 \otimes \usebox\lowerdot \otimes \alpha_s^\vee" pos=0.4] \ar[rrdd, in=120, out=0, bend left=20, "1 \otimes \id \otimes \alpha_s^\vee" near end, "\{2\}" description] && \\
   \cE_s \ar[u, "\usebox\upperdot"'] \ar[loop right, in=-20, out=0, distance=20, "\theta - \alpha_s \otimes \id \otimes \alpha_s^\vee"] \ar[rrdd, in=120, out=0, bend right=20, "1 \otimes \id \otimes \alpha_s^\vee"' near end, "\{2\}" description] \\
   \cE_\varnothing\{-1\} \ar[u, "\usebox\lowerdot"' near start] \ar[uu, bend left=40, "\{-2\}" description, "-\alpha_s \otimes \id \otimes 1" near end] \ar[loop right, in=-110, out=-90, distance=15, "\theta"] \ar[rr, "\id"] && \cE_\varnothing\{-1\} \ar[d, bend left=80, "\{2\}" description, "\ \ 1 \otimes \usebox\lowerdot \otimes \alpha_s^\vee" pos=0.5] \ar[loop right, in=0, out=20, distance=50, "\theta - \alpha_s \otimes \id \otimes \alpha_s^\vee"] \\
   && \cE_s\{-2\}. \ar[loop, in=-20, out=0, distance=40, "\theta - \alpha_s \otimes \id \otimes \alpha_s^\vee" pos=0.6] \ar[u, "\usebox\upperdot"']
   \end{tikzcd}
 \]
Finally, recall that we have defined a ``unit'' chain map $\hat{\eta}_s \in \uHom_\FM(\tD_\varnothing \langle -1 \rangle, \Tmon_s)$ and a ``counit'' chain map
$\hat{\epsilon}_s \in \uHom_\FM(\Tmon_s, \tD_\varnothing \langle 1 \rangle)$ in~\S\ref{sss:unit-freemon}.
It is easy to check that
 \[
 g_s \circ f_s = 1 \otimes \id_{\tD_s} \otimes \alpha_s^\vee, \quad \hat{\epsilon}_s \circ f_s = 0, \quad g_s \circ \hat{\eta}_s = 0, \quad \hat{\epsilon}_s \circ \hat{\eta}_s = 1 \otimes \id_{\tD_\varnothing} \otimes \alpha_s^\vee
 \]
 and that
 $
 (\hat{\eta}_s \oplus f_s) \circ (\hat{\epsilon}_s \oplus g_s) = 1 \otimes \id_{\Tmon_s} \otimes \alpha_s^\vee
 $.
Hence we can set $\vartheta_s^1 = \hat{\eta}_s \oplus f_s$ and $\vartheta_s^2 = \hat{\epsilon}_s \oplus g_s$.

 \eqref{it:tilti-tN-loc}
 The proof is similar, using the chain maps
    \[
   \begin{tikzcd}[column sep=huge]
   \cE_\varnothing\{1\} \ar[loop, in=90, out=110, distance=20, "\theta - \alpha_s \otimes \id \otimes \alpha_s^\vee" pos=0.6] \ar[d, bend left=80, "\{2\}" description, "\ \ 1 \otimes \usebox\lowerdot \otimes \alpha_s^\vee" pos=0.4] \ar[rrdd, in=120, out=0, "1 \otimes \id \otimes \alpha_s^\vee" near start, "\{2\}" description] && \\
   \cE_s \ar[u, "\usebox\upperdot"'] \ar[loop right, in=-20, out=0, distance=20, "\theta - \alpha_s \otimes \id \otimes \alpha_s^\vee"] \ar[rr, "\id" near end] && \cE_s \ar[loop, in=0, out=20, distance=40, "\theta - \alpha_s \otimes \id \otimes \alpha_s^\vee" pos=0.6] \ar[d, bend left=80, "\{2\}" description, "\ \ 1 \otimes \usebox\upperdot \otimes \alpha_s^\vee" pos=0.7] \\
   \cE_\varnothing\{-1\} \ar[u, "\usebox\lowerdot"'] \ar[uu, bend left=40, "\{-2\}" description, "-\alpha_s \otimes \id \otimes 1" near end] \ar[loop right, in=-110, out=-90, distance=15, "\theta"] \ar[rr, "\id"] && \cE_\varnothing\{-1\} \ar[u, "\usebox\lowerdot"'] \ar[loop right, in=-20, out=0, distance=50, "\theta - \alpha_s \otimes \id \otimes \alpha_s^\vee"]
   \end{tikzcd}
 \]
 in $\uHom_\FM^\rhd(\Tmon_s, \tN_s)$ and
 \[
    \begin{tikzcd}[column sep=huge]
    \cE_s \{2\} \ar[rrdd, in=120, out=0, "1 \otimes \id \otimes \alpha_s^\vee" near start, "\{2\}" description] \ar[d, bend left=80, "\{2\}" description, "\ \ 1 \otimes \usebox\upperdot \otimes \alpha_s^\vee" pos=0.5] \ar[loop, in=160, out=180, distance=40, "\theta - \alpha_s \otimes \id \otimes \alpha_s^\vee", pos=0.5] \\
\cE_\varnothing\{1\} \ar[u, "\usebox\lowerdot"] \ar[loop, in=200, out=180, distance=30, "\theta - \alpha_s \otimes \id \otimes \alpha_s^\vee", swap] \ar[rr, "\id"] && \cE_\varnothing\{1\} \ar[loop right, in=0, out=20, distance=40, "\theta - \alpha_s \otimes \id \otimes \alpha_s^\vee" pos=0.6] \ar[d, bend left=80, "\{2\}" description, "\ \ 1 \otimes \usebox\lowerdot \otimes \alpha_s^\vee" pos=0.4] \\
&& \cE_s \ar[u, "\usebox\upperdot"'] \ar[loop right, in=-20, out=0, distance=50, "\theta - \alpha_s \otimes \id \otimes \alpha_s^\vee"] \\
&& \cE_\varnothing\{-1\} \ar[u, "\usebox\lowerdot"'] \ar[uu, bend left=40, "\{-2\}" description, "-\alpha_s \otimes \id \otimes 1" near start] \ar[loop right, in=-20, out=20, distance=35, "\theta"]
    \end{tikzcd}
 \]
 in $\uHom_\FM^\rhd(\tN_s \langle -2 \rangle, \Tmon_s)$.
 \end{proof}
 
\begin{cor}
\label{cor:Tmon-tD-tN}
Let $\uw$ be an expression. 
\begin{enumerate}
\item
\label{it:Tmon-tD}
There exist $k \in \Z_{\geq 0}$, expressions $\uw_1, \sdots, \uw_k$, integers $n_1, \sdots, n_k$ and a morphism
\[
\bigoplus_{i=1}^k \tD_{\uw_i} \langle n_i \rangle \to \Tmon_{\uw}
\]
in $\Conv^\rhd(\UGU,\bk)$ whose image under $\Loc$ is an isomorphism.
\item
\label{it:Tmon-tN}
There exist $k \in \Z_{\geq 0}$, expressions $\uw_1', \sdots, \uw_k'$, integers $n_1', \sdots, n_k'$ and a morphism
\[
\Tmon_{\uw} \to \bigoplus_{i=1}^k \tN_{\uw_i'} \langle n_i' \rangle
\]
in $\Conv^\rhd(\UGU,\bk)$ whose image under $\Loc$ is an isomorphism.
\end{enumerate}
\end{cor}

\begin{proof}
We only prove~\eqref{it:Tmon-tD}; the proof of~\eqref{it:Tmon-tN} is similar.
We proceed by induction on the length of $\uw$. If this length is $0$ there is nothing to prove.

Assume now that the length of $\uw$ is positive. Let $s$ be the last simple reflection appearing in $\uw$, and let $\uv$ be the expression obtained from $\uw$ by omitting $s$. By induction there exist expressions $\uv_1, \sdots, \uv_k$, integers $n_1, \sdots, n_k$ and a morphism
\[
f : \bigoplus_{i=1}^k \tD_{\uv_i} \langle n_i \rangle \to \Tmon_{\uv}
\]
in $\Conv^\rhd(\UGU, \bk)$ such that $\Loc(f)$ is an isomorphism.
Using Lemma~\ref{lem:conv-Loc-isom-tD-tN},
we deduce that
\[
f \hatstar \id : \left( \bigoplus_{i=1}^k \tD_{\uv_i} \langle n_i \rangle \right) \hatstar (\tD_\varnothing \langle -1 \rangle \oplus \tD_s) \to \Tmon_{\uv} \hatstar (\tD_\varnothing \langle -1 \rangle \oplus \tD_s)
\]
also has the property that $\Loc(f \hatstar \id)$ is an isomorphism.

Now by Lemma~\ref{lem:conv-Loc-isom} and Proposition~\ref{prop:loc-tD-tN-tilt}\eqref{it:tilti-tD-loc},
\[
\id \hatstar \vartheta_s^1 : \Tmon_{\uv} \hatstar (\tD_\varnothing \langle -1 \rangle \oplus \tD_s) \to \Tmon_\uv \hatstar \Tmon_s = \Tmon_{\uw}
\]
has the property that $\Loc(\id \hatstar \vartheta_s^1)$ is an isomorphism. Hence $(\id \hatstar \vartheta_s^1) \circ (f \hatstar \id)$ has the desired properties.
\end{proof}

\section{Convolution of standard or costandard objects}

We continue to assume that $\bk$ is a field.

\begin{lem}
\label{lem:tD-tN-ss}
 Let $s$ be a simple reflection. 
 \begin{enumerate}
  \item 
  \label{it:tD-ss}
 There exist morphisms
 \[
  \tD_\varnothing \to \tD_s \hatstar \tD_s \langle 2 \rangle \quad \text{and} \quad \tD_s \hatstar \tD_s \to \tD_\varnothing
 \]
 in $\Conv^\rhd(\UGU,\bk)$ whose images under $\Loc$ are isomorphisms.
 \item
   \label{it:tN-ss}
 There exist morphisms
 \[
  \tN_\varnothing \to \tN_s \hatstar \tN_s \quad \text{and} \quad \tN_s \hatstar \tN_s \to \tN_\varnothing \langle 2 \rangle
 \]
 in $\Conv^\rhd(\UGU,\bk)$ whose images under $\Loc$ are isomorphisms.
 \end{enumerate}
 \end{lem}
 
\begin{proof}
We only prove \eqref{it:tD-ss}; the proof of~\eqref{it:tN-ss} is similar.

\newsavebox\lowerdotline
\savebox\lowerdotline{%
\begin{tikzpicture}[scale=0.3,thick,baseline]
 \draw (0,-0.5) to (0,0.5);
 \draw (0.5,-1) to (0.5,0.5);
 \node at (0,-0.5) {$\bullet$};
\end{tikzpicture}%
}

\newsavebox\upperdotline
\savebox\upperdotline{%
\begin{tikzpicture}[scale=0.3,thick,baseline]
 \draw (0,-0.5) to (0,0.5);
 \draw (0.5,-0.5) to (0.5,1);
 \node at (0,0.5) {$\bullet$};
\end{tikzpicture}%
}

\newsavebox\linelowerdot
\savebox\linelowerdot{%
\begin{tikzpicture}[scale=0.3,thick,baseline]
 \draw (0,-0.5) to (0,0.5);
 \draw (-0.5,-1) to (-0.5,0.5);
 \node at (0,-0.5) {$\bullet$};
\end{tikzpicture}%
}

\newsavebox\lineupperdot
\savebox\lineupperdot{%
\begin{tikzpicture}[scale=0.3,thick,baseline]
 \draw (0,-0.5) to (0,0.5);
 \draw (-0.5,-0.5) to (-0.5,1);
 \node at (0,0.5) {$\bullet$};
\end{tikzpicture}%
}

 The free-monodromic complex $\tD_s \hatstar \tD_s$ can be depicted as follows:
 \[
   \begin{tikzcd}
   & \cE_\varnothing \{2\} \ar[loop, in=90, out=110, distance=20, "\theta" pos=0.6] \ar[ld, "-\usebox\lowerdot \otimes \alpha_s^\vee"] \ar[rd, bend left=20, "\usebox\lowerdot \otimes \alpha_s^\vee"] & \\
   \cE_s \star \cE_\varnothing \{1\} \ar[loop, in=180, out=200, distance=20, "\theta" pos=0.6] \ar[rd, "-\usebox\linelowerdot \otimes \alpha_s^\vee" near end] \ar[ru, bend left=20, "\usebox\upperdot"] & \oplus & \cE_\varnothing\{1\} \star \cE_s \ar[ul, "\usebox\upperdot"] \ar[loop, in=0, out=20, distance=20, "\theta" pos=0.6] \ar[ld, bend left=20, "-\usebox\lowerdotline \otimes \alpha_s^\vee"] \\
   & \cE_s \star \cE_s \ar[loop, in=-110, out=-90, distance=20, "\theta" pos=0.6] \ar[ru, "\usebox\upperdotline"] \ar[lu, bend left=20, "-\usebox\lineupperdot"]
   \end{tikzcd}
 \]
 We consider the element $x_s \in \uHom_\FM(\tD_s \hatstar \tD_s, \tD_1)^0_0$ defined by
 \[
   \begin{tikzcd}[column sep=huge]
    \cE_\varnothing\{2\} \ar[rrdd, bend left=20, "1 \otimes \id \otimes \alpha_s^\vee" near start, "\{2\}" description]&& \\
    \cE_s \{1\} \oplus \cE_s\{1\} && \\
    \cE_s \star \cE_s \ar[rr, "\usebox\capmor"] && \cE_\varnothing
   \end{tikzcd}
 \]
and the element $y_s \in \uHom_\FM(\tD_1, \tD_s \hatstar \tD_s)^0_{-2}$ defined by
 \[
   \begin{tikzcd}[column sep=huge]
    \cE_\varnothing\{2\} \ar[rrdd, bend right=20, "-1 \otimes \usebox\cupmor \otimes \alpha_s^\vee" near start, swap, "\{2\}" description] \ar[rr, "\id"] && \cE_\varnothing\{2\} \\
    && \cE_s \{1\} \oplus \cE_s\{1\} && \\
     && \cE_s \star \cE_s.
   \end{tikzcd}
 \]
 Then $x_s$ and $y_s$ are chain maps from $\tD_s \hatstar \tD_s$ to $\tD_1$ and from $\tD_1$ to $\tD_s \hatstar \tD_s \langle 2 \rangle$ respectively, and we have $x_s \circ y_s = 1 \otimes \id_{\tD_1} \otimes \alpha_s^\vee$. One can also check (using~\eqref{eqn:formula-BsBs}) that $y_s \circ x_s = 1 \otimes \id_{\tD_s \hatstar \tD_s} \otimes \alpha_s^\vee + d(h)$, where $h \in \uEnd_\FM(\tD_s \hatstar \tD_s)^{-1}_{-2}$ is defined by
  \[
   \begin{tikzcd}[column sep=tiny]
   & \cE_\varnothing \{2\} & &&&& & \cE_\varnothing \{2\} & \\
   \cE_s \star \cE_\varnothing \{1\} & \oplus & \cE_\varnothing\{1\} \star \cE_s \ar[rrrrrd, bend right=20, "-1 \otimes \usebox\ymor \otimes \alpha_s^\vee"] &&&& \cE_s \star \cE_\varnothing \{1\} & \oplus & \cE_\varnothing\{1\} \star \cE_s. \\
   & \cE_s \star \cE_s \ar[rrrrru, bend right=20, "1 \otimes \usebox\invymor \otimes 1"] & &&&& & \cE_s \star \cE_s &
   \end{tikzcd}
 \]
 Therefore, $x_s$ and $y_s$ induce the desired morphisms.
\end{proof}

\begin{cor}
\label{cor:Hom-tD-tN-loc}
If $\uv$ and $\uw$ are expressions, we have
  \[
   \gHom_{\loc}(\Loc(\tD_{\uv}), \Loc(\tN_{\uw})) \cong \begin{cases}
                                                Q^\vee & \text{if $\pi(\uv)=\pi(\uw)$;} \\
                                                0 & \text{otherwise.}
                                               \end{cases}
  \]
\end{cor}

\begin{proof}
Using
Proposition~\ref{prop:tD-tN}\eqref{it:tD-tN-independence} and Lemmas~\ref{lem:conv-Loc-isom-tD-tN} and~\ref{lem:tD-tN-ss}, we see that up to isomorphism and Tate twist, $\Loc(\tD_{\uu{u}})$ and $\Loc(\tN_{\uu{u}})$ only depend on $\pi(\uu{u})$. Hence we can assume that $\uv$ and $\uw$ are reduced. In this case, the claim follows from Proposition~\ref{prop:tD-tN}\eqref{it:tD-tN-Hom} and~\eqref{eqn:Hom-mon-loc}.
\end{proof}

\section{Proof of the main theorem}

We come back to the general assumptions of~\S\ref{sec:Cartan-realizations}.  In particular, $\bk$ is an integral domain that admits a ring homomorphism $\Zdem \to \bk$.  In addition, we now assume that $\bk$ is Noetherian and of finite global dimension.  

In this section, we will prove that $\hatstar$ equips
the category $\TiltBSp(\UGU,\bk)$ with the structure of a monoidal category.
As explained in~\S\ref{sec:functoriality-conjecture}, this can be reduced to the ``interchange law'' of~\eqref{eqn:conj-interchange}.  We begin by proving that a very special case of the interchange law holds after localization, under the assumptions of Chapter~\ref{chap:dihedral}. 

\begin{lem}
\label{lem:functoriality-tD-tN}
Assume that, for any pair $(s,t)$ of distinct simple reflections in $S$ generating a finite subgroup of $W$, the assumptions of Chapter~\ref{chap:dihedral} hold.
Let $\uv, \uw$ be expressions,
let $i \in \Z$, and let $k : \tD_\uv \to \tN_{\uw} \langle i \rangle$ be a morphism in the category $\Dmix(\UGU,\bk)$. Then for any $\cF, \cG$ in $\Conv(\UGU,\bk)$ and any morphism $f : \cF \to \cG$, we have
 \begin{equation}
 \label{eqn:functoriality-tD-tN}
 \Loc((\id_\cG \hatstar k) \circ (f \hatstar \id_{\tD_\uv})) = \Loc(f \hatstar k).
 \end{equation}
\end{lem}

\begin{proof}
First, assume that $\pi(\uv) \neq \pi(\uw)$. Then $\Loc(k)=0$ by Corollary~\ref{cor:Hom-tD-tN-loc}, so $\Loc(\id_\cF \hatstar k)=0$ and $\Loc(\id_\cG \hatstar k)=0$ by Lemma~\ref{lem:conv-Loc-isom}. Since $f \hatstar k=(f \hatstar \id_{\tN_\uw \langle i \rangle}) \circ (\id_\cF \hatstar k)$ by Lemma~\ref{lem:interchange-rhd-lhd}, the claim is then clear.

Assume now that
$\pi(\uv)=\pi(\uw)$. Since the vertical arrows in~\eqref{eqn:loc-karoubian} are fully faithful, it is enough to prove~\eqref{eqn:functoriality-tD-tN} after passage to $\Dmix_\loc(\UGU,\bk)_\Kar$.  We will work in $\Dmix(\UGU,\bk)_\Kar$ or $\Dmix_\loc(\UGU,\bk)_\Kar$ for the remainder of the proof.  (We must do this in order for the morphism $g''$ defined in~\eqref{eqn:gpp-defn} below to make sense.)

Choose a reduced expression $\uu{u} \in \hW$ for $\pi(\uv) = \pi(\uw)$. There exist sequences $(\uv_0, \sdots, \uv_n)$ and $(\uw_0, \sdots, \uw_m)$ of expressions such that
\[
\uw_0=\uv_0=\uu{u}, \quad \uv_n = \uv, \quad \uw_m = \uw,
\]
and such that:
\begin{itemize}
\item
for any $i \in \{0, \sdots, n-1\}$, $\uv_{i+1}$ is obtained from $\uv_i$ by either inserting $(s,s)$ between two simple reflections or at an end (for some $s \in S$), or replacing a subsequence $(s,t,\sdots)$ by $(t,s,\sdots)$ (where $s,t$ are distinct simple reflections such that $st$ has finite order $m_{st}$, and each sequence has $m_{st}$ terms);
\item
for any $j \in \{0, \sdots, m-1\}$, $\uw_{j+1}$ is obtained from $\uw_j$ by either inserting $(s,s)$ between two simple reflections or at an end (for some $s \in S$) or replacing a subsequence $(s,t,\sdots)$ by $(t,s,\sdots)$ (where $s,t$ are distinct simple reflections such that $st$ has finite order $m_{st}$, and each sequence has $m_{st}$ terms).
\end{itemize}
We argue by induction on the minimal possible value of $n+m$.

If $n+m=0$, then $\uv=\uw$ is a reduced expression. In this case the $R^\vee$-module $\gHom_\FM(\tD_{\uv}, \tN_\uw)$ is free of rank $1$ by Proposition~\ref{prop:tD-tN}\eqref{it:tD-tN-Hom}, and is generated by a morphism which belongs to $\Conv^\rhd(\UGU,\bk)_\Kar$ by Remark~\ref{rmk:generator-Hom-tD-tN}. Hence the desired equality follows from Lemma~\ref{lem:interchange-rhd-lhd}.

Now assume that $n+m >0$, and that $n>0$. We set $\uv':=\uv_{n-1}$. If $\uv$ is obtained from $\uv'$ by inserting $(s,s)$ between two simple reflections or at an end, we consider the morphism $k' : \tD_{\uv'} \to \tD_{\uv} \langle 2 \rangle$ 
obtained by convolving the first morphism in Lemma~\ref{lem:tD-tN-ss}\eqref{it:tD-ss} with the appropriate identity morphisms. Then $\Loc(k')$ is invertible by Lemma~\ref{lem:conv-Loc-isom} and Lemma~\ref{lem:conv-Loc-isom-tD-tN}. Since $k'$ belongs to $\Conv^\rhd(\UGU,\bk)_\Kar$, by Lemma~\ref{lem:interchange-rhd-lhd} we have
\[
(\id \hatstar k) \circ (f \hatstar \id) \circ (\id \hatstar k') = (\id \hatstar k) \circ (f \hatstar k') = (\id \hatstar kk') \circ (f \hatstar \id).
\]
Applying
induction to $k \circ k'$, we deduce that
\[
\Loc((\id \hatstar k) \circ (f \hatstar \id)) \circ \Loc(\id \hatstar k') = \Loc(f \hatstar (kk')) = \Loc(f \hatstar k) \circ \Loc(\id \hatstar k').
\]
(In the last equality, we again use Lemma~\ref{lem:interchange-rhd-lhd}.)
Since $\Loc(\id \hatstar k')$ is invertible by Lemma~\ref{lem:conv-Loc-isom}, we deduce that
\[
\Loc((\id \hatstar k) \circ (f \hatstar \id)) = \Loc(f \hatstar k)
\]
in this case, as desired.

If now $\uv$ is obtained from $\uv'$ by replacing a subsequence $(s,t,\sdots)$ by $(t,s,\sdots)$, let us choose some isomorphism $g' : \tD_{(s,t,\sdots)} \simto \tD_{(t,s,\sdots)}$ in $\Dmix(\UGU,\bk)$, see Proposition~\ref{prop:tD-tN}\eqref{it:tD-tN-independence}. Consider also a morphism
\begin{equation}
\label{eqn:gpp-defn}
g'' : \tD_{(t,s,\sdots)} \simto \tD_{(t,s,\sdots), \min}
\end{equation}
as in Theorem~\ref{thm:Rouquier-convolution-fm-new}. Then, by the proof of Proposition~\ref{prop:tD-tN}\eqref{it:tD-tN-independence}, $g''g'$ is a multiple of the composition
\[
\tD_{(s,t,\sdots)} \simto \tD_{(s,t,\sdots), \min} \simto \tD_{(t,s,\sdots), \min},
\]
where the first morphism is again as in Theorem~\ref{thm:Rouquier-convolution-fm-new} and the second one is the obvious isomorphism (given by the identity of the underlying parity sequence if $m_{st}$ is odd, and multiplication by $-1$ on $B_{st\cdots}$ and identity elsewhere if $m_{st}$ is even). Hence $g''g'$ belongs to $\Conv^\rhd(\UGU,\bk)_\Kar$.

Now, let $\tD'_\uv$ be the free-monodromic complex obtained by replacing $\tD_{(t,s,\sdots)}$ by $\tD_{(t,s,\sdots), \min}$ in $\tD_\uv$. Let $k' : \tD_{\uv'} \to \tD_\uv$ and $k'' : \tD_\uv \to \tD'_{\uv}$ be the morphisms induced by $g'$ and $g''$ respectively. Then $k''k'$ and $k''$ belong to $\Conv^\rhd(\UGU,\bk)_\Kar$, and $k'$ and $k''$ are isomorphisms in $\Dmix(\UGU,\bk)_\Kar$. Applying induction to $kk'$, we see that
\begin{multline*}
 \Loc \bigl( (\id \hatstar k)(f \hatstar \id)(\id \hatstar k') \bigr) =  \Loc \bigl( (\id \hatstar k(k'')^{-1})(\id \hatstar k'')(f \hatstar \id)(\id \hatstar k') \bigr) \\
=  \Loc \bigl( (\id \hatstar k(k'')^{-1})(f \hatstar \id)(\id \hatstar k''k') \bigr) = \Loc \bigl( (\id \hatstar k(k'')^{-1}) (\id \hatstar k''k') (f \hatstar \id) \bigr) \\
=  \Loc \bigl( (\id \hatstar kk')(f \hatstar \id) \bigr) =  \Loc \bigl( (f \hatstar \id)(\id \hatstar kk') \bigr) \\
=  \Loc \bigl( (f \hatstar \id)(\id \hatstar k)(\id \hatstar k') \bigr).
\end{multline*}
It follows that $(\id \hatstar k)(f \hatstar \id)=(f \hatstar \id)(\id \hatstar k)$ in this case as well.

Finally the case $n=0$, $m>0$ can be treated similarly, using Lemma~\ref{lem:tD-tN-ss}\eqref{it:tN-ss} instead of Lemma~\ref{lem:tD-tN-ss}\eqref{it:tD-ss}. (In these arguments one also needs $\tN$-versions of minimal Rouquier complexes, which are simply obtained from the $\tD$-versions by reflecting all diagrams along an horizontal axis.)
Details are left to the reader.
\end{proof}


\begin{thm}
\label{thm:functoriality-hatstar}
Let $\GKM$, $\BKM$, and $\UKM$ be as in~{\rm\S\ref{sec:parity-complexes}}, and let $\bk$ be a Noetherian integral domain of finite global dimension that admits a ring homomorphism $\Zdem \to \bk$.
 \begin{enumerate}
  \item 
  \label{it:functoriality-UGU}
  The operations $(\cF, \cG) \mapsto \cF \hatstar \cG$ and $(f,g) \mapsto f \hatstar g$ define a functor
  \[
   \TiltBSp(\UGU,\bk) \times \TiltBSp(\UGU,\bk) \to \TiltBSp(\UGU,\bk),
  \]
  and hence equip $\TiltBSp(\UGU,\bk)$ with the structure of a monoidal category.
  
  \item
  \label{it:functoriality-UGB}
  The operations $(\cF, \cG) \mapsto \cF \hatstar \cG$ and $(f,g) \mapsto f \hatstar g$ define a functor
  \[
   \TiltBSp(\UGU,\bk) \times \TiltBSp(\UGB,\bk) \to \TiltBSp(\UGB,\bk),
  \]
  making the category $\TiltBSp(\UGB,\bk)$ into a module for the monoidal category $\TiltBSp(\UGU,\bk)$.
 \end{enumerate}
\end{thm}

\begin{proof}
\eqref{it:functoriality-UGU}
As recalled above, in view of Proposition~\ref{prop:hatstar-coherence}, it is enough to prove the interchange law~\eqref{eqn:conj-interchange}.  Arguing as in the proof of Lem\-ma~\ref{lem:monconv-weak-interchange}, we only have to prove that for any morphisms $f : \Tmon_{\uv} \to \Tmon_{\uv'} \langle n \rangle$ and $k : \Tmon_{\uw} \to \Tmon_{\uw'} \langle m \rangle$ we have
\begin{equation}
\label{eqn:functoriality}
 (\id \hatstar k) \circ (f \hatstar \id) = f \hatstar k
\end{equation}
in $\Hom_{\Dmix(\UGU,\bk)}(\Tmon_{\uv\uw}, \Tmon_{\uv'\uw'} \langle n+m \rangle)$. Since the functors $\bk$ and $\Q$ induce isomorphisms
\[
 \bk \otimes_\Zdem \gHom_{\Dmix(\UGU, \Zdem)}(\Tmon^\Zdem_{\uu{u}}, \Tmon_{\uu{u}'}^\Zdem) \simto \gHom_{\Dmix(\UGU, \bk)}(\Tmon^{\bk}_{\uu{u}}, \Tmon_{\uu{u}'}^{\bk})
\]
and
\[
 \Q \otimes_\Zdem \gHom_{\Dmix(\UGU, \Zdem)}(\Tmon^\Zdem_{\uu{u}}, \Tmon_{\uu{u}'}^\Zdem) \simto \gHom_{\Dmix(\UGU, \Q)}(\Tmon^{\Q}_{\uu{u}}, \Tmon_{\uu{u}'}^{\Q})
\]
respectively by Corollary~\ref{cor:Hom-Tilt-mon}, it suffices to treat the case $\bk=\Q$. 
In this case the assumptions of Chapter~\ref{chap:dihedral} are satisfied for any pair of simple reflections generating a finite subgroup of $W$; see Lemma~\ref{lem:dihedral-faithfulness}.
This is the condition we will need for our arguments below, and we henceforth assume it is satisfied from now on.

Since the functor $\Loc$ is faithful on $\TiltBSp(\UGU, \bk)$ by Lemma~\ref{lem:Loc-Tilt}, to prove~\eqref{eqn:functoriality} it suffices to prove that
\[
\Loc((\id \hatstar k) \circ (f \hatstar \id)) = \Loc(f \hatstar k).
\]

By Corollary~\ref{cor:Tmon-tD-tN}, there exist expressions $\uu{u}_1, \sdots, \uu{u}_i$ and $\uu{u}'_1, \sdots, \uu{u}'_j$, integers $n_1, \sdots, n_i$ and $n'_1, \sdots, n'_j$, and morphisms
\[
k' : \bigoplus_{p=1}^i \tD_{\uu{u}_p} \langle n_p \rangle \to \Tmon_{\uw}, \quad k'' : \Tmon_{\uw'} \langle m \rangle \to \bigoplus_{q=1}^j \tN_{\uu{u}'_q} \langle n'_q \rangle
\]
in $\Conv^\rhd(\UGU,\bk)$ whose images under $\Loc$ are isomorphisms. By Lem\-ma~\ref{lem:functoriality-tD-tN}, we have
\[
\Loc((\id \hatstar (k''kk')) \circ (f \hatstar \id)) = \Loc(f \hatstar (k''kk')).
\]
Using Lemma~\ref{lem:interchange-rhd-lhd}, this implies that we have
\begin{multline*}
\Loc(\id \hatstar k'') \circ \Loc(\id \hatstar k) \circ \Loc(f \hatstar \id) \circ \Loc(\id \hatstar k') \\
= \Loc(\id \hatstar k'') \circ \Loc(f \hatstar k) \circ \Loc(\id \hatstar k').
\end{multline*}
Since $\Loc(\id \hatstar k'')$ and $\Loc(\id \hatstar k')$ are isomorphisms by Lemma~\ref{lem:conv-Loc-isom}, we deduce that $\Loc((\id \hatstar k) \circ (f \hatstar \id)) = \Loc(f \hatstar k)$, as desired.

\eqref{it:functoriality-UGB}
Similarly, what we have to prove is that for any morphisms $f : \Tmon_{\uv} \to \Tmon_{\uv'} \langle n \rangle$ and $k : \cT_{\uw} \to \cT_{\uw'} \langle m \rangle$ we have
\[
 (\id \hatstar k) \circ (f \hatstar \id) = f \hatstar k
\]
in $\Hom_{\Dmix(\UGB,\bk)}(\cT_{\uv\uw}, \cT_{\uv'\uw'} \langle n+m \rangle)$. Since the restriction of the functor $\ForFMLM$ to $\TiltBSp(\UGU,\bk)$ is full by Corollary~\ref{cor:Hom-Tilt-mon}, this equality follows from~\eqref{it:functoriality-UGU}.
\end{proof}


\backmatter



\printindex

\end{document}